\documentclass{amsart}
\usepackage{a4wide}

\usepackage{pgf,tikz,pgfplots}
\pgfplotsset{compat=1.15}
\usepackage{mathrsfs}
\usetikzlibrary{arrows}

\usepackage[T1]{fontenc}
\usepackage{microtype}
\usepackage{lmodern}
\usepackage[colorlinks=true,urlcolor=blue, citecolor=red,linkcolor=blue,linktocpage,pdfpagelabels, bookmarksnumbered,bookmarksopen]{hyperref}
\usepackage[hyperpageref]{backref}
\usepackage{amsthm} %for citing inside theorem header
\usepackage{latexsym,amsmath,amssymb}

\usepackage{bbm}

\usepackage{accents}
\usepackage{esint}
 \usepackage[colorinlistoftodos,prependcaption,textsize=small]{todonotes}

\usepackage{soul}
\usepackage{mathtools} %For \chi with a much lower index (\mychi)
\usepackage{xparse} %For a new definiton of \chi with a slightly lower index
\usepackage[capitalize]{cleveref}

\usepackage[shortlabels]{enumitem}

 \usepackage[textsize=small]{todonotes}
\title{Global existence of the $s$-Schr\"odinger maps in $n \geq 5$.}
\newcommand{\tell}{\mathfrak{l}}
\newcommand{\tn}{\mathfrak{n}}

\newcommand{\N}{{\mathbb N}}

\renewcommand{\S}{{\mathbb S}}

\newtheorem{theorem}{Theorem}
\newtheorem{lemma}[theorem]{Lemma}
\newtheorem{corollary}[theorem]{Corollary}
\newtheorem{proposition}[theorem]{Proposition}

\theoremstyle{definition}
\newtheorem{definition}[theorem]{Definition}

\theoremstyle{remark}
\newtheorem{remark}[theorem]{Remark}

\newcommand\supp{{\rm supp\,}}

\newcommand{\ind}{{\mathbbm{1}} }

\newcommand{\R}{\mathbb{R}}
\newcommand{\C}{\mathbb{C}}

\newcommand{\Z}{\mathbb{Z}}

\newcommand{\brac}[1]{\left (#1 \right )}

\newcommand{\abs}[1]{\left\lvert #1 \right \rvert}

\renewcommand{\i}{{\rm \bf i}}

\newcommand{\barint}{
\rule[.036in]{.12in}{.009in}\kern-.16in \displaystyle\int }

\newcommand{\barcal}{\text{$ \rule[.036in]{.11in}{.007in}\kern-.128in\int $}}

\def\mvint_#1{\mathchoice
          {\mathop{\vrule width 6pt height 3 pt depth -2.5pt
                  \kern -8pt \intop}\nolimits_{\kern -3pt #1}}%
          {\mathop{\vrule width 5pt height 3 pt depth -2.6pt
                  \kern -6pt \intop}\nolimits_{#1}}%
          {\mathop{\vrule width 5pt height 3 pt depth -2.6pt
                  \kern -6pt \intop}\nolimits_{#1}}%
          {\mathop{\vrule width 5pt height 3 pt depth -2.6pt
                  \kern -6pt \intop}\nolimits_{#1}}}

\numberwithin{theorem}{section} \numberwithin{equation}{section}

\newcommand{\aleq}{\lesssim}
\newcommand{\ageq}{\gtrsim}
\newcommand{\aeq}{\approx}

\usepackage{scalerel}[2014/03/10]
\usepackage[usestackEOL]{stackengine}
\def\avint{\,\ThisStyle{\ensurestackMath{%
			\stackinset{c}{.2\LMpt}{c}{.5\LMpt}{\SavedStyle-}{\SavedStyle\phantom{\int}}}%
		\setbox0=\hbox{$\SavedStyle\int\,$}\kern-\wd0}\int}
\let\latexchi\chi
\makeatletter
\renewcommand\chi{\@ifnextchar_\sub@chi\latexchi}
\newcommand{\sub@chi}[2]{% #1 is _, #2 is the subscript
  \@ifnextchar^{\subsup@chi{#2}}{\latexchi^{}_{#2}}%
}
\newcommand{\subsup@chi}[3]{% #1 is the subscript, #2 is ^, #3 is the superscript
  \latexchi_{#1}^{#3}%
}
\makeatother

\author{Ahmed Dughayshim}
\email[Ahmed Dughayshim]{aha80@pitt.edu}

\author{Armin Schikorra}
\email[Armin Schikorra]{armin@pitt.edu}
\address{Department of Mathematics,
University of Pittsburgh,
301 Thackeray Hall,
Pittsburgh, PA 15260, USA}

\begin{document}

\begin{abstract}
We study the $s$-Schr\"odinger map equation
\[
\begin{cases}
\partial_t u = -u \wedge (-\Delta)^s u
    \quad \text{on } \mathbb{R}^n \times \mathbb{R}\\
u: \mathbb{R}^n \times \mathbb{R} \to \mathbb{S}^2\\
u(x,0) = u_0.
\end{cases}
\]
For $s \in (1/2,1)$ and $n \geq 5$ we prove that for small initial data
$u_0$ in the critical Besov space one can obtain a global solution.
\end{abstract}

\maketitle
\tableofcontents
\section{Introduction}
We are interested in studying the following initial value problem $u: \R^n \times \R \to \S^2$
\begin{equation}\label{SSchrodingereq}
\begin{cases}
 \partial_{t} u = - u \wedge (-\Delta)^s u \mbox{    } \mbox{ on } \R^{n} \times \R \\
 u(x,0) = u_{0},
\end{cases}
\end{equation}
here $ \wedge $ stands for the usual cross product in $ \R^{3}$.
When $s=1$ this equation is known as the Schr\"odinger map equation, which arises as the undamped part of the Landau--Lifshitz model describing the evolution of magnetization in ferromagnetic materials; see \cite{SurveySch,ChangShatahUhlenbeck2000,RodnianskiRubinsteinStaffilani2009,Bejenaru2026NewResolutionSpace} and references within. Its well-posedness in the critical spaces was obtained in a series of works \cite{cmp,DIE,Null,AnnalesKI,BejenaruIonescuKenig2007}.

For $s=\frac{1}{2}$ the equation is known as the Halfwave map equation, derived in the physics literature as the continuum limit of Calogero-Moser spin systems, and mathematically investigated more recently in \cite{LenzSchik,halfwav}, with well-posedness discussed in \cite{marsden2024,Uniq,Liu_2023,WaveSystemsArminSilvino,halfn1}, see also \cite{lenzsurvey,GerardLenzmann2018Lax,BerntsonKlabbersLangmann2020}.

% % where $ s \in (1/2,1)$ , $ u : \R^{n+1} \to \S^{2}$, $ \wedge$ stands for the cross product in $ \R^{3}$,  and $ \mathcal{F}((-\Delta)^{s} u ) := \vert \xi \vert^{2s} \hat{u}$. At the end point $ s =1$ the equation \eqref{SSchrodingereq} is the so called Schr\"odinger map equation which was studied in \cite{cmp}, \cite{DIE}, \cite{Null}, \cite{AnnalesKI}, \cite{BejenaruIonescuKenig2007}. Namely, global well-posedness for small initial data $u_{0}$ in the critical Sobolev space $ \dot{H}^{n/2}$ was proven in \cite{AnnalesKI} for all $n \geq 2$. Moreover, for $n =1 $, $ s=1$ well-posedness results were obtained in \cite{ChangShatahUhlenbeck2000}, \cite{RodnianskiRubinsteinStaffilani2009}. For more references on Schr\"odinger maps see \cite{SurveySch},\cite{Bejenaru2026NewResolutionSpace}.\\

% On the other hand, for the end point where $ s=1/2$ the equation \eqref{SSchrodingereq} is the so called halfwave map equation which was studied in \cite{LenzSchik}, \cite{halfwav},\cite{halfwav2}. In particular, see \cite{halfwav} fo well posedness results in high dimension. As for the case $ s= 1/2$, $n=1$ see \cite{halfn1}. The case for which $ n =2,3$, and $ s= 1/2$ seems to be still open. For further references on the halfwave map see \cite{marsden2024}, \cite{Uniq}, \cite{lenzsurvey}, \cite{Liu_2023}.\\

In this article we are interested in initiating the study of well-posednes of \eqref{SSchrodingereq} for the in-between regimes, i.e. for $ s \in (1/2,1)$.  
Our main result concerns global existence of \eqref{SSchrodingereq} with small initial data in the critical Besov norm.

\begin{theorem}\label{mainthm}
Let $ n \geq 5$, $ s \in ( 1/2,1)$,  $ \sigma_{0} \geq \frac{n}{2}+10$, and $Q \in \S^{2}$. Then there exists $ \epsilon := \epsilon(n,s,\sigma_{0}) > 0$ so that for any $ u_{0} : \R^{n} \to \S^{2}$  satisfying $ u_{0} \in B^{\sigma_{0}+1}_{Q}$ and 
$$
\Vert u_{0} \Vert_{\dot{B}_{Q}^{\frac{n}{2}}} \leq \epsilon
$$
then there exists a solution of \eqref{SSchrodingereq}  $u : \R^{n+1} \to \S^{2} $ which satisfies 
$$
u \in C ( \R, \dot{B}^{\sigma_{0}}_{Q}) 
$$ and 
$$
\sup_{t \in \R}  \Vert u \Vert_{\dot{B}_{Q}^{\sigma}} \lesssim \Vert u_{0} \Vert_{\dot{B}_{Q}^{\sigma}},
$$
for any $ \sigma \in [ \frac{n}{2}, \sigma_{0}]$.
\end{theorem}

Here we use the following notation; for $ Q \in \S^{2}$ and $ \sigma >0 $ we define $\dot{B}^{\sigma}_{Q}$ as follows
$$
\dot{B}^{\sigma}_{Q} : = \{ u : \R^{n} \to \S^{2} : (u-Q)_{l} \in \dot{B}^{\sigma}_{2,1} \text{  ,} \text{ for } l=1,2,3 \},
$$
where $ \dot{B}^{\sigma}_{2,1} $ is the homogeneous Besov space. We define the norm 
$$
\Vert u \Vert_{\dot{B}^{\sigma}_{Q}} := \sum_{l=1}^{3} \Vert (u - Q)_{l} \Vert_{\dot{B}^{\sigma}_{2,1}}.
$$
Similarly, 
$$
{B}^{\sigma}_{Q} : = \{ u : \R^{n} \to \S^{2} : (u-Q)_{l} \in {B}^{\sigma}_{2,1} \text{  ,} \text{ for } l=1,2,3 \},
$$
where 
$
{B}^{\sigma}_{2,1}
$
is the inhomogeneous Besov space. And we define the norm
$$
\Vert u \Vert_{{B}^{\sigma}_{Q}} := \sum_{l=1}^{3} \Vert (u - Q)_{l} \Vert_{{B}^{\sigma}_{2,1}}.
$$
\\

There has not been a lot of research on \eqref{SSchrodingereq} in the regime $s \in (1/2,1)$. In  \cite{subcritical} one of the authors studied well-posedness in the subcritical regime, more precisely for small initial data in $B^{\frac{n+1}{2}}_{Q}$. In \cite{model}, well-posedness for a ``model'' equation for \eqref{SSchrodingereq}, namely for
$$
(\i \partial_{t} - ( -\Delta)^{s} ) f = (|\nabla|^{1-2s} |f|^{2}) |\nabla|^{2s-1} f
$$
was established. 

One of the  obstacles in the well-posedness analysis is that if one uses the stereographic projection, as was done in \cite{subcritical}, and study the resulting scalar equation, then the nonlinearity does not have a suitable null-structure(one can write the nonlinearity formally in something reminiscent of a null structure, but there is a critical loss of derivatives to make use of it, see appendix \ref{null}). For related discussions regarding the null structure for Schr\"odinger maps, see \cite{Null,Bejenaru2026NewResolutionSpace,Bejenaru2008GlobalResults}. Instead we will use the stereographic projection (this is one place where we use the Besov-space assumption, instead of the slightly weaker Sobolev-space assumption), combined with a Coulomb gauge $e_i: \R^n \to \S^2$, $i=1,2$, $e_1 \perp e_2 \perp u \perp e_i$. The interesting point is that since $s \in (1/2,1)$ we can use the classical, \emph{local}, Coulomb gauge -- i.e. the same used for $s=1$ in \cite{BejenaruIonescuKenig2007}. This leads to several technical issues, in particular the reduced equation is not simply a reduced equation in terms of $V_{i} : = \langle \partial_{i}u, e_{1} \rangle + \i \langle \partial_{i} u, e_{2} \rangle$ (which was the case  for $s=1$). Instead we find a heavily coupled nonlocal equation of $V$, $u$, and $e$ -- and we exploit that the \emph{leading order} terms \emph{mostly} depends on $V$. Another reason we work in the Besov type resolution spaces are the algebra properties of Besov spaces.

\subsection*{Outline} By rotation invariance, it suffices to prove Theorem \ref{mainthm} for $ Q = (0,0,1)$. For this reason, throughout the paper we denote $ Q = (0,0,1)$ and prove \Cref{mainthm} for $Q=(0,0,1)$.
In \Cref{s:resolution} we define the resolution spaces and prove various estimates for these spaces. Most of these estimates were already proven in \cite{subcritical,model}. In 
\Cref{s:algebraprops} we prove that the main resolution space $F^{\sigma}$ form an algebra for appropriate values of $ \sigma$. In \Cref{s:reduction} we derive the modified $s$-Schr\"odinger equation for the functions $V_{i}$. In the remaining sections we  prove the nonlinear estimates for the modified $s$-Schr\"odinger equation. Lastly, in \Cref{s:proofofthm} we prove Theorem \ref{mainthm}. Appendix \ref{s:taylor} contains basic commutator estimates. In Appendix \ref{s:elementary} we recall some elementary estimates including algebra properties of Besov spaces. And in Appendix \ref{s:shorttime} we prove short time existence of \eqref{SSchrodingereq}. In appendix \ref{null} we include the computations regarding the null structure for $ s \in (0,1)$ and explain why we do not take that route in the proof of Theorem \ref{mainthm}.

\subsection*{Acknowledgement} 
A.S. was an Alexander-von-Humboldt Fellow. S.R.F. and A.S. received funding from NSF Career DMS-2044898.

This work is proudly made without use of AI-tools for any of the mathematical reasoning, except for \Cref{s:shorttime}. Beyond that AI was used only for proof-checking of the manuscript and editing.

\section{Resolution spaces}\label{s:resolution}
\subsection*{Notation}
We use the same notation in \cite{model} and \cite{subcritical}. Namely,  we will use Greek letters $\xi,\tau$ for phase space, and Roman letters $x,t$ for physical space.

Throughout this work we assume, unless otherwise stated, that $n \geq 4$ and $s \in (\frac{1}{2},1)$. Constants, like $C$ can change from line to line and generally may depend on the dimension and $s$. We use the notation $A \aleq B$ if there is a multiplicative, nonnegative constant $C$ such that $A \leq C B$, and $A \aeq B$ if $A \aleq B$ and $A \ageq B$.

We will use the following commutator notation for operators $T$ and maps $f,g$.
\[
[T,f](g) = T(fg) - f Tg
\]
\[
H_T(f,g) = T(fg)-(Tf) g - f Tg
\]
We will write the Fourier transform as 
\[
 \mathcal{F}_{\R^k} f(\zeta) = \int_{\R^k} e^{-\i \zeta \cdot x} f(x) dx
\]
and it's inverse (up to a multiplicative constant which we will ignore for the sake of readability)
\[
 \mathcal{F}_{\R^k}^{-1} f(x) = \int_{\R^k} e^{+\i \zeta \cdot x} f(\zeta) d\zeta.
\]
We will omit the subscript when the space is clear from context. Another abuse of notation is that we will ``define'' the Fourier symbol of the operator $\i \partial_t - (-\Delta^s)$ as
 \[
   \mathcal{F} \brac{(\i \partial_t - (-\Delta^s))f}(\xi,\tau) =  -(\tau +  |\xi|^{2s})\mathcal{F} f(\xi,\tau)
\]
the multiplicative constant that should be on the right-hand side plays no role in the analysis.

\subsection*{Frequency projections}
Let $\varphi_{0}$ be the typical Littlewood-Paley bump function, $\varphi_{0} \in C_c^\infty((-2,2),[0,1])$, $\varphi_{0} \equiv 1$ in $[-1,1]$, $\varphi_{0}(r) = \varphi_{0}(|r|)$. Also, let $\varphi(r) := \varphi_{0}(r) - \varphi_{0}(2r)$ and observe that we have 
\[
 1 = \varphi_{0}(r) + \sum_{k=1}^\infty \varphi(r/2^{k}) \quad \forall r \in [0,\infty).
\]
For vectors $\xi \in \R^n$ we will tacitly mean 
\[
\varphi_{0}(\xi) \equiv \varphi_{0}(|\xi|), \quad \varphi(\xi) := \varphi(|\xi|).
\]
We will use the notation 
$$
\varphi_{[k_{1},k_{2}]}(r) = \sum_{m=k_{1}}^{k_{2}} \varphi(r/2^{m}).
$$
In particular, 
$$
\varphi_{k} (r) : = \varphi (r /2^{k} ).
$$
We also define $\varphi^+$ to be the restriction of $\varphi$ to $r > 0$, e.g. for $k_1,k_2 \in \Z$

\[ \varphi^{+}_{[k_{1},k_{2}]}(r) = \ind_{[0,\infty)}(r) \varphi_{[k_{1},k_{2}]}(r). \]

For $f: \R^{n+1} \to \C$, we denote the frequency projections by
\[
 \Delta_{k} f(x,t) := \mathcal{F}^{-1}_{\R^{n+1}}(\varphi_{k}(\xi) \mathcal{F}_{\R^{n+1}}f(\xi,\tau))  \quad k \in \Z,
\] 
and for $j \geq 0$ we denote the modulation 
\[
 Q_j f(x,t) :=  \begin{cases}
                   \mathcal{F}_{\R^{n+1}}^{-1} \brac{\varphi_{0}({\tau {+} |\xi|^{2s}})\,  \mathcal{F}_{\R^{n+1}} f(\xi,\tau)}   \quad &j = 0, \\
                   \mathcal{F}_{\R^{n+1}}^{-1} \brac{\varphi_{j}\brac{{\tau {+} |\xi|^{2s}}}\,  \mathcal{F}_{\R^{n+1}} f(\xi,\tau)} \quad & j \geq 1.
                   \end{cases}
\]
We also set  
\[
 \Delta_{\leq k}  := \sum_{\ell =- \infty }^{k} \Delta_{\ell} 
\]
and similarly for $Q_{\leq j}$
$$
Q_{\leq j} (f) : = \sum_{l =0}^{j} Q_{l} (f).
$$

We will need one more decomposition: fix an even function
\begin{equation}\label{IntegFunc}
     \chi \in C_c^\infty((-2,2),[0,1]), \quad \text{supp}(\chi) \subset [-2/3,2/3]
\end{equation}

such that
\[
 \sum_{\tell \in \Z} \chi(x-\tell) = 1 \quad \forall x \in \R.
\]
For $ l \in \Z^{n}$, 
we set 
\[
 \chi_{k;\tell}(\xi) := \chi( (\xi^1 -\tell^{1})/2^k)\, \chi( (\xi^2 -\tell^{2})/2^k)\, \ldots \chi( (\xi^n -\tell^{n})/2^k)
\]
Then we have 
\[
 \sum_{\tell \in 2^k \Z^n} \chi_{k;\tell}(\xi) = 1 \quad \forall \xi \in \R^n, k \in \Z.
\]
Set 
\begin{equation}\label{eq:Pktell}
 P_{k,\tell} f:= \mathcal{F}^{-1}\brac{\chi_{k,\tell}\mathcal{F} f}
\end{equation}
Thus, for any $k \in \Z$
\begin{equation}\label{eq:weirdfinedecomp}
 f = \sum_{\tell \in 2^k \Z^n} P_{k,\tell} f.
\end{equation}
Observe that $P_{k,\tell}$ is the only projection operator in this work that is \emph{not} rotation invariant.

\subsection{Definition of the spaces:}
Throughout this paper we fix
\[
 \mathscr{E} \subset \S^{n-1},
\]
a finite set of vectors in $\S^{n-1}$ 
so that whenever $ e \in \mathscr{E}$ then $ -e \in \mathscr{E}$. We also take the cardinality of $ \mathscr{E}$ to be large enough so that we have the conclusion of the following ``geometric observation''.
\begin{lemma}\label{la:geometricobserv}
There exists a finite subset $\mathscr{E} \subset \S^{n-1}$ such that the following holds:

 \item For any $\xi \in \R^n$ there exists $e \in \mathscr{E}$ such that 
 \[
  \abs{\xi-\abs{\xi}e} \leq \vert \xi \vert,
 \]
and in particular
 \[
  \langle \xi,e\rangle \geq  \frac{\vert \xi \vert}{2}.
 \]

 Moreover there exists a decomposition of unity on the $\S^{n-1}$-sphere denoted by $( \tilde{\vartheta}_e)_{e \in \mathscr{E}}$ with 
 \[
  \supp \tilde{\vartheta}_e \subset \{\xi \in \S^{n-1}: \quad \langle \xi, e\rangle \geq \frac{1}{2} \}
 \]
and 
 \[
  \sum_{e \in \mathscr{E}} \tilde{\vartheta}_e(\xi/|\xi|) = 1 \quad \forall \xi \in \R^n \setminus \{0\}.
 \]

 We will denote 
 \[
  \vartheta_{e}(\xi) := \tilde{\vartheta}_{e} (\xi/|\xi|)
 \]
and with some abuse of notation will also use this notation for $e \in \S^{n-1}$ for $e \not \in \mathscr{E}$ as simply a cutoff function on the cone.

Observe that $\vartheta_{e} \not \in C^\infty$ since it has a singularity at $0$.
\end{lemma}
Let $j > 0, k \in \Z$ then define the following set 
$$
D_{k,j} : = \{ (\xi,\tau) \in \R^{n+1} : \vert \xi \vert \in [2^{k-1},2^{k+1}], \vert \vert \xi \vert^{2s} + \tau \vert \in [2^{j-1} ,2^{j+1}] \}.
$$
And for $ j =0$ , and $k \in \Z$ we define
$$
D_{k,0}:= \{ (\xi,\tau) \in \R^{n+1} : \vert \xi \vert \in [2^{k-1},2^{k+1}], \vert \vert \xi \vert^{2s} + \tau \vert \in [0,2] \}.
$$
Moreover, 
$$
D_{  k , \leq j } : = \bigcup_{j'=0}^{j} D_{k,j'},
$$
and in particular 
$$
D_{k,\infty} : = \bigcup_{j'=0}^{\infty} D_{k,j'} 
$$
we define $ D_{k, \geq j }$, $D_{ \leq k , j }$ in the same way.

Further, for any vector $ e \in \S^{n-1}$ consider the following set 
$$
D^{e,k'}_{k,j} : = \{ (\xi,\tau) \in D_{k,j} : \langle \xi ,e \rangle \in [2^{k'-1},2^{k'+1}] \}. 
$$
With this notation we define the following resolution spaces; 
for $f: \R^{n+1} \to \C$ and $k \in \Z$ we define the $X_k$-spaces as follows
\begin{equation}\label{eq:Xdef}
 \|f\|_{X_k} := \begin{cases}
                                       \sum_{j=0}^\infty  2^{\frac{j}{2}} \|Q_{j} f\|_{L^{2}_{t,x}}  \quad &\text{if $\supp \hat{f} \subset D_{k,\infty} $ } \\
                                       \infty \quad \text{otherwise}.
                                      \end{cases}
\end{equation}
Also, for $k \geq 100$, $k' \in \N \cap [1,k+1]$ and $e \in \mathscr{E}$ we define the $Y_{k,k'}^{e}$ on maps whose frequency support is in the cone 
\[
 \|f\|_{Y_{k,k'}^{e}} = \begin{cases}
                                      2^{-k'\frac{2s-1}{2}} \gamma_{k,k'} \|(\i \partial_t - (-\Delta)^s+\i) f \|_{L^1_e L^2_{t,e^\perp}} \quad &\text{if }   \supp \hat{f} \subset D^{e,k'}_{k,\leq 2sk+10} 
                                      \\
                                       \infty \quad \text{otherwise}
                                      \end{cases}
\]
where $ \gamma_{k,k'} : =  2^{2n ( k -k')} $.
We define the space
\[
 Z_{k} := X_{k} + \sum_{e \in \mathscr{E}} \sum_{k'=1}^{k+1} Y_{k,k'}^e
\]
equipped with the norm
\[
 \|f\|_{Z_{k}} = \inf_{f = f_1 + \sum_{e,k'} f_{e,k'}} \brac{\|f_1\|_{X_{k}} + \sum_{e \in \mathscr{E},k'=1,...,k+1}\|f^{k'}_e\|_{Y_{k,k'}^e}}.
\]

For $\sigma \in \R $ the main resolution space $F^\sigma$ is induced by the norm
\begin{equation}\label{eq:Fsigma}
\|u\|_{F^\sigma} := \sum_{k = -\infty}^\infty 2^{k\sigma} \|\Delta_{k} u\|_{Z_k} ,
\end{equation}
and the space for the right-hand side is induced by the norm
\begin{equation}\label{eq:Nsigma}
\|F\|_{N^\sigma} := \sum_{k=- \infty}^\infty 2^{k\sigma} \|(\i \partial_t -(-\Delta)^s + \i)^{{-1}} \Delta_{k} F\|_{Z_k}.
\end{equation}

In order to distinguish the spaces $Y^{e}_{k,k'}$ when $k'$ is small or large, we make the following definition
\begin{definition}
Let $k \geq 100$, $k' \in \N \cap[1,k+1]$. Then set $ q(k,k') = \frac{k'}{k}$. And set $T_{k} : = [\frac{4k}{5 }, k+1]$. 
\end{definition}
If $ k,k'$ are clear from the context, we will simply write $q$. 
The following lemma summarizes various embeddings for the resolution spaces, for a proof see \cite{subcritical}. 
\begin{lemma}\label{embeddings}
\begin{enumerate}
    \item Let $k \in \Z$ and $j \geq 0$, $f \in Z_{k}$ then we have the following 
    \begin{equation}\label{embd1}
    \begin{split}
        \Vert Q_{j}(f) \Vert_{X_{k}} \lesssim \Vert f \Vert_{Z_{k}}.
    \end{split}
    \end{equation}
    Moreover, if $ f $ has Fourier transform supported in $ \{ \langle \xi, e \rangle \geq 2^{k' -c } \}$ for some $ e \in \S^{n-1}$ and $k' \in T_{k}$ then we have 
    \begin{equation}\label{embd2}
    \begin{split}
    & \Vert Q_{j}(f) \Vert_{X_{k}} \lesssim  \min\{ 2^{(k-k')(1-s)}\gamma_{k,k'}^{-1/2}, \gamma_{k,k'}^{-1/2} 2^{(2sk' -j)/2} \} \Vert f \Vert_{Y^{e}_{k,k'}}
    \end{split}
    \end{equation}
    In particular, 
    $$
    \Vert Q_{j}(f) \Vert_{X_{k}}  \lesssim \min\{ 1,  2^{(2sk' -j)/2} \} \Vert f  \Vert_{Y^{e}_{k,k'}}.
    $$
    \item 
     given $ k \in \Z$, $j \geq 0$ and $ f \in Y^{e}_{k,k'}$ with $ k' \notin T_{k}$ we have the following 
    \begin{equation}\label{embd3}
        \Vert f \Vert_{X_{k}} \lesssim \Vert f \Vert_{Y^{e}_{k,k'}}
    \end{equation}
    Moreover, if $ k' \in T_{k}$ then we have 
    \begin{equation}\label{embd4}
     \Vert \sum_{ j \geq 2k( s+q(k,k') -1)} Q_{j}(f) \Vert_{X_{k}} \lesssim \Vert f \Vert_{Y_{k,k'}^{e}}.
    \end{equation}
    \item  for any $j \geq 0$ and $k \in \Z$ we have
    \begin{equation}\label{embd5}
    \Vert Q_{\geq j} (f) \Vert_{L^{2}(\R^{n+1})} \lesssim 2^{-j/2} \Vert f \Vert_{Z_{k}} 
    \end{equation}

\end{enumerate}

\end{lemma}
 Since the spaces $Z_{k}$ are unstable under complex conjugation, we consider the following space 
$$
\bar{Z}_{k} : = \{ f \in L^{2}(\R^{n+1}) : \bar{f} \in Z_{k} \},
$$
with the norm
$$
\Vert f \Vert_{\bar{Z}_{k}} := \Vert \bar{f} \Vert_{Z_{k}}.
$$
And we make the following definition
\begin{definition}\label{tildezk}
Let $k \in \Z$, then we define $ \tilde{Z}_{k}$ as follows
$$
\tilde{Z}_{k} : = Z_{k} + \bar{Z}_{k},
$$
with the infimum norm. Moreover, 
$$
\Vert f \Vert_{\tilde{F}^{\sigma}}: = \sum_{k \in \Z} 2^{\sigma k} \Vert \Delta_{k} f \Vert_{\tilde{Z}_{k}}.
$$
\end{definition}
Lastly, we will need a truncated version of the spaces $ \tilde{F}^{\sigma}, F^{\sigma}$ and $ N^{\sigma}$. Given $ r \in [0,\infty]$ we define 
\begin{equation}\label{nrdef}
\Vert f \Vert_{N^{\sigma}(r)} : = \Vert \ind_{[-r,r]} f \Vert_{N^{\sigma}}.
\end{equation}
As for $F(r)$, $ \tilde{F}(r)$ we cannot cutoff by $ \ind_{[-r,r]}$ since that might ruin the linear estimates. Instead, 
\begin{equation}\label{frdef}
    \Vert f \Vert_{F^{\sigma}(r)} : = \inf \{ \Vert F \Vert_{F^{\sigma}} : F= f \text{ on } [-r,r] \}. 
\end{equation}
In other words, we take the infimum over all possible extensions of $ \ind_{[-r,r]} f$. We define $ \tilde{F}^{\sigma}(r)$ analogously.

\subsection{Estimates for resolution spaces}
The following estimates were proven in \cite{subcritical},  for the inhomogeneous version of the spaces above. The proof for the homogeneous version is identical. 
\begin{lemma}\label{summaryres}
Let $ n \geq 3$ , $ s \in (1/2,1)$ then the following holds
\begin{enumerate}
    \item for any $k \in \Z$, $ f \in Z_{k}$, and any $ e \in \S^{n-1}$ we have
    \begin{equation}\label{local_smoothing}
        \Vert \mathcal{F}^{-1}( \vartheta_{e} \hat{f} )\Vert_{L^{\infty}_{e}L^{2}_{e^{\perp},t}} \lesssim 2^{-k(2s-1)/2} \Vert f \Vert_{Z_{k}}
    \end{equation}
    where $ \tilde{f} \in \{f, \bar{f} \}$.
    \item for any $k \in \Z$, $ f\in Z_{k}$ we have
    \begin{equation}\label{maximal_Estimate}
        \Vert \tilde{f} \Vert_{L^{2}_{e}L^{\infty}_{e^{\perp},t}} \lesssim 2^{k(n-1)/2} \Vert f \Vert_{Z_{k}}
    \end{equation}
    where $ \tilde{f} \in \{ f , \bar{f} \}$.
    \item for any $ k \in \Z$ and $ f \in Z_{k}$ we have
    \begin{equation}\label{energy}
         \Vert \tilde{f} \Vert_{L^{\infty}_{t} L^{2}_{x}} \lesssim \Vert f \Vert_{Z_{k}}.
     \end{equation}
     \item  for any $m \in L^{\infty}(\R^{n}) $ with $ \mathcal{F}^{-1}(m) \in L^{1}(\R^{n})$ we have the following 
     \begin{equation}\label{multilpierlemma}
         \Vert \mathcal{F}^{-1} ( m \hat{f} ) \Vert_{Z_{k}} \lesssim \Vert \mathcal{F}^{-1}(m) \Vert_{L^{1}} \Vert f \Vert_{Z_{k}}
     \end{equation}
     \item And more generally, we have for any $m \in L^{\infty}(\R^{n})$ with $ \mathcal{F}^{-1}(m) \in L^{1}(\R^{n})$ the following 
 \begin{equation}\label{multiplierint}
     \Vert \mathcal{F}^{-1} (m \hat{f}) \Vert_{L^{p}_{e} L^{q}_{e^{\perp},t}} \lesssim \Vert \mathcal{F}^{-1}(m) \Vert_{L^{1}(\R^{n})} \Vert f \Vert_{L^{p}_{e} L^{q}_{e^{\perp},t}},
 \end{equation}
 for any $p,q \in [1,\infty]$.
 \item Moreover, for any $m_{1} , m_{2} \in L^{\infty}(\R^{n})$ with $ \mathcal{F}^{-1}(m_{1}) , \mathcal{F}^{-1}(m_{2}) \in L^{1}(\R^{n})$ consider the following operator, defined via Fourier transform,  
$$
\mathcal{F}(H(f,g,h))(\xi,t) := \int_{\R^{n} \times \R^{n}} m_{1} (\xi-\eta) m_{2}(\kappa) \hat{f}(\xi -\kappa,t) \hat{h}(\kappa-\eta,t) \hat{g}(\eta,t) d \eta d \kappa.
$$
Then for any $p,q \in [1,\infty]$ and any $p_{1},p_{2},p_{3}, q_{1},q_{2},q_{3} \in [1,\infty]$ with 
$$
\frac{1}{p_{1}} + \frac{1}{p_{2}} + \frac{1}{p_{3}} =\frac{1}{p}
$$
and $$
\frac{1}{q_{1}} + \frac{1}{q_{2}} + \frac{1}{q_{3}}= \frac{1}{q}
$$
we have 
\begin{equation}\label{multimultiplier}
\Vert H(f,g,h) \Vert_{L^{p}_{e} L^{q}_{e^{\perp},t}} \lesssim \Vert \mathcal{F}^{-1}(m_{1}) \Vert_{L^{1}} \Vert \mathcal{F}^{-1}(m_{2}) \Vert_{L^{1}}  \Vert f \Vert_{L^{p_{1}}_{e} L^{q_{1}}_{e^{\perp},t}} \Vert g \Vert_{L^{p_{2}}_{e} L^{q_{2}}_{e^{\perp},t}} \Vert h \Vert_{L^{p_{3}}_{e} L^{q_{3}}_{e^{\perp},t}}.
\end{equation}
     \item  for any $ e \in \S^{n-1}$, $k \in \Z$ and any $p,q \in [1,\infty]$ and $ \beta \in \R$ one has 
     \begin{equation}\label{multiplier2}
         \Vert \Delta_{k} ( \vert \nabla \vert^{\beta} f ) \Vert_{L^{p}_{e}L^{q}_{e^{\perp},t}} \lesssim 
2^{k \beta} \Vert \Delta_{ \approx k} f  \Vert_{L^{p}_{e}L^{q}_{e^{\perp},t}}     
     \end{equation}
    \item For any $ j \geq 0$ and $f \in Z_{k}$ we have the following inequality 
    \begin{equation}\label{multiplier3}
        \Vert Q_{ \leq j} (f) \Vert_{Z_{k}} \lesssim \Vert f \Vert_{Z_{k}}.
    \end{equation}
    \item  for any $f \in Z_{k} $ we have 
    \begin{equation}\label{passingtol1l2}
    \Vert f \Vert_{Z_{k}} \lesssim 2^{-k(2s-1)/2} \max_{e \in \mathscr{E}} \Vert ( \i \partial_{t} - ( - \Delta)^{s} +\i) f \Vert_{L^{1}_{e} L^{2}_{e^{\perp},t}}
    \end{equation}
    
\end{enumerate}
\end{lemma}
\begin{proof}
Let $k \leq 0$. Then proofs of \eqref{local_smoothing}, \eqref{maximal_Estimate}, \eqref{energy}, \eqref{multilpierlemma}, \eqref{multiplierint}, \eqref{multiplier3} \eqref{multiplier2}, \eqref{passingtol1l2} is in \cite{model}. If $ k > 0$ then the proofs for \eqref{local_smoothing}, \eqref{maximal_Estimate}, \eqref{energy}, \eqref{multilpierlemma}, \eqref{multiplierint}, \eqref{multiplier3} \eqref{multiplier2}, \eqref{passingtol1l2} can be found in \cite{subcritical}. Therefore, we only need to prove \eqref{multimultiplier}. 
\\
\underline{Proof of \eqref{multimultiplier}} We write $H(f,g,h)$ in physical space (we are suppressing the time variable $t$, for better readability)
\begin{equation*}
\begin{split}
& H(f,g,h)(x) = \int_{\R^{n}} \int_{\R^{n} \times \R^{n}} \int_{\R^{n} \times \R^{n}} e^{\i \xi \cdot x } e^{i ( \xi - \eta) \cdot u} e^{\i \kappa \cdot v} \mathcal{F}^{-1}(m_{1})(u) \mathcal{F}^{-1}(m_{2})(v)
\\
& \times \hat{f}(\xi - \kappa)  \hat{h}( \kappa- \eta) \hat{g}(\eta) d \eta d \kappa d u dv d \xi
\\
&= \int_{\R^{n} \times \R^{n}} \int_{\R^{n} \times \R^{n} }  e^{ -\i  \eta \cdot u} e^{\i \kappa \cdot v} \hat{h}(\kappa - \eta) \hat{g}(\eta) \mathcal{F}^{-1}(m_{1})(u) \mathcal{F}^{-1}(m_{2})(v)
\\
& \times \left( \int_{\R^{n}} e^{\i  \xi \cdot (x+u)} \hat{f}(\xi - \kappa ) d\xi \right) d \kappa d \eta du dv 
\\
& = \int_{\R^{n} \times \R^{n}} \int_{\R^{n} \times \R^{n} }  e^{ -\i  \eta \cdot u} \hat{h}(\kappa - \eta) \hat{g}(\eta) f(x+u) e^{\i \kappa \cdot ( x+u+v)}
\\
& \times  \mathcal{F}^{-1}(m_{1})(u) \mathcal{F}^{-1}(m_{2})(v)  d\eta d\kappa du dv
\\
& = \int_{\R^{n} \times \R^{n}} \int_{\R^{n} \times \R^{n} }  e^{ \i  \eta \cdot (v+x)} \hat{h}(\kappa - \eta) \hat{g}(\eta) f(x+u) e^{\i (\kappa- \eta) \cdot ( x+u+v)} 
\\
& \times \mathcal{F}^{-1}(m_{1})(u) \mathcal{F}^{-1}(m_{2})(v)  d\eta d\kappa du dv
\\
& = \int_{\R^{n} \times \R^{n}} \mathcal{F}^{-1}(m_{1})(u) \mathcal{F}^{-1}(m_{2})(v) h(x+u+v) g(x+v) f(x+u) dudv.
\end{split}
\end{equation*}
Set 
$$
F_{u,v}(x,t) : = h(x+u+v,t) g(x+v,t) f(x+u,t)
$$
Then by Minkowski's inequality we obtain
\begin{equation*}
\begin{split}
& \Vert H(f,g,h) \Vert_{L^{p}_{e}L^{q}_{e^{\perp},t}}
\\
& \lesssim \Vert \mathcal{F}^{-1}(m_{1}) \Vert_{L^{1}(\R^{n})} \Vert \mathcal{F}^{-1}(m_{2}) \Vert_{L^{1}(\R^{n})} \sup_{u,v} \Vert F_{u,v} \Vert_{L^{p}_{e} L^{q}_{e^{\perp},t}}
\\
& \lesssim \Vert \mathcal{F}^{-1}(m_{1}) \Vert_{L^{1}} \Vert \mathcal{F}^{-1}(m_{2}) \Vert_{L^{1}}  \Vert f \Vert_{L^{p_{1}}_{e} L^{q_{1}}_{e^{\perp},t}} \Vert g \Vert_{L^{p_{2}}_{e} L^{q_{2}}_{e^{\perp},t}} \Vert h \Vert_{L^{p_{3}}_{e} L^{q_{3}}_{e^{\perp},t}}.
\end{split}
\end{equation*}
\end{proof}

\subsection{Localized maximal estimate}
The following estimate was proven in \cite{model} for the spaces $Y^{e}_{k,k}$ and $X_{k}$. The proof for $ Y^{e}_{k,k'}$ is almost identical. In fact it is implicitly contained in the arguments made in \cite[Lemma 4.2]{subcritical} and \cite[Lemma 3.8]{model}, which were direct fractional adaptations of \cite[Lemma 4.1]{cmp}. But for completeness we sketch the proof here. 
\begin{lemma}
Let $ n \geq 3$, $ s\in (1/2,1)$. Let $ k, k_{1} \in \Z$ with $ 2^{k_{1}} \lesssim 2^{k}$. Then for any $ f \in Z_{k}$ we have
\begin{equation}\label{eq:max:Xkest:loc}
  \brac{\sum_{\tn \in 2^{k_1} \Z^n} \|P_{k_1,\tn} \tilde{f}\|_{L^{2}_e L^{\infty}_{e^\perp,t}}^2}^{\frac{1}{2}} \aleq
 2^{k_{1} (n-1)/2} 2^{(k -k_{1})/2} (1+|k-k_{1}|)\|f\|_{Z_{k}}.
\end{equation}
where $ \tilde{f} \in \{ f, \bar{f} \} $, and the projections $P_{k_{1},\tn}$ are defined in \eqref{eq:Pktell}.
\end{lemma}
\begin{proof}
If $ f \in X_{k}$ then the result was proven in \cite[(3.34)]{model}. Assume $k \geq 100$ and $ f \in Y^{e'}_{k,k'}$.
In view of \eqref{embd3} we may assume $ k' \in [\frac{4}{5} k , k+1] \cap \N$. Further,
in light of \eqref{embd4}, we may assume $ f $ has Fourier transform supported in $D^{e'}_{k, \leq 2k(s+q(k,k')-1)-c}$. In other words,
$$
f = Q_{ \leq 2k(s+q(k,k')-1) -c }(f).
$$

Now write $ \xi = \xi_{1} e' + \xi' $ where $ \xi' \in e'^{\perp}$, and define $ N(\xi',\tau) : = \left( ( - \tau)^{1/s} - |\xi'|^{2} \right)^{1/2}$ whenever $ \tau < 0$ and $ | \xi'|^{2} \leq (-\tau)^{1/s}$. Notice that for any $ ( \xi,\tau) \in \text{supp}(\hat{f})$ we have that $ N(\xi',\tau)$ is well defined and  $ N(\xi',\tau) \approx 2^{k'}$, see \cite[Lemma 3.3]{subcritical} for a proof. 
Now consider the modified projection $ Q'_{j}$ given by 
$$
\mathcal{F}(Q'_{j}(f)) := \varphi_{j} ( \xi_{1} - N(\xi',\tau)) \hat{f}(\xi,\tau).
$$
We wish to replace $ Q_{\leq 2k( s + q(k,k')-1)} (f)$ with $ Q'_{\leq k' -c} (f)$. To that end, we write 
$$
f = Q'_{\leq k' -c} (f) + E(f)
$$
where $ E(f) =Q_{\leq 2k( s + q(k,k')-1)-c} (f) -  Q'_{\leq k' -c} (f)$.\\
we first prove the following; let $ E(f)$ be as above, then we have 
\begin{equation}\label{claim1}
\begin{split}
& \Vert E(f) \Vert_{X_{k}} \lesssim \Vert f \Vert_{Y^{e'}_{k,k'}}.
\end{split}
\end{equation}
\underline{Proof of \eqref{claim1}. } By Lemma \cite[Lemma 3.3]{subcritical} we have that 
\begin{equation}\label{Nmodcompare}
\frac{  \tau + \vert \xi \vert^{2s} }{ 2^{2k( s + q(k,k') -1)}} \alpha(\xi,\tau) = \frac{ | \xi_{1} - N(\xi',\tau)}{2^{k'}}
\end{equation}
for some $ \alpha$ satisfying $ | \alpha(\xi,\tau) | \approx 1$. Then using this we obtain
\begin{equation}\label{pointwiseEbound}
\begin{split}
& | \varphi_{ \leq 2k(s+q(k,k')-1)-c}( \tau + | \xi |^{2s}) - \varphi_{\leq k' -c} ( \xi_{1} - N(\xi',\tau)) | 
\\
& \lesssim 2^{-2k(s+q-1)} | \tau + | \xi |^{2s} |.
\end{split}
\end{equation}
Therefore, we have proven that in the frequency space we actually have
\begin{equation*}
\begin{split}
& \mathcal{F}(E(f)) (\xi,\tau)= A(\xi,\tau) \hat{f}(\xi,\tau),
\end{split}
\end{equation*}
where $ |A(\xi,\tau) | \lesssim \min \{ 1,2^{-2k(s+q-1)}  | \tau + \vert \xi \vert^{2s} | \} $. Now to prove \eqref{claim1} we argue as follows; for any $ j \leq 2k(s+q(k,k')-1) -c$ 
\begin{equation}\label{lklkl}
\begin{split}
& 2^{j/2} \Vert \varphi_{j}( \tau + | \xi |^{2s} )\mathcal{F}(E(f)) \Vert_{L^{2}( \R^{n+1})} 
\\
& \lesssim  2^{-2k(s+q-1)}  2^{j/2} \Vert \varphi_{j}(\tau+ | \xi |^{2s} ) ( \tau + \vert \xi \vert^{2s})\hat{f} \Vert_{L^{2}(\R^{n+1})}
\\
& \lesssim  2^{-2k(s+q-1)}  2^{j/2} \Vert \varphi_{j} \varphi_{j}(\tau+ \vert \xi \vert^{2s}) \mathcal{F}_{\R^{n+1}} \left( ( \i \partial_{t} - ( - \Delta)^{s} + \i) f \right) \Vert_{L^{2}(\R^{n})}
\\
& \lesssim  2^{-2k(s+q-1)}  2^{j/2} \Vert \varphi_{j} \varphi_{j}(\tau+ \vert \xi \vert^{2s}) \int_{\R} e^{\i x_{1} \xi_{1}} \mathcal{F}_{e^{\perp},t}\left( ( \i \partial_{t} - ( - \Delta)^{s} +\i) f \right)  dx_{1} \Vert_{L^{2}(\R^{n+1})}
\\
& \overset{\text{Minkowski's}}{ \lesssim }2^{-2k(s+q-1)}  2^{j/2} \Vert ( \i \partial_{t} - ( - \Delta)^{s} +\i) f \Vert_{L^{1}_{e} L^{2}_{e^{\perp},t}}  \times | \{ \xi_{1} \approx 2^{k'} : | \tau + | \xi |^{2s} | \lesssim 2^{j} \} |^{1/2}
\\
& \lesssim 2^{-2k(s+q-1)}  2^{j/2} \Vert ( \i \partial_{t} - ( - \Delta)^{s} +\i) f \Vert_{L^{1}_{e} L^{2}_{e^{\perp},t}} 2^{j/2} 2^{-k(2s-2+q(k,k')/2}
\end{split}
\end{equation}
where in the last inequality we used \cite[(3.8)]{subcritical}. Then notice 
\begin{equation*}
\begin{split}
2^{-2k(s+q-1)} 2^{-k(2s-2+q)/2} \sum_{j=0}^{2k(s-1+q(k,k'))}2^{j}  \lesssim  2^{-k(2s-2+q)/2}
\end{split}
\end{equation*}
Plugging this back into \eqref{lklkl} we obtained
\begin{equation*}
\begin{split}
& \sum_{j =0}^{2k(s-1+q(k,k')} 2^{j/2} \Vert \varphi_{j}( \tau + | \xi |^{2s} )\mathcal{F}(E(f)) \Vert_{L^{2}( \R^{n+1})} 
\\
& \lesssim 2^{-k(2s-2+q)/2} \gamma^{-1}_{k,k'} 2^{k'(2s-1)/2} \Vert f \Vert_{Y^{e}_{k,k'}}
\end{split}
\end{equation*}
and \eqref{claim1} follows from the fact that 
$$
2^{-k(2s-2+q)/2} \gamma^{-1}_{k,k'} 2^{k'(2s-1)/2} \lesssim 1.
$$
This proves \eqref{claim1}. 
Therefore, we may assume  
$$
f = Q'_{\leq k'-c} (f).
$$
Now we split
$$
f = Q'_{\leq k_{1} -c} (f) + \sum_{i=k_{1}-c+1}^{k'-c} Q'_{i}(f)
$$
with the understanding that the second term is $0$ if $ k_{1} \geq k'$. Now notice that we have 
$$
\Vert \sum_{i =k_{1}-c+1}^{k'-c} Q'_{i}(f) \Vert_{X_{k}} \lesssim (1+|k-k_{1}|) \Vert f \Vert_{Y^{e}_{k,k'}}.
$$
Indeed, this follows easily from \eqref{embd2} and \eqref{Nmodcompare}. Then we may assume
\begin{equation}\label{lstredmax}
f = Q'_{\leq k_{1}-c}(f) 
\end{equation}
and to conclude the result in this case, it suffices to prove 
\begin{equation}\label{eqlocmx}
  \brac{\sum_{\tn \in 2^{k_1} \Z^n} \|P_{k_1,\tn}Q'_{\leq k_{1}-c}(f)\|_{L^{2}_e L^{\infty}_{e^\perp,t}}^2}^{\frac{1}{2}} \aleq
 2^{k_{1} (n-1)/2} 2^{(k -k_{1})/2} (1+|k-k_{1}|)\|f\|_{Y^{e'}_{k,k'}}.
\end{equation}
Next, we wish to replace the projections $P_{k_{1},\tn}$ appearing on the left hand side of \eqref{eqlocmx} with projections that, in the frequency space, depend on $ (\xi',N(\xi',\tau))$ but not on $ \xi_{1}$. To that end, let
\[(\xi,\tau) \in \supp  \mathcal{F} P_{k_{1},\tn} Q'_{\leq k_{1} -c} f. \]
then we have 
\begin{itemize}
 \item $\abs{\tau + |\xi|^{2s}} \leq 2^{2k(s+q-1)-c}$ 
 \item  $|\xi| \aeq 2^k$ 
 \item $|\xi_{1} - N(\xi',\tau)| \ll \min\{ 2^{k'},2^{k_1} \} \aleq 2^{k}$
 \item $|\xi_{1}| \aeq |N(\xi_{e}',\tau)| \aeq 2^{k'}$
 \item $|\xi^i-\tn^i| \aeq 2^{k_1}$, for each component $i=1,\ldots,n$ 
\end{itemize}
We split $\xi = \xi_{1} e' + \xi'$ and $ \tn = \tn_{1} e' + \tn' $ (notice that $\tn_{1}$ might not be an integer) then we certainly have
\begin{itemize}
    \item $ | \xi' - \tn' |\aleq  2^{k_{1}}$ and
    \item $ | \xi_{1} -\tn_{1} |\aleq  2^{k_{1}}$
\end{itemize}
Consequently,
$$
| N(\xi',\tau)- \tn_{1} |\leq | N(\xi',\tau)- \xi_{1} | + | \xi_{1}- \tn_{1} | \aleq  2^{k_{1}}
$$
Therefore, we may write
\[
\begin{split}
 \mathcal{F}\brac{P_{k_{1},\tn} Q'_{\leq k_{1} -c} f}(\xi,\tau) = \pi_{k_{1},\tn}(\xi_{e}',N(\xi_{e}',\tau)) \mathcal{F}\brac{ P_{k_{1},\tn} Q'_{\leq k_{1} -c} f }(\xi,\tau)
 \end{split}
\]
where for $ ( \xi_{e}',R) \in e'^{\perp} \times \R$. 
$$
\pi_{k_{1},\tn}( \xi',R) = \chi( \frac{R-\tn_{1}}{2^{k_{1}+C}}) \, \eta( \frac{ ( \xi' -\tn')}{2^{k_{1}+C}}),
$$
and $ \chi : \mathbb{R} \to \mathbb{R}$ is defined in \eqref{IntegFunc}.

We denote for $\tilde{f}$ with suitable support
\[
 A_{k_{1},\tn} \tilde{f}  := \mathcal{F}^{-1} \brac{ \pi_{k_{1},\tn}(\xi_{e}',N(\xi_{e}',\tau)) \mathcal{F} \tilde{f}}
\]
and arrive at 
\[
\begin{split}
 P_{k_{1},\tn} Q'_{\leq k_{1} -c} f= A_{k_{1},\tn} {P_{k_{1},\tn} Q'_{\leq k_{1} -c} f }.
 \end{split}
\]
Observe that by \cite[Lemma B.2]{model} we have 
\begin{equation}\label{eq:onlyfinite}
\begin{split}
& | \sum_{\tn \in 2^{k_{1}} \Z^{n}} \pi_{k_{1},\tn}(\xi_{e}',N(\xi_{e}',\tau)) |\aleq  1.
\end{split}
\end{equation}
Now using \eqref{eq:onlyfinite} we see that \eqref{eqlocmx} follows once we prove 
\begin{equation}\label{eqlocmx22}
  \| A_{k_{1},\tn} P_{k_1,\tn}Q'_{\leq k_{1}-c}(f)\|_{L^{2}_e L^{\infty}_{e^\perp,t}} \aleq
 2^{k_{1} (n-1)/2} 2^{(k -k_{1})/2} (1+|k-k_{1}|)\|f\|_{Y^{e}_{k,k'}}.
\end{equation}
And by using \eqref{multilpierlemma}, we reduce matters to proving 
 \begin{equation}\label{eqlocmx3}
  \| A_{k_{1},\tn} Q'_{\leq k_{1}-c}(f)\|_{L^{2}_e L^{\infty}_{e^\perp,t}} \aleq
 2^{k_{1} (n-1)/2} 2^{(k -k_{1})/2} (1+|k-k_{1}|)\|f\|_{Y^{e'}_{k,k'}}.
\end{equation}
Equivalently, we want to prove
\begin{equation}
\begin{split}
  &  \|  \int_{\R \times e^{\perp} \times \R } e^{\i \xi' \cdot x'} e^{\i \xi_{1} x_{1}} e^{\i t \tau }  \pi_{k_{1},\tn}(\xi',N )\varphi_{\leq k_{1}-c}(\xi_{1}-N) \hat{f}(\xi,\tau) d\xi d\tau  \|_{L^{2}_e L^{\infty}_{e^\perp,t}}
  \\
  &\aleq
 2^{k_{1} (n-1)/2} 2^{(k -k_{1})/2} (1+|k-k_{1}|)\|f\|_{Y^{e'}_{k,k'}}. 
\end{split}
\end{equation}

To that end, 
we use \cite[Lemma 3.7]{subcritical}  and \eqref{lstredmax}, to write $  (f)$ as follows 
\begin{equation}\label{redrepform}
\begin{split}
& \mathcal{F}_{\R^{n+1}} ( f ) (\xi,\tau)
\\
& = \varphi_{\leq k_{1}-c} ( \xi_{1} -N) \frac{  \varphi_{\leq k'-c}(\xi_{1} -N) }{ \xi_{1} -N + \frac{\i}{2^{k(2s-2+q)}}} \int_{\R} e^{i \xi_{1} x_{1}} h(x_{1} , \xi', \tau) d x_{1},
\end{split}
\end{equation}
where  $ h $ is supported in $ \R \times \{( \xi',\tau) : N(\xi',\tau) \in [2^{k'-10}, 2^{k'+10}], \vert \xi' \vert \leq 2^{k+2} \}.$ And $h $ satisfy 
\begin{equation}\label{redrepform2}
\begin{split}
\Vert h \Vert_{L^{1}_{x_{1}} L^{2}_{\xi',\tau}} \lesssim 2^{2(1-s) ( k -k')} \gamma^{-1}_{k,k'} 2^{-k'(2s-1)/2} \Vert f \Vert_{Y^{e}_{k,k'}}.
\end{split}
\end{equation}
Then using Minkowki's inequality and this representation formula, we reduce matters to proving 
\begin{equation}
\begin{split}
 2^{-k'(2s-1)/2} &  \|  \int_{\R \times e'^{\perp} \times \R } e^{\i \xi' \cdot x'} e^{\i \xi_{1} x_{1}} e^{\i t \tau }  \pi_{k_{1},\tn}(\xi',N )\varphi_{\leq k_{1}-c}(\xi_{1}-N) 
  \\
  & \times \frac{\varphi_{\leq k' -c} ( \xi_{1} -N)}{ \xi_{1} -N +\frac{\i}{2^{k(2s-2+q)}}} h(\xi',\tau) d\xi d\tau  \|_{L^{2}_e L^{\infty}_{e^\perp,t}}
  \\
  &\aleq
2^{-2(1-s)(k-k')}\gamma_{k,k'} 2^{k_{1} (n-1)/2} 2^{(k -k_{1})/2} \|h\|_{L^{2}_{\xi',\tau}}, 
\end{split}
\end{equation}
for any $ h $ supported in $ \{ \vert \xi' \vert \lesssim 2^{k} \}$.
Carrying the integral in $ \xi_{1}$ we reduced matters into proving
\begin{equation*}
\begin{split}
&   2^{-k'(2s-1)/2}\|  \int_{ e'^{\perp} \times \R } e^{\i \xi' \cdot x'} e^{\i N x_{1}} e^{\i t \tau }  \pi_{k_{1},\tn}(\xi',N )  h(\xi',\tau) d\xi' d\tau  \|_{L^{2}_e L^{\infty}_{e^\perp,t}}
  \\
  &\aleq
 2^{-2(1-s)(k-k')}\gamma_{k,k'} 2^{k_{1} (n-1)/2} 2^{(k -k_{1})/2} \|h\|_{L^{2}_{\xi',\tau}}.
\end{split}
\end{equation*}
Consider the change of variable  $ \tau = -( \mu^{2} + | \xi'|^{2})^{s}$, then the Jacobian is $ \alpha(\mu,\xi') = -2s \mu(\mu^{2} + | \xi'|^{2})^{s-1} \approx 2^{k'} 2^{k(s-1)}$. And set $ h'(\xi',\mu) : = 2^{-k'(2s-1)/2} \alpha(\mu,\xi') h(\xi', - ( \mu^{2} + |\xi'|^{2})^{s} )$ then notice 
$$
\Vert h' \Vert_{L^{2}_{\mu,\xi'}} \lesssim 2^{(k-k')(s-1)} \Vert h \Vert_{L^{2}_{\xi',\tau}}.
$$
and since
$$
2^{(k-k')(s-1)} \lesssim 2^{-2(1-s)(k-k')} \gamma_{k,k'} 
$$
Then we reduced matters to proving 
\begin{equation}
\begin{split}
&  \|  \int_{ e'^{\perp} \times \R } e^{\i \xi' \cdot x'} e^{\i \mu x_{1}} e^{- \i t ( \mu^{2} + |\xi|^{2})^{s}  }  \pi_{k_{1},\tn}(\xi',N )  h(\xi',\mu) d\xi' d\mu  \|_{L^{2}_e L^{\infty}_{e^\perp,t}}
\\
& \lesssim 2^{k_{1} \frac{n-1}{2}} 2^{\frac{k-k_{1}}{2}} \Vert h \Vert_{L^{2}_{\mu,\xi'}}.
\end{split}
\end{equation}
for any $ h \in L^{2}_{\xi',\mu}$ supported in $ ( | \xi'|^{2} + \mu^{2} ) \approx 2^{2k}$. 
This was already proven in \cite[Appendix B]{model}. Lastly, the result for the complex conjugate follows from the fact that the left hand side of \eqref{eq:max:Xkest:loc} is invariant under complex conjugation.

\end{proof}

\subsection{Linear estimates} Let $\phi : [-2,2] \to [0,1]$ be some smooth function with $ \phi(t) = 1 $ on $ [-1,1]$. Then the following was proven in \cite{subcritical} and \cite{model}.
\begin{lemma}\label{LinearEstimates}
Let $ \sigma \in \R $ then the following linear estimates hold
\begin{enumerate}
    \item for any $ f_{0} \in \dot{B}^{\sigma}_{2,1}$ one has
    $$
    \Vert \phi(t) e^{ -\i t (-\Delta)^s} f_{0} \Vert_{F^{\sigma}} \lesssim \Vert f_{0} \Vert_{\dot{B}^{\sigma}_{2,1}}
    $$
    \item Let $ f(x,t) : = \i \phi(t) \int_{0}^{t} e^{-\i(t-r) (-\Delta)^s} F(x,r) dr $ then we have
    $$
    \Vert f \Vert_{F^{\sigma}} \lesssim \Vert F \Vert_{N^{\sigma}}
    $$
\end{enumerate}
\end{lemma}
\begin{proof}
First, we prove (1)
\begin{equation*}
\begin{split}
& \Vert \phi(t) e^{ -\i t (-\Delta)^s} f_{0} \Vert_{F^{\sigma}}
\\
& = \sum_{k=-\infty}^{\infty} 2^{ \sigma k} \Vert \Delta_{k} ( \phi(t) e^{-\i t (-\Delta)^s } f_{0} )\Vert_{Z_{k}}
\\
& \lesssim \sum_{k=-\infty}^{\infty} 2^{ \sigma k} \Vert \Delta_{k} ( \phi(t) e^{-\i t (-\Delta)^s } f_{0} )\Vert_{X_{k}}
\\
& \lesssim_{\phi} \sum_{k=-\infty}^{\infty } 2^{ \sigma k } \Vert  \Delta^{k} f_{0} \Vert_{L^{2}}.
\end{split}
\end{equation*}
Next, we proceed to prove (2). Notice that we have 
\begin{equation*}
\begin{split}
& \mathcal{F}_{\R^{n+1}}( \Delta_{k} f )
\\
& = \mathcal{F}_{\R^{n+1}} (  \Delta_{k} \i \phi(t) \int_{0}^{t} e^{-\i (t-r) (-\Delta)^s} F(x,r) dr )
\\
& = \i \int_{\R} \frac{ \mathcal{F}_{\R^{n+1}} \left( (\i \partial_{t} - (-\Delta)^{s}   + \i)^{-1} \Delta_{k} F \right)(\xi,r) ( r + \vert \xi \vert^{2s} + \i) }{ r + \vert \xi \vert^{2s}} \left( \hat{\phi}(r -\tau) - \hat{\phi}(\tau + \vert \xi\vert^{2s}) \right) dr.
\end{split}
\end{equation*}
Therefore, for $h \in Z_{k}$ set 
$$
T(h) := \mathcal{F}^{-1}_{\R^{n+1}} \left(  \i \int_{\R} \frac{ \hat{h}(\xi,r) ( r + \vert \xi \vert^{2s} + \i)}{ r + \vert \xi \vert^{2s}} \left( \hat{\phi}(r -\tau) - \hat{\phi}(\tau + \vert \xi\vert^{2s}) \right) dr \right)
$$
and to prove (2) it suffices to prove 
\begin{equation}
\Vert T(h) \Vert_{Z_{k}} \lesssim \Vert h \Vert_{Z_{k}}
\end{equation}
If $ h \in X_{k}$ then the result was already proven in \cite[(4.2)]{model}. If $k \geq 100$, $ h \in Y^{e}_{k,k'}$ then the result was proven in \cite[Lemma 9.1]{subcritical}.

\end{proof}

\section{Algebra properties of $F^{\sigma}$ spaces}\label{s:algebraprops}
In this section we prove that the spaces $ F^{n/2}$ are closed under multiplication. This result was already proven in \cite{subcritical} for the inhomogeneous version of $F^{\sigma}$. However, we include the proof here, since we need a version of \cite[Theorem 6.3]{subcritical} that accounts for the complex conjugate. The proof is the same as the one done in \cite[Theorem 6.3]{subcritical}, which was a fractional adaptation of the proof done in \cite[Lemma 5.2]{cmp}. We start with the following, which was proven in \cite[Lemma 6.2]{subcritical}
\begin{lemma}\label{es1}
Let $ j_{1},j_{2}\geq 0 $, $ k_{1} ,k_{2} \in \Z$ with $ k_{1} \leq k_{2} + C $. Assume that $ f_{k_{1}} \in Z_{k_{1}}$, and $ f_{k_{2}} \in Z_{k_{2}}$. Then we have the following inequality 
\begin{equation*}
\begin{split}
 & \Vert \tilde{f}_{k_{1}, \geq j_{1}}   \tilde{f}_{k_{2}, \geq j_{2}} \Vert_{L^{2}} 
 \\
 & \lesssim \left( 2^{j_{2}/2} + 2^{k_{1}/2 +(k_{2})(2s-1)/2} \right)^{-1} 2^{nk_{1}/2} \Vert f_{k_{1}} \Vert_{Z_{k_{1}}} \Vert f_{k_{2}} \Vert_{Z_{k_{2}}} 
\end{split}
\end{equation*}
where $ \tilde{f} \in \{ f , \bar{f} \}.$
\end{lemma}
Next, we recall the following commutator estimate. See \cite[Lemma 5.2]{subcritical} for a proof.

\begin{lemma}\label{comest}
Let $ k_{1}, k_{2} \in \Z $ with $ k_{1} \leq k_{2}+10$. Let $f_{k_{1}} \in Z_{k_{1}}$ and $f_{k_{2}} \in Z_{k_{2}}$. Then the following estimate holds
\begin{equation}\label{a}
\begin{split}
& \Vert  H_{(-\Delta)^{s}}(  \tilde{f}_{k_{1}}, \tilde{f}_{k_{2}}) \Vert_{L^{2}(\R^{n+1})} 
\\
& \lesssim 2^{(k_{1} + (2s-1)k_{2})/2} 2^{nk_{1}/2} \Vert f_{k_{1}} \Vert_{Z_{k_{1}}} \Vert f_{k_{2}} \Vert_{Z_{k_{2}}},
\end{split}
\end{equation}
where $ \tilde{f} \in \{ f , \bar{f} \}$ and $$
H_{(-\Delta)^{s}} (f,g) := (-\Delta)^{s} ( fg) - g (-\Delta)^{s} (f) - f (-\Delta)^{s}(g). 
$$
\end{lemma}

Further, the following elementary estimate was proven in \cite[Lemma 6.1]{subcritical}
\begin{lemma}\label{cl}
Let $ j_{1},j_{2} ,j \geq 0$, $ k, k_{1},k_{2} \in \Z$. Take $ g_{k_{1,j_{1}}}, g_{k_{2},j_{2}} \in L^{2}(\R^{n+1}) $ with the property that $ \text{supp} (\hat{g}_{k_{i},j_{i}} ) \subset D_{k_{i},j_{i}}$ for $i=1,2$. Then the following holds 
$$
\Vert \ind_{D_{k,j}} . \hat{\tilde{g}}_{k_{1},j_{1}} * \hat{\tilde{g}}_{k_{2},j_{2}} \Vert_{L^{2}(\R^{n+1})} \lesssim 2^{ n \min \{ k, k_{1},k_{2} \}/2 } 2^{ \min \{ j, j_{1} , j_{2} \}/2 } \Vert g_{k_{1},j_{1}} \Vert_{L^{2}(\R^{n+1})} \Vert g_{k_{2},j_{2}} \Vert_{L^{2}(\R^{n+1})},
$$ 
where $ \tilde{g} \in \{ g , \bar{g} \}$. 
\end{lemma}
The main result of this section, is the following bilinear estimate 

\begin{theorem}\label{Bilin1}
Let $k_{1},k_{2},k  \in \Z $ with $ k_{1} \leq k_{2}+10 $. And $f_{k_{1}} \in Z_{k_{1}} $ and $ f_{k_{2}} \in Z_{k_{2}}$ then we have the following 
$$
2^{nk/2} \Vert \Delta_{k} (  \tilde{f}_{k_{1}} f_{k_{2}}) \Vert_{Z_{k}} \lesssim 2^{- \frac{n}{2}\vert k - k_{2} \vert} ( 2^{nk_{1}/2} \Vert f_{k_{1}} \Vert_{Z_{k_{1}}} ) ( 2^{nk_{2}/2} \Vert f_{k_{2}} \Vert_{Z_{k_{2}}} ),
$$
where $ \tilde{f}_{k_{1}} \in \{ f_{k_{1}}, \bar{f}_{k_{1}} \}$.
\end{theorem}
\begin{proof}
First, notice that if $k >  k_{2} +20$ then we must have that $ \Delta_{k} ( \tilde{f}_{k_{1}} f_{k_{2}}) = 0 $. Therefore, assume that $ k \leq k_{2} + 20$.
Define $ \tilde{K} :=  \max\{k_{1} + (2s-1) k_{2},0 \} + C' $ for sufficiently large $C'$. Then write 
$$
\Delta_{k} ( \tilde{f}_{k_{1}} f_{k_{2}}) = Q_{\geq \tilde{K}} ( \Delta_{k}(\tilde{f}_{k_{1}} f_{k_{2}})) + Q_{\leq \tilde{K}-1} ( \Delta_{k}(\tilde{f}_{k_{1}} f_{k_{2}})).
$$
We estimate the second term. Using Lemma \ref{es1} ( with $j_{1} = 0 = j_{2} $ ) we get
\begin{equation}
\begin{split}
& 2^{nk/2} \Vert  Q_{\leq \tilde{K}-1} ( \Delta_{k}( \tilde{f}_{k_{1}} f_{k_{2}})) \Vert_{Z_{k}}
\\
& \leq 2^{nk/2} \Vert  Q_{\leq \tilde{K}-1} ( \Delta_{k}(\tilde{f}_{k_{1}} f_{k_{2}})) \Vert_{X_{k}}
\\
& \lesssim 2^{nk/2} 2^{\tilde{K}/2} \Vert \Delta_{k}(\tilde{f}_{k_{1}} f_{k_{2}}) \Vert_{L^{2}} 
\\
& \lesssim 2^{nk/2} 2^{\tilde{K}/2} \left( 2^{-\tilde{K}/2} 2^{nk_{1}/2} \Vert f_{k_{1}} \Vert_{Z_{k_{1}}} \Vert f_{k_{2}} \Vert_{Z_{k_{2}}} \right)
\\
& = 2^{n(k-k_{2})/2}  2^{nk_{1}/2} 2^{nk_{2}/2} \Vert f_{k_{1}} \Vert_{Z_{k_{1}}} \Vert f_{k_{2}} \Vert_{Z_{k_{2}}}
\end{split}
\end{equation}
which is the desired result. It remains to prove the following
\begin{equation}\label{Q6}
2^{nk/2} \Vert Q_{\geq \tilde{K}} ( \Delta_{k}(\tilde{f}_{k_{1}} f_{k_{2}})) \Vert_{Z_{k}} \lesssim 2^{- \frac{n}{2}\vert k - k_{2} \vert } 2^{k_{2}n/2} 2^{k_{1}n/2} \Vert f_{k_{1}} \Vert_{Z_{k_{1}}} \Vert f_{k_{2}}\Vert_{Z_{k_{2}}}. 
\end{equation} 
 First consider the case if $k_{2} <100$.  Then by virtue of $k_{1} \leq k_{2} +10 \lesssim 1 $, we have $\Vert f_{k_{i}} \Vert_{X_{k_{i}}} \lesssim \Vert f_{k_{i}} \Vert_{Z_{k_{i}}} $ for each $i=1,2$. Hence, it is enough to prove the following 
\begin{equation}\label{Q5}
\begin{split}
& 2^{j/2} 2^{nk/2} \Vert Q_{\geq \tilde{K}} \ind_{D_{k,j}} (\tilde{g}_{k_{1},j_{1}} * g_{k_{2}j_{2}}) \Vert_{L^{2}} 
\\
& \lesssim 2^{-\frac{n}{2}\vert k - k_{2} \vert } 2^{nk_{1}/2} 2^{j_{1}/2} \Vert g_{k_{1}j_{1}} \Vert_{L^{2}} 2^{j_{2}/2} 2^{k_{2}n/2} \Vert g_{k_{2}j_{2}} \Vert_{L^{2}}
\end{split}
\end{equation}
for any $g_{k_{i},j_{i}}$ supported in $ D_{k_{i},j_{i}}$. For future use, we prove this without the restriction that $k_{2} < 100$. This follows from Lemma \ref{cl}. Indeed, notice that $j < \max \{ j_{1},j_{2} \} +C$ for some constant $C$. Because otherwise we will have $0$ on the left hand side of \eqref{Q5}, see Lemma \ref{prod1}. Using this, and Lemma \ref{cl} we obtain
\begin{equation*}
\begin{split}
&  2^{j/2} 2^{nk/2} \Vert \ind_{D_{k,j}} (\tilde{g}_{k_{1},j_{1}} * g_{k_{2}j_{2}}) \Vert_{L^{2}} 
\\
& \lesssim 2^{\max \{ j_{1},j_{2} \}/2} 2^{nk/2} 2^{k_{1}n/2} 2^{\min \{j_{1},j_{2}\}/2} \Vert g_{k_{1},j_{1}} \Vert_{L^{2} } \Vert g_{k_{2},j_{2}} \Vert_{L^{2}} 
\\
& \lesssim 2^{n(k-k_{2})/2} 2^{nk_{2}/2}   2^{n k_{1}/2} 2^{j_{1}/2} 2^{j_{2}/2} \Vert g_{k_{1},j_{1}} \Vert_{L^{2} } \Vert g_{k_{2},j_{2}} \Vert_{L^{2}}
\end{split}
\end{equation*}
this gives \eqref{Q5} since $ k \leq k_{2} +20$. And this proves the result under the assumption that $ k_{2} < 100$.  \\
Assume in what follows that $k_{2} \geq 100$. 
To prove \eqref{Q6}, we write 
$$
f_{k_{1}} = \sum_{j \geq  0 } g_{k_{1},j} + \sum_{e \in \mathscr{E} ,k' \in [1,...,k_{1}+1]} f^{e}_{k_{1},k'}, 
$$
where each $ g_{k_{1},j} $ has Fourier transform supported in $ D_{k_{1},j}$ and each $ f^{e}_{k_{1},k'} \in Y^{e}_{k_{1},k'}$. We decompose $f_{k_{2}} $ similarly.  We consider several cases
\\
\textbf{Case 1:} $ f_{k_{1}} = g_{k_{1},j_{1}}$ and $f_{k_{2}} = g_{k_{2},j_{2}}$, this case was already proven in \eqref{Q5}. \\
\textbf{Case 2:} $f_{k_{1}} = f^{e}_{k_{1},k'} \in Y^{e}_{k_{1},k'}$ and $ f_{k_{2}} = g_{k_{2},j_{2}}$. In view of case $1$ we may assume that $ \mathcal{F}f^{e}_{k_{1},k'}$ is supported in $ D^{e,k'}_{k_{1}, \leq 2sk_{1}-80}$ and $ k' \in T_{k_{1}}$.  We claim that $j_{2} \geq \tilde{K}-C$ or the left hand side of \eqref{Q6} is $0$. Indeed, this follows immediately from Lemma \ref{prod1} because $ 2sk_{1} -80 < \tilde{K} - C' - 70$. Therefore, by Lemma \ref{prod1} we may assume that $ j_{2} \geq \tilde{K} -C $. We estimate
\begin{equation*}
\begin{split}
&2^{nk/2} \Vert Q_{\geq \tilde{K}} ( \Delta_{k}(\tilde{f}_{k_{1}}  g_{k_{2},j_{2}}) ) \Vert_{Z_{k}}
\\
& \leq2^{nk/2} \Vert Q_{\geq \tilde{K}} ( \Delta_{k}(\tilde{f}_{k_{1}}  g_{k_{2},j_{2}}) \Vert_{X_{k}}
\\
&  \lesssim2^{nk/2} \sum_{j=  j_{2} -C}^{j_{2}+C} 2^{j/2} \Vert Q_{j}( Q_{\geq \tilde{K}}( \Delta_{k}(\tilde{f}_{k_{1}} g_{k_{2},j_{2}} )) \Vert_{L^{2}}
\\
& \lesssim 2^{nk/2} 2^{j_{2}/2} \Vert \Delta_{k}(\tilde{f}_{k_{1}} g_{k_{2},j_{2}} )\Vert_{L^{2}} 
\\
& \lesssim 2^{nk/2} 2^{j_{2}/2} \Vert g_{k_{2},j_{2}} \Vert_{L^{2}} \Vert \tilde{f}_{k_{1}} \Vert_{L^{\infty}}
\\
& \lesssim 2^{nk/2} 2^{j_{2}/2} \Vert g_{k_{2},j_{2}} \Vert_{L^{2}}  2^{nk_{1}/2} \Vert f_{k_{1}} \Vert_{Z_{k_{1}}}.
\end{split}
\end{equation*}
\textbf{Case 3:} $ f_{k_{2}} \in Y^{e}_{k_{2},k'}$ and $ f_{k_{1}} = g_{k_{1},j_{1}}$. With $k_{1} + 10 \geq k_{2}$ ( This implies that $k_{1} $ is essentially the same as $k_{2}$). In this case we repeat the argument from case 2.  
\begin{equation*}
\begin{split}
& 2^{nk/2} \Vert Q_{\geq \tilde{K}} \left( \Delta_{k} ( f_{k_{2}} \tilde{g}_{k_{1},j_{1}} \right)\Vert_{Z_{k}}
\\
& \lesssim 2^{nk/2} \Vert Q_{\geq \tilde{K}} \left( \Delta_{k} ( f_{k_{2}} \tilde{g}_{k_{1},j_{1}} \right) \Vert_{X_{k}}
\\
& \lesssim 2^{nk/2} 2^{j_{1}/2} \Vert f_{k_{2}} \tilde{g}_{k_{1},j_{1}} \Vert_{L^{2}(\R^{n+1})} 
\\
& \lesssim 2^{nk/2} 2^{j_{1}/2} \Vert g_{k_{1},j_{1}} \Vert_{L^{2}(\R^{n+1})} 
\Vert f_{k_{2}} \Vert_{L^\infty(\R^{n+1})} 
\\
& \lesssim 2^{nk/2} \Vert g_{k_{1},j_{1}} \Vert_{X_{k_{1}}} 2^{nk_{2}/2} \Vert f_{k_{2}} \Vert_{Z_{k_{2}}} 
\end{split}
\end{equation*}
and the result follows because $ 2^{k_{1}} \approx 2^{k_{2}}$ in this case.\\
\textbf{Case 4:} $f_{k_{1}} \in Y^{e'}_{k_{1}}$ and $f_{k_{2}} \in Y^{e}_{k_{2},k'}$ and $ k_{1} +10 \geq k'$. Then in light of case 2 we may assume that $ \hat{f}_{k_{2}} $ is supported in $D_{k_{2}, \leq 2sk'}$.  $ \hat{f}_{k_{1}}$ is supported in $ D_{k_{1},\leq 2sk_{1}+10}$ by definition. So, the modulation of both functions is bounded by $2^{2sk_{1}+10}$. On the other hand, we have $Q_{\geq \tilde{K}}$. Therefore, by Lemma \ref{prod1} we get $ Q_{ \geq \tilde{K}} ( \tilde{f}_{k_{1}} f_{k_{2}}) = 0 $.
\\
These cases imply the result when $k_{1} \geq k_{2} -10$. So assume in the remaining cases that $k_{1} \leq k_{2} -10$. Then we must have $ k \geq k_{2} -c'$ for some constant $c'$. Indeed, if $ 2^{k_{1}} \ll 2^{k_{2}}$ and $ 2^{k} \ll 2^{k_{2}}$ then we will have $ \Delta_{k} (\tilde{f}_{k_{1}} f_{k_{2}}) = 0 $ and this gives the result trivially. So assume in what follows that $ | k - k_{2} | \lesssim 1 $ and $ 2^{k_{1}} \ll 2^{k_{2}}$.\\
\textbf{Case 5 :} $ f_{k_{2},k'} \in Y^{e}_{k_{2},k'}$ with $k'\in T_{k_{2}}$ and $ f_{k_{1}} = g_{k_{1},j_{1}}$ with $k' \leq k_{1} +10$.
In view of case $1$ we may assume that the Fourier transform of $ f_{k_{2}}$ is supported in $D^{e,k'}_{k_{2}, \leq 2k_{2}(s+q-1)-80}$. In particular, by Lemma \ref{prod1} we may assume that $ j_{1} \geq \tilde{K} - C$ for some constant $C$. Then 
\begin{equation*}
\begin{split}
& 2^{nk/2} \Vert \mathcal{F}^{-1} \left( \varphi_{k} (\xi) \varphi_{\geq \tilde{K}} (\tau+\vert \xi \vert^{2s}) ( \hat{f}_{k_{2},k'} * \hat{\tilde{g}}_{k_{1},j_{1}} \right) \Vert_{Z_{k}}
\\
& \leq 2^{nk/2} \Vert \mathcal{F}^{-1} \left( \hat{f}_{k_{2}, k'} * \hat{\tilde{g}}_{k_{1},j_{1}} \right) \Vert_{X_{k}} 
\\
& \lesssim 2^{nk/2} 2^{j_{1}/2} \Vert \hat{f}_{k_{2}, k'} * \hat{\tilde{g}}_{k_{1},j_{1}} \Vert_{L^{2}}.
\end{split}
\end{equation*}
So it suffices to prove 
\begin{equation}\label{M1}
\Vert \hat{f}_{k_{2}} * \hat{\tilde{g}}_{k_{1},j_{1}} \Vert_{L^{2}} \lesssim 2^{nk_{1}/2} \Vert g_{k_{1},j_{1}} \Vert_{L^{2}} \Vert f_{k_{2}} \Vert_{Y^{e}_{k_{2},k'}}.
\end{equation}
whenever $\hat{f}_{k_{2}}$ is supported in $D^{e,k'}_{k_{2},\leq 2k_{2}(s+q -1)-80}$. 
For future use, we prove \eqref{M1} without the restriction that $k' \leq k_{1} +10$. Since $ \hat{f}_{k_{2}}$ is supported $D^{e,k'}_{k_{2},\leq 2k_{2}(s+q-1) -80}$, then using \cite[Lemma 3.7]{subcritical} we may assume
$$
\hat{f}_{k_{2}} = \frac{\varphi_{\leq k' -80} (\xi_{1} - N)}{ \xi_{1} -N + \i/2^{k(2s-2+q)}} h(\xi',\tau)
$$
with 
$$ 
\Vert h \Vert_{L^{2}_{\xi',\tau}} \lesssim 2^{2(1-s)(k-k')} \gamma^{-1}_{k,k'} 2^{-k'(2s-1)/2} \Vert f_{k_{2}} \Vert_{Y^{e}_{k_{2},k'}}, 
$$ and $h$ is supported in $ \{ \vert \xi' \vert \lesssim 2^{k_{2}}, N(\xi',\tau) \approx 2^{k'}, N^{2 } + \vert \xi' \vert^{2} \approx 2^{2k} \}$. In particular,
$$
\Vert h \Vert_{L^{2}_{\xi',\tau}} \lesssim 2^{-k'(2s-1)/2} \Vert f_{k_{2}} \Vert_{Y^{e}_{k,k'}}.
$$
(Notice that in view of case 1, we are ignoring the $X_{k}$ error term coming from \cite[Lemma 8]{subcritical}).\\
Then \eqref{M1} is reduced to showing 
\begin{equation}\label{M2}
\Vert \hat{f}_{k_{2}} * \hat{\tilde{g}}_{k_{1},j_{1}} \Vert_{L^{2}} \lesssim 2^{(2s-1)k'/2} 2^{nk_{1}/2} \Vert h \Vert_{L^{2}_{\xi',\tau}} \Vert g_{k_{1},j_{1}} \Vert_{L^{2}} 
\end{equation}
using duality we obtain that the left hand side of \eqref{M2} is equal to 
\begin{equation}\label{M3}
\begin{split}
& \Vert \hat{f}_{k_{2}} * \hat{\tilde{g}}_{k_{1},j_{1}} \Vert_{L^{2}}
\\
& \Vert \int_{\R \times e^{\perp} \times \R} \hat{\tilde{g}}_{k_{1},j_{1}}(\eta_{1}e+ \eta',\beta) \hat{f}_{k_{2}}((\xi_{1}-\eta_{1})e+ \xi'- \eta', \tau - \beta) d\eta_{1} d \eta' d \beta \Vert_{L^{2}} 
\\
&
= \sup_{a \in L^{2}, \Vert a \Vert =1} \vert \int_{\R \times e^{\perp} \times \R} \int_{\R \times e^{\perp} \times \R}  \hat{\tilde{g}}_{k_{1},j_{1}}(\eta,\beta) \hat{f}_{k_{2}}(\xi,\tau) a( \xi_{1} + \eta_{1}, \xi'+ \eta', \tau+ \beta) d \xi d \tau d \eta d \beta \vert
\end{split}
\end{equation}
fix $ a $ with norm 1. Then notice 
\begin{equation}\label{M4}
\begin{split}
& \int_{\R \times e^{\perp} \times \R} \hat{f}_{k_{2}} (\xi,\tau) a( \xi_{1} + \eta_{1}, \xi'+ \eta', \tau+ \beta) d \xi_{1} d \xi' d \tau
\\
& = \int_{\R \times e^{\perp} } h(\xi',\tau) \int_{\R} \frac{ \varphi_{\leq k'- 80}(\xi_{1} - N)}{ \xi_{1} - N + \frac{i}{2^{k(2s-2+q)}}} a(\xi_{1}+ \eta_{1}, \xi'+\eta', \tau+ \beta) d\xi_{1} d \tau d \xi' 
\end{split}
\end{equation}
use the change of variable $ \tau = - ( N^{2} + \vert \xi' \vert^{2})^{s} $ then the Jacobian would be $c_{s} N ( N^{2} + \vert \xi' \vert^{2})^{s-1} $ and define $ \tilde{h}(\xi',N) = h(\xi', -(N^{2}+ \vert \xi' \vert^{2})^{s}) N ( N^{2}+\vert \xi' \vert^{2})^{s-1}$. Then using this change of variable, \eqref{M4} becomes
\begin{equation}\label{M5}
\int_{\R \times e^{\perp} } \tilde{h}( \xi',N) \int_{\R} \frac{ \eta_{\leq k'-80}(\xi_{1} -N)}{\xi_{1} -N + \frac{i}{2^{k(2s-2+q)}}} a(\xi_{1} + \eta_{1}, \xi'+ \eta', -( N^{2}+ \vert \xi' \vert^{2})^{s} + \beta) d \xi_{1} d N d \xi'
\end{equation}
plug \eqref{M5} into \eqref{M3} to get 
\begin{equation}\label{M6}
\begin{split}
& \vert \int_{\R \times e^{\perp} \times \R}  \hat{\tilde{g}}_{k_{1},j_{1}} (\eta_{1}e + \eta',\beta) \int_{e^{\perp}} \int_{\R} \tilde{h}(\xi',N) \\
& \times \left( \int_{\R} \frac{ \varphi_{\leq k'-80}(\xi_{1} -N)}{\xi_{1} -N + \frac{i}{2^{(k(2s-2+q)}}} a(\xi_{1}+ \eta_{1}, \xi'+\eta', -(N^{2}+\vert \xi'\vert^{2})^{s}+\beta) d \xi_{1} \right) d N d \xi' d \beta d \eta \vert.
\end{split}
\end{equation}
Apply H\"{o}lder in the variable $ \beta$ to obtain that \eqref{M6} is bounded by 
\begin{equation}\label{M7}
\begin{split}
& \vert \int_{ \R \times e^{\perp}} \Vert \hat{\tilde{g}}_{k_{1},j_{1}}( \eta_{1} e+ \eta',\cdot ) \Vert_{L^{2}_{\beta}} \int_{e^{\perp} \times \R} \tilde{h}(\xi',N) \Vert H( a( \cdot+ \eta_{1}, \xi'+ \eta', \beta )(N) \Vert_{L^{2}_{\beta}} d N d \xi' d \eta_{1} d \eta' \vert,
\end{split}
\end{equation}
where $H$ is the Hilbert transform. 
Next, apply H\"{o}lder in the variable $N$ and using boundedness of the Hilbert transform on $L^{2} $ we bound \eqref{M7} by 
\begin{equation}\label{M8}
\begin{split}
& \vert \int_{ \R \times e^{\perp}} \Vert \hat{\tilde{g}}_{k_{1},j_{1}}( \eta_{1} e+ \eta',\cdot ) \Vert_{L^{2}_{\beta}} \int_{e^{\perp} } \Vert \tilde{h}(\xi',\cdot)\Vert_{L^{2}_{N}} \Vert a( \cdot, \xi'+ \eta', \cdot ) \Vert_{L^{2}_{\xi_{1}, \beta}} d \xi' d \eta_{1} d \eta' \vert
\end{split}
\end{equation}
Lastly, we apply H\"{o}lder in the variable $ \xi'$ and obtain that \eqref{M8} is bounded by 
\begin{equation}
\begin{split}
& \Vert \tilde{h} \Vert_{L^{2}_{\xi',N}} \int_{\R \times e^{\perp}} \Vert \hat{\tilde{g}}_{k_{1},j_{1}} ( \eta_{1}e+\eta',\cdot) \Vert_{L^{2}_{\beta}} 
\\
& \lesssim 2^{k_{1}n/2} \Vert \tilde{h} \Vert_{L^{2}_{\xi',N}} \Vert g_{k_{1},j_{1}} \Vert_{L^{2}}
\end{split}
\end{equation}
and notice that $ \Vert \tilde{h} \Vert_{L^{2}} \lesssim 2^{(2s-1)k'/2} \Vert h \Vert_{L^{2}}$. Thus, we have proven \eqref{M2}. Which proves the result in this case.
\\
\textbf{Case 6: } $f_{k_{2}} \in Y^{e}_{k_{2},k'}$, $k'\in T_{k_{2}}$ and $ f_{k_{1}} = g_{k_{1},j_{1}}$ , $k' \geq k_{1} +10$ and $j_{1} \geq \tilde{K}$. In light of case $1$ we may assume that $f_{k_{2}}$ is supported in $D^{e,k'}_{k_{2}, \leq 2k_{2}(s+q-1)-80} $. Then we want to prove 
\begin{equation}\label{Q9}
2^{nk/2} \Vert \mathcal{F}^{-1} \left( \varphi_{k}(\xi) \varphi_{\geq \tilde{K}} (\tau+\vert \xi \vert^{2s}) ( \hat{f}_{k_{2}} * \hat{\tilde{g}}_{k_{1},j_{1}}) \right) \Vert_{Z_{k}} \lesssim 2^{j_{1}/2} 2^{nk_{1}/2} \Vert g_{k_{1},j_{1}} \Vert_{L^{2}} 2^{k_{2}n/2} \Vert f_{k_{2}} \Vert_{Z_{k_{2}}}.
\end{equation}
To that end, write $ f_{k_{2}} = f_{k_{2},\leq j_{1}-c} + f_{k_{2}, \geq j_{1}+c} + f_{k_{2}, [j_{1}-c,j_{1}+c]} $ then 
\begin{equation}\label{sxcv}
\begin{split}
& 2^{nk/2} \Vert \mathcal{F}^{-1} \left( \varphi_{k}(\xi) \varphi_{\geq \tilde{K}} (\tau+\vert \xi \vert^{2s}) ( \hat{f}_{k_{2}} * \hat{\tilde{g}}_{k_{1},j_{1}}) \right) \Vert_{Z_{k}}
\\
& \leq 2^{nk/2} \Vert  \mathcal{F}^{-1} \left( \varphi_{k}(\xi) \varphi_{\geq \tilde{K}} (\tau+\vert \xi \vert^{2s}) ( \hat{f}_{k_{2},\leq j_{1}-c} * \hat{\tilde{g}}_{k_{1},j_{1}} ) \right) \Vert_{Z_{k}} 
\\
& +  2^{nk/2} \Vert \mathcal{F}^{-1} \left( \varphi_{k}(\xi) \varphi_{\geq \tilde{K}} (\tau+\vert \xi \vert^{2s}) ( \hat{f}_{k_{2}, \geq j_{1}+c} * \hat{\tilde{g}}_{k_{1},j_{1}})  \right) \Vert_{Z_{k}} 
\\
&+ 
2^{nk/2} \Vert \mathcal{F}^{-1} \left( \varphi_{k}(\xi) \varphi_{\geq \tilde{K}} (\tau+\vert \xi \vert^{2s}) ( \hat{f}_{k_{2},[j_{1}-c,j_{1}+c]} * \hat{\tilde{g}}_{k_{1},j_{1}} ) \right) \Vert_{Z_{k}}.
\end{split}
\end{equation}
For the last term we use the embedding of $X_{k} $ into $Y^{e}_{k_{2},k'}$ and case 1. It remains to estimate the first two terms. To that end, we start with the first term on the right hand side of \eqref{sxcv}
\begin{equation*}
\begin{split}
& 2^{nk/2} \Vert \mathcal{F}^{-1} \left( \varphi_{k}(\xi) \varphi_{\geq \tilde{K}} (\tau+\vert \xi \vert^{2s}) ( \hat{f}_{k_{2},\leq j_{1}-c} * \hat{\tilde{g}}_{k_{1},j_{1}} ) \right) \Vert_{Z_{k}}
\\
& \leq 2^{nk/2} \Vert \mathcal{F}^{-1} \left( \varphi_{\geq \tilde{K}}(\tau + \vert \xi \vert^{2s}) (\hat{f}_{k_{2}, \leq j_{1}-c} * \hat{\tilde{g}}_{k_{1},j_{1}} ) \right) \Vert_{X_{k}} 
\\
& \lesssim 2^{nk/2} 2^{j_{1}/2} \Vert \hat{f}_{k_{2}, \leq j_{1} -c} * \hat{\tilde{g}}_{k_{1},j_{1}} \Vert_{L^{2}}
\end{split}
\end{equation*}
and we want to prove 
\begin{equation}\label{Q10}
\Vert \hat{f}_{k_{2}} * \hat{\tilde{g}}_{k_{1},j_{1}} \Vert_{L^{2}} \lesssim 2^{nk_{1}/2} \Vert g_{k_{1},j_{1}} \Vert_{L^{2}} \Vert f_{k_{2}} \Vert_{Y^{e}_{k_{2},k'}}
\end{equation}
But this follows from \eqref{M1}. To conclude case 6, we need to prove the estimate for the second term in \eqref{Q9}. In other words, we want to prove
\begin{equation}\label{Q18}
2^{nk/2} \Vert \Delta_{k} Q_{\geq \tilde{K}} ( {f}_{k_{2}, \geq j_{1} + c} \tilde{g}_{k_{1},j_{1}}) \Vert_{Z_{k}} \lesssim 2^{nk_{1}/2}  2^{nk_{2}/2} \Vert f_{k_{2}} \Vert_{Z_{k_{2}}} \Vert g_{k_{1},j_{1}} \Vert_{X_{k_{1}}}.
\end{equation}
To that end, write
\begin{equation}\label{Q19}
\begin{split}
& ( \i \partial_{t} -  (-\Delta)^s + \i ) (f_{k_{2},\geq j_{1} +c} \tilde{g}_{k_{1},j_{1}} )
\\
 & = ( \i \partial_{t} -  (-\Delta)^s + \i )(f_{k_{2},\geq j_{1}+c}). \tilde{g}_{k_{1},j_{1}} 
 \\
 & + f_{k_{2}, \geq j_{1}+c} ( \i \partial_{t} -  (-\Delta)^s ) \tilde{g}_{k_{1},j_{1}} 
 \\
 & - H_{(-\Delta)^{s}}(f_{k_{2}, \geq j_{1}+c } ,  \tilde{g}_{k_{1},j_{1}}) 
 \end{split}
\end{equation}
where the last term is the commutator of the fractional laplacian. Then we bound the left hand side of \eqref{Q18} by 
\begin{equation}\label{Q20}
\begin{split}
& 2^{nk/2} \Vert \Delta_{k} Q_{\geq \tilde{K}}  \left(\frac{   ( \i \partial_{t} -  (-\Delta)^s + \i ) (f_{k_{2},\geq j_{1} +c} ) . \tilde{g}_{k_{1},j_{1}}}{ \i \partial_{t} -   (-\Delta)^s +\i}  \right) \Vert_{Z_{k}}
\\
& + 2^{nk/2} \Vert \Delta_{k} Q_{\geq \tilde{K}} \left( \frac{  f_{k_{2}, \geq j_{1}+c} ( \i \partial_{t} -  (-\Delta)^s ) \tilde{g}_{k_{1},j_{1}}}{\i \partial_{t} - (-\Delta)^s +\i} \right) \Vert_{Z_{k}}
\\
& + 2^{nk/2} \Vert \Delta_{k} Q_{\geq \tilde{K}} \left( \frac{  H_{(-\Delta)^{s}}(f_{k_{2}, \geq j_{1}+c } ,  \tilde{g}_{k_{1},j_{1}})}{\i \partial_{t} -  (-\Delta)^s + \i} \right) \Vert_{Z_{k}}
\end{split}
\end{equation}
if we bound each term  in \eqref{Q20} with the right hand side of \eqref{Q18} then this concludes case 6. Starting with the first term, notice that since $2^{k_{1}} \ll 2^{k'}$ then the support of $ \hat{f}_{k_{2} \geq j_{1}+c} * \hat{\tilde{g}}_{k_{1},j_{1}} $ is contained in $\{ (\xi,\tau): \xi. e \approx 2^{k'} \}$. Therefore, we bound the first term as follows 
\begin{equation*}
\begin{split}
& 2^{nk/2} \Vert \Delta_{k} Q_{\geq \tilde{K}}  \left(\frac{   ( \i \partial_{t} - (-\Delta)^s + \i ) (f_{k_{2},\geq j_{1} +c} ) . \tilde{g}_{k_{1},j_{1}}}{ \i \partial_{t} -   (-\Delta)^s +\i}  \right) \Vert_{Z_{k}}
\\
& \sum_{k'' = k'-c}^{\min\{k'+c,k+1\}} \Vert \Delta_{k} Q_{\geq \tilde{K}}  \left(\frac{   ( \i \partial_{t}  - (-\Delta)^s + \i ) (f_{k_{2},\geq j_{1} +c} ) . \tilde{g}_{k_{1},j_{1}}}{ \i \partial_{t} -  (-\Delta)^s +\i}  \right) \Vert_{Y^{e}_{k,k''}}
\\
& \lesssim 2^{nk/2} \gamma_{k,k'} 2^{-k'(2s-1)/2} \Vert ( \i \partial_{t} - (-\Delta)^s + \i ) (f_{k_{2},\geq j_{1} +c} ) . \tilde{g}_{k_{1},j_{1}} \Vert_{L^{1}_{e}L^{2}_{e^{\perp},t}}
\\
& \leq 2^{nk/2} \Vert g_{k_{1},j_{1}} \Vert_{L^{\infty}} \gamma_{k,k'} 2^{-k'(2s-1)/2} \Vert ( \i \partial_{t} -(-\Delta)^s + \i ) (f_{k_{2},\geq j_{1} +c} ) \Vert_{L^{1}_{e}L^{2}_{e^{\perp},t}}
\\
& \lesssim  2^{nk/2} \Vert f_{k_{2}} \Vert_{Z_{k_{2}}} \Vert g_{k_{1},j_{1}} \Vert_{L^{\infty}}
\\
& \lesssim 2^{nk_{2}/2} \Vert f_{k_{2}} \Vert_{Z_{k_{2}}} 2^{nk_{1}/2} \Vert g_{j_{1},k_{1}} \Vert_{X_{k_{1}}}
\end{split}
\end{equation*}
which is the right hand side of \eqref{Q18}. Next, we do the same with second term in \eqref{Q20}. Using Lemma \ref{prod2} we obtain
\begin{equation*}
\begin{split}
& 2^{nk/2} \Vert \Delta_{k} Q_{\geq \tilde{K}} \left( \frac{  f_{k_{2}, \geq j_{1}+c} ( \i \partial_{t} -  (-\Delta)^s ) \tilde{g}_{k_{1},j_{1}}}{\i \partial_{t} - (-\Delta)^s +\i} \right) \Vert_{Z_{k}}
\\
& \lesssim 2^{nk/2} \Vert \Delta_{k} Q_{\geq \tilde{K}} \left( \frac{  f_{k_{2}, \geq j_{1}+c} ( \i \partial_{t} -   (-\Delta)^s ) \tilde{g}_{k_{1},j_{1}}}{\i \partial_{t} - (-\Delta)^s +\i} \right) \Vert_{X_{k}}
\\
& \lesssim 2^{nk/2} \Vert f_{k_{2}, \geq j_{1}+c}   ( \i \partial_{t} - (-\Delta)^s ) \tilde{g}_{k_{1},j_{1}} \Vert_{L^{2}} \sum_{j=j_{1}-c}^{\infty} 2^{-j/2} 
\\
& \lesssim 2^{nk/2} 2^{-j_{1}/2} \Vert f_{k_{2}, \geq j_{1}+c}   ( \i \partial_{t} -   (-\Delta)^s ) \tilde{g}_{k_{1},j_{1}} \Vert_{L^{2}}
\end{split}
\end{equation*}
using Lemma \ref{es1} we bound this by 
\begin{equation*}
2^{nk/2} \Vert f_{k_{2}} \Vert_{Z_{k_{2}}} 2^{nk_{1}/2} \Vert g_{k_{1},j_{1}} \Vert_{X_{k_{1}}}
\end{equation*}
which gives the result since $ 2^{k } \approx 2^{k_{2}}$.  It remains to prove the estimate for the third term in \eqref{Q20}. To that end, we estimate as above
\begin{equation*}
\begin{split}
& 2^{nk/2} \Vert \Delta_{k}Q_{\geq \tilde{K}} \left( \frac{  H_{(-\Delta)^{s}}(f_{k_{2}, \geq j_{1}+c } ,  \tilde{g}_{k_{1},j_{1}})}{\i \partial_{t} - (-\Delta)^s + \i} \right) \Vert_{Z_{k}}
\\
& \lesssim 2^{nk/2} \Vert \Delta_{k} Q_{\geq \tilde{K}} \left( \frac{  H_{(-\Delta)^{s}}(f_{k_{2}, \geq j_{1}+c } ,  \tilde{g}_{k_{1},j_{1}})}{\i \partial_{t} - (-\Delta)^s + \i} \right) \Vert_{X_{k}}
\\
& \lesssim 2^{nk/2} \Vert H_{(-\Delta)^{s}} ( f_{k_{2}, \geq j_{1}+c} , \tilde{g}_{k_{1,j_{1}}} )\Vert_{L^{2}} \sum_{j= \tilde{K} }^{\infty} 2^{-j/2} 
\\
& \lesssim 2^{nk/2} 2^{-\tilde{K}/2} \Vert H_{(-\Delta)^{s}}(f_{k_{2},\geq j_{1}+c} , \tilde{g}_{k_{1},j_{1}}) \Vert_{L^{2}}
\\
& \lesssim 2^{nk/2} 2^{-\tilde{K} /2} 2^{\tilde{K}/2} 2^{nk_{1}/2} \Vert f_{k_{2}} \Vert_{Z_{k_{2}}} \Vert g \Vert_{Z_{k_{1}}}
\end{split}
\end{equation*}
where we used Lemma \ref{comest} in the last line. Using that $ 2^{k_{2}} \approx 2^{k}$ we obtain
\begin{equation*}
\begin{split}
& 2^{nk/2} \Vert \Delta_{k} Q_{\geq \tilde{K}} \left( \frac{  H_{(-\Delta)^{s}}(f_{k_{2}, \geq j_{1}+c } ,  \tilde{g}_{k_{1},j_{1}})}{\i \partial_{t} - (-\Delta)^s + \i} \right) \Vert_{Z_{k}}
\\
& \lesssim 2^{nk_{2}/2} \Vert f_{k_{2}} \Vert_{Z_{k_{2}}} . 2^{nk_{1}/2} \Vert g_{k_{1},j_{1}} \Vert_{Z_{k_{1}}}
\end{split}
\end{equation*}
this concludes case 6. 
\\
\textbf{Case 7:} the last case to consider is $ f_{k_{2}} \in Y^{e}_{k_{2},k'}$ and $ f_{k_{1} \leq \tilde{K}} \in Z_{k_{1}}$ with $ 2^{k_{1}} \leq 2^{k'-10}$, and $k' \in T_{k_{2}}$, and $f_{k_{1},\leq \tilde{K}}$ has modulation controlled by $2^{\tilde{K}}$.  Using  \eqref{Q20} we estimate 
\begin{equation}\label{Q21}
\begin{split}
& 2^{nk/2} \Vert \Delta_{k} Q_{\geq \tilde{K}} ( \tilde{f}_{k_{1}, \leq \tilde{K}} f_{k_{2}}) \Vert_{Z_{k}}
\\
& \lesssim  2^{nk/2} \Vert \Delta_{k} Q_{\geq \tilde{K}}  \left(\frac{   ( \i \partial_{t} - (-\Delta)^s + \i ) (f_{k_{2}} ) . \tilde{f}_{k_{1} , \leq \tilde{K}}}{ \i \partial_{t} -  (-\Delta)^s +\i}  \right) \Vert_{Z_{k}}
\\
& + 2^{nk/2} \Vert \Delta_{k} Q_{\geq \tilde{K}} \left( \frac{  f_{k_{2}} ( \i \partial_{t} -   (-\Delta)^s ) \tilde{f}_{k_{1}, \leq \tilde{K}}}{\i \partial_{t} - (-\Delta)^s +\i} \right) \Vert_{Z_{k}}
\\
& + 2^{nk/2} \Vert \Delta_{k} Q_{\geq \tilde{K}} \left( \frac{  H_{(-\Delta)^{s}}(f_{k_{2} } ,  f_{k_{1}, \leq \tilde{K}})}{\i \partial_{t} - (-\Delta)^s + \i} \right) \Vert_{Z_{k}}
\end{split}
\end{equation}
and it is enough to estimate each term in the right hand side of \eqref{Q21} by 
$$
2^{nk_{1}/2} \Vert f_{k_{1}, \leq \tilde{K}} \Vert_{Z_{k_{1}}} . 2^{nk_{2}/2} \Vert f_{k_{2}} \Vert_{Z_{k_{2}}}.
$$
To that end, we start with the first term.  Notice that the support of $ \mathcal{F} (f_{k_{2}} \tilde{f}_{k_{1}, \leq \tilde{K}})$ is contained in $ \{ (\xi,\tau) : \langle \xi , e\rangle \approx 2^{k'} \} $. Hence, 

\begin{equation*}
\begin{split}
& 2^{nk/2} \Vert \Delta_{k} Q_{\geq \tilde{K}}  \left(\frac{   ( \i \partial_{t} - (-\Delta)^s + \i ) (f_{k_{2}} ) . \tilde{f}_{k_{1}, \leq \tilde{K}}}{ \i \partial_{t} -  (-\Delta)^s +\i}  \right) \Vert_{Z_{k}}
\\
& \leq 2^{nk/2} \sum_{k''=k'-c}^{\min \{ k'+c,k+1 \} } \Vert \Delta_{k} Q_{\geq \tilde{K}}  \left(\frac{   ( \i \partial_{t} - (-\Delta)^s + \i ) (f_{k_{2}} ) . \tilde{f}_{k_{1}, \leq \tilde{K}}}{ \i \partial_{t} -   (-\Delta)^s +\i}  \right) \Vert_{Y^{e}_{k,k''}}
\\
& \lesssim 2^{nk/2} \gamma_{k_{2},k'} 2^{-k'(2s-1)/2} \Vert ( \i \partial_{t} -  (-\Delta)^s + \i ) (f_{k_{2}} ) . \tilde{f}_{k_{1}, \leq \tilde{K}} \Vert_{L^{1}_{e}L^{2}_{e^{\perp},t}}
\\
& \leq 2^{nk/2} \Vert f_{k_{1}, \leq \tilde{K}} \Vert_{L^{\infty}} \gamma_{k_{2},k'} 2^{-k'(2s-1)/2} \Vert ( \i \partial_{t} - (-\Delta)^s + \i ) (f_{k_{2}} ) \Vert_{L^{1}_{e}L^{2}_{e^{\perp},t}}
\\
& \lesssim  2^{nk/2} \Vert f_{k_{2}} \Vert_{Z_{k_{2}}} \Vert f_{k_{1}, \leq \tilde{K}} \Vert_{L^{\infty}}
\\
& \lesssim 2^{nk_{2}/2} \Vert f_{k_{2}} \Vert_{Z_{k_{2}}} 2^{nk_{1}/2} \Vert f_{k_{1}, \leq \tilde{K}} \Vert_{Z_{k_{1}}}.
\end{split}
\end{equation*}
Next, we estimate the second term on the right hand side of \eqref{Q21}
\begin{equation*}
\begin{split}
& 2^{nk/2} \Vert \Delta_{k} Q_{\geq \tilde{K}} \left( \frac{  f_{k_{2}} ( \i \partial_{t} -  (-\Delta)^s ) \tilde{f}_{k_{1}, \leq \tilde{K}}}{\i \partial_{t} -(-\Delta)^s +\i} \right) \Vert_{Z_{k}}
\\
& \lesssim 2^{nk/2} \Vert \Delta_{k} Q_{\geq \tilde{K}} \left( \frac{  f_{k_{2}} ( \i \partial_{t} -  (-\Delta)^s)  \tilde{f}_{k_{1}, \leq \tilde{K}}}{\i \partial_{t} - (-\Delta)^s +\i} \right) \Vert_{X_{k}}
\\
& \lesssim 2^{nk/2} \Vert f_{k_{2}}   ( \i \partial_{t} - (-\Delta)^{s}   ) \tilde{f}_{k_{1}, \leq \tilde{K}} \Vert_{L^{2}} \sum_{j=\tilde{K}}^{\infty} 2^{-j/2} 
\\
& \lesssim 2^{nk/2} 2^{-\tilde{K}/2} \Vert f_{k_{2}}   ( \i \partial_{t} - (-\Delta)^s ) \tilde{f}_{k_{1}, \leq \tilde{K}} \Vert_{L^{2}}
\end{split}
\end{equation*}
using Lemma \ref{es1} and the fact that the modulation of $f_{k_{1}, \leq \tilde{K}}$ is controlled by $ 2^{\tilde{K}}$ we obtain the bound
\begin{equation*}
\begin{split}
& 2^{nk/2} 2^{-\tilde{K}/2} \Vert f_{k_{2}}   ( \i \partial_{t} -   (-\Delta)^s)  \tilde{f}_{k_{1}, \leq \tilde{K}} \Vert_{L^{2}}
\\
& \lesssim 2^{-\tilde{K}/2}  2^{nk_{2}/2} \Vert f_{k_{2}} \Vert_{Z_{k_{2}}} 2^{n k_{1}/2} \Vert f_{k_{1}, \leq \tilde{K}} \Vert_{Z_{k_{1}}}  2^{\tilde{K}} 2^{-\tilde{K}/2}
\end{split}
\end{equation*}
this gives the desired bound. Lastly, we estimate the third term in the right hand side of \eqref{Q21}. 
\begin{equation*}
\begin{split}
& 2^{nk/2} \Vert \Delta_{k} Q_{\geq \tilde{K}} \left( \frac{  H_{(-\Delta)^{s}}(f_{k_{2} } ,  \tilde{f}_{k_{1}, \leq \tilde{K}})}{\i \partial_{t} -(-\Delta)^s + \i} \right) \Vert_{Z_{k}}
\\
& \lesssim 2^{nk/2} \Vert \Delta_{k} Q_{\geq \tilde{K}} \left( \frac{  H_{(-\Delta)^{s}}(f_{k_{2} } ,  \tilde{f}_{k_{1}, \leq \tilde{K}})}{\i \partial_{t} - (-\Delta)^s + \i} \right) \Vert_{X_{k}}
\\
& \lesssim 2^{nk/2} \Vert H_{(-\Delta)^{s}} ( f_{k_{2}} , \tilde{f}_{k_{1}, \leq \tilde{K}} )\Vert_{L^{2}} \sum_{j= \tilde{K} }^{\infty} 2^{-j/2} 
\\
& \lesssim 2^{nk/2} 2^{-\tilde{K}/2} \Vert H_{(-\Delta)^{s}}(f_{k_{2}} , \tilde{f}_{k_{1}, \leq \tilde{K}}) \Vert_{L^{2}}
\\
& \lesssim 2^{nk/2} 2^{-\tilde{K} /2} 2^{nk_{1}/2} 2^{\tilde{K}/2} \Vert f_{k_{1}, \leq \tilde{K}} \Vert_{Z_{k_{1}}} \Vert f_{k_{2}} \Vert_{Z_{k_{2}}}
\end{split}
\end{equation*}
where we used Lemma \ref{comest} in the last line. 
This finishes the proof of case 7.

\end{proof}
As consequence of Theorem \ref{Bilin1} we obtain the following corollaries.
\begin{corollary}\label{AlgebraF}
Let $ n \geq 3 $ and $ s \in (1/2,1)$ then the following holds
\begin{enumerate}
    \item For any $ \sigma \in (0, \frac{n}{2} ]$ and any $ f,g \in F^{\frac{n}{2}} \cap F^{\sigma}$ the following inequality holds
    $$
    \Vert f g \Vert_{F^{\sigma}} \lesssim \Vert f \Vert_{F^{\frac{n}{2}}} \Vert g \Vert_{F^{\sigma}}.
    $$
    In particular, we have 
    $$
    \Vert f g \Vert_{F^{\frac{n}{2}}} \lesssim \Vert f \Vert_{F^{\frac{n}{2}}} \Vert g \Vert_{F^{\frac{n}{2}}}.
    $$
    \item For any $ \sigma \geq \frac{n}{2}$ we have the following
    $$
    \Vert f g \Vert_{F^{\sigma}} \lesssim \Vert f \Vert_{F^{\sigma}} \Vert g \Vert_{F^{\frac{n}{2}}} + \Vert f \Vert_{F^{\frac{n}{2}}} \Vert g \Vert_{F^{\sigma}}.
    $$
\item for any $ \sigma >0$ we have 
        $$
        \Vert f | \nabla|^{-1}g \Vert_{{F}^{\sigma}} \lesssim \Vert f \Vert_{{F}^{\frac{n-2}{2}}} \Vert g \Vert_{{F}^{\sigma}} + \Vert f \Vert_{{F}^{\sigma}} \Vert g \Vert_{F^{\frac{n-2}{2}}}.
        $$
\item If $ n \geq 5$, $ \sigma \geq \frac{n-2}{2}$ then
$$
\Vert | \nabla|^{-1} (fg) \Vert_{F^{\sigma}} \lesssim \Vert f \Vert_{F^{\frac{n-2}{2}}} \Vert g \Vert_{F^{\sigma}} + \Vert f \Vert_{F^{\sigma}} \Vert g \Vert_{F^{\frac{n-2}{2}}}.
$$
\item for any $ \sigma >0$ the following asymmetric inequality holds 
$$
\Vert f g \Vert_{F^{\sigma}} \lesssim \Vert f \Vert_{F^{\frac{n-2}{2}}} \Vert g \Vert_{F^{\sigma +1}} + \Vert f \Vert_{F^{\sigma}} \Vert g \Vert_{F^{\frac{n}{2}}}.
$$
\item and more generally, for any $ \sigma > 0 $, and any $ \epsilon > 0$ with $ \epsilon < \frac{n}{2}$ we have 
$$
\Vert fg \Vert_{{F}^{\sigma
}} \lesssim \Vert f \Vert_{{F}^{\frac{n-2\epsilon}{2}}} \Vert g \Vert_{{F}^{\sigma+\epsilon}} + \Vert f \Vert_{{F}^{\sigma}} \Vert g \Vert_{{F}^{\frac{n}{2}}}.
$$
\end{enumerate}
\end{corollary}
\begin{proof}
Notice that $ (3) $ follows from $(5)$ and $(5)$ follows from $(6)$. So we first begin by proving $ (6)$. It suffices to prove 
\begin{equation}\label{reddalg1}
\begin{split}
& \sum_{k =- \infty}^{\max \{ k_{1},k_{2} \} +C} 2^{k \sigma} \Vert f_{k_{1}} g_{k_{2}} \Vert_{Z_{k}} 
\\
& \lesssim  2^{k_{1} \frac{ n-2 \epsilon}{2}} \Vert f_{k_{1}} \Vert_{Z_{k_{1}}} 2^{k_{2} ( \sigma +\epsilon)} \Vert g_{k_{2}} \Vert_{Z_{k_{2}}} + 2^{\sigma k_{1}} \Vert f_{k_{1}} \Vert_{Z_{k_{1}}} 2^{k_{2} \frac{n}{2}} \Vert g_{k_{2}} \Vert_{Z_{k_{2}}}.
\end{split}
\end{equation}
If $ \max \{ k_{1},k_{2} \} =k_{1}$, then by applying Theorem \ref{Bilin1} we obtain
\begin{equation*}
\sum_{k =- \infty}^{\max \{ k_{1},k_{2} \} +C} 2^{k \sigma} \Vert f_{k_{1}} g_{k_{2}} \Vert_{Z_{k}} \lesssim  2^{\sigma k_{1}} \Vert f_{k_{1}} \Vert_{Z_{k_{1}}} 2^{k_{2} \frac{n}{2}} \Vert g_{k_{2}} \Vert_{Z_{k_{2}}}.
\end{equation*}
If on the other hand $ \max \{ k_{1},k_{2} \} =k_{2}$ then again we apply Theorem \ref{Bilin1} to obtain
\begin{equation*}
\begin{split}
& \sum_{k =- \infty}^{\max \{ k_{1},k_{2} \} +C} 2^{k \sigma} \Vert f_{k_{1}} g_{k_{2}} \Vert_{Z_{k}} 
\\
& \lesssim  2^{\frac{n}{2} k_{1}} \Vert f_{k_{1}} \Vert_{Z_{k_{1}}} 2^{k_{2} \sigma} \Vert g_{k_{2}} \Vert_{Z_{k_{2}}}
\\
& \leq  2^{-\epsilon k_{2} } 2^{\frac{n}{2} k_{1}} \Vert f_{k_{1}} \Vert_{Z_{k_{1}}} 2^{k_{2} (\sigma+\epsilon)} \Vert g_{k_{2}} \Vert_{Z_{k_{2}}}
\\
& \lesssim  2^{k_{1} \frac{ n-2 \epsilon}{2}} \Vert f_{k_{1}} \Vert_{Z_{k_{1}}} 2^{k_{2} ( \sigma +\epsilon)} \Vert g_{k_{2}} \Vert_{Z_{k_{2}}},
\end{split}
\end{equation*}
where in the last line we used the assumption that $k_{1} \leq k_{2}$. This proves \eqref{reddalg1}, which implies $(6)$ and hence proves $(5),(3)$. Next, we prove $(1)$. It suffices to prove 
\begin{equation}\label{reddalg2}
\sum_{k=-\infty}^{\max \{ k_{1},k_{2} \} +C } 2^{\sigma k} \Vert f_{k_{1}} g_{k_{2}} \Vert_{Z_{k}} \lesssim 2^{k_{1}\frac{n}{2}} \Vert f_{k_{1}} \Vert_{Z_{k_{1}}} 2^{k_{2} \sigma} \Vert g_{k_{2}} \Vert_{Z_{k_{2}}}.
\end{equation}
If $ \max\{ k_{1},k_{2} \} = k_{2}$ then
by Theorem \ref{Bilin1} we obtain \eqref{reddalg2} directly. If on the other hand $ \max\{ k_{1},k_{2} \} = k_{1}$ then by Theorem \ref{Bilin1} we obtain 
\begin{equation*}
\begin{split}
& \sum_{k=-\infty}^{ k_{1} +C } 2^{\sigma k} \Vert f_{k_{1}} g_{k_{2}} \Vert_{Z_{k}} \lesssim 2^{\sigma k_{1}} \Vert f_{k_{1}} \Vert_{Z_{k_{1}}} 2^{k_{2} \frac{n}{2}} \Vert g_{k_{2}} \Vert_{Z_{k_{2}}}
\\
& \lesssim 2^{(k_{1}-k_{2})(\sigma - \frac{n}{2})} 2^{k_{1}\frac{n}{2}} \Vert f_{k_{1}} \Vert_{Z_{k_{1}}} 2^{k_{2} \sigma} \Vert g_{k_{2}} \Vert_{Z_{k_{2}}}
\\
& \lesssim 2^{k_{1}\frac{n}{2}} \Vert f_{k_{1}} \Vert_{Z_{k_{1}}} 2^{k_{2} \sigma} \Vert g_{k_{2}} \Vert_{Z_{k_{2}}}
\end{split}
\end{equation*}
where in the last line we used that $ k_{1} \geq k_{2}$ and $ 0< \sigma \leq \frac{n}{2}$. This proves $(1)$. As for $(2)$ it follows directly from applying Theorem \ref{Bilin1}. It remains to prove $(4)$. It suffices to prove 
\begin{equation}\label{reddalge4}
\begin{split}
& \sum_{k=-\infty}^{\max \{k_{1},k_{2} \}+C} 2^{(\sigma-1)k} \Vert f_{k_{1}} g_{k_{2}} \Vert_{Z_{k}} \lesssim 2^{ \min\{k_{1},k_{2} \} \frac{n-2}{2} } 2^{ \max \{ k_{1},k_{2} \} \sigma } \Vert f_{k_{1}} \Vert_{Z_{k_{1}}} \Vert g_{k_{2}} \Vert_{Z_{k_{2}}}.
\end{split}
\end{equation}
But this follows directly from Theorem \ref{Bilin1}. Indeed, by Theorem \ref{Bilin1} we obtain
\begin{equation*}
\begin{split}
& \sum_{k=-\infty}^{\max \{k_{1},k_{2} \}+C} 2^{(\sigma-1)k} \Vert f_{k_{1}} g_{k_{2}} \Vert_{Z_{k}} \lesssim 2^{ \max \{ k_{1},k_{2} \} (\sigma-1) } 2^{\min \{ k_{1} ,k_{2} \} \frac{n}{2}} \Vert f_{k_{1}} \Vert_{Z_{k_{1}}} \Vert g_{k_{2}} \Vert_{Z_{k_{2}}}.
\end{split}
\end{equation*}
Then \eqref{reddalge4} follows from the following 
\begin{equation*}
2^{ \max \{ k_{1},k_{2} \} (\sigma-1) } 2^{\min \{ k_{1} ,k_{2} \} \frac{n}{2}} \lesssim 2^{ \max \{ k_{1},k_{2} \} ( \sigma )}  2^{ \min \{ k_{1},k_{2} \} \frac{n-2}{2}}.
\end{equation*}

\end{proof}
We note that the same algebra properties hold for the spaces $ \tilde{F}^{\sigma}$.
\begin{corollary}\label{algebraforftilderr}
    Recall the space  $\tilde{F}^{\sigma}$ defined in Definition \ref{tildezk}. Then the following inequalities hold
    \begin{enumerate}
        \item for any $ \sigma \in (0,\frac{n}{2}]$
        $$
        \Vert f g \Vert_{\tilde{F}^{\sigma}} \lesssim \Vert f \Vert_{\tilde{F}^{\frac{n}{2}}} \Vert g \Vert_{\tilde{F}^{\sigma}}
        $$
        \item for any $ \sigma \geq \frac{n}{2}$ we have 
        $$
        \Vert fg \Vert_{\tilde{F}^{\sigma}} \lesssim \Vert f \Vert_{\tilde{F}^{\frac{n}{2}}} \Vert g \Vert_{\tilde{F}^{\sigma}} + \Vert f \Vert_{\tilde{F}^{\sigma}} \Vert g \Vert_{\tilde{F}^{\frac{n}{2}}}.
        $$
        \item for any $ \sigma >0$ we have 
        $$
        \Vert f | \nabla|^{-1}g \Vert_{\tilde{F}^{\sigma}} \lesssim \Vert f \Vert_{\tilde{F}^{\frac{n-2}{2}}} \Vert g \Vert_{\tilde{F}^{\sigma}} + \Vert f \Vert_{\tilde{F}^{\sigma}} \Vert g \Vert_{\tilde{F}^{\frac{n-2}{2}}}.
        $$
\item If $ n \geq 5$, $ \sigma \geq \frac{n-2}{2}$ then
$$
\Vert | \nabla|^{-1} (fg) \Vert_{\tilde{F}^{\sigma}} \lesssim \Vert f \Vert_{\tilde{F}^{\frac{n-2}{2}}} \Vert g \Vert_{\tilde{F}^{\sigma}}+ \Vert f \Vert_{\tilde{F}^{\sigma}} \Vert g \Vert_{\tilde{F}^{\frac{n-2}{2}}}.
$$   
\item for any $ \sigma >0$ the following asymmetric inequality holds
$$
\Vert f g \Vert_{\tilde{F}^{\sigma}} \lesssim \Vert f \Vert_{\tilde{F}^{\frac{n-2}{2}}} \Vert g \Vert_{ \tilde{F}^{\sigma +1}} + \Vert f \Vert_{\tilde{F}^{\sigma}} \Vert g \Vert_{\tilde{F}^{\frac{n}{2}}}
$$
\item and more generally, for any $ \sigma > 0 $, and any $ \epsilon > 0$ with $ \epsilon < \frac{n}{2}$ we have 
$$
\Vert fg \Vert_{\tilde{F}^{\sigma
}} \lesssim \Vert f \Vert_{\tilde{F}^{\frac{n-2\epsilon}{2}}} \Vert g \Vert_{\tilde{F}^{\sigma+\epsilon}} + \Vert f \Vert_{\tilde{F}^{\sigma}} \Vert g \Vert_{\tilde{F}^{\frac{n}{2}}}.
$$
    \end{enumerate}
    
\end{corollary}
We close this section by mentioning that the same algebra properties hold for $ \tilde{F}^{\sigma}(r)$, $ F^{\sigma}(r)$. Indeed, this follows easily from Corollary \ref{AlgebraF}, and Corollary \ref{algebraforftilderr}.
\begin{corollary}\label{algebraforftilde}
Let $ r \in [0,\infty]$. In what follows let $ X^{\sigma}(r)$ denote either $ F^{\sigma}(r)$ or $ \tilde{F}^{\sigma}(r)$.  Then the following inequalities hold
    \begin{enumerate}
        \item for any $ \sigma \in (0,\frac{n}{2}]$
        \begin{equation}\label{eqalgebraic1}
             \Vert f g \Vert_{X^{\sigma}(r)} \lesssim \Vert f \Vert_{X^{\frac{n}{2}}(r)} \Vert g \Vert_{X^{\sigma}(r)}
        \end{equation}
        
        \item for any $ \sigma \geq \frac{n}{2}$ we have 
        \begin{equation}\label{eqalgebraic2}
            \Vert fg \Vert_{X^{\sigma}(r)} \lesssim \Vert f \Vert_{X^{\frac{n}{2}}(r)} \Vert g \Vert_{X^{\sigma}(r)} + \Vert f \Vert_{X^{\sigma}(r)} \Vert g \Vert_{X^{\frac{n}{2}}(r)}.
        \end{equation}
        \item for any $ \sigma >0$ we have 
        \begin{equation}\label{eqalgebraic3}
            \Vert f | \nabla|^{-1}g \Vert_{X^{\sigma}(r)} \lesssim \Vert f \Vert_{X^{\frac{n-2}{2}}(r)} \Vert g \Vert_{X^{\sigma}(r)} + \Vert f \Vert_{X^{\sigma}(r)} \Vert g \Vert_{X^{\frac{n-2}{2}}(r)}.
        \end{equation}
\item If $ n \geq 5$, $ \sigma \geq \frac{n-2}{2}$ then
\begin{equation}\label{eqalgebraic4}
\Vert | \nabla|^{-1} (fg) \Vert_{X^{\sigma}(r)} \lesssim \Vert f \Vert_{X^{\frac{n-2}{2}}(r)} \Vert g \Vert_{X^{\sigma}(r)}+ \Vert f \Vert_{X^{\sigma}(r)} \Vert g \Vert_{X^{\frac{n-2}{2}}(r)}.
\end{equation}
\item for any $ \sigma >0$ the following asymmetric inequality holds
\begin{equation}\label{eqalgebraic5}
    \Vert f g \Vert_{X^{\sigma}(r)} \lesssim \Vert f \Vert_{X^{\frac{n-2}{2}}(r)} \Vert g \Vert_{ X^{\sigma +1}(r)} + \Vert f \Vert_{X^{\sigma}(r)} \Vert g \Vert_{X^{\frac{n}{2}}(r)}
\end{equation}
\item and more generally, for any $ \sigma > 0 $, and any $ \epsilon > 0$ with $ \epsilon < \frac{n}{2}$ we have 
\begin{equation}\label{eqalgebraic6}
    \Vert fg \Vert_{X^{\sigma
}(r)} \lesssim \Vert f \Vert_{X^{\frac{n-2\epsilon}{2}}(r)} \Vert g \Vert_{X^{\sigma+\epsilon}(r)} + \Vert f \Vert_{X^{\sigma}(r)} \Vert g \Vert_{X^{\frac{n}{2}}(r)}.
\end{equation}

    \end{enumerate}
    
\end{corollary}

\section{Reduction to a modified Schr\"odinger equation}\label{s:reduction}
Throughout this section, we fix $ u_{0} :\R^{n} \to \S^{2}$ with $ \Vert u_{0} \Vert_{\dot{B}^{n/2}_{Q}} \ll 1$.  And take $ u : \R^{n+1} \to \S^{2}$, $ u \in L^{\infty}([-T,T];B_{Q}^{ \sigma})$, that solves \eqref{SSchrodingereq} on a time interval $ [-T,T]$, where $T > 0$ depends on the norm of the initial date $ u_{0}$ and $ \sigma \geq \frac{n}{2} +10$. The existence of such a solution is guaranteed by Proposition \ref{localexistencethm}. By shrinking $T$, if necessary, we may assume $|u(x,t) - Q| \leq \frac{1}{2}$ for any $ (x, t) \in \R^{n} \times [-T,T]$.  Further, we will use the following notation;
 for any $i \in \{1,..,n \}$ we write $ \nabla^{2s-1}_{i}$ to mean the operator with symbol $ \i \vert \xi \vert^{2s-2} \xi_{i}$. And sometimes, to simplify the notation, for any $F: \R^{n+1} \to \C^{n}$ we write 
 $$
 \nabla^{2s-1}_{x} (F) : = \sum_{i =1}^{n} \nabla^{2s-1}_{i} F_{i}. 
 $$
 Then notice the following identity 
 $$
  (- \Delta)^{s}(h) =- \sum_{i=1}^{n} \nabla^{2s-1}_{i} \partial_{i} h.
 $$
The plan is to use frames $ e_{1} , e_{2}$ and study equations in terms of 
$$
V_{i} : = \langle e_{1} , \partial_{i} u \rangle + \i \langle e_{2} , \partial_{i} u \rangle,
$$
for $ i \in \{1,..., n \}$. And 
$$
V_{0} : = \langle e_{1} , \partial_{t} u \rangle + \i \langle e_{2} , \partial_{t} u \rangle.
$$
To that end, define $L : \S^{2} \to \C$ to be the stereographic projection. And let $f_{1} + \i f_{2}= f : = L(u)$. And take $ e'_{i} $ to be 
$$
e'_{i} : = \frac{ \partial_{f_{i}} L^{-1} (f)}{ \vert\partial_{f_{i}} L^{-1} (f) \vert}
$$
for $i \in \{1,2\}$. 
Then notice that we have 
$$
\langle \partial_{t} u , e'_{i} \rangle = \frac{2}{1+ \vert f \vert^{2}} \partial_{t} f_{i}
$$
for $i \in \{1,2\}$. This implies 
\begin{equation}\label{eqq1}
\partial_{t} f = \frac{1+\vert f \vert^{2}}{2} \left( \langle e'_{1} \partial_{t} u \rangle + \i \langle e'_{2}, \partial_{t} u \rangle \right).
\end{equation}
Similarly, we obtain, for any $i \in \{1,..., n \}$, 
\begin{equation}\label{eqq2}
\partial_{i} f =  \frac{1+\vert f \vert^{2}}{2} \left( \langle e'_{1} \partial_{i} u \rangle + \i \langle e'_{2}, \partial_{i} u \rangle \right).
\end{equation}
Set $$ e_{1} : = cos(\theta)e'_{1} - sin(\theta) e'_{2}$$
and $$ e_{2} : = sin(\theta) e'_{1} + cos(\theta) e'_{2}$$ where $ \theta : \R^{n} \to \R$ is to be chosen later. Then we have
\begin{equation}\label{eqq3}
e^{\i \theta} \partial_{t} f = \frac{1 + \vert f \vert^{2}}{2} \left( \langle e_{1} , \partial_{t} u \rangle + \i \langle e_{2} , \partial_{t} u \rangle \right).
\end{equation}
Similarly, we have 
\begin{equation}\label{eqq4}
e^{\i \theta} \partial_{i} f = \frac{1 + \vert f \vert^{2}}{2} \left( \langle e_{1} , \partial_{i} u \rangle + \i \langle e_{2} , \partial_{i} u \rangle \right),
\end{equation}
for $i \in \{1,...,n \}$. Next, we define 
\begin{equation}\label{Vxdef}
V_{i} : = \langle e_{1} , \partial_{i} u \rangle + \i \langle e_{2} , \partial_{i} u \rangle
\end{equation}
for $i \in \{1,..,n \}$. And 
\begin{equation}\label{VtDef}
V_{0} : = \langle e_{1} , \partial_{t} u \rangle + \i \langle e_{2} , \partial_{t} u \rangle.
\end{equation}
Now we show that the $V_{i}'$s, for $i \in \{1,..,n\},$ solve a system of Schr\"odinger type equations. 
First, notice that by the fact that $u$ solves the s-Schr\"odinger map, we obtain
\begin{equation}\label{eqq5}
\begin{split}
& \langle e_{1}, \partial_{t} u\rangle
\\
& = - \langle e_{1} , u \wedge ( - \Delta)^{s} u \rangle
\\
& = \langle e_{2} , (-\Delta)^{s} u \rangle
\\
& = -\langle e_{2} , \sum_{i=1}^{n} \nabla^{2s-1}_{i} \left( \langle \partial_{i} u, {e}_{1} \rangle e_{1} + \langle \partial_{i} u , e_{2} \rangle e_{2}\right) \rangle.
\end{split}
\end{equation}
For each $i \in \{1,..,n\}$ define the following commutator acting on scalar functions as follows 
$$
[\nabla^{2s-1}_{i}, e_{2}](h) : = \nabla^{2s-1}_{i}(h e_{2}) - \nabla_{i}^{2s-1}(h) e_{2}.
$$
Then we obtain, using \eqref{eqq5}, and the above notation,
\begin{equation*}
\begin{split}
& \langle e_{1}, \partial_{t} u\rangle
\\
& = -\langle e_{2} , \sum_{i=1}^{n} [ \nabla^{2s-1}_{i} , e_{1} ] ( \langle \partial_{i} u , e_{1} \rangle) \rangle
\\
& - \langle e_{2} , \sum_{i=1}^{n} [ \nabla^{2s-1}_{i} , e_{2} ] ( \langle \partial_{i} u , e_{2} \rangle) \rangle
\\
& - \sum_{i=1}^{n} \nabla^{2s-1}_{i} (  \langle \partial_{i} u , e_{2} \rangle ) 
\end{split}
\end{equation*}
Similarly, 
\begin{equation*}
\begin{split}
& \i \langle e_{2}, \partial_{t} u\rangle
\\
& =- \i \langle e_{2} , u \wedge ( - \Delta)^{s} u \rangle
\\
& = -\i \langle e_{1} , (-\Delta)^{s} u \rangle
\\
& = \i \langle e_{1} , \sum_{i=1}^{n} \nabla^{2s-1}_{i} \left( \langle \partial_{i} u ,e_{1} \rangle e_{1} + \langle \partial_{i} u , e_{2} \rangle e_{2} \right) \rangle
\\
&= \i \langle e_{1} , \sum_{i=1}^{n} [ \nabla^{2s-1}_{i},e_{2}] ( \langle \partial_{i} u , e_{2} \rangle) \rangle 
\\
&+ \i \langle e_{1} , \sum_{i=1}^{n} [ \nabla^{2s-1}_{i},e_{1}] ( \langle \partial_{i} u , e_{1} \rangle) \rangle
\\
&+ \i \sum_{i =1}^{n} \nabla^{2s-1}_{i} ( \langle \partial_{i} u , e_{1} \rangle).
\end{split}
\end{equation*}
Combining the above, we arrive at the following
\begin{equation}\label{eqq6}
\begin{split}
& V_{0} - \i \sum_{i=1}^{n} \nabla^{2s-1}_{i} (V_{i}) 
\\
& = - \langle e_{2} , \sum_{i=1}^{n} [ \nabla^{2s-1}_{i} , e_{1} ] ( \langle \partial_{i} u , e_{1} \rangle) \rangle
\\
& - \langle e_{2} , \sum_{i=1}^{n} [ \nabla^{2s-1}_{i} , e_{2} ] ( \langle \partial_{i} u , e_{2} \rangle) \rangle
\\
&+ \i \langle e_{1} , \sum_{i=1}^{n} [ \nabla^{2s-1}_{i},e_{2}] ( \langle \partial_{i} u , e_{2} \rangle) \rangle 
\\
&+ \i \langle e_{1} , \sum_{i=1}^{n} [ \nabla^{2s-1}_{i},e_{1}] ( \langle \partial_{i} u , e_{1} \rangle) \rangle.
\end{split}
\end{equation}
Next, consider the following operator acting on scalar valued functions $h_{1},....,h_{n}$ as follows
\begin{equation}
\begin{split}
& \mathscr{H}( e_{1} ,e_{2})(h_{1},..,h_{n}) : = \langle e_{2}, \sum_{i=1}^{n} [ \nabla^{2s-1}_{i} , e_{1}](h_{i}) \rangle + \langle e_{1} , \sum_{i=1}^{n} [ \nabla^{2s-1}_{i} , e_{2}] (h_{i}) \rangle
\\
& = \langle e_{2} , \sum_{i=1}^{n} \nabla^{2s-1}_{i}( h_{i} e_{1}) \rangle + \langle e_{1} , \sum_{i=1}^{n} \nabla^{2s-1}_{i} ( h_{i} e_{2}) \rangle. 
\end{split}
\end{equation}
Notice that when $ s = 1 $ then 
$$
\mathscr{H} ( e_{1} , e_{2}) = 0.
$$
Using this notation, we rewrite \eqref{eqq6} as follows
\begin{equation}\label{eqqpre}
\begin{split}
&  V_{0} - \i \sum_{i=1}^{n} \nabla^{2s-1}_{i} (V_{i})
\\
& = -\langle e_{2}, \sum_{i=1}^{n} [ \nabla^{2s-1}_{i}, e_{2}] \left( \langle \partial_{i} u , e_{2} \rangle \right) \rangle 
\\
& +\i \langle e_{1} ,\sum_{i=1}^{n} [ \nabla^{2s-1}_{i}, e_{1}] ( \langle \partial_{i} u , e_{1} \rangle)
\\
& - \mathscr{H}(e_{1}, e_{2} )( \langle \partial_{1} u, e_{1}\rangle,..., \langle \partial_{n} u, e_{1}\rangle )
\\
& + \langle e_{1} , \sum_{i=1}^{n} [ \nabla^{2s-1}_{i}, e_{2} ] ( \langle \partial_{i} u , e_{1} \rangle) \rangle
\\
& + \i \langle e_{1} , \sum_{ i =1}^{n} [ \nabla^{2s-1}_{i}, e_{2}] ( \langle \partial_{i} u, e_{2} \rangle ) \rangle
\\
& = -\langle e_{2}, \sum_{i=1}^{n} [ \nabla^{2s-1}_{i}, e_{2}] \left( \langle \partial_{i} u , e_{2} \rangle \right) \rangle 
\\
& + \i \langle e_{1} ,\sum_{i=1}^{n} [ \nabla^{2s-1}_{i}, e_{1}] ( \langle \partial_{i} u , e_{1} \rangle)
\\
& - \mathscr{H}(e_{1}, e_{2} )( \langle \partial_{1} u, e_{1}\rangle,..., \langle \partial_{n} u, e_{1}\rangle )
\\
& + \langle e_{1} , \sum_{i =1}^{n} [ \nabla^{2s-1}_{i},e_{2}] ( V_{i} ) \rangle.
\end{split}
\end{equation}
Now to get a Schr\"odinger equation, we use the following trick; fix any $ i,j \in \{1,...,n\}$ then notice
\begin{equation}\label{eqq7}
\begin{split}
& \partial_{j} V_{i} - \partial_{i} V_{j} 
\\
& = \partial_{j} \left( \langle \partial_{i} u , e_{1} \rangle + \i \langle \partial_{i} u , e_{2} \rangle \right) - \partial_{i} \left( \langle \partial_{j} u , e_{1} \rangle + \i \langle \partial_{j} u , e_{2} \rangle \right)
\\
& = \langle \partial_{i} u , \partial_{j} e_{1} \rangle + \i \langle \partial_{i} u , \partial_{j} e_{2} \rangle - \langle \partial_{j} u , \partial_{i} e_{1} \rangle - \i \langle \partial_{j}u, \partial_{i} e_{2} \rangle
\\
& = \langle \partial_{i} u , e_{2} \rangle \langle e_{2}, \partial_{j} e_{1} \rangle + \i \langle \partial_{i} u, e_{1} \rangle \langle e_{1}, \partial_{j} e_{2} \rangle
\\
& - \langle \partial_{j} u , e_{2} \rangle \langle e_{2}, \partial_{i} e_{1} \rangle - \i \langle \partial_{j} u , e_{1} \rangle \langle e_{1}, \partial_{i} e_{2} \rangle
\\
& = - \langle \partial_{i} u, e_{2} \rangle \langle e_{1} , \partial_{j} e_{2} \rangle + \i \langle \partial_{i} u , e_{1} \rangle \langle e_{1}, \partial_{j} e_{2} \rangle
\\
& + \langle \partial_{j} u , e_{2} \rangle \langle \partial_{i} e_{2} , e_{1} \rangle - \i \langle \partial_{j} u , e_{1} \rangle \langle e_{1}, \partial_{i} e_{2} \rangle
\\
& = \i V_{i} \langle e_{1} , \partial_{j} e_{2} \rangle - \i V_{j} \langle \partial_{i} e_{2} , e_{1} \rangle.
\end{split}
\end{equation}
Similarly, for any $ i \in \{1,..,n \}$ we have
\begin{equation}\label{eqq8}
\begin{split}
& \partial_{i} V_{0} - \partial_{t} V_{i} 
\\
& = \i V_{0} \langle {e}_{1}, \partial_{i} {e}_{2} \rangle - \i V_{i} \langle \partial_{t} {e}_{2}, {e}_{1} \rangle.
\end{split}
\end{equation}
We use the identities \eqref{eqq7}, and \eqref{eqq8} to obtain the following; for any fixed $ i \in \{1,...,n \}$ we write
\begin{equation*}
\begin{split}
& \partial_{t} V_{i} + \i ( - \Delta)^{s} V_{i} 
\\
& = \partial_{i} V_{0} - \i \sum_{j=1}^{n} \nabla^{2s-1}_{j} \partial_{j} V_{i} 
\\
& + \i V_{i} \langle \partial_{t} e_{2}, e_{1} \rangle - \i V_{0} \langle e_{1} , \partial_{i} e_{2} \rangle
\\
& = \partial_{i} V_{0} - \i \partial_{i} \sum_{j=1}^{n} \nabla^{2s-1}_{j} V_{j} 
\\
& + \i V_{i} \langle \partial_{t} e_{2}, e_{1} \rangle - \i V_{0} \langle e_{1} , \partial_{i} e_{2} \rangle
\\
& +  \sum_{j=1}^{n} \nabla^{2s-1}_{j} \left( V_{i} \langle e_{1} , \partial_{j} e_{2}\rangle \right) -  \sum_{j=1}^{n} \nabla^{2s-1}_{j} \left(  V_{j} \langle \partial_{i} e_{2}, e_{1} \rangle \right).
\end{split}
\end{equation*}
Using the equation \eqref{eqqpre} we arrive at the following
\begin{equation}\label{recudedSchro}
\begin{split}
& \partial_{t} V_{i} + \i ( - \Delta)^{s} V_{i} 
\\
& =- \partial_{i} \big( \langle e_{2}, \sum_{j=1}^{n} [ \nabla^{2s-1}_{j}, e_{2}] \left( \langle \partial_{j} u , e_{2} \rangle \right) \rangle \big)
\\
& + \partial_{i} \big( \i \langle e_{1} ,\sum_{j=1}^{n} [ \nabla^{2s-1}_{j}, e_{1}] ( \langle \partial_{j} u , e_{1} \rangle) \big)
\\
& - \partial_{i} \big( \mathscr{H}(e_{1}, e_{2} )( \langle \partial_{1} u, e_{1}\rangle,..., \langle \partial_{n} u, e_{1}\rangle ) \big)
\\
& +\partial_{i} \big(  \langle e_{1} , \sum_{j =1}^{n} [ \nabla^{2s-1}_{j},e_{2}] ( V_{j} ) \rangle \big)
\\
&
+ \i V_{i} \langle \partial_{t} e_{2}, e_{1} \rangle - \i V_{0} \langle e_{1} , \partial_{i} e_{2} \rangle
\\
& + \sum_{j=1}^{n} \nabla^{2s-1}_{j}  \left( V_{i} \langle e_{1} , \partial_{j} e_{2}\rangle \right) -  \sum_{j=1}^{n} \nabla^{2s-1}_{j} \left( V_{j} \langle \partial_{i} e_{2}, e_{1} \rangle \right).
\end{split}
\end{equation}
for any $ i \in \{1,...,n \}$. 
\begin{proposition}\label{modifiedsschrodeq}
Let $ u : \R^{n+1} \to \S^{2}$ be a solution of \eqref{SSchrodingereq} on a time interval $ [-T,T]$. Then on the same interval $[-T,T]$, the functions $ \{ V_{i} \}_{i=1}^{n}$ defined above solve the following equation 
$$
\partial_{t} V_{i} + \i ( - \Delta)^{s} V_{i} = \mathcal{N} ( V_{0},V_{1},...,V_{n},e_{1},e_{2} )(V_{i})
$$
where $\mathcal{N}(V_{0},V_{1},...,V_{n}, e_{1},e_{2}) (V_{i})$ is the right hand side of \eqref{recudedSchro}. 
\end{proposition}
Now we make a choice for $ \theta : \R^{n+1} \to \R$ as follows; we choose $ \theta$ so that 
$$
\sum_{i=1}^{n} \partial_{i} \left( \langle  e_{1}, \partial_{i} e_{2} \rangle \right)=0.
$$
This can be done by choosing 
$$
\theta : = -\Delta^{-1} \sum_{i=1}^{n} \partial_{i} \left( \langle  e'_{1} , \partial_{i} e'_{2} \rangle \right).
$$
Notice that a direct computation gives that $ \theta \in C ( [-T,T] : B^{\sigma}_{2,1} )$. In fact we have 
\begin{lemma}\label{thetainteg}
Let $ \theta : \R^{n+1} \to \R$ be given as above, $ \sigma \geq \frac{n}{2}$. Then we have
\begin{equation}\label{theta1}
    \Vert \theta ( \cdot,t) \Vert_{\dot{B}^{\sigma}_{2,1}} \lesssim 
 C ( \Vert f(\cdot,t) \Vert_{L_{x}^{\infty}})\Vert u(\cdot,t) \Vert_{\dot{B}^{\sigma}_{Q}}.
    \end{equation}
    for any $ t \in [-T,T]$. In particular, if we know that $ \Vert f(\cdot, t) \Vert_{L_{x}^{\infty}} \ll 1$ then we obtain 
    \begin{equation}\label{theta0}
    \Vert \theta ( \cdot,t) \Vert_{\dot{B}^{\sigma}_{2,1}} \lesssim \Vert u(\cdot,t) \Vert_{\dot{B}^{\sigma}_{Q}}.
    \end{equation}
Moreover, for any $t \in [-T,T]$ the following holds
\begin{equation}\label{theta2}
\Vert \theta(t) \Vert_{{B}^{\sigma}_{2,1}} <\infty.
\end{equation}
\end{lemma}
\begin{proof}
by definition of $\theta$, it suffices to estimate terms of the form
$$
\Delta^{-1} \partial_{i} \langle e'_{1} , \partial_{i} e'_{2} \rangle.
$$
A direct computation gives that 
\begin{equation}\label{iuie}
\langle e'_{1} , \partial_{i} e'_{2} \rangle = \frac{2}{1+|f|^{2}} ( f_{1} \partial_{i} f_{2} - f_{2} \partial_{i} f_{1}).
\end{equation}
Then we claim the following is true for any $t \in [-T,T]$ and any $ \sigma \geq \frac{n}{2}$ 
\begin{equation}\label{iop}
\Vert |\nabla|^{-1} ( \frac{1}{1+|f(\cdot,t)|^{2}} f_{2}(\cdot,t) \partial_{i} f_{1}(\cdot,t) ) \Vert_{\dot{B}^{\sigma}_{2,1}} \lesssim C( \Vert f(\cdot,t) \Vert_{L^{\infty}(\R^{n})}) \Vert f(\cdot,t) \Vert_{\dot{B}^{\sigma}_{2,1}}.
\end{equation}
To prove \eqref{iop}, fix $t \in [-T,T]$ and set $F = (\frac{1}{1+|f|^{2}} - 1 )$. Then we estimate the left hand side of \eqref{iop} with ( for better readability, we stop writing $t$ in the argument) 
\begin{equation}\label{iopp}
\Vert |\nabla|^{-1} ( \frac{1}{1+|f|^{2}} f_{2} \partial_{i} f_{1} ) \Vert_{\dot{B}^{\sigma}_{2,1}} \lesssim
\Vert |\nabla|^{-1}  (F f_{2} \partial_{i} f_{1})  \Vert_{\dot{B}^{\sigma}_{2,1}} + \Vert |\nabla|^{-1} (f_{2} \partial_{i} f_{1}) \Vert_{\dot{B}^{\sigma}_{2,1}}.
\end{equation}
To prove \eqref{iop} we estimate each term on the right hand side of \eqref{iopp}. Starting with the second term, we have from Lemma \ref{elemantaryestbesov} 
\begin{equation}\label{theta4}
\Vert |\nabla|^{-1} (f_{2} \partial_{i} f_{1}) \Vert_{\dot{B}^{\sigma}_{2,1}} \lesssim C( \Vert f(\cdot,t) \Vert_{L^{\infty}}) \Vert f \Vert_{\dot{B}^{\sigma}_{2,1}}.
\end{equation}
We also claim the following holds
 \begin{equation}\label{theta55}
\Vert |\nabla|^{-1}  (F f_{2} \partial_{i} f_{1})  \Vert_{\dot{B}^{\sigma}_{2,1}} \lesssim C( \Vert f(\cdot,t) \Vert_{L^{\infty}}) \Vert f \Vert_{\dot{B}^{\sigma}_{2,1}}
\end{equation}
Indeed, to see this we again use Lemma \ref{elemantaryestbesov} and the fact that $ \Vert F \Vert_{L^{\infty}} \lesssim 1$ to obtain
\begin{equation*}
\begin{split}
& \Vert |\nabla|^{-1}  (F f_{2} \partial_{i} f_{1})  \Vert_{\dot{B}^{\sigma}_{2,1}}
\\
& \lesssim C( \Vert f ( \cdot,t) \Vert_{L^{\infty}}) \left( \Vert f \Vert_{\dot{B}^{\sigma}_{2,1}} + \Vert F \Vert_{\dot{B}^{\sigma}_{2,1}} \right)
\end{split}
\end{equation*}
now using Lemma \ref{Besovstuff} we have
$$
 \Vert F \Vert_{\dot{B}^{\sigma}_{2,1}} \lesssim C( \Vert f(\cdot,t) \Vert_{L^{\infty}}) \Vert f \Vert_{\dot{B}^{\sigma}_{2,1}}
$$
This prove \eqref{theta55}. Combining \eqref{theta4}, \eqref{theta55} and Lemma \ref{Besovstuff} we obtain \eqref{theta1}. It remains to prove \eqref{theta2}. Notice that by \eqref{theta1}, we only need to control low frequencies of $ |\nabla|^{-1} \langle e'_{1} ( \cdot,t) , \partial_{i} e'_{2}(\cdot,t) \rangle$. That is, it suffices to show 
\begin{equation*}
\Vert \Delta_{\leq 0} |\nabla|^{-1} \langle e'_{1} ( \cdot,t) , \partial_{i} e'_{2}(\cdot,t) \rangle \Vert_{L^{2}(\R^{n})} < \infty
\end{equation*}
for any $t \in [-T,T]$. To that end, using the identity \eqref{iuie}, we need to prove 
\begin{equation*}
\Vert \Delta_{\leq 0} | \nabla |^{-1} G(\cdot,t) \partial_{i} f_{1}(\cdot,t) \Vert_{L^{2}(\R^{n})} < \infty,
\end{equation*}
where $ G := \frac{f_{2}}{1+|f|^{2}}$. We use boundedness of the Riesz potential to obtain
\begin{equation*}
\begin{split}
& \Vert | \nabla|^{-1} \left( G( \cdot,t) \partial_{i} f_{1}(\cdot,t) \right) \Vert_{L^{2}(\R^{n})}
\\
& \lesssim \Vert G(\cdot,t) \partial_{i} f_{1}(\cdot,t) \Vert_{L^{\frac{2n}{n+2}}(\R^{n})}
\\
& \lesssim \Vert G(\cdot,t) \Vert_{L^{n}(\R^{n})} \Vert \nabla f(\cdot,t) \Vert_{L^{2}(\R^{n})}
\\
& \lesssim \Vert f(\cdot,t) \Vert_{L^{n}(\R^{n})} \Vert \nabla f(\cdot,t) \Vert_{L^{2}(\R^{n})} <\infty.
\end{split}
\end{equation*}
This concludes the proof of \eqref{theta2}.
\end{proof}

Next, we prove the following quantitative estimate regrading the functions $V_{i}$
\begin{proposition}\label{controlvu}
Let $ \sigma \geq \frac{n-2}{2}$. And let $u, V_{i}'s$ be given as above for $i=1,...,n$. Assume $ \Vert u|_{t=0} \Vert_{\dot{B}^{\frac{n}{2}}_{Q}} \leq \epsilon \ll 1$. Then we have
\begin{enumerate}
    \item 
    \begin{equation}\label{cmcm1}
    \Vert V_{i}( \cdot,t) \Vert_{\dot{B}^{\sigma }_{2,1}} \lesssim C(\Vert f(\cdot,t) \Vert_{L^{\infty}} \Vert) \Vert u(\cdot,t) \Vert_{\dot{B}^{\sigma+1}_{Q}}.
    \end{equation}
    In particular, 
    $$
    \Vert V_{i}|_{t =0} \Vert_{\dot{B}^{\sigma }_{2,1}} \lesssim \Vert u|_{t=0} \Vert_{\dot{B}^{\sigma+1}_{Q}}
    $$
\item Similarly, we have 
\begin{equation}\label{cmcm}
\Vert V_{i} (\cdot,t) \Vert_{{B}^{\sigma }_{2,1}} \lesssim C( \Vert f( \cdot, t) \Vert_{L^\infty}) \Vert u ( \cdot, t) \Vert_{{B}^{\sigma+1}_{Q}}.
\end{equation}
In particular, 
$$
\Vert V_{i}|_{t =0} \Vert_{{B}^{\sigma }_{2,1}} \lesssim \Vert u|_{t=0} \Vert_{{B}^{\sigma+1}_{Q}}.
$$
\end{enumerate}
\end{proposition}
\begin{proof}
We prove \eqref{cmcm1} since the proof of \eqref{cmcm} follows almost verbatim. 
Using the identity \eqref{eqq4} we know that for each $i \in \{ 1,...,n\}$ we have the equality 
$$
V_{i} = \frac{2}{1+ \vert f \vert^{2}} e^{\i \theta} \partial_{i} f 
$$
where $ f = L(u)$ and $L: \S^{2} \to \C$ is the stereographic projection. Fix $t \in [-T,T]$. Then for better readability, we omit writing $t$ in the argument. 
In view of Lemma \ref{Besovstuff}, and  \eqref{theta1} it suffices to prove 
\begin{equation}\label{modvu1}
\begin{split}
& \Vert (\frac{1}{1+ \vert f \vert^{2}} e^{\i \theta} \partial_{i} f)  \Vert_{\dot{B}^{\sigma}_{2,1}} \lesssim C ( \Vert f(\cdot,t) \Vert_{L^{\infty}}) \left( \Vert f \Vert_{\dot{B}^{\sigma+1}_{2,1}} + \Vert \theta \Vert_{\dot{B}^{\sigma+1}_{2,1}} \right).
\end{split}
\end{equation}
To that end,
set $ F : = \frac{1}{1+|f|^{2}}( e^{\i \theta} -1)$. Then we have
\begin{equation}\label{fgfg1}
\begin{split}
& \Vert V_{i} \Vert_{\dot{B}^{\sigma}_{2,1}} \lesssim \Vert \frac{1}{1+|f|^{2}} \partial_{i} f \Vert_{\dot{B}^{\sigma}_{2,1}} + \Vert F \partial_{i} f \Vert_{\dot{B}^{\sigma}_{2,1}}.
\end{split}
\end{equation}
We estimate each term on the right hand side of \eqref{fgfg1}. For first term, using the algebra properties of Besov spaces, we have
\begin{equation*}
\begin{split}
& \Vert \frac{1}{1+|f|^{2}} \partial_{i} f \Vert_{\dot{B}^{\sigma}_{2,1}}
\\
& \lesssim \Vert f \Vert_{\dot{B}^{\sigma +1}_{2,1}} + \Vert \frac{|f|^{2}}{1+|f|^{2}} \partial_{i} f \Vert_{\dot{B}^{\sigma}_{2,1}}
\\
& \overset{\eqref{Besovalgest1}, \eqref{adad}}{\lesssim}C (\Vert f( \cdot,t) \Vert_{L^{\infty}}) \Vert f \Vert_{\dot{B}^{\sigma +1}_{2,1}}.
\end{split}
\end{equation*}
Similarly, for the second term on the right hand side of \eqref{fgfg1} we have  
\begin{equation*}
\begin{split}
& \Vert F \partial_{i} f \Vert_{\dot{B}^{\sigma}_{2,1}}
\\
& \overset{\eqref{adad}}{\lesssim} \Vert f \Vert_{\dot{B}^{\sigma+1}_{2,1}} + \Vert F \Vert_{\dot{B}^{\sigma+1}_{2,1}}
\end{split}
\end{equation*}
and using the definition of $F$ and the algebra properties of Besov spaces and Lemma \ref{Besovstuff}, we obtain
$$
\Vert F \Vert_{\dot{B}^{\sigma+1}_{2,1}} \lesssim C( \Vert f(\cdot,t) \Vert_{L^{\infty}}) \left( \Vert f \Vert_{\dot{B}^{\sigma+1}_{2,1}} + \Vert \theta  \Vert_{\dot{B}^{\sigma +1}_{2,1}} \right).
$$
This completes the estimate for the second term on the right hand side of \eqref{fgfg1} and hence completes the proof.

\end{proof}

We close this section with integrability properties of $e_{i}$'s. Notice that, locally, $e_{i}$'s are as smooth as $f$. However, in the nonlinear estimates we will need identities of the form $ e_{i} = \sum_{k \in \Z} \Delta_{k} e_{i} $. For example, if $ e_{i} \in C( [-T,T] : L^{2}) $ then such an identity holds. However, we clearly have $ e_{i} \not\in L^{2}$ for any time $t \in [-T,T]$. This is because $ e_{i}$ maps into $ \S^{2}$. With that being said, the next lemma shows the existence of vectors $ Q_{i}$ so that $ e_{i} - Q_{i} \in C([-T,T]: L^{2})$.  
\begin{lemma}\label{eintegra}
Let $ e_{i} : \R^{n+1} \to \S^{2}$ be as above. Then there exists $ Q_{i} \in \S^{2}$ so that $ e_{i} - Q_{i} \in C(  [-T,T]: B^{\sigma}_{2,1})$ 
\end{lemma}
\begin{proof}
Set $ Q_{1} : = (1,0,0) $ and $ Q_{2} : = (0,1,0)$. We first show that $ e'_{i} - Q_{i} \in C(  [-T,T]: B^{\sigma}_{2,1})$. Indeed, this follows directly from the definition of $ e'_{i}$. We do the case $i=1$ , the case $ i=2$ is identical. By definition of $ e'_{1}$ we have 
$$
e'_{1} = \frac{1}{1+ |f|^{2}} ( 1+ |f|^{2} - 2 f^{2}_{1} , -2f_{1} f_{2} , -2 f_{1} ),
$$
then  
$$
e'_{1} - Q_{1} = \frac{1}{1+ |f|^{2}} (  - 2 f^{2}_{1} , -2f_{1} f_{2} , -2f_{1})
$$
and this is clearly in $C ([-T,T]:B^{\sigma}_{2,1})$ by the algebra properties of Besov spaces. Next, we have from Lemma \ref{thetainteg} that  $ \theta \in C([-T,T] : B^{\sigma}_{2,1})$.  Now by definition of $ e_{1}$ we have
\begin{equation}\label{stiii}
\begin{split}
& e_{1} -Q_{1} = cos(\theta) e'_{1} - sin(\theta)e'_{2} - Q_{1}
\\
& = cos(\theta) ( e'_{1} - Q_{1} ) + ( cos(\theta) - 1) Q_{1} - sin(\theta) (e'_{2} - Q_{2}) - sin(\theta) Q_{2}
\end{split}
\end{equation}
since $ e'_{i} -Q_{i} \in C( [-T,T] : B^{\sigma}_{2,1})$ and $ \theta \in C([-T,T]: B^{\sigma}_{2,1} )$. Then
each term on the right hand side of \eqref{stiii} is in $ C ( [-T,T]: {B}^{\sigma}_{2,1})$ which readily implies that $ e_{1} -Q_{1} \in C( [-T,T]:B^{\sigma}_{2,1} )$. The same proof yields that $ e_{2} -Q_{2} \in C([-T,T]:B^{\sigma}_{2,1} )$.
\end{proof}
One consequence of the above lemma is that whenever we have a term of the form $ e_{i}(x) -e_{i}(y)$ then we can replace it with $ |\nabla|^{-1} ( |\nabla| e_{i}(x) - |\nabla|e_{i}(y))$, which then allows us to measure such certain commutators in the resolution spaces $F^{\sigma}$.

\section{Nonlinear estimates I }
Throughout this section let $u , V_{i}$, $ e_{1},e_{2}$, and $T> 0$ be as in Proposition \ref{modifiedsschrodeq}.

In this section we prove the nonlinear estimates for the first three terms on the right hand side of \eqref{recudedSchro}.  First we prove the following abstract trilinear estimate 
\begin{proposition}\label{trilineares1}
Let $ n\geq 3 $, $s \in (1/2,1), \sigma \geq \frac{n-2}{2}$ and $ f,g,h \in F^{\frac{n-2}{2}}$. Then define for any $i \in \{ 1,...,n \}$ the following operator
\begin{equation}\label{trilinearoperator1}
\begin{split}
& T^{i}_{1}(f,g,h)(x,t)  
\\
& : = \int_{\R^{n}} \frac{ ( |\nabla|^{-1}h(x,t) - | \nabla|^{-1}h(y,t))( |\nabla |^{-1} g(x,t) - |\nabla|^{-1}g(y,t) ) f(y,t)  (x-y)_{i}}{ | x-y|^{n+2s}} dy
\end{split} 
\end{equation}
then we have
\begin{equation}\label{trlnes1}
\begin{split}
& \Vert \ind_{[-T',T']} \partial_{j} T_{1}^{i} ( \tilde{f},\tilde{g},\tilde{h}) \Vert_{N^{\sigma}}
\\
& \lesssim \Vert f \Vert_{F^{\sigma}} \Vert g \Vert_{F^{\frac{n-2}{2}}} \Vert h \Vert_{F^{\frac{n-2}{2}}}
\\
& + \Vert f \Vert_{F^{\frac{n-2}{2}}} \Vert g \Vert_{F^{\sigma}} \Vert h \Vert_{F^{\frac{n-2}{2}}}
\\
& + \Vert f \Vert_{F^{\frac{n-2}{2}}} \Vert g \Vert_{F^{\frac{n-2}{2}}} \Vert h \Vert_{F^{\sigma}}
\end{split}
\end{equation}
for any $ j \in \{1,..,n \}$ and any $T' \in [0,\infty]$. Here $ \tilde{f} \in \{ f, \bar{f} \}$.
\end{proposition}
\begin{proof}[Proof of Proposition \ref{trilineares1}]
To prove \eqref{trlnes1} it suffices to prove the following 
\begin{equation}\label{trl1}
\begin{split}
&  \sum_{k \in \Z} 2^{\sigma k} \Vert \Delta_{k}( \i \partial_{t} - ( - \Delta)^{s} + \i)^{-1} \ind_{[-T',T']} \partial_{j} T_{1}^{i}( \tilde{f}_{k_{1}} , \tilde{g}_{k_{2}} , \tilde{h}_{k_{3}}) \Vert_{Z_{k}}
\\
& \lesssim  2^{ \sigma k_{1}}\Vert f_{k_{1}} \Vert_{Z_{k_{1}}}  2^{ \frac{n-2}{2} k_{2}}\Vert g_{k_{2}} \Vert_{Z_{k_{2}}}  2^{ \frac{n-2}{2} k_{3}}\Vert h_{k_{3}} \Vert_{Z_{k_{3}}}
\\
& + 2^{ \frac{n-2}{2} k_{1}}\Vert f_{k_{1}} \Vert_{Z_{k_{1}}}  2^{ \sigma k_{2}}\Vert g_{k_{2}} \Vert_{Z_{k_{2}}}  2^{ \frac{n-2}{2} k_{3}}\Vert h_{k_{3}} \Vert_{Z_{k_{3}}}
\\
& + 2^{ \frac{n-2}{2} k_{1}}\Vert f_{k_{1}} \Vert_{Z_{k_{1}}}  2^{ \frac{n-2}{2} k_{2}}\Vert g_{k_{2}} \Vert_{Z_{k_{2}}}  2^{ \sigma k_{3}}\Vert h_{k_{3}} \Vert_{Z_{k_{3}}}
\end{split}
\end{equation}
with the understanding that $ f_{k} : = \Delta_{k} f$. To prove \eqref{trl1} we need to consider several cases, depending on which function has the highest frequency. 
\\
\textbf{Case 1:} If $ \max\{ k_{1},k_{2},k_{3} \} = k_{1}$. In view of \eqref{multilpierlemma} we may assume that $ f_{k_{1}}$ has Fourier transform supported in $ \{ \langle \xi, e' \rangle \geq 2^{k_{1}-c} \}$ for some $ e' \in \mathscr{E}$. Then we consider three sub-cases
\\
\textbf{Case 1a:} if $2^{\min\{ k_{2},k_{3} \} } \approx 2^{\max\{ k_{2},k_{3} \} } \approx 2^{k_{1}}$ then to prove \eqref{trl1} it suffices to prove the following 
\begin{equation}\label{trl2}
\begin{split}
& \sum_{k \leq k_{1} +C} 2^{k \sigma} \Vert \Delta_{k} ( \i \partial_{t} - ( - \Delta)^{s} + \i)^{-1} \ind_{[-T',T']}\partial_{j} T_{1}^{i}( \tilde{f}_{k_{1}} , \tilde{g}_{k_{2}} , \tilde{h}_{k_{3}}) \Vert_{Z_{k}}
\\
& \lesssim 2^{ \sigma k_{1}}\Vert f_{k_{1}} \Vert_{Z_{k_{1}}}     2^{ \frac{n-2}{2} k_{2}}\Vert g_{k_{2}} \Vert_{Z_{k_{2}}}  2^{ \frac{n-2}{2} k_{3}}\Vert h_{k_{3}} \Vert_{Z_{k_{3}}}.
\end{split}
\end{equation}
To that end, notice that $T_{1}^{i}( \tilde{f}_{k_{1}},\tilde{g}_{k_{2}},\tilde{h}_{k_{3}})$ can be written as follows
\begin{equation}\label{trl3}
\begin{split}
& T_{1}^{i}(\tilde{f}_{k_{1}}, \tilde{g}_{k_{2}},\tilde{h}_{k_{3}})(x,t)
\\
& = \nabla^{2s-1}_{i} ( \tilde{f}_{k_{1}} | \nabla |^{-1} ( \tilde{g}_{k_{2}}) | \nabla |^{-1}( \tilde{h}_{k_{3}} ) )
\\
& - |\nabla|^{-1}\tilde{h}_{k_{3}}\nabla^{2s-1}_{i} ( \tilde{f}_{k_{1}} |\nabla|^{-1} \tilde{g}_{k_{2}} )
\\
& - | \nabla |^{-1} \tilde{g}_{k_{2}} \nabla^{2s-1}_{i} ( \tilde{f}_{k_{1}} |\nabla|^{-1} \tilde{h}_{k_{3}}) 
\\
& + |\nabla|^{-1} \tilde{g}_{k_{2}} |\nabla|^{-1}\tilde{h}_{k_{3}} \nabla^{2s-1}_{i} ( \tilde{f}_{k_{1}}). 
\end{split}
\end{equation}
This follows directly from Lemma \ref{pointwiseidentitycomm}. 
We estimate each term on the right hand side of \eqref{trl3}. By \eqref{multilpierlemma} we may assume $ T^{i}_{1} ( \tilde{f}_{k_{1}}, \tilde{g}_{k_{2}},\tilde{h}_{k_{3}})$ has Fourier transform supported in $ \{ \xi : \langle \xi, e \rangle \geq 2^{k -c} \}$ for some $ e \in \mathscr{E}$.  Then we start with the first term on the right hand side of \eqref{trl3}
\begin{equation*}
\begin{split}
& \Vert \Delta_{k} ( \i \partial_{t} - ( - \Delta)^{s} + \i)^{-1} \ind_{[-T',T']} \partial_{j} 
\nabla^{2s-1}_{i} ( \tilde{f}_{k_{1}} | \nabla |^{-1} ( \tilde{g}_{k_{2}}) | \nabla |^{-1}( \tilde{h}_{k_{3}} ) )
\Vert_{Z_{k}}
\\
& \lesssim 2^{k} 2^{ -(2s-1)k/2} \Vert \nabla^{2s-1}_{i} ( \tilde{f}_{k_{1}} | \nabla |^{-1} ( \tilde{g}_{k_{2}}) | \nabla |^{-1}( \tilde{h}_{k_{3}} ) ) \Vert_{L^{1}_{e}L^{2}_{e^{\perp},t}}
\\
& \overset{\eqref{multiplierint}}{\lesssim} 2^{k} 2^{-(2s-1)k/2} 2^{k(2s-1)} \Vert \tilde{f}_{k_{1}} \Vert_{L^{\infty}_{e'} L^{2}_{e'^{\perp},t}} 2^{-k_{2}} \Vert \tilde{g}_{k_{2}} \Vert_{L^{2}_{e'}L^{\infty}_{e'^{\perp},t}} 2^{-k_{3}} \Vert \tilde{h}_{k_{3}} \Vert_{L^{2}_{e}L^{\infty}_{e^{\perp},t}}
\end{split}
\end{equation*}
applying the smoothing, and maximal estimates and using the assumption that $ 2^{k_{1}} \approx 2^{k_{2}} \approx 2^{k_{3}}$ we obtain
\begin{equation*}
\begin{split}
& \Vert \Delta_{k} ( \i \partial_{t} - ( - \Delta)^{s} + \i)^{-1} \ind_{[-T',T']}\partial_{j} 
\nabla^{2s-1}_{i} ( \tilde{f}_{k_{1}} | \nabla |^{-1} ( \tilde{g}_{k_{2}}) | \nabla |^{-1}( \tilde{h}_{k_{3}} ) )
\Vert_{Z_{k}}
\\
& \lesssim   \Vert f_{k_{1}} \Vert_{Z_{k_{1}}}  2^{k_{2} \frac{n-2}{2}} \Vert g_{k_{2}} \Vert_{Z_{k_{2}}} 2^{k_{3} \frac{n-2}{2}} \Vert h_{k_{3}} \Vert_{Z_{k_{3}}}.
\end{split}
\end{equation*}
Multiplying by $ 2^{\sigma k} $ and summing over $k \leq k_{1} +C$ yields \eqref{trl2}

The remaining terms on the right hand side of of \eqref{trl3} are handled in exactly the same way. Namely, we only use the smoothing and maximal inequalities, together with the assumption that $ 2^{k_{1}} \approx 2^{k_{2}} \approx 2^{k_{3}}$.
\\
\textbf{Case 1b:} If $ \min\{ k_{2},k_{3} \} \leq \max\{ k_{2},k_{3} \} - C  $ for sufficiently large $C$, and $ 2^{\max\{ k_{2},k_{3} \}} \approx 2^{k_{1}}$. Without loss of generality assume $ k_{3} = \min \{ k_{2},k_{3} \}$. In this case, we make use of the identity \eqref{trl3} again, to write
\begin{equation*}
\begin{split}
& T^{i}_{1} ( \tilde{f}_{k_{1}}, \tilde{g}_{k_{2}} , \tilde{h}_{k_{3}} ) 
\\
& = \nabla^{2s-1}_{i} ( \tilde{f}_{k_{1}} | \nabla |^{-1} ( \tilde{g}_{k_{2}}) | \nabla |^{-1}( \tilde{h}_{k_{3}} ) )  - |\nabla|^{-1}\tilde{h}_{k_{3}}\nabla^{2s-1}_{i} ( \tilde{f}_{k_{1}} |\nabla|^{-1} \tilde{g}_{k_{2}} )
\\
& + |\nabla|^{-1} \tilde{g}_{k_{2}} |\nabla|^{-1}\tilde{h}_{k_{3}} \nabla^{2s-1}_{i} ( \tilde{f}_{k_{1}})
- | \nabla |^{-1} \tilde{g}_{k_{2}} \nabla^{2s-1}_{i} ( \tilde{f}_{k_{1}} |\nabla|^{-1} \tilde{h}_{k_{3}}) 
\\
& : = I +II
\end{split}
\end{equation*}
\underline{Estimate for $II$:}
we estimate as follows 
\begin{equation}\label{trl4}
\begin{split}
& \Vert \Delta_{k} ( \i \partial_{t} - ( - \Delta)^{s} + \i )^{-1} \ind_{[-T',T']}\partial_{j} II \Vert_{Z_{k}}
\\
& \lesssim \max_{ e \in \mathscr{E} } 2^{k} 2^{-(2s-1)k/2} \Vert II \Vert_{L^{1}_{e} L^{2}_{e^{\perp},t} }
\\
& \lesssim 2^{k} 2^{-(2s-1)k/2} \max_{e \in \mathscr{E}} \Vert |\nabla|^{-1} \tilde{g}_{k_{2}} \Vert_{L^{2}_{e} L^{\infty}_{e^{\perp},t}} \Vert [\nabla^{2s-1}_{i} , |\nabla|^{-1} \tilde{h}_{k_{3}} ] ( \tilde{f}_{k_{1}}) \Vert_{L^{2}}.
\end{split}
\end{equation}
We need to estimate the last factor on the right hand side of \eqref{trl4}. To that end, we apply Lemma \ref{Taylorstuff2}, which is possible since $ s \in (1/2,1)$ and $ k_{1} \geq k_{3} +C$, to obtain
\begin{equation}\label{trl5}
\begin{split}
& \Vert [\nabla^{2s-1}_{i} , |\nabla|^{-1} \tilde{h}_{k_{3}} ] ( \tilde{f}_{k_{1}}) \Vert_{L^2}
\\
& \lesssim \sum_{ | \alpha | =1}^{\infty} \frac{ a_{|\alpha|}}{\alpha !} 2^{k_{1} ( 2s-1 - \vert \alpha \vert)} \Vert \tilde{f}_{k_{1}} D^{\alpha}( |\nabla|^{-1}\tilde{h}_{k_{3}} ) \Vert_{L^2}
\\
& \lesssim \sum_{ | \alpha | =1}^{\infty} \frac{ a_{| \alpha|}}{ \alpha !} 2^{k_{1} ( 2s-1 - | \alpha|)} \Vert \tilde{f}_{k_{1}} \Vert_{L^{\infty}_{e'} L^{2}_{e'^{\perp},t}} 2^{ | \alpha  | k_{3}} \Vert |\nabla|^{-1}\tilde{h}_{k_{3}} \Vert_{L^{2}_{e'} L^{\infty}_{e'^{\perp},t}}
\\
& \overset{\eqref{local_smoothing},\eqref{maximal_Estimate}}{\lesssim} 2^{-(2s-1)k_{1}/2} \Vert f_{k_{1}} \Vert_{Z_{k_{1}}} 2^{(n-1)k_{3}/2} 2^{-k_{3}} \Vert h_{k_{3}} \Vert_{Z_{k_{3}}} \sum_{| \alpha|=1}^{\infty} \frac{ a_{|\alpha|}}{\alpha !} 2^{k_{1} ( 2s-1 -|\alpha|)} 2^{k_{3} |\alpha|}
\\
& \lesssim 2^{-(2s-1)k_{1}/2} \Vert f_{k_{1}} \Vert_{Z_{k_{1}}} 2^{(n-1)k_{3}/2} 2^{-k_{3}} \Vert h_{k_{3}} \Vert_{Z_{k_{3}}}   2^{k_{3}} 2^{(2s-2)k_{1}}  \sum_{| \alpha|=1}^{\infty} \frac{2^{-C  | \alpha|} a_{|\alpha|}}{\alpha !} 
\\
& \overset{\eqref{qpppa}}{\lesssim} 2^{-(2s-1)k_{1}/2 } \Vert f_{k_{1}} \Vert_{Z_{k_{1}}} 2^{(n-1)k_{3}/2} 2^{-k_{3}} \Vert h_{k_{3}} \Vert_{Z_{k_{3}}} 2^{k_{3}} 2^{(2s-2)k_{1}}. 
\end{split}
\end{equation}
Plugging the bound \eqref{trl5} back into \eqref{trl4}, and using the maximal estimate for the norm of $g_{k_{2}}$, and the assumption that $ 2^{k_{3}} \lesssim 2^{k_{2}}$ we obtain  
\begin{equation*}
\begin{split}
& \Vert \Delta_{k} ( \i \partial_{t} - ( - \Delta)^{s} + \i )^{-1} \partial_{j} II \Vert_{Z_{k}}
\\
& \lesssim 2^{k} 2^{-(2s-1)k/2} 2^{-k_{1}(2s-1)/2}  2^{k_{1}(2s-2)} \Vert f_{k_{1}} \Vert_{Z_{k_{1}}}  2^{k_{2} \frac{n-2}{2}}\Vert g_{k_{2}} \Vert_{Z_{k_{2}}} 2^{k_{3} \frac{n-2}{2}} \Vert h_{k_{3}} \Vert_{Z_{k_{3}}}.
\end{split}
\end{equation*}
Multiplying by $ 2^{k \sigma}$ and summing over $ k \leq k_{1} +C$ gives \eqref{trl12}. \\
\underline{Estimate for $I$:}
Set $ F : = \tilde{f}_{k_{1}} | \nabla |^{-1} \tilde{g}_{k_{2}}$. Then we would like to use the same trick we used for $II$; however $ F $ may have frequency close to $ 2^{k_{3}}$ and hence the convergence in Lemma \ref{Taylorstuff2} may fail. To overcome this, we rewrite $ I $ as follows
\begin{equation}\label{trl6}
\begin{split}
& I = \nabla^{2s-1}_{i} ( F | \nabla |^{-1}( \tilde{h}_{k_{3}} ) )  - |\nabla|^{-1}\tilde{h}_{k_{3}}\nabla^{2s-1}_{i} ( F )
\\
& =  \nabla^{2s-1}_{i}( F_{ \leq k_{3} +C} | \nabla |^{-1}( \tilde{h}_{k_{3}} ) )  - |\nabla|^{-1}\tilde{h}_{k_{3}}\nabla^{2s-1}_{i} ( F_{\leq k_{3} +C} ) 
\\
& +\sum_{ k_{3} +C  < l \leq k_{1} +C } \nabla^{2s-1}_{i} ( F_{l} | \nabla |^{-1}( \tilde{h}_{k_{3}} ) )  - |\nabla|^{-1}\tilde{h}_{k_{3}}\nabla^{2s-1}_{i} ( F_{l} )
\end{split}
\end{equation}
 where $F_{l} : = \Delta_{l} (F)$. We begin by handling the first term on the right hand side of \eqref{trl6}
 \begin{equation}\label{trl7}
    \begin{split}
& \Vert \Delta_{k} ( \i\partial_{t} - ( - \Delta)^{s} + \i)^{-1} \ind_{[-T',T']}\partial_{j} \left( \nabla^{2s-1}_{i}( F_{\leq k_{3} +C} | \nabla |^{-1}( \tilde{h}_{k_{3}} ) )  - |\nabla|^{-1}\tilde{h}_{k_{3}}\nabla^{2s-1}_{i} ( F_{\leq k_{3} +C} ) \right) \Vert_{Z_{k}}
\\
& \lesssim 2^{k } 2^{-k(2s-1)/2}\max_{e \in \mathscr{E}} \Vert \left( \nabla^{2s-1}_{i}( F_{\leq k_{3} + C} | \nabla |^{-1}( \tilde{h}_{k_{3}} ) )  - |\nabla|^{-1}\tilde{h}_{k_{3}}\nabla^{2s-1}_{i} ( F_{\leq k_{3 } +C} ) \right) \Vert_{L^{1}_{e} L^{2} _{e^{\perp},t}}
\end{split}
 \end{equation}
 and for any $ e \in \mathscr{E}$ we have (recall that $F : = \tilde{f}_{k_{1}} | \nabla |^{-1} \tilde{g}_{k_{2}})$
 \begin{equation}\label{trl8}
\begin{split}
&  \Vert \left( \nabla^{2s-1}_{i}( F_{\leq k_{3} +C} | \nabla |^{-1}( \tilde{h}_{k_{3}} ) )  - |\nabla|^{-1}\tilde{h}_{k_{3}}\nabla^{2s-1}_{i} ( F_{\leq k_{3} +C} ) \right) \Vert_{L^{1}_{e} L^{2} _{e^{\perp},t}}
\\
& \overset{\eqref{multiplierint}}{\lesssim}2^{k_{3} (2s-1)} \Vert \tilde{f}_{k_{1}} \Vert_{L^{\infty}_{e'} L^{2}_{e'^{\perp},t}} \Vert | \nabla |^{-1} \tilde{g}_{k_{2}} \Vert_{L^{2}_{e'} L^{\infty}_{e'^{\perp},t}} \Vert | \nabla |^{-1} \tilde{h}_{k_{3}} \Vert_{L_{e}^{2} L^{\infty}_{e^{\perp},t}}
\\
&\overset{\eqref{local_smoothing}, \eqref{maximal_Estimate}}{\lesssim }2^{k_{3} (2s-1)}  2^{-(2s-1)k_{1}/2} 2^{-k_{1}/2} \Vert f_{k_{1}} \Vert_{Z_{k_{1}}} 2^{k_{2} \frac{n-2}{2}} \Vert g_{k_{2}} \Vert_{Z_{k_{2}}} \Vert 2^{k_{3} \frac{n-1}{2} -k_{3}} \Vert h_{k_{3}} \Vert_{Z_{k_{3}}}. 
\end{split}
 \end{equation}
We plug the bound \eqref{trl8} back into \eqref{trl7}, to obtain
\begin{equation}
\begin{split}
& \Vert \Delta_{k} ( \i\partial_{t} - ( - \Delta)^{s} + \i)^{-1} \ind_{[-T',T']}\partial_{j} \left( \nabla^{2s-1}_{i}( F_{\leq k_{3} + C} | \nabla |^{-1}( \tilde{h}_{k_{3}} ) )  - |\nabla|^{-1}\tilde{h}_{k_{3}}\nabla^{2s-1}_{i} ( F_{\leq k_{3} +C} ) \right) \Vert_{Z_{k}}
\\
& \lesssim 2^{k_{3} ( 2s-1) } 2^{-(2s-1)k/2} 2^{-(2s-1)k_{1}/2} \Vert f_{k_{1}} \Vert_{Z_{k_{1}}} 2^{k_{2} \frac{n-2}{2}} \Vert g_{k_{2}} \Vert_{Z_{k_{2}}} 2^{k_{3} \frac{n-2}{2}} \Vert h_{k_{3}} \Vert_{Z_{k_{3}}}.
\end{split}
\end{equation}
Multiplying by $ 2^{k \sigma}$ and summing over $ k \leq k_{3} + C$ yields the result. To conclude case 1b, we need to estimate the last sum on the right hand side of \eqref{trl6}. For this term we can use Lemma \ref{Taylorstuff2}. Indeed, fix $l \in [ k_{3} +C , k_{1} +C ] \cap \Z$, then we have
\begin{equation}\label{trl9}
\begin{split}
& \Vert \Delta_{k} ( \i \partial_{t} - ( - \Delta)^{s} + \i)^{-1} \ind_{[-T',T']} \partial_{j} \left( \nabla^{2s-1}_{i} ( F_{l} | \nabla |^{-1}( \tilde{h}_{k_{3}} ) )  - |\nabla|^{-1}\tilde{h}_{k_{3}}\nabla^{2s-1}_{i} ( F_{l} )\right) \Vert_{Z_{k}}
\\
& \lesssim 2^{l} 2^{-(2s-1)l/2} \max_{e \in \mathscr{E}} \Vert  \nabla^{2s-1}_{i} ( F_{l} | \nabla |^{-1}( \tilde{h}_{k_{3}} ) )  - |\nabla|^{-1}\tilde{h}_{k_{3}}\nabla^{2s-1}_{i} ( F_{l} )\Vert_{L^{1}_{e}L^{2}_{e^{\perp},t}}
\end{split}
\end{equation}
and notice that for any $ e \in \mathscr{E}$ we have by Lemma \ref{Taylorstuff2} 
\begin{equation}\label{trl10}
\begin{split}
& \Vert  \nabla^{2s-1}_{i} ( F_{l} | \nabla |^{-1}( \tilde{h}_{k_{3}} ) )  - |\nabla|^{-1}\tilde{h}_{k_{3}}\nabla^{2s-1}_{i} ( F_{l} )\Vert_{L^{1}_{e}L^{2}_{e^{\perp},t}}
\\
& \lesssim \sum_{ \vert \alpha \vert=1}^{\infty} \frac{a_{|\alpha|}}{\alpha !} 2^{l (2s-1 - | \alpha|)} \Vert F_{l} D^{\alpha} |\nabla |^{-1}\tilde{h}_{k_{3}} \Vert_{L^{1}_{e} L^{2}_{e^{\perp},t}}
\\
& \lesssim \sum_{ \vert \alpha \vert=1}^{\infty} \frac{a_{|\alpha|}}{\alpha !} 2^{l (2s-1 - | \alpha|)} 2^{k_{3} \vert \alpha \vert} \Vert  \vert \nabla\vert^{-1}\tilde{h}_{k_{3}} \Vert_{L^{2}_{e} L^{\infty}_{e^{\perp},t}} \Vert F_{l} \Vert_{L^{2}}
\\
& \lesssim 2^{l ( 2s-2)} 2^{k_{3}} \Vert \tilde{f}_{k_{1}} \Vert_{L^{\infty}_{e'}L^{2}_{e'^{\perp},t}} \Vert \vert \nabla \vert^{-1} \tilde{g}_{k_{2}} \Vert_{L^{2}_{e'}L^{\infty}_{e'^{\perp},t}} 2^{k_{3} \frac{n-1}{2}} \Vert \vert \nabla \vert^{-1} h_{k_{3}} \Vert_{Z_{k_{3}}}
\\
& \lesssim 2^{l(2s-2)} 2^{-k_{1} (2s-1)/2} \Vert f_{k_{1}} \Vert_{Z_{k_{1}}} 2^{k_{2} \frac{n-2}{2}} \Vert g_{k_{2}} \Vert_{Z_{k_{2}}} 2^{k_{3} \frac{n-2}{2}} \Vert h_{k_{3}} \Vert_{Z_{k_{3}}}
\end{split}
\end{equation}
Plugging \eqref{trl10} back into \eqref{trl9} yields 
\begin{equation*}
\begin{split}
& \Vert \Delta_{k} ( \i \partial_{t} - ( - \Delta)^{s} + \i)^{-1} \ind_{[-T',T']} \partial_{j} \left( \nabla^{2s-1}_{i} ( F_{l} | \nabla |^{-1}( \tilde{h}_{k_{3}} ) )  - |\nabla|^{-1}\tilde{h}_{k_{3}}\nabla^{2s-1}_{i} ( F_{l} )\right) \Vert_{Z_{k}}
\\
& \lesssim 2^{(l-k_{1})(2s-1)/2}  \Vert f_{k_{1}} \Vert_{Z_{k_{1}}} 2^{k_{2} \frac{n-2}{2} } \Vert g_{k_{2}} \Vert_{Z_{k_{2}}} 2^{k_{3} \frac{n-2}{2}} \Vert h_{k_{3}} \Vert_{Z_{k_{3}}}
\end{split}
\end{equation*}
summing over $l \in [k_{3}+ C , k_{1} +C] $ and then multiplying by $ 2^{k \sigma}$ and summing over $ k \leq k_{1} +C$ gives the result. This concludes case 1b.
\\
\textbf{Case 1c:} If $ \max \{ k_{2} , k_{3} \} \leq k_{1} - C$ for sufficently large $C$. Assume without loss of generality that $k_{2} \geq k_{3}$. In this case we make use of Lemma \ref{Taylorstuff3}. More precisely, recall the identity \eqref{trl3}  
\begin{equation*}
\begin{split}
&  T_{1}^{i}(\tilde{f}_{k_{1}}, \tilde{g}_{k_{2}},\tilde{h}_{k_{3}})(x,t)
\\
& = \nabla^{2s-1}_{i} ( \tilde{f}_{k_{1}} | \nabla |^{-1} ( \tilde{g}_{k_{2}}) | \nabla |^{-1}( \tilde{h}_{k_{3}} ) )
\\
& - |\nabla|^{-1}\tilde{h}_{k_{3}}\nabla^{2s-1}_{i} ( \tilde{f}_{k_{1}} |\nabla|^{-1} \tilde{g}_{k_{2}} )
\\
& - | \nabla |^{-1} \tilde{g}_{k_{2}} \nabla^{2s-1}_{i} ( \tilde{f}_{k_{1}} |\nabla|^{-1} \tilde{h}_{k_{3}}) 
\\
& + |\nabla|^{-1} \tilde{g}_{k_{2}} |\nabla|^{-1}\tilde{h}_{k_{3}} \nabla^{2s-1}_{i} ( \tilde{f}_{k_{1}}) 
\end{split}
\end{equation*}
take the derivative $ \partial_{j}$ of both sides to we obtain
\begin{equation}\label{trl11}
\begin{split}
& \partial_{j} T^{i}_{1} ( \tilde{f}_{k_{1}}, \tilde{g}_{k_{2}} ,  h_{k_{3}} ) 
\\
& = \partial_{j} \left( \nabla^{2s-1}_{i} ( \tilde{f}_{k_{1}} | \nabla |^{-1} ( \tilde{g}_{k_{2}}) | \nabla |^{-1}( \tilde{h}_{k_{3}} ) ) \right) - | \nabla |^{-1} \tilde{h}_{k_{3}} \partial_{j} \left( \nabla^{2s-1}_{i}(\tilde{f}_{k_{1}} |\nabla|^{-1} \tilde{g}_{k_{2}} ) \right)
\\
& + \partial_{j} |\nabla|^{-1} \tilde{h}_{k_{3}} \left( |\nabla|^{-1} \tilde{g}_{k_{2}} \nabla^{2s-1}_{i} ( \tilde{f}_{k_{1}}) - \nabla^{2s-1}_{i} ( \tilde{f}_{k_{1}} |\nabla|^{-1}\tilde{g}_{k_{2}}) \right) 
\\
& + |\nabla|^{-1} \tilde{g}_{k_{2}} \left( |\nabla|^{-1}\tilde{h}_{k_{3}} \partial_{j}\nabla^{2s-1}_{i}(\tilde{f}_{k_{1}}) - \partial_{j} \nabla^{2s-1}_{i} (|\nabla|^{-1} \tilde{h}_{k_{3}} \tilde{f}_{k_{1}} ) \right)
\\
& - \partial_{j} | \nabla|^{-1} \tilde{g}_{k_{2}} \left( \nabla^{2s-1}_{i} ( \tilde{f}_{k_{1} } | \nabla|^{-1} \tilde{h}_{k_{3}} ) - | \nabla|^{-1} \tilde{h}_{k_{3}} \nabla^{2s-1}_{i} ( \tilde{f}_{k_{1}}) \right).
\end{split}
\end{equation}
The idea is to use Lemma \ref{Taylorstuff3} for the first line, and third line on the right hand side of \eqref{trl11} and Lemma \ref{Taylorstuff2} for second, and fourth lines on the right hand side of \eqref{trl11}.  Define $ F : = \tilde{f}_{k_{1}} |\nabla|^{-1} \tilde{g}_{k_{2}}$, and define $ \nabla^{2s-1}_{i,j}:= \partial_{j} \nabla^{2s-1}_{i}$. Then the first line on the right hand side \eqref{trl11} can be written as
\begin{equation*}
\begin{split}
& \partial_{j} \left( \nabla^{2s-1}_{i} ( \tilde{f}_{k_{1}} | \nabla |^{-1} ( \tilde{g}_{k_{2}}) | \nabla |^{-1}( \tilde{h}_{k_{3}} ) ) \right) - | \nabla |^{-1} \tilde{h}_{k_{3}} \partial_{j} \left( \nabla^{2s-1}_{i}(\tilde{f}_{k_{1}} |\nabla|^{-1} \tilde{g}_{k_{2}} ) \right)
\\
& = \nabla^{2s-1}_{i,j} ( F |\nabla|^{-1} \tilde{h}_{k_{3}} ) - |\nabla|^{-1} \tilde{h}_{k_{3}} \nabla^{2s-1}_{i,j} (F).
\end{split}
\end{equation*}
Therefore, by Lemma \ref{Taylorstuff3}, the assumption that $ k_{2} \geq k_{3},$ we obtain
\begin{equation*}
\begin{split}
& \Vert \Delta_{k} ( \i \partial_{t} - ( -\Delta)^{s} +\i)^{-1} \ind_{[-T',T']} \left(   \nabla^{2s-1}_{i,j} ( F |\nabla|^{-1} \tilde{h}_{k_{3}} ) - |\nabla|^{-1} \tilde{h}_{k_{3}} \nabla^{2s-1}_{i,j} (F) \right) \Vert_{Z_{k}}
\\
& \lesssim 2^{-k_{1} (2s-1)/2} \Vert \nabla^{2s-1}_{i,j} ( F |\nabla|^{-1} \tilde{h}_{k_{3}} ) - |\nabla|^{-1} \tilde{h}_{k_{3}} \nabla^{2s-1}_{i,j} (F) \Vert_{L^{1}_{e'}L^{2}_{e'^{\perp},t}}
\\
& \lesssim 2^{-k_{1}(2s-1)/2} 2^{k_{1}(2s-1)} 2^{k_{3}} \Vert F \Vert_{L^{2}} \Vert |\nabla|^{-1} \tilde{h}_{k_{3}} \Vert_{L^{2}_{e'}L^{\infty}_{e'^{\perp},t}}
\\
& \lesssim \Vert f_{k_{1}} \Vert_{Z_{k_{1}}} 2^{k_{2} \frac{n-1}{2}} 2^{-k_{2}} \Vert g_{k_{2}} \Vert_{Z_{k_{2}}} 2^{ k_{3} \frac{n+1}{2}} 2^{-k_{3}} \Vert h_{k_{3}} \Vert_{Z_{k_{3}}}
\\
& \lesssim \Vert f_{k_{1}} \Vert_{Z_{k_{1}}} 2^{ k_{2} \frac{n-2}{2}} \Vert g_{k_{2}} \Vert_{Z_{k_{2}}} 2^{ k_{3} \frac{n-2}{2}} \Vert h_{k_{3}} \Vert_{Z_{k_{3}}}.
\end{split}
\end{equation*}
Which readily implies the estimates for the first term on the right hand side of \eqref{trl11}. Next, we deal with the second term on the right hand side of \eqref{trl11}. For this we use Lemma \ref{Taylorstuff2}
\begin{equation}\label{trl12}
\begin{split}
& \Vert \Delta_{k} ( \i \partial_{t} - ( - \Delta)^{s} + \i)^{-1}\ind_{[-T',T']} \partial_{j} |\nabla|^{-1} \tilde{h}_{k_{3}} \left( |\nabla|^{-1} \tilde{g}_{k_{2}} \nabla^{2s-1}_{i} ( \tilde{f}_{k_{1}}) - \nabla^{2s-1}_{i} ( \tilde{f}_{k_{1}} |\nabla|^{-1}\tilde{g}_{k_{2}}) \right) \Vert_{Z_{k}}
\\
& \lesssim 2^{-k_{1} (2s-1)/2} \Vert \partial_{j} |\nabla|^{-1} \tilde{h}_{k_{3}} \left( |\nabla|^{-1} \tilde{g}_{k_{2}} \nabla^{2s-1}_{i} ( \tilde{f}_{k_{1}}) - \nabla^{2s-1}_{i} ( \tilde{f}_{k_{1}} |\nabla|^{-1}\tilde{g}_{k_{2}}) \right) \Vert_{L^{1}_{e'}L^{2}_{e'^{\perp},t}}
\\
& \lesssim 2^{-k_{1} (2s-1)/2} 2^{k_{3} \frac{n-1}{2}} \Vert h_{k_{3}} \Vert_{Z_{k_{3}}} \Vert |\nabla|^{-1} \tilde{g}_{k_{2}} \nabla^{2s-1}_{i} ( \tilde{f}_{k_{1}}) - \nabla^{2s-1}_{i} ( \tilde{f}_{k_{1}} |\nabla|^{-1}\tilde{g}_{k_{2}}) \Vert_{L^{2}}
\\
& \lesssim 2^{-k_{1} (2s-1)/2} 2^{ k_{3} \frac{n-1}{2}} \Vert h_{k_{3}} \Vert_{Z_{k_{3}}}  2^{k_{1} (2s-2)} \Vert \tilde{f}_{k_{1}} \Vert_{L^{\infty}_{e'}L^{2}_{e'^{\perp},t}} \Vert \tilde{g}_{k_{2}} \Vert_{L^{2}_{e'}L^{\infty}_{e'^{\perp},t}}
\\
& \lesssim \Vert f_{k_{1}} \Vert_{Z_{k_{1}}} 2^{ k_{2} \frac{n-2}{2}} \Vert g_{k_{2}} \Vert_{Z_{k_{2}}} 2^{ k_{3} \frac{n-2}{2}} \Vert h_{k_{3}} \Vert_{Z_{k_{3}}}.
\end{split}
\end{equation}
 Next, we deal with the third line on the right hand side of \eqref{trl11}.
 By Lemma \ref{Taylorstuff3} we have
 \begin{equation*}
\begin{split}
& \Vert \Delta_{k} ( \i\partial_{t} - ( -\Delta)^{s} +\i)^{-1} \ind_{[-T',T']}\left(  |\nabla|^{-1} \tilde{g}_{k_{2}} \left( |\nabla|^{-1}\tilde{h}_{k_{3}} \partial_{j}\nabla^{2s-1}_{i}(\tilde{f}_{k_{1}}) - \partial_{j} \nabla^{2s-1}_{i} (|\nabla|^{-1} \tilde{h}_{k_{3}} \tilde{f}_{k_{1}} ) \right) \right)\Vert_{Z_{k}}
\\
& \lesssim 2^{-k_{1} (2s-1)/2} \max_{ e \in \mathscr{E}} \Vert  |\nabla|^{-1} \tilde{g}_{k_{2}} \left( |\nabla|^{-1}\tilde{h}_{k_{3}} \partial_{j}\nabla^{2s-1}_{i}(\tilde{f}_{k_{1}}) - \partial_{j} \nabla^{2s-1}_{i} (|\nabla|^{-1} \tilde{h}_{k_{3}} \tilde{f}_{k_{1}} ) \right) \Vert_{L^{1}_{e} L^{2}_{e^{\perp},t}}
\\
& \lesssim 2^{-k_{1}(2s-1)/2} 2^{k_{2} \frac{n-1}{2}} 2^{-k_{2}} \Vert g_{k_{2}} \Vert_{Z_{k_{2}}} \Vert |\nabla|^{-1}\tilde{h}_{k_{3}} \partial_{j}\nabla^{2s-1}_{i}(\tilde{f}_{k_{1}}) - \partial_{j} \nabla^{2s-1}_{i} (|\nabla|^{-1} \tilde{h}_{k_{3}} \tilde{f}_{k_{1}} ) \Vert_{L^{2}(\R^{n+1})} 
\\
& \lesssim 2^{-k_{1}(2s-1)/2} 2^{k_{2} \frac{n-1}{2}} 2^{-k_{2}} \Vert g_{k_{2}} \Vert_{Z_{k_{2}}} 2^{k_{1} (2s-1)} \Vert \tilde{f}_{k_{1}} \Vert_{L^{\infty}_{e'} L^{2}_{e'^{\perp},t}} 2^{k_{3}} \Vert |\nabla |^{-1}\tilde{h}_{k_{3}} \Vert_{L^{2}_{e'}L^{\infty}_{e'^{\perp},t}}
\\
& \lesssim 2^{k_{2} \frac{n-1}{2}} 2^{-k_{2}} \Vert g_{k_{2}} \Vert_{Z_{k_{2}}} \Vert f_{k_{1}} \Vert_{Z_{k_{1}}} 2^{k_{3} \frac{n+1}{2}} 2^{-k_{3}} \Vert h_{k_{3}} \Vert_{Z_{k_{3}}} 
\\
& \lesssim \Vert f_{k_{1}} \Vert_{Z_{k_{1}}} 2^{ k_{2} \frac{n-2}{2}} \Vert g_{k_{2}} \Vert_{Z_{k_{2}}} 2^{ k_{3} \frac{n-2}{2}} \Vert h_{k_{3}} \Vert_{Z_{k_{3}}}.
\end{split}
 \end{equation*}
Lastly, we deal with the fourth line on the right hand side of \eqref{trl11}.
 By Lemma \ref{Taylorstuff2}
 \begin{equation*}
\begin{split}
& \Vert \Delta_{k} ( \i \partial_{t} - ( - \Delta)^{s} + \i)^{-1} \ind_{[-T',T']} \partial_{j} | \nabla|^{-1} \tilde{g}_{k_{2}} \left( \nabla^{2s-1}_{i} ( \tilde{f}_{k_{1} } | \nabla|^{-1} \tilde{h}_{k_{3}} ) - | \nabla|^{-1} \tilde{h}_{k_{3}} \nabla^{2s-1}_{i} ( \tilde{f}_{k_{1}}) \right) \Vert_{Z_{k}}
\\
& \lesssim 2^{-k_{1} (2s-1)/2} \max_{e \in \mathscr{E}} \Vert \partial_{j} | \nabla|^{-1} \tilde{g}_{k_{2}} \left( \nabla^{2s-1}_{i} ( \tilde{f}_{k_{1} } | \nabla|^{-1} \tilde{h}_{k_{3}} ) - | \nabla|^{-1} \tilde{h}_{k_{3}} \nabla^{2s-1}_{i} ( \tilde{f}_{k_{1}}) \right) \Vert_{L^{1}_{e}L^{2}_{e^{\perp},t}}
\\
& \lesssim 2^{-k_{1}(2s-1)/2} \max_{e \in \mathscr{E}}\Vert \tilde{g}_{k_{2}} \Vert_{L^{2}_{e} L^{\infty}_{e^{\perp},t}} \Vert  \nabla^{2s-1}_{i} ( \tilde{f}_{k_{1} } | \nabla|^{-1} \tilde{h}_{k_{3}} ) - | \nabla|^{-1} \tilde{h}_{k_{3}} \nabla^{2s-1}_{i} ( \tilde{f}_{k_{1}}) \Vert_{L^{2}(\R^{n+1})}
\\
& \lesssim 2^{-k_{1} (2s-1)/2} 2^{k_{2} \frac{n-1}{2}} \Vert g_{k_{2}} \Vert_{Z_{k_{2}}} 2^{k_{1} (2s-2)} \Vert \tilde{f}_{k_{1}} \Vert_{L^{\infty}_{e'}L^{2}_{e'^{\perp},t}} \Vert \tilde{h}_{k_{3}} \Vert_{L^{2}_{e'}L^{\infty}_{e'^{\perp},t}} 
\\
& \lesssim \Vert f_{k_{1}} \Vert_{Z_{k_{1}}} 2^{ k_{2} \frac{n-2}{2}} \Vert g_{k_{2}} \Vert_{Z_{k_{2}}} 2^{ k_{3} \frac{n-2}{2}} \Vert h_{k_{3}} \Vert_{Z_{k_{3}}}.
\end{split}
 \end{equation*}
 This concludes case 1c, and hence case 1.
 \\
 \textbf{Case 2:} If $ \max\{ k_{1},k_{2} ,k_{3} \} \in \{ k_{2} ,k_{3} \}$. By symmetry, we may assume $ k_{2} \geq k_{3}$. Further, we may assume $ g_{k_{2}}$ has Fourier transform supported in $ \{ \langle \xi , e' \rangle \geq 2^{k_{2} -c} \}$, for some $ e' \in \mathscr{E}$. Then to prove \eqref{trl1} we consider two sub-cases. 
 \\
  \textbf{Case 2a:} If $2^{k_{1}} \lesssim 2^{k_{3}}$ Then recall the identity \eqref{trl3} 
\begin{equation}\label{trl13}
\begin{split}
&  T_{1}^{i}(\tilde{f}_{k_{1}}, \tilde{g}_{k_{2}},\tilde{h}_{k_{3}})(x,t)
\\
& = \nabla^{2s-1}_{i} ( \tilde{f}_{k_{1}} | \nabla |^{-1} ( \tilde{g}_{k_{2}}) | \nabla |^{-1}( \tilde{h}_{k_{3}} ) )
\\
& - |\nabla|^{-1}\tilde{h}_{k_{3}}\nabla^{2s-1}_{i} ( \tilde{f}_{k_{1}} |\nabla|^{-1} \tilde{g}_{k_{2}} )
\\
& - | \nabla |^{-1} \tilde{g}_{k_{2}} \nabla^{2s-1}_{i} ( \tilde{f}_{k_{1}} |\nabla|^{-1} \tilde{h}_{k_{3}}) 
\\
& + |\nabla|^{-1} \tilde{g}_{k_{2}} |\nabla|^{-1}\tilde{h}_{k_{3}} \nabla^{2s-1}_{i} ( \tilde{f}_{k_{1}}) 
\end{split}
\end{equation}
we handle each term on the right hand side of \eqref{trl13} separately. For the first term we have
\begin{equation*}
\begin{split}
& \Vert \Delta_{k} ( \i \partial_{t} - ( -\Delta)^{s} + \i)^{-1} \ind_{[-T',T']}\partial_{j} \left(     \nabla^{2s-1}_{i} ( \tilde{f}_{k_{1}} | \nabla |^{-1} ( \tilde{g}_{k_{2}}) | \nabla |^{-1}( \tilde{h}_{k_{3}} ) ) \right) \Vert_{Z_{k}}
\\
& \lesssim 2^{-k (2s-1)/2} 2^{k} \max_{e \in \mathscr{E}} \Vert \Delta_{k} \nabla^{2s-1}_{i} ( \tilde{f}_{k_{1}} | \nabla |^{-1} ( \tilde{g}_{k_{2}}) | \nabla |^{-1}( \tilde{h}_{k_{3}} ) ) \Vert_{L^{1}_{e} L^{2}_{e^{\perp},t}}
\\
& \lesssim 2^{-k(2s-1)/2} 2^{k} 2^{k(2s-1)} \max_{e \in \mathscr{E}} \Vert \tilde{f}_{k_{1}} \Vert_{L^{2}_{e} L^{\infty}_{e^{\perp},t}} 2^{-k_{3}} \Vert \tilde{h}_{k_{3}} \Vert_{L^{2}_{e'}L^{\infty}_{e'^{\perp},t}} 2^{-k_{2} } \Vert \tilde{g}_{k_{2}} \Vert_{L^{\infty}_{e'} L^{2}_{e'^{\perp},t}}
\\
& \lesssim 2^{k(2s-1)/2} 2^{-k_{2}(2s-1)/2} 2^{k} 2^{k_{1} \frac{n-1}{2}} \Vert f_{k_{1}} \Vert_{Z_{k_{1}}} 2^{-k_{3}} 2^{k_{3} \frac{n-1}{2}} \Vert h_{k_{3}} \Vert_{Z_{k_{3}}} 2^{-k_{2}} \Vert g_{k_{2}} \Vert_{Z_{k_{2}}} 
\\
& \lesssim  2^{k(2s-1)/2} 2^{-k_{2}(2s-1)/2} 2^{k-k_{2}} 2^{k_{1} \frac{n-2}{2}} \Vert f_{k_{1}} \Vert_{Z_{k_{1}}}  2^{k_{3} \frac{n-2}{2}} \Vert h_{k_{3}} \Vert_{Z_{k_{3}}} \Vert g_{k_{2}} \Vert_{Z_{k_{2}}}
\end{split}
\end{equation*}
multiplying by $ 2^{k \sigma}$ and summing over $ k \leq k_{2} +C$ yields the result. The remaining terms on the right hand side of \eqref{trl13} are handled in the same way. Namely, using only the maximal, smoothing estimates and the assumption that $k_{3} \geq k_{1} -C$.
\\
\textbf{Case 2b:} if $k_{3} \leq k_{1} -C$. Then in view of case 1b we may assume that $ k_{1} \leq k_{2} -C.$ In this case we make use of Lemma \eqref{Taylorstuff2}. Namely using \eqref{trl13}
\begin{equation}\label{trl14}
\begin{split}
& \Vert \Delta_{k} (\i \partial_{t} - ( -\Delta)^{s} + \i)^{-1} \ind_{[-T',T']} \partial_{j} T^{i}_{1}(\tilde{f}_{k_{1}} ,\tilde{g}_{k_{2}} , \tilde{h}_{k_{3}} ) \Vert_{Z_{k}}
\\
& \lesssim 2^{-k_{2} ( 2s-1)/2} 2^{k_{2}}  \Vert  T^{i}_{1}(\tilde{f}_{k_{1}} ,\tilde{g}_{k_{2}} , \tilde{h}_{k_{3}} ) \Vert_{L^{1}_{e'} L^{2}_{e'^{\perp},t}}
\\
& \lesssim 2^{-k_{2} ( 2s-1)/2} 2^{k_{2}}   \Vert \nabla^{2s-1}_{i} ( \tilde{f}_{k_{1}} | \nabla |^{-1} ( \tilde{g}_{k_{2}}) | \nabla |^{-1}( \tilde{h}_{k_{3}} ) ) - |\nabla|^{-1}\tilde{h}_{k_{3}}\nabla^{2s-1}_{i} ( \tilde{f}_{k_{1}} |\nabla|^{-1} \tilde{g}_{k_{2}} ) \Vert_{L^{1}_{e'} L^{2}_{e'^{\perp},t}} 
\\
& +2^{-k_{2} ( 2s-1)/2} 2^{k_{2}}  \Vert   | \nabla |^{-1} \tilde{g}_{k_{2}} \nabla^{2s-1}_{i} ( \tilde{f}_{k_{1}} |\nabla|^{-1} \tilde{h}_{k_{3}}) -  |\nabla|^{-1} \tilde{g}_{k_{2}} |\nabla|^{-1}\tilde{h}_{k_{3}} \nabla^{2s-1}_{i} ( \tilde{f}_{k_{1}})  \Vert_{L^{1}_{e'}L^{2}_{e'^{\perp},t}}
\end{split}
\end{equation}
We estimate each term on the right hand side of \eqref{trl14} separately. Starting with the first term, we have by Lemma \ref{Taylorstuff2}
\begin{equation}\label{trl15}
\begin{split}
& \Vert \nabla^{2s-1}_{i} ( \tilde{f}_{k_{1}} | \nabla |^{-1} ( \tilde{g}_{k_{2}}) | \nabla |^{-1}( \tilde{h}_{k_{3}} ) ) - |\nabla|^{-1}\tilde{h}_{k_{3}}\nabla^{2s-1}_{i} ( \tilde{f}_{k_{1}} |\nabla|^{-1} \tilde{g}_{k_{2}} ) \Vert_{L^{1}_{e'} L^{2}_{e^{\perp},t}} 
\\
& \lesssim \sum_{ | \alpha | =1}^{\infty} \frac{ a_{ \vert \alpha \vert}}{\alpha !} 2^{k_{2} (2s-1 - \vert \alpha \vert)} \Vert \tilde{f}_{k_{1}} |\nabla|^{-1}\tilde{g}_{k_{2}} | \nabla |^{-1} D^{\alpha} \tilde{h}_{k_{3}} \Vert_{L^{1}_{e'} L^{2}_{e'^{\perp},t}}
\\
& \lesssim 2^{k_{2} (2s-2)} \Vert |\nabla|^{-1}\tilde{g}_{k_{2}} \Vert_{L^{\infty}_{e'} L^{2}_{e'^{\perp},t}} \Vert \tilde{f}_{k_{1}} \Vert_{L^{2}_{e'}L^{\infty}_{e'^{\perp},t}} 2^{k_{3}} \Vert |\nabla|^{-1} \tilde{h}_{k_{3}} \Vert_{L^{2}_{e'}L^{\infty}_{e'^{\perp},t}} \sum_{ |\alpha|=1}^{\infty} \frac{a_{ \vert \alpha \vert}}{\alpha !} 2^{ - C \vert \alpha \vert}
\\
& \lesssim 2^{k_{2}(2s-1)/2} 2^{-k_{2}} \Vert g_{k_{2}} \Vert_{Z_{k_{2}}} 2^{k_{1} \frac{n-2}{2} } \Vert f_{k_{1}} \Vert_{Z_{k_{1}}} 2^{k_{3} \frac{n-2}{2}} \Vert h_{k_{3}} \Vert_{Z_{k_{3}}}. 
\end{split}
\end{equation}
Similarly, we estimate the last term on the right hand side of \eqref{trl14}. By Lemma \ref{Taylorstuff2} again, we obtain
\begin{equation}\label{trl16}
\begin{split}
& \Vert   | \nabla |^{-1} \tilde{g}_{k_{2}} \nabla^{2s-1}_{i} ( \tilde{f}_{k_{1}} |\nabla|^{-1} \tilde{h}_{k_{3}}) -  |\nabla|^{-1} \tilde{g}_{k_{2}} |\nabla|^{-1}\tilde{h}_{k_{3}} \nabla^{2s-1}_{i} ( \tilde{f}_{k_{1}})  \Vert_{L^{1}_{e'}L^{2}_{e'^{\perp},t}}
\\
& \lesssim 2^{-k_{2}} \Vert \tilde{g}_{k_{2}} \Vert_{L^{\infty}_{e'}L^{2}_{e'^{\perp},t}} \sum_{ \vert \alpha \vert =1}^{\infty} \frac{a_{\vert \alpha \vert}}{\alpha !} 2^{ (2s-1 - \vert \alpha \vert) k_{1}} \Vert \tilde{f}_{k_{1}} D^{\alpha} |\nabla |^{-1} \tilde{h}_{k_{3}} \Vert_{L^{1}_{e'}L^{\infty}_{e'^{\perp},t}}
\\
& \lesssim 2^{-k_{2}} 2^{ - k_{2} ( 2s-1)/2} \Vert g_{k_{2}} \Vert_{Z_{k_{2}}} 2^{(2s-2)k_{1}} \Vert \tilde{f}_{k_{1}} \Vert_{L^{2}_{e'}L^{\infty}_{e'^{\perp},t}} \Vert \tilde{h}_{k_{3}} \Vert_{L^{2}_{e'}L^{\infty}_{e'^{\perp},t}} \sum_{|\alpha|=1}^{\infty} \frac{a_{\vert \alpha \vert}}{\alpha !} 2^{-C |\alpha |}
\\
& \overset{\eqref{local_smoothing},\eqref{maximal_Estimate}}{\lesssim} 2^{-k_{2}} \Vert g_{k_{2}} \Vert_{Z_{k_{2}}} 2^{k_{1} (2s-1)/2} 2^{k_{1} \frac{n-2}{2}} \Vert f_{k_{1}} \Vert_{Z_{k_{1}}} 2^{k_{3} \frac{n-2}{2}} \Vert h_{k_{3}} \Vert_{Z_{k_{3}}}.
\end{split}
\end{equation}
Plugging the bounds \eqref{trl15} and \eqref{trl16} back into \eqref{trl14} we obtain
\begin{equation*}
\begin{split}
& \Vert \Delta_{k} (\i \partial_{t} - ( -\Delta)^{s} + \i)^{-1} \partial_{j} T^{i}_{1}(\tilde{f}_{k_{1}} ,\tilde{g}_{k_{2}} , \tilde{h}_{k_{3}} ) \Vert_{Z_{k}}
\\
& \lesssim \Vert g_{k_{2}} \Vert_{Z_{k_{2}}} 2^{k_{1} \frac{n-2}{2}} \Vert f_{k_{1}} \Vert_{Z_{k_{1}}} 2^{k_{3} \frac{n-2}{2}} \Vert h_{k_{3}} \Vert_{Z_{k_{3}}}
\end{split}
\end{equation*}
which concludes case 2b.

\end{proof}
Now we see how the trilinear estimate \eqref{trlnes1} implies the nonlinear estimates for the first two terms on the right hand side of \eqref{recudedSchro}.

\begin{proposition}\label{Firstnonlinearestimate}
Let $n \geq 3$, $s \in (\frac{1}{2},1)$, $ \sigma \geq \frac{n-2}{2}$. Then we have the following inequality for any $ i, j \in \{ 1,...,n \}$. 
\begin{equation}\label{firstnonlines1}
\begin{split}
& \Vert  \partial_{j} \left( \langle e , [ \nabla^{2s-1}_{i} , e ](f) \rangle \right) \Vert_{N^{\sigma}(T')}
\\
& \lesssim \Vert f\Vert_{\tilde{F}^\sigma} \Vert |\nabla | e \Vert^{2}_{\tilde{F}^{\frac{n-2}{2}}(T')} 
\\
& + \Vert f \Vert_{\tilde{F}^{\frac{n-2}{2}}} \Vert | \nabla| e \Vert_{\tilde{F}^{\frac{n-2}{2}}(T')} \Vert |\nabla| e \Vert_{\tilde{F}^{\sigma}(T')}.
\end{split}
\end{equation}
where $e \in \{ e_{1},e_{2} \}$ and $ T'\in [0,T]$. Here $[-T,T]$ is the interval for which the local solution $u$ exists. 
\end{proposition}
\begin{proof} 
Notice the following identity
\begin{equation}\label{eqz1}
\begin{split}
& \langle e , [ \nabla^{2s-1}_{i} , e ](f) \rangle
\\
& =c \int_{\R^{n}} \frac{ | e(x,t) - e(y,t) |^{2} f(y,t) ( x-y )_{i}}{|x-y|^{n +2s}}
\end{split}
\end{equation}
Indeed, this follows from the fact that 
$$
[ \nabla^{2s-1}_{i} , e ](f)(x,t) = c \int_{\R^{n}} \frac{(x-y)_{i} f(y,t) ( e(x,t) - e(y,t))}{|x-y|^{n+2s}}
$$
and the fact that 
$$
\langle e(x,t) , e(x,t) -e(y,t) \rangle = \frac{1}{2} | e(x,t) - e(y,t) |^{2}
$$
combining these two identities yield \eqref{eqz1}. Then  \eqref{firstnonlines1} follows directly from applying \eqref{trlnes1} with $ g = E_{T'} (|\nabla| e) = h$ where $ E_{T'}(|\nabla| e)$ is any extension of $ \ind_{[-T',T']} |\nabla| e$. 
\end{proof}
Next, we show how \eqref{trlnes1} implies the nonlinear estimates for the term concerning $ \mathscr{H}(e_{1},e_{2})(V_{1},...,V_{n})$

\begin{proposition}\label{errnonlinears}
Let $ n \geq 3 $, $ s \in (1/2,1)$. Let $ e_{1},e_{2} : \R^{n+1} \to \S^{2}$ and $V_{i}'s$ be as in section \ref{s:reduction}.  Then the following inequality holds for any $j \in \{ 1,...,n \}$
\begin{equation*}
\begin{split}
& \Vert \ \partial_{j} \mathscr{H}(e_{1},e_{2})(V_{1},....,V_{n}) \Vert_{N^{\sigma}(T')} \\
& \lesssim  \Vert  | \nabla |e_{1}  \Vert_{\tilde{F}^{\frac{n-2}{2}}(T')} \Vert |\nabla| e_{2} \Vert_{\tilde{F}^{\frac{n-2}{2}}(T')} \Vert V_{x} \Vert_{\tilde{F}^{\sigma}(T')}
\\
& + \Vert  | \nabla |e_{1}  \Vert_{\tilde{F}^{\sigma}(T')} \Vert |\nabla| e_{2} \Vert_{\tilde{F}^{\frac{n-2}{2}}(T')} \Vert V_{x} \Vert_{\tilde{F}^{\frac{n-2}{2}}(T')}
\\
& + \Vert  | \nabla |e_{1}  \Vert_{\tilde{F}^{\frac{n-2}{2}}(T')} \Vert |\nabla| e_{2} \Vert_{\tilde{F}^{\sigma}} \Vert V_{x} \Vert_{\tilde{F}^{\frac{n-2}{2}}(T')}
\end{split}
\end{equation*}
where $ V_{x} : = ( V_{1},....,V_{n})$, $ T' \in [0,T]$. 
\end{proposition}
\begin{proof}
Recall that by definition we have 
$$
\mathscr{H}(e_{1},e_{2}) (V_{x}) : = \langle e_{2} , \sum_{i=1}^{n} \nabla^{2s-1}_{i}( V_{i} e_{1} ) \rangle + \langle e_{1} , \sum_{i=1}^{n} \nabla^{2s-1}_{i} (V_{i} e_{2} ) \rangle.
$$
Further, notice the following identity 
\begin{equation}\label{w1}
\begin{split}
& 2 \langle e_{2} , \sum_{i=1}^{n} \nabla^{2s-1}_{i}( V_{i} e_{1} ) \rangle + 2 \langle e_{1} , \sum_{i=1}^{n} \nabla^{2s-1}_{i} (V_{i} e_{2} ) \rangle
\\
& = \langle e_{1}, \sum_{i=1}^{n} \nabla^{2s-1}_{i} ( e_{2} V_{i} ) \rangle - \langle e_{1} , \sum_{i=1}^{n} \nabla^{2s-1}_{i} ( e_{2} )V_{i} \rangle 
\\
& + \langle e_{1} , \sum_{i=1}^{n} \nabla^{2s-1}_{i} ( e_{2}) V_{i} \rangle + \langle e_{2} , \sum_{i=1}^{n} \nabla^{2s-1}_{i} ( e_{1} V_{i} ) \rangle
\\
& + \langle e_{2} , \sum_{i=1}^{n} \nabla^{2s-1}_{i} ( e_{1} V_{i} ) \rangle  - \langle e_{2} , \sum_{i=1}^{n} \nabla^{2s-1}_{i} ( e_{1} ) V_{i} \rangle
\\
& + \langle e_{2}, \sum_{i=1}^{n} \nabla^{2s-1}_{i} (e_{1} ) V_{i} \rangle + \langle e_{1} , \sum_{i=1}^{n} \nabla^{2s-1}_{i}(e_{2} V_{i}) \rangle
\end{split}
\end{equation}
For the first term on the right hand side of \eqref{w1} we write it as follows
\begin{equation*}
\begin{split}
& \langle e_{1}, \sum_{i=1}^{n} \nabla^{2s-1}_{i} ( e_{2} V_{i} ) \rangle - \langle e_{1} , \sum_{i=1}^{n} \nabla^{2s-1}_{i} ( e_{2} )V_{i} \rangle
\\
& = \langle e_{1} , \sum_{i=1}^{n} H_{\nabla^{2s-1}_{i}} ( e_{2} , V_{i} ) \rangle
\end{split}
\end{equation*}
where 
$$
H_{\nabla^{2s-1}_{i}}(e , h ): = \nabla^{2s-1}_{i}( e h) - \nabla^{2s-1}_{i} (e ) h - \nabla^{2s-1}_{i}(h) e.
$$
Similarly, the third term on the right hand side of \eqref{w1} is equal to 
$$
\langle e_{2} , \sum_{i=1}^{n} H_{\nabla^{2s-1}_{i}} ( e_{1} ,V_{i} ) \rangle. 
$$
Moreover, we write the second term on the right hand side of \eqref{w1} as 
\begin{equation*}
\begin{split}
& \langle e_{1}, \sum_{i=1}^{n} \nabla^{2s-1}_{i} ( e_{2}) V_{i}  \rangle + \langle e_{2} , \sum_{i=1}^{n} \nabla^{2s-1}_{i} ( e_{1} V_{i}) \rangle 
\\
& = - \sum_{i=1}^{n} \tilde{H}_{\nabla^{2s-1}_{i}}( V_{i} e_{1} , e_{2}) 
\end{split}
\end{equation*}
where for any $ u, v : \R^{n} \to \R^{3}$ we define 
$$
\tilde{H}_{\nabla^{2s-1}_{i}} ( u,v) = \sum_{j=1}^{3} H_{\nabla^{2s-1}_{i}}( u_{j} , v_{j} ) 
$$
The same identity holds for the fourth term on the right hand side of \eqref{w1}. Putting all of these identities together, we arrive at 
\begin{equation*}
\begin{split}
& 2 \mathscr{H}(e_{1} ,e_{2} ) (V_{x} ) 
\\
& = \sum_{i=1}^{n} \left( \langle e_{1} , H_{\nabla^{2s-1}_{i}} ( e_{2},V_{i} ) \rangle  - \tilde{H}_{\nabla^{2s-1}_{i}} ( V_{i} e_{1} ,e_{2} ) \right) + \sum_{i=1}^{n} \left( \langle e_{2} , H_{\nabla^{2s-1}_{i} }(e_{1} , V_{i} ) \rangle - \tilde{H}_{\nabla^{2s-1}_{i}} ( e_{1} , V_{i} e_{2} ) \right).
\end{split}
\end{equation*}
Fix an arbitrary $ i\in \{ 1,...,n \}$. Then the result follows if we estimate
\begin{equation}\label{w2}
\begin{split}
& \Vert \ind_{[-T',T']} \partial_{j} \left( \langle e_{1} , H_{\nabla^{2s-1}_{i}} ( e_{2} , V_{i} ) \rangle - \tilde{H}_{\nabla^{2s-1}_{i}}( V_{i} e_{1} ,e_{2} ) \right) \Vert_{N^{\sigma}} 
\\
& + \Vert \ind_{[-T',T']} \partial_{j} \left( \langle e_{2} , H_{\nabla^{2s-1}_{i} }(e_{1} , V_{i} ) \rangle - \tilde{H}_{\nabla^{2s-1}_{i}} ( e_{1} , V_{i} e_{2} ) \right) \Vert_{N^{\sigma}}
\\
& \lesssim \Vert  | \nabla |e_{1}  \Vert_{\tilde{F}^{\frac{n-2}{2}}(T')} \Vert |\nabla| e_{2} \Vert_{\tilde{F}^{\frac{n-2}{2}}(T')} \Vert V_{x} \Vert_{\tilde{F}^{\sigma}(T')}
\\
& + \Vert  | \nabla |e_{1}  \Vert_{\tilde{F}^{\sigma}(T')} \Vert |\nabla| e_{2} \Vert_{\tilde{F}^{\frac{n-2}{2}}(T')} \Vert V_{x} \Vert_{\tilde{F}^{\frac{n-2}{2}}(T')}
\\
& + \Vert  | \nabla |e_{1}  \Vert_{\tilde{F}^{\frac{n-2}{2}}(T')} \Vert |\nabla| e_{2} \Vert_{\tilde{F}^{\sigma}} \Vert V_{x} \Vert_{\tilde{F}^{\frac{n-2}{2}}(T')}
\end{split}
\end{equation}
This follows from \eqref{trlnes1}. Indeed, notice that for any functions $ f,g,h : \mathbb{R}^{n+1} \to \C$ we have 
\begin{equation*}
\begin{split}
& g H_{\nabla^{2s-1}_{i}} (h,f) - H_{\nabla^{2s-1}_{i}}(fg,h ) 
\\
& =g \nabla^{2s-1}_{i} ( f h ) - g h \nabla^{2s-1}_{i} (f) - gf \nabla^{2s-1}_{i}(h) \\
& - \nabla^{2s-1}_{i} (fgh) + fg \nabla^{2s-1}_{i} (h) + h \nabla^{2s-1}_{i} (fg)
\\
& = - \nabla^{2s-1}_{i} ( fgh) + h \nabla^{2s-1}_{i} (fg)  + g \nabla^{2s-1}(fh) - gh \nabla^{2s-1} (f) 
\\
& \overset{\eqref{trl3}}{=} -T^{i}_{1} (f, |\nabla|g , | \nabla| h ). 
\end{split}
\end{equation*}
So we apply \eqref{trlnes1} with $ f = E_{T'}(V_{x})$ , $ g = E_{T'}( |\nabla| e_{1} )$, $ h = E_{T'}( |\nabla| e_{2})$, where $ E_{T'}(V_{x})$ , $  E_{T'}( |\nabla| e_{1} )$, $E_{T'}( |\nabla| e_{2})$ are any extensions of $\ind_{[-T',T']} V_{x}$, $ \ind_{[-T',T']} | \nabla | e_{1}$, $ \ind_{[-T',T']} |\nabla| e_{2}$ .

\end{proof}

\section{Nonlinear estimates II}
Throughout this section let $ u, V_{i}, e_{1},e_{2}$, and $T>0$ be as in Proposition \ref{modifiedsschrodeq}.\\
In this section we prove the nonlinear estimates concerning the  fourth term on the right hand side of \eqref{recudedSchro}. It consists of terms of the form 
$$
\partial_{j} \langle e_{1} , [ \nabla^{2s-1}_{i},e_{2} ](V_{j}) \rangle 
$$
where we recall that the frames constructed in section \ref{s:reduction} satisfy $ \sum_{i=1}^{n} \partial_{i} ( \langle  e_{1} , \partial_{i} e_{2} \rangle ) =0$. Therefore, it suffices to prove the following nonlinear estimate 
\begin{proposition}\label{trlCgauge}
Let $ n \geq 4$ and $ s \in (1/2,1)$, $ \sigma \geq \frac{n-2}{2}$ and $ T' \in [0,T]$. Let $ e_{1},e_{2} : \R^{n+1} \to \S^{2}$ be given as in section \ref{s:reduction}, namely 
$$
\text{Div} ( \langle e_{1} , \nabla e_{2} \rangle ) =0.
$$
Then we have the following estimate for any $ f \in F^{\sigma}$
\begin{equation}
\begin{split}
& \Vert  \partial_{j} \langle e_{1} , [ \nabla^{2s-1}_{i},e_{2} ](f) \rangle \Vert_{N^{\sigma}(T')} 
\\
& \lesssim \Vert | \nabla | e_{1} \Vert_{\tilde{F}^{\sigma}(T')} \Vert | \nabla | e_{2} \Vert_{\tilde{F}^{\frac{n-2}{2}}(T')} \Vert f \Vert_{\tilde{F}^{\frac{n-2}{2}}(T')}
\\
& + \Vert | \nabla | e_{1} \Vert_{\tilde{F}^{\frac{n-2}{2}}(T')} \Vert | \nabla | e_{2} \Vert_{\tilde{F}^{\sigma}(T')} \Vert f \Vert_{\tilde{F}^{\frac{n-2}{2}}(T')}
\\
&+ \Vert | \nabla | e_{1} \Vert_{\tilde{F}^{\frac{n-2}{2}}(T')} \Vert | \nabla | e_{2} \Vert_{\tilde{F}^{\frac{n-2}{2}}(T')} \Vert f \Vert_{\tilde{F}^{\sigma}(T')}.
\end{split}
\end{equation}
\end{proposition}
Since $ e_{1} \neq e_{2}$ we cannot hope to reduce the nonlinearity in \eqref{trlCgauge} to $T^{i}_{1}$. However, we use the Coulomb gauge condition to reduce the nonlinearity into a sum of manageable trilinear operators. To make the calculations more readable, we do all the calculations in frequency space.
\\
Define $ \psi : \R^{n} \to \R^{n} $ by $ \psi(\zeta) : = \zeta |\zeta|^{2s-2}$. Then
take the spatial Fourier transform of $ \partial_{j} \langle e_{1} , [ \nabla^{2s-1}_{i},e_{2} ] (f) $ to obtain ( we ignore the time variable $t$ for better readability)
\begin{equation}\label{gaugecal1}
\begin{split}
& \mathcal{F} \left( \partial_{j} \langle e_{1} , [ \nabla^{2s-1}_{i},e_{2} ] (f) \right)(\xi)
\\
& = \int_{\R^{n} \times \R^{n}} \xi_{j} \big[ \psi(\xi - \eta) - \psi(\xi -\kappa)    \big]_{i} f( \xi -\kappa) \langle \hat{e}_{2} ( \kappa - \eta) , \hat{e}_{1} ( \eta) \rangle_{\R^{3}} d \eta d \kappa
\end{split}
\end{equation}
We simplify the right hand side of \eqref{gaugecal1} further; notice that ( for $ \kappa \neq \eta$)
\begin{equation}\label{gaugecal2}
\begin{split}
& 1 = \frac{ \langle \kappa - \eta, \kappa - \eta \rangle}{ | \kappa - \eta |^{2}}
\\
& = \frac{ \langle \kappa- \eta , ( I - \frac{ \kappa}{|\kappa|} \otimes \frac{\kappa}{|\kappa|} ) (\kappa - \eta) \rangle }{| \kappa - \eta |^{2}} + \frac{ \langle \kappa- \eta ,  \frac{ \kappa}{|\kappa|} \otimes \frac{\kappa}{|\kappa|} ) (\kappa - \eta) \rangle }{| \kappa - \eta |^{2}}.
\end{split}
\end{equation}
Set 
\begin{equation}\label{defbij}
B_{i,j}(\xi,\eta,\kappa) : =  \xi_{j}  [ \psi( \xi- \eta) - \psi( \xi - \kappa)]_{i}.
\end{equation}
Then using \eqref{gaugecal2} we obtain that \eqref{gaugecal1} becomes
\begin{equation*}
\begin{split}
& \int_{\R^{n} \times \R^{n}} \xi_{j} \big[ \psi(\xi - \eta) - \psi(\xi -\kappa) \big]_{i} f( \xi - \kappa) \langle \hat{e}_{2} ( \kappa - \eta) , \hat{ e}_{1} ( \eta) \rangle_{\mathbb{R}^{3}} d \eta d \kappa
\\
& = \sum_{l,a=1}^{n} \int_{\R^{n} \times \R^{n}} B_{i,j} (\xi,\eta,\kappa) f(\xi-\kappa) \frac{ ( \kappa - \eta)_{l}}{|\kappa - \eta|^{2}} (  I - \frac{ \kappa}{|\kappa|} \otimes \frac{\kappa}{|\kappa|} )_{al} ( \kappa - \eta)_{a} \langle \hat{e}_{2} (  \kappa- \eta) , \hat{e}_{1}( \eta) \rangle_{\R^{3}} 
\\
& + \sum_{l=1}^{n} \int_{\R^{n} \times \R^{n}} B_{i,j}(\xi,\eta,\kappa) f(\xi -\kappa) \frac{ (\kappa- \eta)_{l}}{|\kappa - \eta|^{2}} \frac{\kappa_{l}}{|\kappa|} \langle \frac{\kappa}{|\kappa|} , \kappa- \eta \rangle \langle \hat{e}_{2} ( \kappa- \eta), \hat{e}_{1} (\eta) \rangle_{\R^{3}}.
\end{split}
\end{equation*}
Define 
\begin{equation}\label{Gamma1guy}
    \Gamma_{1} : = \sum_{l,a=1}^{n} \int_{\R^{n} \times \R^{n}} B_{i,j}(\xi,\eta,\kappa) f(\xi- \kappa) \frac{ ( \kappa - \eta)_{l}}{|\kappa - \eta|^{2}} (  I - \frac{ \kappa}{|\kappa|} \otimes \frac{\kappa}{|\kappa|} )_{al} ( \kappa - \eta)_{a} \langle \hat{e}_{2} (  \kappa- \eta) , \hat{e}_{1}( \eta) \rangle_{\R^{3}}, 
    \end{equation}
and 
\begin{equation}\label{Gamma2guy}
\Gamma_{2} : =\sum_{l=1}^{n} \int_{\R^{n} \times \R^{n}} B_{i,j}(\xi,\eta,\kappa)f(\xi-\kappa) \frac{ (\kappa- \eta)_{l}}{|\kappa - \eta|^{2}} \frac{\kappa_{l}}{|\kappa|} \langle \frac{\kappa}{|\kappa|} , \kappa- \eta \rangle \langle \hat{e}_{2} ( \kappa- \eta), \hat{e}_{1} (\eta) \rangle_{\R^{3}}.
\end{equation}
Therefore, we end up with the following 
\begin{equation*}
\begin{split}
& \mathcal{F}_{\R^{n}} \left( \partial_{j} \langle e_{1} , [ \nabla^{2s-1}_{i},e_{2} ] (f) \right) = \Gamma_{1} + \Gamma_{2}. 
\end{split}
\end{equation*}
And it suffices to prove the nonlinear estimates for each $ \mathcal{F}^{-1}(\Gamma_{i}) $ with $ i \in \{ 1,2 \}$. 
\subsection{Nonlinear estimates for $\Gamma_{1}$}
We start with $\Gamma_{1}$,the estimates for $ \Gamma_{1}$ will follow from the following trilinear estimate
\begin{proposition}\label{GaugeTrl1}
Let $ n \geq 4$, $ s \in (1/2,1)$ and $T' \in [0,\infty]$. Then define the trilinear operator $T^{i,j}_{2}$ via the Fourier transform (we are suppressing the $t$ dependence for readability)
\begin{equation}
\begin{split}
& \mathcal{F}_{\R^{n}} ( T^{i,j}_{2} (f,g,h) )(\xi) : = \sum_{l ,a=1}^{n} \int_{\R^{n} \times \R^{n}} B_{i,j} ( \xi, \eta, \kappa) \hat{f}(\xi-\kappa) 
\\
& 
\times \frac{ ( \kappa - \eta)_{l}}{|\kappa - \eta|^{2}} ( I- \frac{ \kappa}{|\kappa|} \otimes \frac{\kappa}{|\kappa|})_{al} ( \kappa- \eta)_{a} \frac{\hat{h}(\kappa- \eta)}{|\kappa-\eta|} \frac{\hat{g}(\eta)}{|\eta|} d \eta d \kappa.
\end{split}
\end{equation}
Then we have the following estimate for any $ \sigma \geq \frac{n-2}{2}$
\begin{equation*}
\begin{split}
& \Vert \ind_{[-T',T']} T^{i,j}_{2} ( \tilde{f} ,  \tilde{g} ,  \tilde{h} ) \Vert_{N^{\sigma}} 
\\
& \lesssim \Vert f \Vert_{F^{\sigma}} \Vert g \Vert_{F^{\frac{n-2}{2}}} \Vert h \Vert_{F^{\frac{n-2}{2}}}
\\
& + \Vert f \Vert_{F^{\frac{n-2}{2}}} \Vert g \Vert_{F^{\sigma}} \Vert h \Vert_{F^{\frac{n-2}{2}}}
\\
& + \Vert f \Vert_{F^{\frac{n-2}{2}}} \Vert g \Vert_{F^{\frac{n-2}{2}}} \Vert h \Vert_{F^{\sigma}}.
\end{split}
\end{equation*}

\end{proposition}
Before we start the proof, we explain the idea formally. We examine the nonlinearity at a dyadic level; let $f_{k_{1}} \in Z_{k_{1}}$ , $g_{k_{2}} \in Z_{k_{2}} $, and $h_{k_{3}} \in Z_{k_{3}}$. For simplicity assume that $2^{k_{1}} \gg 2^{\max\{ k_{2}, k_{3} \} } $. Then, in light of Lemma \ref{Taylorstuff3}, $B_{i,j}(\xi,\eta,\kappa)$ is practically placing $ (2s-1)$ derivatives on $ f_{k_{1}}$ and $ h_{k_{3}}$ gets one derivative. Moreover, notice that, for each fixed $a \in \{1,...,n \}$ we have 
\begin{equation}\label{identitycolumbguy1}
\sum_{l=1}^{n}  ( I- \frac{ \kappa}{|\kappa|} \otimes \frac{\kappa}{|\kappa|})_{al} ( \kappa- \eta)_{l} = -\sum_{l=1}^{n} ( I- \frac{ \kappa}{|\kappa|} \otimes \frac{\kappa}{|\kappa|})_{al} ( \eta)_{l}.
\end{equation} 
Therefore, using \eqref{multimultiplier}, the nonlinearity is essentially of the form (recall that we have $|\nabla|^{-1}  $ applied to $g_{k_{2}}$ and $h_{k_{3}}$)
$$
|\nabla|^{2s-1} f_{k_{1}} |\nabla|^{-1/2} g_{k_{2}} | \nabla|^{-1/2} h_{k_{3}}
$$
 and this is easy to handle since we can place $f_{k_{1}}$ in smoothing space while the remaining functions are placed in maximal space. Of course, in order for us to use Lemma \ref{Taylorstuff2} and Lemma \ref{Taylorstuff3} we must have that the frequencies are far apart ( i.e $k_{1} \geq k_{3} +C $ for large $C$). Moreover, if $k_{2}$ is close to $k_{3}$ then $ \frac{\kappa}{|\kappa|}$ might be singular and so we cannot hope to apply \eqref{multimultiplier} or \eqref{multiplierint}.  However, as we will see below, one can use the localized maximal estimate to handle the case when $ 2^{k_{2}} \approx 2^{k_{3}}$.

 Throughout this section, we will make use of the following elementary fact; given  $ r \in \Z$  we have the following
 \begin{equation}\label{cardn1n2}
     \# \{ \tn_{2} \in 2^{r} \Z^{n} :| \tn_{1} + \tn_{2} | \lesssim 2^{r} \} \lesssim 1 \text{  , } \forall \tn_{1} \in 2^{r} \Z.
 \end{equation}
The advantage of this is that whenever we have $ 2^{k_{2}} \approx 2^{k_{3}}$ then we obtain 
\begin{equation}\label{maximalestcardinality}
\begin{split}
& \sum_{ \tn_{1} ,\tn_{2} \in 2^{r} \Z^{n}} \Vert  \Delta_{r} \left( P_{r,\tn_{1}} g_{k_{2}} P_{r,\tn_{2}} h_{k_{3}} \right) \Vert_{L^{1}_{e} L^{\infty}_{e^{\perp},t}}
\\
& \lesssim \sum_{ \tn_{1} ,\tn_{2} \in 2^{r}\Z^{n} : | \tn_{1} + \tn_{2} | \lesssim 2^{r} } \Vert P_{r, \tn_{1}} g_{k_{2}} \Vert_{L^{2}_{e} L^{\infty}_{e^{\perp},t}} \Vert P_{r, \tn_{2}} h_{k_{3}} \Vert_{L^{2}_{e} L^{\infty}_{e^{\perp},t}} 
\\
& \lesssim \left( \sum_{\tn \in 2^{r} \Z^{n}} \Vert P_{r,\tn} g_{k_{2}} \Vert^{2}_{L^{2}_{e} L^{\infty}_{e^{\perp},t}}\right)^{1/2} \left( \sum_{\tn \in 2^{r} \Z^{n}} \Vert P_{r,\tn} h_{k_{3}} \Vert^{2}_{L^{2}_{e} L^{\infty}_{e^{\perp},t}}\right)^{1/2}, 
\end{split}
\end{equation}
where in the last inequality we used \eqref{cardn1n2}.

\begin{proof}[Proof of Proposition \ref{GaugeTrl1}]
It suffices to prove the following 
\begin{equation}\label{aszx1}
\begin{split}
& \sum_{k =-\infty}^{ \max \{ k_{1} ,k_{2},k_{3} \} +C} 2^{ k \sigma} \Vert \Delta_{k}( \i \partial_{t} - ( -\Delta)^{s} +\i)^{-1} \ind_{[-T',T']} T^{i,j}_{2} ( \tilde{f}_{k_{1}},  \tilde{g}_{k_{2}} ,  \tilde{h}_{k_{3}} ) \Vert_{Z_{k}}
\\
& \lesssim 2^{k_{1} \sigma } \Vert f_{k_{1}} \Vert_{Z_{k_{1}}} 2^{k_{2} \frac{n-2}{2}} \Vert g_{k_{2}} \Vert_{Z_{k_{2}}} 2^{k_{3} \frac{n-2}{2}} \Vert h_{k_{3}} \Vert_{Z_{k_{3}}}
\\
& + 2^{k_{1} \frac{n-2}{2} } \Vert f_{k_{1}} \Vert_{Z_{k_{1}}} 2^{k_{2}\sigma } \Vert g_{k_{2}} \Vert_{Z_{k_{2}}} 2^{k_{3} \frac{n-2}{2}} \Vert h_{k_{3}} \Vert_{Z_{k_{3}}}
\\
& + 2^{k_{1} \frac{n-2}{2} } \Vert f_{k_{1}} \Vert_{Z_{k_{1}}} 2^{k_{2} \frac{n-2}{2}} \Vert g_{k_{2}} \Vert_{Z_{k_{2}}} 2^{k_{3} \sigma } \Vert h_{k_{3}} \Vert_{Z_{k_{3}}}.
\end{split}
\end{equation}
To prove \eqref{aszx1}, we consider several cases depending on which frequency is the highest. 
\\
\textbf{Case 1:} If $ \max\{ k_{1},k_{2},k_{3} \} = k_{1}$. Then in view of \eqref{multilpierlemma}, we may assume that $ f_{k_{1}}$ has Fourier transform supported in the set $ \{ (\xi, \tau) : \langle \xi , e' \rangle \geq 2^{k_{1} -2} \}$ for some $ e' \in \mathscr{E}$. Next, we consider three sub-cases, depending on whether $k_{2},k_{3}$ are close to $k_{1}$ or not. 
\\
\textbf{Case 1a: } If $ 2^{k_{1}} \approx 2^{k_{2}} \approx 2^{k_{3}}$. Then in this case we wish to place $ f_{k_{1}}$ in smoothing space, while the remaining functions in the maximal space and use \eqref{multimultiplier}. However, if $k_{2}$ is  close to $k_{3}$ then the product could have low frequency. In this case, we cannot apply \eqref{multimultiplier}. Therefore, we instead apply the localized maximal estimate. More precisely
we write 
\begin{equation}\label{editedcase1a}
\begin{split}
& \mathcal{F}_{\R^{n}}(T^{i,j}_{2}(\tilde{f}_{k_{1}},  \tilde{g}_{k_{2}} , \tilde{h}_{k_{3}} ))(\xi,t)
\\
& = \sum_{l ,a=1}^{n} \int_{\R^{n} \times \R^{n}} B_{i,j} ( \xi, \eta, \kappa) \hat{f}(\xi-\kappa,t) 
\\
& 
\times \frac{ ( \kappa - \eta)_{l}}{|\kappa - \eta|^{2}} ( I- \frac{ \kappa}{|\kappa|} \otimes \frac{\kappa}{|\kappa|})_{al} ( \kappa- \eta)_{a} | \kappa - \eta|^{-1} \hat{h}_{k_{3}}(\kappa- \eta,t) |\eta|^{-1} \hat{g}_{k_{2}}(\eta,t) d \eta d \kappa
\\
& = \sum_{r \leq k_{1} +C} \sum_{l ,a=1}^{n} \int_{\R^{n} \times \R^{n}} B_{i,j} ( \xi, \eta, \kappa) \hat{f}(\xi-\kappa,t) 
\\
& 
\times \frac{ ( \kappa - \eta)_{l}}{|\kappa - \eta|^{2}} ( I- \frac{ \kappa}{|\kappa|} \otimes \frac{\kappa}{|\kappa|})_{al} ( \kappa- \eta)_{a} | \kappa - \eta|^{-1} \varphi_{r}(\kappa) \hat{h}_{k_{3}}(\kappa- \eta,t) |\eta|^{-1} \hat{g}_{k_{2}}(\eta,t) d \eta d \kappa.
\end{split}
\end{equation}
To handle the right hand side of \eqref{editedcase1a}, we split the sum in $r$ to when $ 2^{r} \ll 2^{k_{1}}$ and when $ 2^{r} \approx 2^{k_{1}}$. 
To that end, define $ H^{1}_{r} (\tilde{f}_{k_{1}},\tilde{g}_{k_{2}}, \tilde{h}_{k_{3}}) $ via the Fourier transform
\begin{equation}\label{hrdef}
\begin{split}
& \mathcal{F}_{\R^{n}}(H^{1}_{r} ( \tilde{f}_{k_{1}} , \tilde{g}_{k_{2}} ,\tilde{h}_{k_{3}} ) (\xi,t) 
\\
& := \sum_{l ,a=1}^{n} \int_{\R^{n} \times \R^{n}} B_{i,j} ( \xi, \eta, \kappa) \hat{f}(\xi-\kappa,t) 
\\
& \times
\frac{ ( \kappa - \eta)_{l}}{|\kappa - \eta|^{2}} ( I- \frac{ \kappa}{|\kappa|} \otimes \frac{\kappa}{|\kappa|})_{al} ( \kappa- \eta)_{a}
| \kappa - \eta|^{-1} \varphi_{r}(\kappa)  \hat{h}_{k_{3}}(\kappa- \eta,t) |\eta|^{-1} \hat{g}_{k_{2}}(\eta,t) d \eta d \kappa
\end{split}
\end{equation}
Fix $r \leq k_{1} -C$ for sufficiently large $C$. Then write 
\begin{equation*}
\begin{split}
& \mathcal{F}_{\R^{n}}(H^{1}_{r} ( \tilde{f}_{k_{1}} , \tilde{g}_{k_{2}} ,\tilde{h}_{k_{3}} ) (\xi,t) 
\\
& = \sum_{\tn_{1} \tn_{2} \in 2^{r} \Z} \sum_{l ,a=1}^{n} \int_{\R^{n} \times \R^{n}} B_{i,j} ( \xi, \eta, \kappa) \hat{f}_{k_{1}} (\xi-\kappa,t) 
 \frac{ ( \kappa - \eta)_{l}}{|\kappa - \eta|^{2}} ( I- \frac{ \kappa}{|\kappa|} \otimes \frac{\kappa}{|\kappa|})_{al} ( \kappa- \eta)_{a} 
 \\
 & \times
 | \kappa - \eta|^{-1} \varphi_{r}(\kappa)   \chi_{r,\tn_{2}}(\kappa-\eta)  \hat{h}_{k_{3}}(\kappa- \eta,t) |\eta|^{-1}\chi_{r,\tn_{1}}(\eta) \hat{g}_{k_{2}}(\eta,t) d \eta d \kappa
\end{split}
\end{equation*}
Notice for each fixed $ \tn_{1} \in 2^{r} \Z$, the sum in $ \tn_{2}$ is finite, see \eqref{cardn1n2}. 
 Now we estimate $H^{1}_{r}$ as follows, using \eqref{multimultiplier} and the assumption that $ 2^{k_{1}} \approx 2^{k_{2}} \approx 2^{k_{3}}$ we obtain 
\begin{equation*}
\begin{split}
&  \Vert \Delta_{k}( \i \partial_{t} - ( -\Delta)^{s} + \i)^{-1} \ind_{[-T',T']} H^{1}_{r}( \tilde{f}_{k_{1}}, \tilde{g}_{k_{2}}, \tilde{h}_{k_{3}}) \Vert_{Z_{k}}
\\
& \lesssim \sum_{\tn_{1},\tn_{2} \in 2^{r} \Z} \Vert \Delta_{k}( \i \partial_{t} - ( -\Delta)^{s} + \i)^{-1} \ind_{[-T',T']} H_{r}( \tilde{f}_{k_{1}}, P_{r,\tn} \tilde{g}_{k_{2}}, P_{r,\tn} \tilde{h}_{k_{3}}) \Vert_{Z_{k}}
\\
& \lesssim 2^{-k_{1}(2s-1)/2} \sum_{\tn_{1},\tn_{2} \in 2^{r} \Z^{n}} \Vert H^{1}_{r} ( \tilde{f}_{k_{1}}, P_{r,\tn_{1}} \tilde{g}_{k_{2}}, P_{r,\tn_{2}} \tilde{h}_{k_{3}} ) \Vert_{L^{1}_{e'}L^{2}_{e'^{\perp},t}}
\\
& \overset{\eqref{multimultiplier}}{\lesssim} 2^{-k_{1}(2s-1)/2} \Vert \tilde{f}_{k_{1}} \Vert_{L^{\infty}_{e'} L^{2}_{e'^{\perp},t}} \sum_{\tn_{1},\tn_{2} } \Vert P_{r,\tn_{1}}\tilde{g}_{k_{2}} \Vert_{L^{2}_{e'}L^{\infty}_{e'^{\perp},t}} \Vert P_{r,\tn_{2}} \tilde{h}_{k_{3}} \Vert_{L^{2}_{e'}L^{\infty}_{e'^{\perp},t}} 2^{k_{1}} 2^{k_{1}(2s-1)} 2^{-2k_{1}}
\\
& \overset{\eqref{maximalestcardinality}}{\lesssim}2^{-k_{1}(2s-1)/2} \Vert \tilde{f}_{k_{1}} \Vert_{L^{\infty}_{e'} L^{2}_{e'^{\perp},t}} \left( \sum_{\tn \in 2^{r} \Z^{n}} \Vert P_{r,\tn} \tilde{g}_{k_{2}} \Vert_{L^{2}_{e'}L^{\infty}_{e'^{\perp},t}}^{2} \right)^{1/2}
\\
& \times \left( \sum_{\tn \in 2^{r} \Z^{n}} 
  \Vert P_{r,\tn} \tilde{g}_{k_{3}} \Vert_{L^{2}_{e'}L^{\infty}_{e'^{\perp},t}}^{2} \right)^{1/2}2^{k_{1}} 2^{k_{1}(2s-1)} 2^{-2k_{1}}
\\
& \overset{\eqref{local_smoothing}, \eqref{eq:max:Xkest:loc}}{\lesssim} 2^{-k_{1}(2s-1)} \Vert f_{k_{1}} \Vert_{Z_{k_{1}}} 2^{r \frac{n-1}{2}} 2^{(k_{2}-r)/2} \Vert g_{k_{2}} \Vert_{Z_{k_{2}}} 2^{ r \frac{n-1}{2}} 2^{(k_{3} -r)/2} \Vert h_{k_{3}} \Vert_{Z_{k_{3}}} (1+|k_{1} -r|)^{2} 2^{k_{1}} 2^{k_{1}(2s-1)} 2^{-2k_{1}}
\\
& \approx (1+ |k_{1} -r|)^{2} 2^{r (n-2)} \Vert f_{k_{1}} \Vert_{Z_{k_{1}}} \Vert g_{k_{2}} \Vert_{Z_{k_{2}}} \Vert h_{k_{3}} \Vert_{Z_{k_{3}}}
\end{split}
\end{equation*}
summing over $ r \leq k_{1} -C$, and multiplying by $2^{k \sigma} \approx 2^{k_{1} \sigma}$ we obtain
\begin{equation}\label{locar1}
\begin{split}
& \sum_{r \leq k_{1} -C}   2^{k_{1} \sigma} \Vert \Delta_{k}( \i \partial_{t} - ( -\Delta)^{s} + \i)^{-1}\ind_{[-T',T']} H^{1}_{r}( \tilde{f}_{k_{1}}, \tilde{g}_{k_{2}}, \tilde{h}_{k_{3}}) \Vert_{Z_{k}}
\\
& \lesssim 2^{k_{1} \sigma} \Vert f_{k_{1}} \Vert_{Z_{k_{1}}} 2^{k_{2} \frac{n-2}{2}} \Vert g_{k_{2}} \Vert_{Z_{k_{2}}} 2^{k_{3} \frac{n-2}{2}} \Vert h_{k_{3}} \Vert_{Z_{k_{3}}}.
\end{split}
\end{equation}
Next, take $r \in [k_{1} -C,k_{1} +C]$. Then in this case we can apply the formal argument alluded to at the beginning of the proof without any issues. This is because we have that $ \kappa$ is localized at $ 2^{k_{1}}$, so we can apply \eqref{multimultiplier}. More precisely, for any $r \in [k_{1} -C,k_{1} +C]$ we have 
\begin{equation}
\begin{split}
& \Vert \Delta_{k} ( \i \partial_{t} - ( - \Delta)^{s} + \i)^{-1} \ind_{[-T',T']} H^{1}_{r} (\tilde{f}_{k_{1}},   \tilde{g}_{k_{2}} ,  \tilde{h}_{k_{3}} ) \Vert_{Z_{k}}
\\
& \lesssim 2^{-k (2s-1)/2} \max_{e \in \mathscr{E}} \Vert \Delta_{k} H^{1}_{r} (\tilde{f}_{k_{1}} ,  \tilde{g}_{k_{2}} ,  \tilde{h}_{k_{3}} ) \Vert_{L^{1}_{e}L^{2}_{e^{\perp},t}}
\\
& \overset{\eqref{multimultiplier}}{\lesssim} 2^{-k(2s-1)/2} 2^{k} 2^{k_{1} (2s-1)} \max_{e \in \mathscr{E}} \Vert f_{k_{1}} \Vert_{L^{\infty}_{e'}L^{2}_{e'^{\perp},t}} \Vert |\nabla|^{-1} h_{k_{3}} \Vert_{L^{2}_{e'}L^{\infty}_{e'^{\perp},t}} \Vert |\nabla|^{-1} g_{k_{2}} \Vert_{L^{2}_{e}L^{\infty}_{e^{\perp},t}}
\\
& \overset{\eqref{local_smoothing},\eqref{maximal_Estimate}}{\lesssim} 2^{-k (2s-1)/2} 2^{k} 2^{k_{1}(2s-1)/2} 2^{- k_{1}}  \Vert f_{k_{1}} \Vert_{Z_{k_{1}}} 2^{k_{2} \frac{n-2}{2}} \Vert g_{k_{2}} \Vert_{Z_{k_{2}}} 2^{k_{3} \frac{n-2}{2}} \Vert h_{k_{3}} \Vert_{Z_{k_{3}}}
\end{split}
\end{equation}
multiplying by $ 2^{k \sigma}$ and summing over $k \leq k_{1} +C$ gives us (for $r \in [k_{1}-C,k_{1}+C]$)
\begin{equation}\label{locar2}
\begin{split}
&\sum_{k \leq k_{1}+C} 2^{k \sigma} \Vert \Delta_{k} ( \i \partial_{t} - ( - \Delta)^{s} + \i)^{-1} \ind_{[-T',T']} H^{1}_{r} (\tilde{f}_{k_{1}},   \tilde{g}_{k_{2}} ,  \tilde{h}_{k_{3}} ) \Vert_{Z_{k}}
\\
& \lesssim 2^{k_{1} \sigma } \Vert f_{k_{1}} \Vert_{Z_{k_{1}}} 2^{k_{2} \frac{n-2}{2}} \Vert g_{k_{2}} \Vert_{Z_{k_{2}}} 2^{k_{3} \frac{n-2}{2}} \Vert h_{k_{3}} \Vert_{Z_{k_{3}}}.
\end{split}
\end{equation}
Combining \eqref{locar1} and \eqref{locar2} together with the obvious identity
$$
T^{i,j}_{2}(\tilde{f}_{k_{1}}, \tilde{g}_{k_{2}}, \tilde{h}_{k_{3}} ) = \sum_{r \leq k_{1} +C} H^{1}_{r} ( \tilde{f}_{k_{1}},\tilde{g}_{k_{2}},\tilde{h}_{k_{3}})
$$
yields
\eqref{aszx1} in this case. 
\\
\textbf{Case 1b:} If $  2^{k_{2}} \ll 2^{k_{1} }$ while $ 2^{k_{3}} \approx 2^{k_{1}}$. Then we make use of the identity \eqref{identitycolumbguy1} to place the derivatives on  $\tilde{g}_{k_{2}}$. In other words, in the Fourier space we have 
\begin{equation*}
\begin{split}
& \mathcal{F}_{\R^{n}}(T^{i,j}_{2}(\tilde{f}_{k_{1}}, \tilde{g}_{k_{2}} , \tilde{h}_{k_{3}} ))(\xi,t)
\\
&   = - \sum_{l ,a=1}^{n} \int_{\R^{n} \times \R^{n}} B_{i,j} ( \xi, \eta, \kappa) \hat{f}_{k_{1}}(\xi-\kappa,t) 
\\
& 
\times \frac{ ( \kappa - \eta)_{l}}{|\kappa - \eta|^{2}} ( I- \frac{ \kappa}{|\kappa|} \otimes \frac{\kappa}{|\kappa|})_{al} ( \eta)_{a} \frac{\hat{h}_{k_{3}}(\kappa- \eta,t)}{|\kappa-\eta|} \frac{ \hat{g}_{k_{2}}(\eta,t)}{|\eta|} d \eta d \kappa.
\end{split}
\end{equation*}
Now we use \eqref{multimultiplier} to obtain 
\begin{equation*}
\begin{split}
& \Vert \Delta_{k} ( \i \partial_{t} - ( -\Delta)^{s} +\i)^{-1} \ind_{[-T',T']} T^{2}_{i,j} (\tilde{f}_{k_{1}}, \tilde{g}_{k_{2}},\tilde{h}_{k_{3}}) \Vert_{Z_{k}}
\\
& \lesssim  2^{-k(2s-1)/2} \max_{e \in \mathscr{E}} \Vert T^{2}_{i,j} \Vert_{L^{1}_{e} L^{2}_{e^{\perp},t}}
\\
& \overset{\eqref{multimultiplier}}{ \lesssim} 2^{k} 2^{-k(2s-1)/2} 2^{k_{1}(2s-1)}  \Vert g_{k_{2}} \Vert_{L^{2}_{e}L^{\infty}_{e^{\perp},t}} \Vert f_{k_{1}} \Vert_{L^{\infty}_{e'} L^{2}_{e'^{\perp},t}} 2^{-2k_{3}} \Vert h_{k_{3}} \Vert_{L^{2}_{e'}L^{\infty}_{e'^{\perp},t}}
\\
& \overset{\eqref{local_smoothing}, \eqref{maximal_Estimate}}{\lesssim} 2^{k} 2^{-k(2s-1)/2} 2^{k_{1} (2s-1)} 2^{-2k_{3}} 2^{-k_{1}(2s-1)/2} \Vert f_{k_{1}} \Vert_{Z_{k_{1}}} 2^{k_{2} \frac{n-1}{2}} \Vert g_{k_{2}} \Vert_{Z_{k_{2}}} 2^{k_{3} \frac{n-1}{2}} \Vert h_{k_{3}} \Vert_{Z_{k_{3}}}
\\
& \overset{2^{k_{1}} \approx 2^{k_{3}}}{ \lesssim } 2^{k -k_{1}} 2^{( k_{1} -k)(2s-1)/2} \Vert f_{k_{1}} \Vert_{Z_{k_{1}}} 2^{k_{2} \frac{n-2}{2}} \Vert g_{k_{2}} \Vert_{Z_{k_{2}}} 2^{k_{3} \frac{n-2}{2}} \Vert h_{k_{3}} \Vert_{Z_{k_{3}}}.
\end{split}
\end{equation*}
Multiplying by $ 2^{k \sigma }$ and summing over $k \leq k_{1} +C$ yields \eqref{aszx1} in this case. 
\\
\textbf{Case 1c:} if $ 2^{ k_{3} } \ll 2^{k_{1}} $. Assume without loss of generality that $\min \{ k_{2} , k_{3} \} =k_{3}$ (The argument is almost identical for the other case).
In this case we make use of Lemma \ref{Taylorcalc} to write 
\begin{equation*}
\begin{split}
& (\xi-\eta)_{i} | \xi -\eta|^{2s-2} - (\xi -\kappa)_{i} | \xi -\kappa|^{2s-2}
\\
& = \sum_{|\alpha| =1}^{\infty} \frac{ m_{\alpha} ( \xi -\eta) }{\alpha !} |\xi-\eta|^{2s-1-|\alpha|} (\kappa -\eta)^{\alpha} 
\end{split}
\end{equation*}
where $ m_{\alpha}$ satisfy \eqref{qpppa}. Fix any multi-index $ \alpha$. Then define $H^{2}_{\alpha}$ as follows 
\begin{equation}\label{Halpha}
\begin{split}
& \mathcal{F}_{\R^{n}} ( H^{2}_{\alpha} ( \tilde{f}_{k_{1}},\tilde{g}_{k_{2}},\tilde{h}_{k_{3}})(\xi,t)
\\
& = \frac{1}{\alpha !} \sum_{a,l=1}^{n} \int_{\R^{n} \times \R^{n}} m_{\alpha} ( \xi - \eta)|\xi -\eta|^{(2s-1)-|\alpha|}  (\kappa - \eta)^{\alpha} \hat{f}(\xi - \kappa,t)   
\\
& \times \frac{ ( \kappa - \eta)_{l}}{|\kappa - \eta|^{2}} ( I- \frac{ \kappa}{|\kappa|} \otimes \frac{\kappa}{|\kappa|})_{al} ( \kappa-\eta)_{a} \frac{\hat{h}_{k_{3}}(\kappa- \eta,t)}{|\kappa-\eta|} \frac{ \hat{g}_{k_{2}}(\eta,t)}{|\eta|} d \eta d \kappa.
\end{split}
\end{equation}
In view Lemma \ref{Taylorcalc} and the definition of $T^{i,j}_{2}$ it suffices to prove the following 
\begin{equation}\label{aszx2}
\begin{split}
& \sum_{k\leq k_{1}+C} 2^{k \sigma } \sum_{|\alpha|=1}^{\infty} \Vert \Delta_{k} ( \i \partial_{t} - ( -\Delta)^{s} +\i)^{-1} \ind_{[-T',T']} \partial_{j} H^{2}_{\alpha} ( \tilde{f}_{k_{1}},\tilde{g}_{k_{2}},\tilde{h}_{k_{3}} ) \Vert_{Z_{k}}
\\
& \lesssim  2^{k_{1} \sigma } \Vert f_{k_{1}} \Vert_{Z_{k_{1}}} 2^{k_{2} \frac{n-2}{2}} \Vert g_{k_{2}} \Vert_{Z_{k_{2}}} 2^{k_{3} \frac{n-2}{2}} \Vert h_{k_{3}} \Vert_{Z_{k_{3}}}.
\end{split}
\end{equation}
To that end, fix any multi-index $ \alpha$ and assume first that $ 2^{ k_{3}  } \ll 2^{ k_{2},}$. Then set 
$$
a_{\alpha} : = \Vert \mathcal{F}^{-1}_{\R^{n}} ( \varphi_{[k_{1}-C,k_{1}+C]} m_{\alpha}) \Vert_{L^{1}}. 
$$
Then we estimate, using \eqref{multimultiplier}, 
\begin{equation*}
\begin{split}
& \Vert \Delta_{k} ( \i \partial_{t} - ( -\Delta)^{s} +\i)^{-1} \ind_{[-T',T']} \partial_{j} H^{2}_{\alpha} ( \tilde{f}_{k_{1}}, \tilde{g}_{k_{2}},\tilde{h}_{k_{3}} ) \Vert_{Z_{k}}
\\
& \lesssim 2^{-k(2s-1)/2} 2^{k} \max_{e \in \mathscr{E}} \Vert H^{2}_{\alpha}( \tilde{f}_{k_{1}}, \tilde{g}_{k_{2}}, \tilde{h}_{k_{3}} ) \Vert_{L^{1}_{e} L^{2}_{e^{\perp},t}} 
\\
& \overset{\eqref{multimultiplier}}{ \lesssim} 2^{-k(2s-1)/2}2^{k}  2^{k_{1}(2s-1 - | \alpha| ) } 2^{ | \alpha| k_{3}} \frac{a_{\alpha}}{\alpha !} 2^{-k_{2} - k_{3}} \Vert \tilde{f}_{k_{1}} \Vert_{L^{\infty}_{e'} L^{2}_{e'^{\perp},t}} \Vert \tilde{g}_{k_{2}} \Vert_{L^{2}_{e'}L^{\infty}_{e'^{\perp},t}} \max_{e \in \mathscr{E}} \Vert \tilde{h}_{k_{3}} \Vert_{L^{2}_{e} L^{\infty}_{e^{\perp},t}}
\\
& \overset{\eqref{local_smoothing}, \eqref{maximal_Estimate}}{ \lesssim} 2^{-k(2s-1)/2}2^{k}  2^{k_{1}(2s-1 - | \alpha| ) } 2^{ | \alpha| k_{3}} \frac{a_{\alpha}}{\alpha !} 2^{-k_{2} -k_{3}} 2^{-k_{1} ( 2s-1)/2} 2^{k_{2} \frac{n-1}{2}} 2^{k_{3} \frac{n-1}{2}} \Vert f_{k_{1}} \Vert_{Z_{k_{1}}} \Vert g_{k_{2}} \Vert_{Z_{k_{2}}} \Vert  h_{k_{3}} \Vert_{Z_{k_{3}}}
\end{split}
\end{equation*}
Now we use the assumption that $ k_{3} \leq k_{1} -C$ and $k_{3} \leq k_{2} -C$ to obtain
\begin{equation*}
\begin{split}
& 2^{-k(2s-1)/2}2^{k}  2^{k_{1}(2s-1 - | \alpha| ) } 2^{ | \alpha| k_{3}} a_{\alpha} 2^{-k_{2} -k_{3}} 2^{-k_{1} ( 2s-1)/2} 2^{k_{2} \frac{n-1}{2}} 2^{k_{3} \frac{n-1}{2}} 
\\
& \lesssim 2^{ ( k_{3} -k_{1} ) ( |\alpha | -1) } 2^{( k -k_{1} ) } 2^{(k_{1} -k)(2s-1)/2} 2^{k_{2} \frac{n-2}{2}} 2^{k_{3} \frac{n-2}{2}} 
\\
& \lesssim 2^{ -C  |\alpha|} 2^{( k -k_{1} ) } 2^{(k_{1} -k)(2s-1)/2} 2^{k_{2} \frac{n-2}{2}} 2^{k_{3} \frac{n-2}{2}} 
\end{split}
\end{equation*}
And therefore, we end up with 
\begin{equation*}
\begin{split}
& \Vert \Delta_{k} ( \i \partial_{t} - ( -\Delta)^{s} +\i)^{-1} \ind_{[-T',T']}\partial_{j} H^{2}_{\alpha} ( \tilde{f}_{k_{1}}, \tilde{g}_{k_{2}},\tilde{h}_{k_{3}} ) \Vert_{Z_{k}}
\\
& \lesssim \frac{a_{\alpha} 2^{-C |\alpha|}}{\alpha !} 2^{( k -k_{1} ) } 2^{(k_{1} -k)(2s-1)/2} 2^{k_{2} \frac{n-2}{2}} 2^{k_{3} \frac{n-2}{2}}\Vert f_{k_{1}} \Vert_{Z_{k_{1}}} \Vert g_{k_{2}} \Vert_{Z_{k_{2}}} \Vert  h_{k_{3}} \Vert_{Z_{k_{3}}}
\end{split}
\end{equation*}
Summing over $ |\alpha|$ we obtain, using the convergence properties of $ a_{\alpha}$ stated in Lemma \ref{Taylorcalc}, 
\begin{equation*}
\begin{split}
& \sum_{|\alpha|=1}^{\infty} \Vert \Delta_{k} ( \i \partial_{t} - ( -\Delta)^{s} +\i)^{-1} \ind_{[-T',T']} \partial_{j}H^{2}_{\alpha} ( \tilde{f}_{k_{1}}, \tilde{g}_{k_{2}},\tilde{h}_{k_{3}} ) \Vert_{Z_{k}}
\\
& \lesssim  2^{( k -k_{1} ) } 2^{(k_{1} -k_)(2s-1)/2} 2^{k_{2} \frac{n-2}{2}} 2^{k_{3} \frac{n-2}{2}}\Vert f_{k_{1}} \Vert_{Z_{k_{1}}} \Vert g_{k_{2}} \Vert_{Z_{k_{2}}} \Vert  h_{k_{3}} \Vert_{Z_{k_{3}}}
\end{split}
\end{equation*}
Multiplying by $ 2^{k \sigma}$ and summing over $k \leq k_{1} +C$ yields \eqref{aszx2} under the assumption that $ 2^{k_{3}  } \ll 2^{k_{2}  }$. \\
If $ 2^{k_{2}} \approx 2^{k_{3}}$ then we cannot use \eqref{multimultiplier} directly since $ \frac{1}{|\kappa|}$ might be singular. We again use the localized maximal estimate instead. To that end, define 
\begin{equation*}
\begin{split}
& \mathcal{F}_{\R^{n}} ( H^{2}_{\alpha,r} ( \tilde{f}_{k_{1}},\tilde{g}_{k_{2}},\tilde{h}_{k_{3}})(\xi,t)
\\
& := \frac{1}{\alpha !} \sum_{a,l=1}^{n} \int_{\R^{n} \times \R^{n}} m_{\alpha} ( \xi - \eta)|\xi -\eta|^{(2s-1)-|\alpha|}  (\kappa - \eta)^{\alpha} \hat{f}(\xi - \kappa,t)   
\\
& \times \frac{ ( \kappa - \eta)_{l}}{|\kappa - \eta|^{2}} ( I- \frac{ \kappa}{|\kappa|} \otimes \frac{\kappa}{|\kappa|})_{al} ( \kappa-\eta)_{a} \frac{\hat{h}_{k_{3}}(\kappa- \eta,t)}{|\kappa-\eta|} \frac{ \hat{g}_{k_{2}}(\eta,t)}{|\eta|} \varphi_{r}(\kappa) d \eta d \kappa.
\end{split}
\end{equation*}
Then to obtain \eqref{aszx2}, it suffices to prove 
\begin{equation}\label{newaszx}
\begin{split}
& \sum_{k \leq k_{1} +C} 2^{k \sigma } \sum_{| \alpha|=1}^{\infty} \sum_{r \leq k_{2} +C} \Vert \Delta_{k} ( \i \partial_{t} - ( -\Delta)^{s} + \i)^{-1} \ind_{[-T',T']} \partial_{j} H^{2}_{\alpha,r} ( \tilde{f}_{k_{1}} , \tilde{g}_{k_{2}} , \tilde{h}_{k_{3}} ) \Vert_{Z_{k}}
\\
& \lesssim 2^{k_{1} \sigma} \Vert f_{k_{1}} \Vert_{Z_{k_{1}}} 2^{k_{2} \frac{n-2}{2}} \Vert g_{k_{2}} \Vert_{Z_{k_{2}}} 2^{k_{3} \frac{n-2}{2}} \Vert h_{k_{3}} \Vert_{Z_{k_{3}}}.
\end{split}
\end{equation}
To that end, fix any $r,\alpha$ then we estimate using \eqref{multimultiplier}, the localized maximal estimate, and the assumption that $ 2^{k_{2}} \approx 2^{k_{3}}$ 
\begin{equation*}
\begin{split}
& \Vert \Delta_{k} ( \i \partial_{t} - ( -\Delta)^{s} + \i)^{-1} \ind_{[-T',T']} \partial_{j} H^{2}_{\alpha ,r}(\tilde{f}_{k_{1}},\tilde{g}_{k_{2}}, \tilde{h}_{k_{3}} ) \Vert_{Z_{k}}
\\
& \lesssim 2^{-k_{1} (2s-1)/2} 2^{k_{1}} \Vert H^{2}_{\alpha ,r} (\tilde{f}_{k_{1}},\tilde{g}_{k_{2}} ,\tilde{h}_{k_{3}} ) \Vert_{L^{1}_{e'} L^{2}_{e'^{\perp},t}}
\\
& \overset{ \eqref{multimultiplier}}{\lesssim} 2^{k_{1} (2s-1)/2} \frac{a _{ \alpha} 2^{ - C |\alpha |}}{\alpha !}  2^{-k_{3}} \Vert f_{k_{1}} \Vert_{L^{\infty}_{e'} L^{2}_{e'^{\perp},t}} 
\\
& \times \left( \sum_{\tn} \Vert P_{r,\tn} \tilde{g}_{k_{2}} \Vert^{2} _{L^{2}_{e'}L^{\infty}_{e'^{\perp},t}} \right)^{1/2}
\left( \sum_{\tn} \Vert P_{r,\tn} \tilde{h}_{k_{3}} \Vert^{2} _{L^{2}_{e'}L^{\infty}_{e'^{\perp},t}} \right)^{1/2}
\\
& \overset{\eqref{local_smoothing}, \eqref{eq:max:Xkest:loc}}{\lesssim} \frac{a_{\alpha} 2^{-C  | \alpha|}}{\alpha !} \Vert f_{k_{1}} \Vert_{Z_{k_{1}}} 2^{r \frac{n-1}{2}} 2^{(k_{2} -r)/2} \Vert g_{k_{2}} \Vert_{Z_{k_{2}}} 2^{r \frac{n-1}{2}} 2^{(k_{3} -r)/2} \Vert h_{k_{3}} \Vert_{Z_{k_{3}}}  2^{-k_{3}}(1+|k_{2}-r|)^{2}.
\end{split}
\end{equation*}
This readily implies \eqref{newaszx} ( since $2^{k_{2}} \approx 2^{k_{3}})$.
This concludes case 1. 
\\
\textbf{Case 2:} If $ \max \{k_{1} , k_{2} ,k_{3} \} =k_{2}$. In view of \eqref{multilpierlemma}, we may assume that $g_{k_{2}}$ has Fourier transform supported in $ \{ ( \xi, \tau) : \langle \xi, e' \rangle \geq 2^{k_{2} -2}\}$ for some $ e' \in \mathscr{E}$. Again we consider several sub-cases. The case in which $ 2^{k_{1}} \approx 2^{k_{2}} \approx 2^{k_{3}}$ is covered in case 1a above. And the case in which $ 2^{k_{2}} \approx 2^{k_{1}}$ while $ 2^{k_{3}} \ll 2^{k_{2}}$ is covered in case 1c. The two cases that needs to be done are when $ 2^{\max \{k_{1},k_{3} \}} \ll 2^{k_{2}}$ and when $ 2^{k_{2}} \approx 2^{k_{3}}$ but $ 2^{k_{1}} \ll 2^{k_{2}}$. 
\\
\textbf{Case 2a:} If $ 2^{k_{2}} \approx 2^{k_{3}}$ and $2^{k_{1}} \ll 2^{k_{2}}.$ As was done in case 1a, we need to be careful when we apply \eqref{multimultiplier}. This is because the support in variable $ \kappa$ could include $0$ when $ 2^{k_{2}} \approx 2^{k_{3}}$. To that end, we use the operator $H^{1}_{r}$, defined in \eqref{hrdef}, to write 
\begin{equation*}
\begin{split}
& T^{i,j}_{2} ( \tilde{f}_{k_{1}},\tilde{g}_{k_{2}},\tilde{h}_{k_{3}} ) = \sum_{r \leq k_{2} +C} H^{1}_{r}(\tilde{f}_{k_{1}},\tilde{g}_{k_{2}},\tilde{h}_{k_{3}} ). 
\end{split}
\end{equation*}
And repeating the argument leading to \eqref{locar1} we easily obtain, 
\begin{equation}\label{case2r1}
\begin{split}
&\sum_{k \leq k_{1} + C} 2^{k \sigma} \sum_{r \leq k_{1}-C} \Vert \Delta_{k}( \i \partial_{t} - ( - \Delta)^{s} + \i)^{-1} \ind_{[-T',T']}  H^{1}_{r} ( \tilde{f}_{k_{1}}, \tilde{g}_{k_{2}} , \tilde{h}_{k_{3}}) \Vert_{Z_{k}}
\\
& \lesssim 2^{k_{1} \sigma } \Vert f_{k_{1}} \Vert_{Z_{k_{1}}} 2^{k_{2} \frac{n-2}{2}} \Vert g_{k_{2}} \Vert_{Z_{k_{2}}} 2^{k_{3} \frac{n-2}{2}} \Vert h_{k_{3}} \Vert_{Z_{k_{3}}}. 
\end{split}
\end{equation}
Now fix $ r \in [k_{1} -C , k_{2} +C]$. Then we estimate using \eqref{multimultiplier}, and the assumption that $2^{k_{2}} \approx 2^{k_{3}}$
\begin{equation*}
\begin{split}
& \Vert \Delta_{k}( \i \partial_{t} - ( - \Delta)^{s} + \i)^{-1} \ind_{[-T',T']} H^{1}_{r} ( \tilde{f}_{k_{1}}, \tilde{g}_{k_{2}} , \tilde{h}_{k_{3}}) \Vert_{Z_{k}}
\\
& \lesssim 2^{-k(2s-1)/2} \max_{e \in \mathscr{E}} \Vert H_{r} ( \tilde{f}_{k_{1}}, \tilde{g}_{k_{2}} , \tilde{h}_{k_{3}}) \Vert_{L^{1}_{e} L^{2}_{e^{\perp},t}}
\\
& \overset{\eqref{multimultiplier}}{ \lesssim} 2^{-k(2s-1)/2} 2^{r} 2^{k_{2}(2s-1)} 2^{-k_{2} -k_{3}} \Vert g_{k_{2}} \Vert_{L^{\infty}_{e'} L^{2}_{e'^{\perp},t}} \Vert  h_{k_{3}} \Vert_{L^{2}_{e'} L^{\infty}_{e'^{\perp},t}} \max_{e \in \mathscr{E}} \Vert f_{k_{1}} \Vert_{L^{2}_{e} L^{\infty}_{e^{\perp},t}}
\\
& \overset{\eqref{local_smoothing}, \eqref{maximal_Estimate}}{\lesssim}  2^{-k(2s-1)/2} 2^{r} 2^{k_{2}(2s-1)/2} 2^{-k_{2}}   2^{k_{1} \frac{n-2}{2}} \Vert f_{k_{1}} \Vert_{Z_{k_{1}}} \Vert g_{k_{2}} \Vert_{Z_{k_{2}}} 2^{k_{3} \frac{n-2}{2}} \Vert h_{k_{3}} \Vert_{Z_{k_{3}}}.
\end{split}
\end{equation*}
This readily implies 
\begin{equation}\label{case2r2}
\begin{split}
& \sum_{k \leq k_{2} +C} 2^{k \sigma } \sum_{r = k_{1}-C}^{k_{2}+C} \Vert \Delta_{k}( \i \partial_{t} - ( - \Delta)^{s} + \i)^{-1} \ind_{[-T',T']} H^{1}_{r} ( \tilde{f}_{k_{1}}, \tilde{g}_{k_{2}} , \tilde{h}_{k_{3}}) \Vert_{Z_{k}}
\\
& \lesssim 2^{k_{1} \frac{n-2}{2}} \Vert f_{k_{1}} \Vert_{Z_{k_{1}}} 2^{k_{2} \sigma } \Vert g_{k_{2}} \Vert_{Z_{k_{2}}} 2^{k_{3} \frac{n-2}{2}} \Vert h_{k_{3}} \Vert_{Z_{k_{3}}}. 
\end{split}
\end{equation}
Combining \eqref{case2r1} and \eqref{case2r2} we obtain 
\begin{equation*}
\begin{split}
& \sum_{k \leq k_{2} +C} 2^{k \sigma } \sum_{r \leq k_{2}+C} \Vert \Delta_{k}( \i \partial_{t} - ( - \Delta)^{s} + \i)^{-1} \ind_{[-T',T']}  H^{1}_{r} ( \tilde{f}_{k_{1}}, \tilde{g}_{k_{2}} , \tilde{h}_{k_{3}}) \Vert_{Z_{k}}
\\
& \lesssim 2^{k_{1} \sigma } \Vert f_{k_{1}} \Vert_{Z_{k_{1}}} 2^{k_{2} \frac{n-2}{2}} \Vert g_{k_{2}} \Vert_{Z_{k_{2}}} 2^{k_{3} \frac{n-2}{2}} \Vert h_{k_{3}} \Vert_{Z_{k_{3}}}
\\
& + 2^{k_{1} \frac{n-2}{2} } \Vert f_{k_{1}} \Vert_{Z_{k_{1}}} 2^{k_{2} \sigma } \Vert g_{k_{2}} \Vert_{Z_{k_{2}}} 2^{k_{3} \frac{n-2}{2}} \Vert h_{k_{3}} \Vert_{Z_{k_{3}}}
\end{split}
\end{equation*}
Which implies \eqref{aszx1} in this case. 
\\
\textbf{Case 2b:} If $ 2^{k_{1}} \ll 2^{k_{2}}$ and $ 2^{k_{3}} \ll 2^{k_{2}}$. Then recall the definition of $T^{i,j}_{2}$
\begin{equation*}
\begin{split}
& \mathcal{F}_{\R^{n}} ( T^{2}_{i,j}( \tilde{f}_{k_{1}}, \tilde{g}_{k_{2}},\tilde{h}_{k_{3}} ) )(\xi,t)
\\
& =  \sum_{l ,a=1}^{n} \int_{\R^{n} \times \R^{n}} B_{i,j} ( \xi, \eta, \kappa) \hat{f}_{k_{1}}(\xi-\kappa,t) 
\\
& 
\times \frac{ ( \kappa - \eta)_{l}}{|\kappa - \eta|^{2}} ( I- \frac{ \kappa}{|\kappa|} \otimes \frac{\kappa}{|\kappa|})_{al} ( \kappa - \eta)_{a} \frac{\hat{h}_{k_{3}}(\kappa- \eta,t)}{|\kappa-\eta|} \frac{ \hat{g}_{k_{2}}(\eta,t)}{|\eta|} d \eta d \kappa.
\end{split}
\end{equation*}
If $ 2^{k_{1}} \lesssim 2^{k_{3}}$ 
then we use \eqref{multimultiplier} to obtain
\begin{equation*}
\begin{split}
& \Vert \Delta_{k} ( \i \partial_{t} - ( - \Delta)^{s} + \i)^{-1} \ind_{[-T',T']} T^{2}_{i,j} ( \tilde{f}_{k_{1}},\tilde{g}_{k_{2}},\tilde{h}_{k_{3}}) \Vert_{Z_{k}}
\\
& \lesssim 2^{-k_{2} (2s-1)/2} \Vert T^{2}_{i,j} ( \tilde{f}_{k_{1}},\tilde{g}_{k_{2}},\tilde{h}_{k_{3}}) \Vert_{L^{1}_{e'} L^{2}_{e'^{\perp},t}}
\\
& \overset{\eqref{multimultiplier}}{ \lesssim } 2^{-k_{2} (2s-1)/2} 2^{k_{2} (2s-1)} 2^{-k_{2}} \Vert f_{k_{1}} \Vert_{L^{2}_{e'} L^{\infty}_{e'^{\perp},t}} \Vert h_{k_{3}} \Vert_{L^{2}_{e'} L^{\infty}_{e'^{\perp},t}}  \Vert g_{k_{2}} \Vert_{L^{\infty}_{e'} L^{2}_{e'^{\perp},t}}
\\
& \overset{2^{k_{1}} \lesssim 2^{k_{3}}}{\lesssim} 2^{k_{1} \frac{n-2}{2}} \Vert f_{k_{1}} \Vert_{Z_{k_{1}}} 2^{k_{3} \frac{n-2}{2}} \Vert h_{k_{3}} \Vert_{Z_{k_{3}}} \Vert g_{k_{2}} \Vert_{Z_{k_{2}}}.
\end{split}
\end{equation*}
Multiplying by $ 2^{k \sigma}$ and summing over $ k \leq k_{2} +C$ yields \eqref{aszx1} in this case. On the other hand, if $ 2^{k_{3}} \ll 2^{k_{1}} $ ( i.e $ k_{3} \leq k_{1} - C_{0}$ with $C_{0}$ as in Lemma \ref{Taylorcalc}). Then we use the operator $H^{2}_{\alpha}$, defined in \eqref{Halpha}, to obtain
\begin{equation*}
\begin{split}
& T^{i,j}_{2} ( \tilde{f}_{k_{1}} , \tilde{g}_{k_{2}},\tilde{h}_{k_{3}} ) = \sum_{ | \alpha|=1}^{\infty} \partial_{j} H^{2}_{\alpha}(\tilde{f}_{k_{1}},\tilde{g}_{k_{2}},\tilde{h}_{k_{3}} )
\end{split}
\end{equation*}
Therefore, it suffices to prove 
\begin{equation}\label{case2reduced}
\begin{split}
& \sum_{\vert \alpha \vert=1}^{\infty} 2^{k_{2}} \Vert \Delta_{k} ( \i \partial_{t} - (-\Delta)^{s} +\i)^{-1} \ind_{[-T',T']} H^{2}_{\alpha} ( \tilde{f}_{k_{1}},\tilde{g}_{k_{2}},\tilde{h}_{k_{3}}) \Vert_{Z_{k}}
\\
& \lesssim 2^{k_{1} \frac{n-2}{2}} \Vert f_{k_{1}} \Vert_{Z_{k_{1}}} 2^{k_{3} \frac{n-2}{2}} \Vert h_{k_{3}} \Vert_{Z_{k_{3}}} \Vert g_{k_{2}} \Vert_{Z_{k_{2}}}
\end{split}
\end{equation}
To that end, fix any multi-index $ \alpha$, then as was done in case 1c, we make use of Lemma \ref{Taylorcalc}, and \eqref{multimultiplier} to obtain
\begin{equation*}
\begin{split}
& 2^{k_{2}} \Vert \Delta_{k} ( \i \partial_{t} - ( - \Delta)^{s} +\i)^{-1} \ind_{[-T',T']} H^{2}_{\alpha} (\tilde{f}_{k_{1}},\tilde{g}_{k_{2}},\tilde{h}_{k_{3}}) \Vert_{Z_{k}}
\\
& \lesssim 2^{-k_{2} (2s-1)/2} 2^{k_{2}} \Vert H^{2}_{\alpha}( \tilde{f}_{k_{1}},\tilde{g}_{k_{2}},\tilde{h}_{k_{3}}) \Vert_{L^{1}_{e'} L^{2}_{e'^{\perp},t}}
\\
& \overset{\eqref{multimultiplier}}{\lesssim} \frac{a_{\alpha} 2^{- |\alpha| C }}{\alpha !} 2^{-k_{2}(2s-1)/2}  2^{(2s-2)k_{1}} \Vert g_{k_{2}} \Vert_{L^{\infty}_{e'}L^{2}_{e'^{\perp},t}} \Vert f_{k_{1}} \Vert_{L^{2}_{e' } L^{\infty}_{e'^{\perp},t}}  \Vert h_{k_{3}} \Vert_{L^{2}_{e' } L^{\infty}_{e'^{\perp},t}} 
\\
& \overset{\eqref{local_smoothing}, \eqref{maximal_Estimate}}{\lesssim}\frac{a_{\alpha} 2^{- |\alpha| C }}{\alpha !} 2^{k_{1} \frac{n-2}{2}} \Vert f_{k_{1}} \Vert_{Z_{k_{1}}} 2^{k_{3} \frac{n-2}{2}} \Vert h_{k_{3}} \Vert_{Z_{k_{3}}} \Vert g_{k_{2}} \Vert_{Z_{k_{2}}}
\end{split}
\end{equation*}
Summing over $ |\alpha|$ and using the convergence stated in Lemma \ref{Taylorcalc} gives \eqref{case2reduced}. 
\\
\textbf{Case 3:} If $ \max \{ k_{1},k_{2} , k_{3} \} = k_{3}$. In view of the analysis above, we may assume that $ \max\{ k_{1} ,k_{2} \} \leq k_{3} -C$ for sufficiently large $ C$. Then in view of \eqref{multilpierlemma} we may assume $ h_{k_{3}}$ has Fourier transform supported in $ \{ (\xi,\tau) : \langle \xi, e \rangle \geq 2^{k_{3} -2} \}$ for some $ e \in \mathscr{E}$. We make use of the identity \eqref{identitycolumbguy1} to write 
\begin{equation*}
\begin{split}
& \mathcal{F}_{\R^{n}}(T^{2}_{i,j}(\tilde{f}_{k_{1}}, \tilde{g}_{k_{2}} , \tilde{h}_{k_{3}} ))(\xi,t)
\\
&   = - \sum_{l ,a=1}^{n} \int_{\R^{n} \times \R^{n}} B_{i,j} ( \xi, \eta, \kappa) \hat{f}_{k_{1}}(\xi-\kappa,t) 
\\
& 
\times \frac{ ( \kappa - \eta)_{l}}{|\kappa - \eta|^{2}} ( I- \frac{ \kappa}{|\kappa|} \otimes \frac{\kappa}{|\kappa|})_{al} ( \eta)_{a} \frac{\hat{h}_{k_{3}}(\kappa- \eta,t)}{|\kappa-\eta|} \frac{ \hat{g}_{k_{2}}(\eta,t)}{|\eta|} d \eta d \kappa.
\end{split}
\end{equation*}
Then using \eqref{multimultiplier} we obtain
\begin{equation*}
\begin{split}
&  \Vert \Delta_{k} ( \i \partial_{t} - ( -\Delta)^{s} +\i)^{-1} \ind_{[-T',T']} T^{2}_{i,j} (\tilde{f}_{k_{1}}, \tilde{g}_{k_{2}},\tilde{h}_{k_{3}}) \Vert_{Z_{k}}
\\
& \lesssim  2^{-k_{3}(2s-1)/2} \Vert T^{2}_{i,j} \Vert_{L^{1}_{e} L^{2}_{e^{\perp},t}}
\\
& \overset{\eqref{multimultiplier}}{\lesssim} 2^{k_{3}(2s-1)/2} 2^{-k_{3}}  \Vert f_{k_{1}} \Vert_{L^{2}_{e} L^{\infty}_{e}} \Vert g_{k_{2}} \Vert_{L^{2}_{e} L^{\infty}_{e}} \Vert h_{k_{3}} \Vert_{L^{\infty}_{e} L^{2}_{e^{\perp},t}}
\\
& \overset{\eqref{local_smoothing},\eqref{maximal_Estimate}}{ \lesssim} 2^{k_{1} \frac{n-2}{2}} \Vert f_{k_{1}} \Vert_{Z_{k_{1}}} 2^{k_{2} \frac{n-2}{2}} \Vert g_{k_{2}} \Vert_{Z_{k_{2}}} \Vert h_{k_{3}} \Vert_{Z_{k_{3}}}.
\end{split}
\end{equation*}
This concludes case 3, and hence proves the proposition.
\end{proof}
As a corollary we obtain the estimates concerning the $\Gamma_{1}$ term. 
\begin{corollary}
Let $ \Gamma_{1}$ be as in \eqref{Gamma1guy}. Then we have the following estimate 
\begin{equation*}
\begin{split}
& \Vert \ind_{[-T',T']} \mathcal{F}_{\R^{n}}^{-1}\Gamma_{1} \Vert_{N^{\sigma}}
\\
& \lesssim \Vert f \Vert_{\tilde{F}^{\sigma}} \Vert |\nabla | e_{1} \Vert_{\tilde{F}^{\frac{n-2}{2}}} \Vert |\nabla| e_{2} \Vert_{\tilde{F}^{\frac{n-2}{2}}}
\\
& + \Vert f \Vert_{\tilde{F}^{\frac{n-2}{2}}} \Vert |\nabla | e_{1} \Vert_{\tilde{F}^{\sigma}} \Vert |\nabla| e_{2} \Vert_{\tilde{F}^{\frac{n-2}{2}}}
\\
& + \Vert f \Vert_{\tilde{F}^{\frac{n-2}{2}}} \Vert |\nabla | e_{1} \Vert_{\tilde{F}^{\frac{n-2}{2}}} \Vert |\nabla| e_{2} \Vert_{\tilde{F}^{\sigma}}.
\end{split}
\end{equation*}
\end{corollary}
\begin{proof}
In the definition of $ \Gamma_{1}$, we may replace $ e_{i}$ with $ |\nabla|^{-1} ( | \nabla| e_{i} )$ this is because of the identity \eqref{identitycolumbguy1}, which kills any dirac delta coming from taking the Fourier transform of $e_{i}$. Now we apply Proposition \ref{GaugeTrl1} with $ g = | \nabla| e_{1}$ and $ h = | \nabla |e_{2}$. 
\end{proof}

\subsection{Nonlinear estimates for $\Gamma_{2}$} Recall that we reduced the proof of \eqref{trlCgauge} to proving  the nonlinear estimates for $\Gamma_{1}$ and $\Gamma_{2}$, where $\Gamma_{1},\Gamma_{2}$ are defined in \eqref{Gamma1guy} and \eqref{Gamma2guy}. The nonlinear estimates for $\Gamma_{1}$  followed from Proposition \ref{GaugeTrl1}. For $\Gamma_{2}$ we need to use the Coulomb gauge condition to rewrite $\Gamma_{2}$ as a sum of two manageable trilinear operators. First, recall the definition of $ \Gamma_{2}$
\begin{equation*}
\begin{split}
& \Gamma_{2}(\xi,t) = \sum_{l=1}^{n} \int_{\R^{n} \times \R^{n}} 
 B_{i,j}(\xi,\eta,\kappa) \hat{f}(\xi-\kappa,t) 
 \\
 & \times \frac{ (\kappa- \eta)_{l}}{|\kappa - \eta|^{2}} \frac{\kappa_{l}}{|\kappa|} \langle \frac{\kappa}{|\kappa|} , \kappa- \eta \rangle
  \langle \hat{e}_{2} ( \kappa- \eta,t), \hat{e}_{1} (\eta,t) \rangle_{\R^{3}} d \eta d \kappa
\end{split}
\end{equation*}
where $B_{i,j} $ was defined in \eqref{defbij}. 
Then notice that the Coulomb gauge condition is equivalent to the following
\begin{equation}\label{columbfourier}
0 = \int_{\R^{n}} \langle \kappa, \kappa-\eta \rangle \langle \hat{e}_{2}(\kappa-\eta,t) ,\hat{e}_{1}(\eta,t) \rangle_{\R^{3}} d \eta.
\end{equation}
Using this, we rewrite $\Gamma_{2}$ as follows 
\begin{equation*}
\begin{split}
& \Gamma_{2}(\xi,t)
\\
& =  \int_{\R^{n} \times \R^{n}}  B_{i,j}(\xi,\eta,\kappa) \hat{f}(\xi -\kappa,t) \frac{ (\langle \kappa - \eta , \kappa \rangle)^{2}}{| \kappa - \eta|^{2} | \kappa|^{2}} \langle \hat{e}_{2}(\kappa-\eta,t) ,\hat{e}_{1}(\eta,t) \rangle_{\R^{3}} d\eta d\kappa
\\
& =  \int_{\R^{n} \times \R^{n}}  B_{i,j}(\xi,\eta,\kappa) \hat{f}(\xi -\kappa,t) \frac{\langle \kappa - \eta, \kappa \rangle }{ | \kappa |^{2}} (\frac{ (\langle \kappa - \eta , \kappa \rangle)}{| \kappa - \eta|^{2} } -1) \langle \hat{e}_{2}(\kappa-\eta,t) ,\hat{e}_{1}(\eta,t) \rangle_{\R^{3}} d\eta d\kappa
\\
& + \int_{\R^{n} \times \R^{n}}  B_{i,j}(\xi,\eta,\kappa) \hat{f}(\xi -\kappa,t) \frac{\langle \kappa - \eta,\kappa \rangle }{ | \kappa|^{2}} \hat{e}_{2}(\kappa-\eta,t) ,\hat{e}_{1}(\eta,t) \rangle_{\R^{3}} d\eta d\kappa
\\
& \overset{\eqref{columbfourier}}{=} \int_{\R^{n} \times \R^{n}}  B_{i,j}(\xi,\eta,\kappa) \hat{f}(\xi -\kappa,t) \frac{\langle \kappa - \eta, \kappa \rangle }{ | \kappa |^{2}} (\frac{ (\langle \kappa - \eta , \kappa \rangle)}{| \kappa - \eta|^{2} } -1) \langle \hat{e}_{2}(\kappa-\eta,t) ,\hat{e}_{1}(\eta,t) \rangle_{\R^{3}} d\eta d\kappa
\\
& + \int_{\R^{n} \times \R^{n}} ( B_{i,j}(\xi,\eta,\kappa) - B(\xi,0,\kappa)) \hat{f}(\xi -\kappa,t)     \frac{\langle \kappa - \eta,\kappa \rangle }{ | \kappa|^{2}}   \langle \hat{e}_{2}(\kappa-\eta,t) ,\hat{e}_{1}(\eta,t) \rangle_{\R^{3}} d\eta d\kappa.
\end{split}
\end{equation*}
Define 
\begin{equation}\label{Gamma2aa}
\Gamma^{a}_{2}: = \int_{\R^{n} \times \R^{n}}  B_{i,j}(\xi,\eta,\kappa) \hat{f}(\xi -\kappa,t) \frac{\langle \kappa - \eta, \kappa \rangle }{ | \kappa |^{2}} (\frac{ (\langle \kappa - \eta , \kappa \rangle)}{| \kappa - \eta|^{2} } -1) \langle \hat{e}_{2}(\kappa-\eta,t) ,\hat{e}_{1}(\eta,t) \rangle_{\R^{3}} d\eta d\kappa
\end{equation}
and
\begin{equation}\label{Gamma2bb}
\Gamma^{b}_{2} : = \int_{\R^{n} \times \R^{n}} ( B_{i,j}(\xi,\eta,\kappa) - B_{i,j}(\xi,0,\kappa)) \hat{f}(\xi -\kappa,t)    
     \frac{\langle \kappa - \eta,\kappa \rangle }{ | \kappa|^{2}}   \langle \hat{e}_{2}(\kappa-\eta,t) ,\hat{e}_{1}(\eta,t) \rangle_{\R^{3}} d\eta d\kappa.
\end{equation}
Then we have proven 
$$
\Gamma_{2} = \Gamma^{a}_{2} + \Gamma^{b}_{2},
$$
and it suffices to prove the nonlinear estimates for $ \Gamma^{a}_{2} $ and $ \Gamma^{b}_{2}$. 
\\
We start with $ \Gamma^{a}_{2}$. 
\begin{proposition}\label{Proofforgamma2a}
Let $ n \geq 4$, $ s \in (\frac{1}{2},1)$, $ \sigma \geq \frac{n-2}{2},$ and $T' \in [0,\infty]$. 
Then for any $ f,g,h \in F^{\frac{n-2}{2}} \cap F^{\sigma}$ we have the following 
\begin{equation*}
\begin{split}
& \Vert \ind_{[-T',T']} T^{i,j}_{3}( \tilde{f},\tilde{g} , \tilde{h} ) \Vert_{N^{\sigma}} 
\\
& \lesssim \Vert f \Vert_{F^{\sigma}} \Vert g \Vert_{F^{\frac{n-2}{2}}} \Vert h \Vert_{F^{\frac{n-2}{2}}}
\\
& + \Vert f \Vert_{F^{\frac{n-2}{2}}} \Vert g \Vert_{F^{\sigma}} \Vert h \Vert_{F^{\frac{n-2}{2}}}
\\
&+ \Vert f \Vert_{F^{\frac{n-2}{2}}} \Vert g \Vert_{F^{\frac{n-2}{2}}} \Vert h \Vert_{F^{\sigma}}
\end{split}
\end{equation*}
where $T^{i,j}_{3}$ is defined via the Fourier transform 
\begin{equation*}
\begin{split}
& \mathcal{F}_{\R^{n} } T^{i,j}_{3} ( f,g,h)(\xi,t):=
\\
& \int_{\R^{n} \times \R^{n}} B_{i,j}(\xi,\eta ,\kappa) \hat{f}(\xi - \kappa ,t) \frac{ \langle \kappa - \eta , \kappa \rangle}{|\kappa |^{2}} \left( \frac{ \langle \kappa - \eta, \eta \rangle}{| \kappa - \eta |^{2}} \right) \frac{ \hat{ h }(\kappa - \eta,t)}{ | \kappa - \eta|} \frac{ \hat{g}(\eta ,t)}{ | \eta |} d \eta d\kappa.
\end{split}
\end{equation*}
\end{proposition}
\begin{proof}
To prove the result, it suffices to prove the following 
\begin{equation}\label{dyadictij3}
\begin{split}
& \sum_{k = - \infty}^{\max \{ k_{1},k_{2},k_{3} \} +C} 2^{k \sigma} \Vert \Delta_{k} ( \i \partial_{t} - ( - \Delta)^{s} + \i)^{-1} \ind_{[-T',T']} T^{i,j}_{3} ( \tilde{f}_{k_{1}} , \tilde{g}_{k_{2}},\tilde{h}_{k_{3}} ) \Vert_{Z_{k}}
\\
& \lesssim 2^{k_{1} \sigma} \Vert f_{k_{1}} \Vert_{Z_{k_{1}}} 2^{k_{2} \frac{n-2}{2}} \Vert g_{k_{2}} \Vert_{Z_{k_{2}}} 2^{k_{3} \frac{n-2}{2}} \Vert h_{k_{3}} \Vert_{Z_{k_{3}}}
\\
& + 2^{k_{1} \frac{n-2}{2}} \Vert f_{k_{1}} \Vert_{Z_{k_{1}}} 2^{k_{2} \sigma} \Vert g_{k_{2}} \Vert_{Z_{k_{2}}} 2^{k_{3} \frac{n-2}{2}} \Vert h_{k_{3}} \Vert_{Z_{k_{3}}}
\\
& + 2^{k_{1} \frac{n-2}{2}} \Vert f_{k_{1}} \Vert_{Z_{k_{1}}} 2^{k_{2} \frac{n-2}{2}} \Vert g_{k_{2}} \Vert_{Z_{k_{2}}} 2^{k_{3} \sigma} \Vert h_{k_{3}} \Vert_{Z_{k_{3}}}.
\end{split}
\end{equation}
To that end, we consider several cases. 
\\
\textbf{Case 1:} If $ \max \{ k_{1} ,k_{2} ,k_{3} \} = k_{1}$. Then in light of \eqref{multilpierlemma}, we may assume that $ f_{k_{1}}$ has Fourier transform supported in $ \{ \langle \xi , e' \rangle \geq 2^{k_{1} -2} \}$ for some $ e' \in \mathscr{E}$. We consider several sub-cases depending on whether $ k_{2},k_{3}$ are close to $k_{1}$ or not. \\
\textbf{Case 1a:} If $ 2^{k_{1}} \approx 2^{k_{2}} \approx 2^{k_{3}}$. Then in this case we wish to make use of \eqref{multimultiplier}. However, the presence of $ \frac{1}{|\kappa|} $ could cause problems when $ 2^{k_{2}} \approx 2^{k_{3}}$. To that end, we write $T^{i,j}_{3}$ in the Fourier space as follows; 
\begin{equation*}
\begin{split}
& \mathcal{F}_{\R^{n}} T^{i,j}_{3} (\tilde{f}_{k_{1}},\tilde{g}_{k_{2}},\tilde{h}_{k_{3}}) (\xi,t)
\\
& = \sum_{r \leq k_{1} +C} \mathcal{F}_{\R^{n}} H^{3}_{r} (\tilde{f}_{k_{1}},\tilde{g}_{k_{2}},\tilde{h}_{k_{3}}) ( \xi,t)
\end{split}
\end{equation*}
where 
\begin{equation}\label{defh3r}
\begin{split}
&\mathcal{F}_{\R^{n}} H^{3}_{r} ( f,g,h) ( \xi,t) 
\\
& := \int_{\R^{n} \times \R^{n}} B_{i,j}(\xi,\eta ,\kappa) \hat{f}(\xi - \kappa ,t) \frac{ \langle \kappa - \eta , \kappa \rangle}{|\kappa |^{2}} \left( \frac{ \langle \kappa - \eta, \eta \rangle}{| \kappa - \eta |^{2}} \right) \frac{ \hat{ h }(\kappa - \eta,t)}{ | \kappa - \eta|} \frac{ \hat{g}(\eta ,t)}{ | \eta |} \varphi_{r}(\kappa) d \eta d\kappa.
\end{split}
\end{equation}
If $ r \in [k_{1}-C,k_{1} +C]$ then we make use of \eqref{multimultiplier} and the assumption that $ 2^{k_{1}} \approx 2^{k_{2}} \approx 2^{k_{3}}$ to obtain for any $e \in \mathscr{E}$
\begin{equation}
\begin{split}
& 2^{-k(2s-1)/2} \Vert H^{3}_{r} ( \tilde{f}_{k_{1}}, \tilde{g}_{k_{2}} , \tilde{h}_{k_{3}} ) \Vert_{L_{e}^{1} L^{2}_{e^{\perp},t}}
\\
& \lesssim 2^{k} 2^{-k(2s-1)/2} 2^{k_{1}(2s-1)} \Vert \tilde{f}_{k_{1}} \Vert_{L^{\infty}_{e'}L^{2}_{e'^{\perp},t}} \Vert \tilde{g}_{k_{2}} \Vert_{L^{2}_{e'}L^{\infty}_{e'^{\perp},t}} \Vert \tilde{h}_{k_{3}} \Vert_{L^{2}_{e} L^{\infty}_{e^{\perp},t}} 2^{-2k_{1}} 
\\
 & \lesssim 2^{k -k_{1}} 2^{(k_{1} -k)(2s-1)/2}  \Vert f_{k_{1}} \Vert_{Z_{k_{1}}} 2^{k_{2} \frac{n-2}{2}} \Vert g_{k_{2}} \Vert_{Z_{k_{2}}} 2^{k_{3} \frac{n-2}{2}} \Vert h_{k_{3}} \Vert_{Z_{k_{3}}}
\end{split}
\end{equation}
This readily implies , for $ r \in [k_{1} -C ,k_{1} +C]$, 
\begin{equation}\label{firstestforh3r}
\begin{split}
& \sum_{k \leq k_{1} +C} 2^{k \sigma} \Vert \Delta_{k} ( \i \partial_{t} - ( - \Delta)^{s} + \i)^{-1} \ind_{[-T',T']} H^{3}_{r} ( \tilde{f}_{k_{1}},\tilde{g}_{k_{2}},\tilde{h}_{k_{3}} ) \Vert_{Z_{k}}
\\
& \lesssim 2^{k_{1} \sigma} \Vert f_{k_{1}} \Vert_{Z_{k_{1}}} 2^{k_{2} \frac{n-2}{2}} \Vert g_{k_{2}} \Vert_{Z_{k_{2}}} 2^{k_{3} \frac{n-2}{2}} \Vert h_{k_{3}} \Vert_{Z_{k_{3}}}.
\end{split}
\end{equation}
On the other hand, if $ r \leq k_{1} -C$ for sufficiently large $C$ then we use the localized maximal estimate and \eqref{multimultiplier} to obtain 
\begin{equation*}
\begin{split}
& \Vert \Delta_{k} ( \i \partial_{t} - ( -\Delta)^{s} + \i)^{-1} \ind_{[-T',T']} H^{3}_{r} ( \tilde{f}_{k_{1}},\tilde{g}_{k_{2}},\tilde{h}_{k_{3}}) \Vert_{Z_{k}}
\\
& \lesssim 2^{-k_{1}(2s-1)/2} \Vert H^{3}_{r} ( \tilde{f}_{k_{1}} , \tilde{g}_{k_{2}} , \tilde{h}_{k_{3}} ) \Vert_{L^{1}_{e'} L^{2}_{e'^{\perp},t}}
\\
& = 2^{-k_{1}(2s-1)/2} \sum_{\tn_{1}, \tn_{2}} \Vert H^{3}_{r} ( \tilde{f}_{k_{1}} ,P_{r,\tn_{1}} \tilde{g}_{k_{2}} , P_{r,\tn_{2}} \tilde{h}_{k_{3}} ) \Vert_{L^{1}_{e'} L^{2}_{e'^{\perp},t}}
\\
& \lesssim 2^{k_{1}(2s-1)/2}  2^{-r} \left( \sum_{\tn \in 2^{r} \Z^{n}} \Vert P_{r,\tn} \tilde{g}_{k_{2}} \Vert^{2}_{L^{2}_{e'} L^{\infty}_{e'^{\perp},t}} \right)^{1/2}
\\
& \times  \left( \sum_{\tn \in 2^{r} \Z^{n}} \Vert P_{r,\tn} \tilde{h}_{k_{3}} \Vert^{2}_{L^{2}_{e'} L^{\infty}_{e'^{\perp},t}} \right)^{1/2} \Vert \tilde{f} \Vert_{L^{\infty}_{e'} L^{2}_{e'^{\perp},t}}
\\
& \lesssim (1+ | k_{1} -r|)^{2} 2^{r (n-1)} 2^{(k_{1} -r)} 2^{-r} \Vert f_{k_{1}} \Vert_{Z_{k_{1}}} \Vert g_{k_{2}} \Vert_{Z_{k_{2}}} \Vert h_{k_{3}} \Vert_{Z_{k_{3}}}
\end{split}
\end{equation*}
summing over $ r \leq k_{1} -C$ we obtain (since $ n \geq 4$) 
\begin{equation*}
\begin{split}
&  \sum_{r \leq k_{1} -C} \Vert \Delta_{k} (\i \partial_{t} - (-\Delta)^{s} + \i)^{-1} \ind_{[-T',T']} H^{3}_{r} ( \tilde{f}_{k_{1}},\tilde{g}_{k_{2}},\tilde{h}_{k_{3}}) \Vert_{Z_{k}}
\\
& \lesssim 2^{k_{2} \frac{n-2}{2}} \Vert g_{k_{2}} \Vert_{Z_{k_{2}}} 2^{k_{3} \frac{n-2}{2}} \Vert h_{k_{3}} \Vert_{Z_{k_{3}}} \Vert f_{k_{1}} \Vert_{Z_{k_{1}}}.
\end{split}
\end{equation*}
This readily implies 
\begin{equation}\label{secindestforh3r}
\begin{split}
    &  \sum_{k \leq k_{1} +C} 2^{k \sigma} \sum_{r \leq k_{1} -C} \Vert \Delta_{k} ( \i \partial_{t} - (-\Delta)^{s} + \i )^{-1} \ind_{[-T',T']} H^{3}_{r} ( \tilde{f}_{k_{1}},\tilde{g}_{k_{2}},\tilde{h}_{k_{3}}) \Vert_{Z_{k}}
\\
& \lesssim 2^{k_{2} \frac{n-2}{2}} \Vert g_{k_{2}} \Vert_{Z_{k_{2}}} 2^{k_{3} \frac{n-2}{2}} \Vert h_{k_{3}} \Vert_{Z_{k_{3}}} 2^{k_{1}\sigma} \Vert f_{k_{1}} \Vert_{Z_{k_{1}}}.
\end{split}
\end{equation}
Combining \eqref{firstestforh3r} and \eqref{secindestforh3r} gives us \eqref{dyadictij3} in this case. \\
\textbf{Case 1b:} If $ 2^{k_{3}} \approx 2^{k_{1}}$ while $ 2^{k_{2}} \ll 2^{k_{3}}$.  In this case, notice that $ \kappa$ is localized at $ 2^{k_{3}}$. Therefore, we can use \eqref{multimultiplier} directly to obtain 
\begin{equation*}
\begin{split}
& \Vert \Delta_{k} ( \i \partial_{t} - ( - \Delta)^{s} + \i)^{-1} \ind_{[-T',T']} T^{i,j}_{3} ( \tilde{f}_{k_{1}} , \tilde{g}_{k_{2}}, \tilde{h}_{k_{3}} ) \Vert_{Z_{k}} 
\\
& \lesssim \max_{e \in \mathscr{E}} 2^{-k (2s-1)/2} \Vert T^{i,j}_{3} (\tilde{f}_{k_{1}},\tilde{g}_{k_{2}}, \tilde{h}_{k_{3}} ) \Vert_{L^{1}_{e} L^{2}_{e^{\perp},t}}
\\
& \overset{\eqref{multimultiplier}}{\lesssim} 2^{ -k (2s-1)/2} 2^{k_{1} (2s-1) } 2^{- k_{3}} \max_{e \in \mathscr{E}} \Vert g_{k_{2}} \Vert_{L^{2}_{e} L^{\infty}_{e^{\perp},t}} \Vert \tilde{f}_{k_{1}} \Vert_{L^{\infty}_{e'} L^{2}_{e'^{\perp},t}} \Vert \tilde{h}_{k_{3}} \Vert_{L^{2}_{e' } L^{\infty}_{e'^{\perp},t}}
\\
& \overset{\eqref{local_smoothing}, \eqref{maximal_Estimate}}{\lesssim}2^{(k_{1} -k)(2s-1)/2} \Vert f_{k_{1}} \Vert_{Z_{k_{1}}} 2^{k_{2} \frac{n-2}{2}} \Vert g_{k_{2}} \Vert_{Z_{k_{2}}} 2^{k_{3} \frac{n-2}{2}} \Vert h_{k_{3}} \Vert_{Z_{k_{3}}}.
\end{split}
\end{equation*}
Multiplying  by $ 2^{k \sigma}$ and summing over $k \leq k_{1} +C$ yields \eqref{dyadictij3} in this case. 
\\
\textbf{Case 1c:} If $ 2^{\min \{ k_{2} ,k_{3} \}} \ll 2^{k_{1}} $. In view of the cases above, we may assume $ k_{3} \leq k_{1} -C$ for sufficiently large $C$. 
In this case we will use Lemma \ref{Taylorcalc}. In other words, we write 
\begin{equation*}
\begin{split}
& (\xi-\eta)_{i} | \xi -\eta|^{2s-2} - (\xi -\kappa)_{i} | \xi -\kappa|^{2s-2}
\\
& = \sum_{|\alpha| =1}^{\infty} \frac{ m_{\alpha} ( \xi -\eta) }{\alpha !} |\xi-\eta|^{2s-|\alpha|-1} (\kappa -\eta)^{\alpha} 
\end{split}
\end{equation*}
where $ m_{\alpha}$ satisfy \eqref{qpppa}. Fix any multi-index $ \alpha$ then define $H^{4}_{\alpha}$ as follows 
\begin{equation}\label{Hfouralphadef}
\begin{split}
& \mathcal{F}_{\R^{n}} H^{4}_{\alpha}(\tilde{f}_{k_{1}},\tilde{g}_{k_{2}} ,\tilde{h}_{k_{3}} )(\xi,t) 
\\
& := \frac{1}{\alpha !}\int_{\R^{n} \times \R^{n}} m_{\alpha}(\xi -\eta) | \xi - \eta|^{2s-1 -|\alpha|}\hat{f}_{k_{1}}(\xi - \kappa ,t) ( \kappa - \eta)^{\alpha} \frac{ \langle \kappa - \eta , \kappa \rangle}{|\kappa |^{2}}
\\
& \times \left( \frac{ \langle \kappa - \eta, \eta \rangle}{| \kappa - \eta |^{2}} \right) \frac{ \hat{ h }_{k_{3}}(\kappa - \eta,t)}{ | \kappa - \eta|} \frac{ \hat{g}_{k_{2}}(\eta ,t)}{ | \eta |} d \eta d\kappa.
\end{split}
\end{equation}
And to prove \eqref{dyadictij3} in this case, it suffices to prove the following 
\begin{equation}\label{aszx09}
\begin{split}
& \sum_{k \leq k_{1} +C} 2^{k \sigma} \sum_{|\alpha |=1}^{\infty} \Vert \Delta_{k} ( \i \partial_{t} - ( - \Delta)^{s} + \i)^{-1} \ind_{[-T',T']} \partial_{j} H^{4}_{\alpha} ( \tilde{f}_{k_{1}} , \tilde{g}_{k_{2}} , \tilde{h}_{k_{3}} ) \Vert_{Z_{k}}
\\
& \lesssim 2^{k_{1} \sigma} \Vert f_{k_{1}} \Vert_{Z_{k_{1}}} 2^{k_{2} \frac{n-2}{2}} \Vert g_{k_{2}} \Vert_{Z_{k_{2}}} 2^{k_{3} \frac{n-2}{2}} \Vert h_{k_{3}} \Vert_{Z_{k_{3}}}
\end{split}
\end{equation}
The proof of \eqref{aszx09} is similar to the proof of \eqref{aszx2}. Set 
$$
a_{\alpha} : = \Vert \mathcal{F}^{-1}_{\R^{n}} \left( \varphi_{[k_{1} -C , k_{1} +C]} m_{\alpha} \right) \Vert_{L^{1}}.
$$
Then to avoid issues when $ \frac{1}{\kappa}$ is singular, assume first that $ 2^{\min \{ k_{2} , k_{3} \} } \ll 2^{ \max \{ k_{2} ,k_{3} \} }$. Then we estimate using \eqref{multimultiplier} ( notice that under the assumption that $   2^{\min \{ k_{2} , k_{3} \} } \ll 2^{ \max \{ k_{2} ,k_{3} \} }  $, we have that $ \kappa$ is localized at $ 2^{\max \{ k_{2} ,k_{3} \} }$)
\begin{equation*}
\begin{split}
& \Vert \Delta_{k} ( \i \partial_{t} - ( - \Delta)^{s} + \i)^{-1} \ind_{[[-T,T]}\partial_{j} H^{4}_{\alpha} ( \tilde{f}_{k_{1}} , \tilde{g}_{k_{2}} , \tilde{h}_{k_{3}} ) \Vert_{Z_{k}}
\\
& \lesssim 2^{-k (2s-1)/2} 2^{k} \max_{e \in \mathscr{E}} \Vert H^{4}_{\alpha} ( \tilde{f}_{k_{1}} , \tilde{g}_{k_{2}},\tilde{h}_{k_{3}} ) \Vert_{L^{1}_{e} L^{2}_{e^{\perp},t}}
\\
& \overset{\eqref{multimultiplier}}{\lesssim} \frac{ a_{\alpha} 2^{- C | \alpha|}}{\alpha ! } 2^{-k (2s-1)/2} 2^{k - k_{1}} 2^{k_{1} (2s-1)}  2^{- \max \{ k_{2} , k_{3} \} } \Vert \tilde{f}_{k_{1}} \Vert_{L^{\infty}_{e'} L^{2}_{e'^{\perp},t}} \Vert \tilde{g}_{k_{2}} \Vert_{L^{2}_{e'} L^{\infty}_{e'^{\perp},t}} \max_{e \in \mathscr{E}}  \Vert \tilde{h}_{k_{3}} \Vert_{L^{2}_{e} L^{\infty}_{e^{\perp},t}} 
\\
& \overset{\eqref{local_smoothing},\eqref{maximal_Estimate}}{\lesssim }\frac{a_{\alpha} 2^{-C |\alpha|}}{\alpha !} 2^{k- k_{1}} 2^{ ( k_{1} - k)(2s-1)/2}  \Vert f_{k_{1}} \Vert_{Z_{k_{1}}} 2^{k_{2} \frac{n-2}{2}} \Vert g_{k_{2}} \Vert_{Z_{k_{2}}} 2^{k_{3} \frac{n-2}{2}} \Vert h_{k_{3}} \Vert_{Z_{k_{3}}}
\end{split}
\end{equation*}
summing over $ \alpha$, and using the convergence stated in Lemma \ref{Taylorcalc}, we obtain 
\begin{equation*}
\begin{split}
& \sum_{|\alpha |=1}^{\infty} \Vert \Delta_{k} ( \i \partial_{t} - ( - \Delta)^{s} + \i)^{-1} \ind_{[-T',T']} \partial_{j} H^{4}_{\alpha} ( \tilde{f}_{k_{1}} , \tilde{g}_{k_{2}} , \tilde{h}_{k_{3}} ) \Vert_{Z_{k}}
\\
& \lesssim 2^{k- k_{1}} 2^{ ( k_{1} - k)(2s-1)/2}  \Vert f_{k_{1}} \Vert_{Z_{k_{1}}} 2^{k_{2} \frac{n-2}{2}} \Vert g_{k_{2}} \Vert_{Z_{k_{2}}} 2^{k_{3} \frac{n-2}{2}} \Vert h_{k_{3}} \Vert_{Z_{k_{3}}}
\end{split}
\end{equation*}
 multiplying by $ 2^{k \sigma}$ and summing over $k \leq k_{1} +C$ yields \eqref{aszx09}. Next, if $ 2^{k_{2}} \approx 2^{k_{3}}$ then it might be the case that $ \frac{1}{|\kappa|}$ is singular and we cannot use \eqref{multimultiplier} directly. Therefore, we use the localized maximal estimate instead. To that end, for each fixed multi-index $ \alpha$, define 
\begin{equation*}
\begin{split}
& \mathcal{F}_{\R^{n}} H^{4}_{\alpha ,r} ( \tilde{f}_{k_{1}},\tilde{g}_{k_{2}},\tilde{h}_{k_{3}})(\xi,t)
\\
& :=  \frac{1}{\alpha !}\int_{\R^{n} \times \R^{n}} m_{\alpha}(\xi -\eta) | \xi - \eta|^{2s-1 -|\alpha|}\hat{f}_{k_{1}}(\xi - \kappa ,t) ( \kappa - \eta)^{\alpha} \frac{ \langle \kappa - \eta , \kappa \rangle}{|\kappa |^{2}}
\\
& \times \left( \frac{ \langle \kappa - \eta, \eta \rangle}{| \kappa - \eta |^{2}} \right) \frac{ \hat{ h }_{k_{3}}(\kappa - \eta,t)}{ | \kappa - \eta|} \frac{ \hat{g}_{k_{2}}(\eta ,t)}{ | \eta |} \varphi_{r}(\kappa)  d \eta d\kappa.
\end{split}
\end{equation*}
Then clearly, we have 
$$
H^{4}_{\alpha}(f,g,h) = \sum_{r \leq k_{3} +C} H^{4}_{\alpha ,r}(f,g,h).
$$
Therefore, to conclude \eqref{aszx09} in the case $ 2^{k_{2}} \approx 2^{k_{3}}$, it suffices to prove the following 
\begin{equation}\label{aszx090}
\begin{split}
& \sum_{k \leq k_{1} +C} 2^{k \sigma} \sum_{|\alpha |=1}^{\infty} \sum_{r \leq k_{3} +C} \Vert \Delta_{k} ( \i \partial_{t} - ( - \Delta)^{s} + \i)^{-1} \ind_{[-T',T']}\partial_{j} H^{4}_{\alpha,r} ( \tilde{f}_{k_{1}} , \tilde{g}_{k_{2}} , \tilde{h}_{k_{3}} ) \Vert_{Z_{k}}
\\
& \lesssim 2^{k_{1} \sigma } \Vert f_{k_{1}} \Vert_{Z_{k_{1}}} 2^{k_{2} \frac{n-2}{2}} \Vert g_{k_{2}} \Vert_{Z_{k_{2}}} 2^{k_{3} \frac{n-2}{2}} \Vert h_{k_{3}} \Vert_{Z_{k_{3}}}.
\end{split}
\end{equation}
Fix any $ r \leq k_{3} +C$ and any multi-index $ \alpha$ then we estimate using the localized maximal estimate and \eqref{multimultiplier} 
\begin{equation*}
\begin{split}
& \Vert \Delta_{k} ( \i \partial_{t} - ( - \Delta)^{s} + \i)^{-1} \ind_{[-T',T']} \partial_{j} H^{4}_{\alpha,r} ( \tilde{f}_{k_{1}} , \tilde{g}_{k_{2}} , \tilde{h}_{k_{3}} ) \Vert_{Z_{k}}
\\
& \lesssim 2^{-k_{1} (2s-1)/2} 2^{k_{1}} \Vert H^{4}_{\alpha ,r} ( \tilde{f}_{k_{1}},\tilde{g}_{k_{2}} ,\tilde{h}_{k_{3}} ) \Vert_{L^{1}_{e'} L^{2}_{e'^{\perp},t}}
\\
& \overset{\eqref{multimultiplier}}{\lesssim} \frac{ a_{\alpha} 2^{- C | \alpha|}}{ \alpha !}  2^{k_{1} (2s-1)/2} 2^{-r} \Vert \tilde{f}_{k_{1}} \Vert_{L^{\infty}_{e'} L^{2}_{e'^{\perp},t}}
\\
& \times \sum_{\tn_{1} ,\tn_{2}} \Vert P_{r,\tn_{1}} \tilde{g}_{k_{2}} \Vert_{L^{2}_{e'} L^{\infty}_{e'^{\perp},t}}  \Vert P_{r,\tn_{2}} \tilde{h}_{k_{3}} \Vert_{L^{2}_{e'} L^{\infty}_{e'^{\perp},t}}
\\
& \overset{\eqref{maximalestcardinality}}{\lesssim}   \frac{ a_{\alpha} 2^{- C | \alpha|}}{ \alpha !}  2^{k_{1} (2s-1)/2} 2^{-r} \Vert \tilde{f}_{k_{1}} \Vert_{L^{\infty}_{e'} L^{2}_{e'^{\perp},t}}
\\
& \times \left( \sum_{\tn \in 2^{r} \Z^{n}} \Vert  P_{r,\tn} \tilde{g}_{k_{2}} \Vert^{2}_{L^{2}_{e'} L^{\infty}_{e'^{\perp},t}}  \right)^{1/2}  \left( \sum_{\tn \in 2^{r} \Z^{n}} \Vert  P_{r,\tn} \tilde{h}_{k_{3}} \Vert^{2}_{L^{2}_{e'} L^{\infty}_{e'^{\perp},t}}  \right)^{1/2} 
\\
& \overset{\eqref{local_smoothing}, \eqref{eq:max:Xkest:loc}}{\lesssim }\frac{ a_{\alpha} 2^{- C | \alpha|}}{ \alpha !} 2^{-r} 2^{r(n-1)} 2^{k_{3}-r} (1 + | r-k_{3}|)^{2} \Vert f_{k_{1}} \Vert
_{Z_{k_{1}}} \Vert g_{k_{2}} \Vert_{Z_{k_{2}}} \Vert h_{k_{3}} \Vert_{Z_{k_{3}}}.
\end{split}
\end{equation*}
Summing over $ r \leq k_{3} +C$ first we obtain 
\begin{equation*}
\begin{split}
&  \sum_{r \leq k_{3} +C} \Vert \Delta_{k} ( \i \partial_{t} - ( - \Delta)^{s} + \i)^{-1} \ind_{[-T',T']} \partial_{j} H^{4}_{\alpha,r} ( \tilde{f}_{k_{1}} , \tilde{g}_{k_{2}} , \tilde{h}_{k_{3}} ) \Vert_{Z_{k}}
\\
& \lesssim \frac{a_{\alpha} 2^{-C  | \alpha |}}{\alpha !} \Vert f_{k_{1}} \Vert_{Z_{k_{1}}} 2^{k_{2} \frac{n-2}{2}} \Vert g_{k_{2}} \Vert_{Z_{k_{2}}} 2^{k_{3} \frac{n-2}{2}} \Vert h_{k_{3}} \Vert_{Z_{k_{3}}}.
\end{split}
\end{equation*}
Summing over $ \alpha $ now we obtain, using the convergence stated in Lemma \ref{Taylorcalc}, 
\begin{equation*}
\begin{split}
& \sum_{|\alpha|=1}^{\infty} \sum_{r \leq k_{3} +C} \Vert \Delta_{k} ( \i \partial_{t} - ( - \Delta)^{s} + \i)^{-1} \ind_{[-T',T']} \partial_{j} H^{4}_{\alpha,r} ( \tilde{f}_{k_{1}} , \tilde{g}_{k_{2}} , \tilde{h}_{k_{3}} ) \Vert_{Z_{k}}
\\
& \lesssim \Vert f_{k_{1}} \Vert_{Z_{k_{1}}} 2^{k_{2} \frac{n-2}{2}} \Vert g_{k_{2}} \Vert_{Z_{k_{2}}} 2^{k_{3} \frac{n-2}{2}} \Vert h_{k_{3}} \Vert_{Z_{k_{3}}}.
\end{split}
\end{equation*}
This readily implies \eqref{aszx090}. This concludes case 1c, and hence completes the proof for case 1. 
\\
\textbf{Case 2:} If $ \max \{ k_{1} ,k_{2} ,k_{3}  \} = k_{2}$. Then in view of \eqref{multilpierlemma} we may assume that $ g_{k_{2}}$ has Fourier transform supported in $ \{ \langle \xi , e' \rangle \geq 2^{k_{2} -2} \}$ for some $ e' \in \mathscr{E}$. In view of the analysis above, the only sub-cases we need to consider is when $ 2^{k_{2}} \approx 2^{k_{3}}$ while $ 2^{k_{1}} \ll 2^{k_{2}}$, and the case when $ 2^{\max \{ k_{1} ,k_{3} \}} \ll 2^{k_{2}}$. 
\\
\textbf{Case 2a:} If $ 2^{k_{2}} \approx 2^{k_{3}}$ while $ 2^{k_{1}} \ll 2^{k_{2}}$. Since $ 2^{k_{2}} \approx 2^{k_{3}}$ then $ \frac{1}{|\kappa|}  $ might be singular, and hence one needs to be careful when applying \eqref{multimultiplier}. We make use of the operator $H^{3}_{r}$ defined in \eqref{defh3r} to write 
$$
T^{i,j}_{3} ( \tilde{f}_{k_{1}}, \tilde{g}_{k_{2}}, \tilde{h}_{k_{3}}) = \sum_{r \leq k_{2} + C} H^{3}_{r} ( \tilde{f}_{k_{1}},\tilde{g}_{k_{2}},\tilde{h}_{k_{3}}).
$$
 For $ r \leq k_{1} -C$ we have that $ H^{3}_{r}( \tilde{f}_{k_{1}},\tilde{g}_{k_{2}},\tilde{h}_{k_{3}})$ has frequency localized around $ 2^{k_{1}}$. In particular, $ \Delta_{k} H^{3}_{r} ( \tilde{f}_{k_{1}},\tilde{g}_{k_{2}},\tilde{h}_{k_{3}}) = 0 $ for $k \geq k_{1} +C$ and $ r \leq k_{1} -C$. Then by the same argument leading to \eqref{secindestforh3r}, we obtain
\begin{equation*}
\begin{split}
& \sum_{k \leq k_{2} +C}2^{k \sigma} \sum_{r \leq k_{1} -C} \Vert \Delta_{k} ( \i \partial_{t} - ( - \Delta)^{s} + \i)^{-1} \ind_{[-T',T']} H^{3}_{r} ( \tilde{f}_{k_{1}}, \tilde{g}_{k_{2}} , \tilde{h}_{k_{3}}) \Vert_{Z_{k}}
\\
& \lesssim 2^{k_{1} \sigma } \Vert f_{k_{1}} \Vert_{Z_{k_{1}}} 2^{k_{2} \frac{n-2}{2}} \Vert g_{k_{2}} \Vert_{Z_{k_{2}}} 2^{k_{3} \frac{n-2}{2}} \Vert h_{k_{3}} \Vert_{Z_{k_{3}}}.
\end{split}
\end{equation*}
Next, fix $ r \in [k_{1} -C , k_{2} +C]$. We use \eqref{multimultiplier} and the assumption that $ 2^{k_{1}} \lesssim 2^{r}$ to obtain
\begin{equation*}
\begin{split}
& \Vert \Delta_{k} ( \i \partial_{t} - ( - \Delta)^{s} + \i)^{-1} \ind_{[-T',T']} H^{3}_{r} ( \tilde{f}_{k_{1}}, \tilde{g}_{k_{2}}, \tilde{h}_{k_{3}} ) \Vert_{Z_{k}}
\\
& \lesssim 2^{-k (2s-1)/2} \max_{e \in \mathscr{E}} \Vert H^{3}_{r} ( \tilde{f}_{k_{1}} , \tilde{g}_{k_{2}} , \tilde{h}_{k_{3}} ) \Vert_{L^{1}_{e} L^{2}_{e^{\perp},t}} 
\\
& \overset{\eqref{multimultiplier}}{\lesssim} 2^{-k(2s-1)/2} 2^{r} 2^{k_{2}(2s-1)} 2^{-r} 2^{-k_{2}} \Vert \tilde{g}_{k_{2}} \Vert_{L^{\infty}_{e'} L^{2}_{e'^{\perp},t}} \Vert \tilde{h}_{k_{3}} \Vert_{L^{2}_{e'} L^{\infty}_{e'^{\perp},t}} \max_{e \in \mathscr{E}} \Vert \tilde{f}_{k_{1}} \Vert_{L^{2}_{e} L^{\infty}_{e^{\perp},t}}
\\
& \overset{\eqref{local_smoothing},\eqref{maximal_Estimate}}{\lesssim} 2^{-k(2s-1)/2} 2^{k_{2}(2s-1)/2} 2^{k_{1} \frac{n-2}{2}} \Vert f_{k_{1}} \Vert_{Z_{k_{1}}} 2^{k_{2} \frac{n-2}{2}} \Vert h_{k_{3}} \Vert_{Z_{k_{3}}} \Vert g_{k_{2}} \Vert_{Z_{k_{2}}}.
\end{split}
\end{equation*}
Since for fixed $ r \in [k_{1} -C, k_{2} +C]$ the sum in $k$ runs up to $ r+C$ we obtain
\begin{equation*}
\begin{split}
&\sum_{k \leq r +C}2^{k \sigma} \Vert \Delta_{k} ( \i \partial_{t} - ( - \Delta)^{s} + \i)^{-1} \ind_{[-T',T']}H^{3}_{r} ( \tilde{f}_{k_{1}}, \tilde{g}_{k_{2}} , \tilde{h}_{k_{3}}) \Vert_{Z_{k}}
\\
& \lesssim 2^{k_{1} \frac{n-2}{2}} \Vert f_{k_{1}} \Vert_{Z_{k_{1}}} 2^{r \sigma} 2^{(k_{2} -r)(2s-1)/2} \Vert g_{k_{2}} \Vert_{Z_{k_{2}}} 2^{k_{3} \frac{n-2}{2}} \Vert h_{k_{3}} \Vert_{Z_{k_{3}}}.
\end{split}
\end{equation*}
This readily implies 
\begin{equation*}
\begin{split}
& \sum_{k \leq k_{2} +C}2^{k \sigma} \sum_{ k_{1} -C \leq r \leq k_{2} +C} \Vert \Delta_{k} ( \i \partial_{t} - ( - \Delta)^{s} + \i)^{-1} \ind_{[-T',T']}H^{3}_{r} ( \tilde{f}_{k_{1}}, \tilde{g}_{k_{2}} , \tilde{h}_{k_{3}}) \Vert_{Z_{k}}
\\
& \lesssim 2^{k_{1} \frac{n-2}{2}} \Vert f_{k_{1}} \Vert_{Z_{k_{1}}} 2^{k_{2} \sigma} \Vert g_{k_{2}} \Vert_{Z_{k_{2}}} 2^{k_{3} \frac{n-2}{2}} \Vert h_{k_{3}} \Vert_{Z_{k_{3}}}.
\end{split}
\end{equation*}
This concludes case 2a. 
\\
\textbf{Case 2b:} If $ 2^{\max \{ k_{1}, k_{3} \} } \ll 2^{k_{2}}$. In this case $ \kappa$ is localized around $ 2^{k_{2}}$ and we can apply \eqref{multimultiplier} directly. Assume first $ 2^{k_{1}} \lesssim 2^{k_{3}} $. Then we apply \eqref{multimultiplier}
\begin{equation*}
\begin{split}
& \Vert \Delta_{k} ( \i \partial_{t} - (- \Delta)^{s} + \i)^{-1} \ind_{[-T',T']}T^{i,j}_{3} ( \tilde{f}_{k_{1}},\tilde{g}_{k_{2}},\tilde{h}_{k_{3}}) \Vert_{Z_{k}}
\\
& \lesssim 2^{-k_{2} ( 2s-1)/2} \Vert T^{i,j}_{3} ( \tilde{f}_{k_{1}},\tilde{g}_{k_{2}},\tilde{h}_{k_{3}}) \Vert_{L^{1}_{e'} L^{2}_{e'^{\perp},t}}
\\
& \overset{\eqref{multimultiplier}}{\lesssim} 2^{-k_{2}(2s-1)/2} \Vert \tilde{g}_{k_{2}} \Vert_{L^{\infty}_{e'} L^{2}_{e'^{\perp},t}} \Vert \tilde{f}_{k_{1}} \Vert_{L^{2}_{e'} L^{\infty}_{e'^{\perp},t}} \Vert \tilde{h}_{k_{3}} \Vert_{L^{2}_{e'} L^{\infty}_{e'^{\perp},t}} 2^{(2s-1)k_{2}} 2^{-k_{3}}
\\
& \lesssim 2^{k_{1} \frac{n-2}{2}} \Vert f_{k_{1}} \Vert_{Z_{k_{1}}} 2^{k_{3} \frac{n-2}{2}} \Vert h_{k_{3}} \Vert_{Z_{k_{3}}} \Vert g_{k_{2}} \Vert_{Z_{k_{2}}},
\end{split}
\end{equation*}
where in the last line we used the maximal estimate, smoothing estimate and the assumption that $ 2^{k_{1}} \lesssim 2^{k_{3}}$. Next, if $ 2^{k_{3}} \ll 2^{k_{1}}$. Then we make use of Lemma \ref{Taylorcalc}. More precisely, we write 
$$
T^{i,j}_{3}( \tilde{f}_{k_{1}},\tilde{g}_{k_{2}},\tilde{h}_{k_{3}}) = \sum_{|\alpha|=1}^{\infty} \partial_{j} H^{4}_{\alpha}( \tilde{f}_{k_{1}},\tilde{g}_{k_{2}},\tilde{h}_{k_{3}})
$$
where $H^{4}_{\alpha}$ was defined in \eqref{Hfouralphadef}. As before, set
$$
a_{\alpha} : = \Vert \mathcal{F}^{-1}_{\R^{n}} \left( \varphi_{[k_{1}-C,k_{1}+C]} m_{\alpha} \right) \Vert_{L^{1}}
$$
Then using \eqref{multimultiplier} we obtain for each fixed $ \alpha$
\begin{equation*}
\begin{split}
& \Vert \Delta_{k} ( \i \partial_{t} - ( - \Delta)^{s} + \i)^{-1} \ind_{[-T',T']} \partial_{j} H^{4}_{\alpha} ( \tilde{f}_{k_{1}} , \tilde{g}_{k_{2}} , \tilde{h}_{k_{3}}) \Vert_{Z_{k}}
\\
& \lesssim 2^{-k_{2} ( 2s-1)/2} 2^{k_{2}} \Vert H^{4}_{\alpha} ( \tilde{f}_{k_{1}}, \tilde{g}_{k_{2}} , \tilde{h}_{k_{3}} ) \Vert_{L^{1}_{e'},L^{2}_{e'^{\perp},t}}
\\
& \overset{\eqref{multimultiplier}}{ \lesssim} 2^{-k_{2} (2s-1)/2}\frac{ a_{ \alpha} 2^{-C | \alpha|}}{ \alpha !} 2^{k_{1} ( 2s-1)} 2^{-k_{1}} \Vert \tilde{g}_{k_{2}} \Vert_{L^{\infty}_{e'} L^{2}_{e'^{\perp},t}} \Vert \tilde{f}_{k_{1}} \Vert_{L^{2}_{e'}L^{\infty}_{e'^{\perp},t}} \Vert \tilde{h}_{k_{3}} \Vert_{L^{2}_{e'}L^{\infty}_{e'^{\perp},t}}
\\
& \lesssim \frac{ 2^{-C |\alpha|} a_{ \alpha}}{\alpha !} 2^{k_{1} \frac{n-2}{2}} \Vert f_{k_{1}} \Vert_{Z_{k_{1}} } 2^{k_{3} \frac{n-2}{2}} \Vert h_{k_{3}} \Vert_{Z_{k_{3}}} \Vert g_{k_{2}} \Vert_{Z_{k_{2}}}.
\end{split}
\end{equation*}
Summing over $ \alpha$ and then multiplying by $ 2^{k \sigma}$ and summing over $ k \leq k_{2} +C$ yields the result. This concludes case 2b and hence case 2. 
\\
\textbf{Case 3:} If $ \max \{ k_{1} , k_{2} ,k_{3} \} = k_{3}$. In view of \eqref{multilpierlemma} we may assume that $h_{k_{3}}$ has Fourier transform supported in $ \{ \langle \xi ,e \rangle \geq 2^{k_{3} -2} \}$ for some $e \in \mathscr{E}$. In view the analysis above, we only need to consider the case in which $ 2^{\max \{ k_{1} , k_{2} \} } \ll 2^{k_{3}}$. Then in this case, notice that $ \kappa$ is localized at $ 2^{k_{3}}$. Then we can apply \eqref{multimultiplier} directly to obtain 
\begin{equation*}
\begin{split}
& \Vert \Delta_{k} ( \i \partial_{t} - ( - \Delta)^{s} + \i)^{-1} \ind_{[-T',T']} T^{i,j}_{3} ( \tilde{f}_{k_{1}} , \tilde{g}_{k_{2}} , \tilde{h}_{k_{3}} ) \Vert_{Z_{k}}
\\
& \lesssim 2^{-k_{3}(2s-1)/2} \Vert T^{i,j}_{3} ( \tilde{f}_{k_{1}},\tilde{g}_{k_{2}}, \tilde{h}_{k_{3}} ) \Vert_{L^{1}_{e} L^{2}_{e^{\perp},t} } 
\\
& \overset{\eqref{multimultiplier}}{ \lesssim } 2^{k_{3} ( 2s-1)/2} 2^{-k_{3}} \Vert \tilde{h}_{k_{3}} \Vert_{L^{\infty}_{e} L^{2}_{e^{\perp},t}} \Vert \tilde{g}_{k_{2}} \Vert_{L^{2}_{e} L^{\infty}_{e^{\perp},t}} \Vert \tilde{f}_{k_{1}} \Vert_{L^{2}_{e} L^{\infty}_{e^{\perp},t}}
\\
& \overset{\eqref{local_smoothing},\eqref{maximal_Estimate}}{\lesssim } 2^{k_{1} \frac{n-2}{2}} \Vert f_{k_{1}} \Vert_{Z_{k_{1}}} 2^{k_{2} \frac{n-2}{2}} \Vert g_{k_{2}} \Vert_{Z_{k_{2}}} \Vert h_{k_{3}} \Vert_{Z_{k_{3}}}
\end{split}
\end{equation*}
this implies the result and concludes case 3. 
\end{proof}
Next, we deal with $ \Gamma^{b}_{2}$. The estimates for $ \Gamma^{b}_{2}$ will follow from the following trilinear estimate
\begin{proposition}\label{proofforgamma2bb}
Let $ n \geq 4$, $ s \in (1/2 ,1)$, $ \sigma \geq \frac{n-2}{2}$, and $T' \in [0,\infty]$. Then the following holds for any $ f,g,h \in F^{\frac{n-2}{2}} \cap F^{\sigma}$
\begin{equation}\label{trtij4}
\begin{split}
& \Vert \ind_{[-T',T']} T^{i,j}_{4} ( \tilde{f}, \tilde{g} , \tilde{h} ) \Vert_{N^{\sigma}}
\\
& \lesssim \Vert f \Vert_{F^{\sigma}} \Vert g \Vert_{F^{\frac{n-2}{2}}} \Vert h \Vert_{F^{\frac{n-2}{2}}}
\\
&+ \Vert f \Vert_{F^{\frac{n-2}{2}}} \Vert g \Vert_{F^{\sigma}} \Vert h \Vert_{F^{\frac{n-2}{2}}}
\\
& + \Vert f \Vert_{F^{\frac{n-2}{2}}} \Vert g \Vert_{F^{\frac{n-2}{2}}} \Vert h \Vert_{F^{\sigma}}
\end{split}
\end{equation}
where $T^{i,j}_{4}$ is defined via the Fourier transform as follows 
\begin{equation*}
\begin{split}
& \mathcal{F}_{\R^{n}} ( T^{i,j}_{4} ( f,g,h) ) (\xi,t) 
\\
& = \int_{\R^{n} \times \R^{n}} \xi_{j} [ \psi( \xi - \eta)  - \psi(\xi) ]_{i} \hat{f}(\xi - \kappa,t) \frac{ \langle \kappa - \eta, \kappa \rangle}{ | \kappa |^{2}}
\\
& \times \frac{ \hat{g}(\eta,t)}{|\eta|} \frac{ \hat{h}(\kappa - \eta,t)}{ | \kappa - \eta|} d \eta d \kappa
\end{split}
\end{equation*}
where recall that $ \psi( \zeta) = \zeta | \zeta|^{2s-2}$. 
\end{proposition}
\begin{proof}
The proof is very similar to the proof of Proposition \ref{Proofforgamma2a}. Namely, to prove \eqref{trtij4} it suffices to prove the following 
\begin{equation}\label{dyadictij4}
\begin{split}
& \sum_{k \leq \max \{ k_{1}, k_{2} ,k_{3} \} +C } 2^{k \sigma } \Vert \Delta_{k} ( \i \partial_{t} - ( - \Delta)^{s} + \i)^{-1} \ind_{[-T',T']} T^{i,j}_{4} ( \tilde{f}_{k_{1}}, \tilde{g}_{k_{2}}, \tilde{h}_{k_{3}}) \Vert_{Z_{k}}
\\
& \lesssim 2^{k_{1} \sigma } \Vert f_{k_{1}} \Vert_{Z_{k_{1}}} 2^{k_{2} \frac{n-2}{2}} \Vert g_{k_{2}} \Vert_{Z_{k_{2}}} 2^{k_{3} \frac{n-2}{2}} \Vert h_{k_{3}} \Vert_{Z_{k_{3}}}
\\
& + 2^{k_{1} \frac{n-2}{2}} \Vert f_{k_{1}} \Vert_{Z_{k_{1}}} 2^{k_{2}\sigma} \Vert g_{k_{2}} \Vert_{Z_{k_{2}}} 2^{k_{3} \frac{n-2}{2}} \Vert h_{k_{3}} \Vert_{Z_{k_{3}}}
\\
& + 2^{k_{1} \frac{n-2}{2}} \Vert f_{k_{1}} \Vert_{Z_{k_{1}}} 2^{k_{2} \frac{n-2}{2}} \Vert g_{k_{2}} \Vert_{Z_{k_{2}}} 2^{k_{3} \sigma} \Vert h_{k_{3}} \Vert_{Z_{k_{3}}}
\end{split}
\end{equation}
And to prove \eqref{dyadictij4}, we consider several cases depending on which function has the highest frequency. \\
\textbf{Case 1:} If $ \max \{ k_{1},k_{2} ,k_{3} \} = k_{1}$. Then in view of \eqref{multilpierlemma} we may assume that $ f_{k_{1}}$ has Fourier transform supported in the set $ \{ \langle \xi , e' \rangle \geq 2^{k_{1} -2} \}$ for some $e' \in \mathscr{E}$. Then we consider several sub-cases depending on whether the frequencies $2^{k_{2}},2^{k_{3}}$ are close to $2^{k_{1}}$ or not. 
\\
\textbf{Case 1a:} If $ 2^{k_{1}} \approx 2^{k_{2}} \approx 2^{k_{3}}$. Then as was done in the proof of Proposition \ref{Proofforgamma2a}, the idea is to use \eqref{multimultiplier}, the smoothing and maximal estimates and the assumption $ 2^{k_{1}} \approx 2^{k_{2}} \approx 2^{k_{3}}$. However, we have to use the localized maximal estimate since $ \frac{1}{|\kappa|}$ might be singular in this case. To that end, write 
$$
T^{i,j}_{4} ( \tilde{f}_{k_{1}}, \tilde{g}_{k_{2}}, \tilde{h}_{k_{3}}) = \sum_{ r \leq k_{1} +C} H^{5}_{r} ( \tilde{f}_{k_{1}},\tilde{g}_{k_{2}},\tilde{h}_{k_{3}})
$$
where $H^{5}_{r}$ is defined via the Fourier transform by 
\begin{equation}\label{h5rdef}
\begin{split}
& \mathcal{F}_{\R^{n}}( H^{5}_{r} ( f,g,h) )(\xi,t) 
\\
& =  \int_{\R^{n} \times \R^{n}} \xi_{j} [ \psi( \xi - \eta)  - \psi(\xi) ]_{i} \hat{f}(\xi - \kappa,t) \frac{ \langle \kappa - \eta, \kappa \rangle}{ | \kappa |^{2}}
\\
& \times \frac{ \hat{g}(\eta,t)}{|\eta|} \frac{ \hat{h}(\kappa - \eta,t)}{ | \kappa - \eta|} \varphi_{r}( \kappa) d \eta d \kappa.
\end{split}
\end{equation}
Fix $ r \in [k_{1} -C ,k_{1} +C]$ then by \eqref{multimultiplier} and the smoothing, maximal estimates we obtain
\begin{equation*}
\begin{split}
& \Vert \Delta_{k} ( \i \partial_{t} - ( - \Delta)^{s} + \i)^{-1} \ind_{[-T',T']} H^{5}_{r} ( \tilde{f}_{k_{1}},\tilde{g}_{k_{2}},\tilde{h}_{k_{3}}) \Vert_{Z_{k}}
\\
& \lesssim 2^{-k (2s-1)/2} \max_{e \in \mathscr{E}} \Vert H^{5}_{r} ( \tilde{f}_{k_{1}}, \tilde{g}_{k_{2}}, \tilde{h}_{k_{3}}) \Vert_{L^{1}_{e} L^{2}_{e^{\perp},t}}
\\
& \overset{\eqref{multimultiplier}}{\lesssim} 2^{-k(2s-1)/2} 2^{k} 2^{k_{1} ( 2s-1)} 2^{-2k_{1}} \Vert \tilde{f}_{k_{1}} \Vert_{L^{\infty}_{e'} L^{2}_{e'^{\perp},t}} \Vert \tilde{g}_{k_{2}} \Vert_{L^{2}_{e'} L^{\infty}_{e'^{\perp},t}} \max_{e \in \mathscr{E}} \Vert \tilde{h}_{k_{3}} \Vert_{L^{2}_{e} L^{\infty}_{e^{\perp},t}}
\\
&\overset{\eqref{local_smoothing},\eqref{maximal_Estimate}}{ \lesssim} 2^{(k_{1}-k)(2s-1)/2} 2^{k - k_{1}} \Vert f_{k_{1}} \Vert_{Z_{k_{1}}} 2^{k_{2} \frac{n-2}{2}} \Vert g_{k_{2}} \Vert_{Z_{k_{2}}} 2^{k_{3} \frac{n-2}{2}} \Vert h_{k_{3}} \Vert_{Z_{k_{3}}}.
\end{split}
\end{equation*}
Multiplying by $ 2^{k \sigma}$ and summing over $k$ we obtain
\begin{equation}\label{firstestimateforh5r}
\begin{split}
& \sum_{k \leq k_{1} +C} 2^{k \sigma} \sum_{r \in [k_{1} -C,k_{1} +C]} \Vert \Delta_{k} ( \i \partial_{t} - ( - \Delta)^{s} + \i)^{-1} \ind_{[-T',T']} H^{5}_{r} ( \tilde{f}_{k_{1}} , \tilde{g}_{k_{2}} , \tilde{h}_{k_{3}} ) \Vert_{Z_{k}}
\\
& \lesssim 2^{k_{1} \sigma } \Vert f_{k_{1}} \Vert_{Z_{k_{1}}} 2^{k_{2} \frac{n-2}{2}} \Vert g_{k_{2}} \Vert_{Z_{k_{2}}} 2^{k_{3} \frac{n-2}{2}} \Vert h_{k_{3}} \Vert_{Z_{k_{3}}}.
\end{split}
\end{equation}
Next, if $ r \leq k_{1} -C$ for sufficiently large constant $ C $. Then $H^{5}_{r} ( \tilde{f}_{k_{1}},\tilde{g}_{k_{2}},\tilde{h}_{k_{3}} )$ has Fourier transform supported in $ \{ \langle \xi , e' \rangle \geq 2^{k_{1} -c}, |\xi| \approx 2^{k_{1}} \}$. Then we have
\begin{equation*}
\begin{split}
& \Vert \Delta_{k} ( \i \partial_{t} - ( - \Delta)^{s} + \i)^{-1} \ind_{[-T',T']} H^{5}_{r} ( \tilde{f}_{k_{1}}, \tilde{g}_{k_{2}} , \tilde{h}_{k_{3}}) \Vert_{Z_{k}} 
\\
& \lesssim 2^{-k_{1} (2s-1)/2} \Vert H^{5}_{r} ( \tilde{f}_{k_{1}} , \tilde{g}_{k_{2}} , \tilde{h}_{k_{3}} \Vert_{L^{1}_{e'} L^{2}_{e'^{\perp},t}}
\\
& \leq 2^{-k_{1} (2s-1)/2} \sum_{\tn_{1} , \tn_{2}} \Vert H^{5}_{r} (\tilde{f}_{k_{1}}, P_{\tn_{1} ,r} \tilde{g}_{k_{2}} , P_{\tn_{2},r} \tilde{h}_{k_{3}} ) \Vert_{L^{1}_{e'} L^{2}_{e'^{\perp},t}}
\\
& \overset{\eqref{multimultiplier}}{\lesssim} 2^{k_{1} (2s-1)/2} 2^{-k_{1}}  2^{(k_{1}-r)}\Vert \tilde{f}_{k_{1}} \Vert_{L^{\infty}_{e'} L^{2}_{e'^{\perp},t}} \sum_{\tn_{1} , \tn_{2}} \Vert P_{r,\tn_{1}} \tilde{g}_{k_{2}} \Vert_{L^{2}_{e'} L^{\infty}_{e'^{\perp},t}} \Vert P_{r, \tn_{2} } \tilde{h}_{k_{3}} \Vert_{L^{2}_{e'} L^{\infty}_{e'^{\perp},t}} \\
& \overset{\eqref{maximalestcardinality}}{\lesssim} 2^{k_{1} (2s-1)/2} 2^{-r} \Vert \tilde{f}_{k_{1}} \Vert_{L^{\infty}_{e'} L^{2}_{e'^{\perp},t}} 
\\
& \times \left( \sum_{\tn \in 2^{r} \Z^{n}} \Vert P_{r, \tn  } \tilde{g}_{k_{2}} \Vert^{2}_{L^{2}_{e'}L^{\infty}_{e'^{\perp},t}} \right)^{1/2}\left( \sum_{\tn \in 2^{r} \Z^{n}} \Vert P_{r, \tn  } \tilde{h}_{k_{3}} \Vert^{2}_{L^{2}_{e'}L^{\infty}_{e'^{\perp},t}} \right)^{1/2}
\\
& \overset{\eqref{local_smoothing},\eqref{eq:max:Xkest:loc}}{\lesssim} 2^{-r } \Vert f_{k_{1}} \Vert_{Z_{k_{1}}} ( 1+ | k_{1} -r|)^{2} 2^{r(n-1)} 2^{k_{2} - r} \Vert g_{k_{2}} \Vert_{Z_{k_{2}}} \Vert h_{k_{3}} \Vert_{Z_{k_{3}}}
\\
& = \Vert f_{k_{1}} \Vert_{Z_{k_{1}}} 2^{r(n-3)}( 1+|r-k_{1})^{2} 2^{k_{1}} \Vert g_{k_{2}} \Vert_{Z_{k_{2}} } \Vert h_{k_{3}} \Vert_{Z_{k_{3}}}
\end{split}
\end{equation*}
Then this implies the following 
\begin{equation}\label{secondestforh5r}
\begin{split}
& \sum_{k \leq k_{1} + C} 2^{k \sigma } \sum_{r \leq k_{1} -C}  \Vert \Delta_{k} ( \i \partial_{t} - ( - \Delta)^{s} + \i)^{-1} \ind_{[-T',T']} H^{5}_{r} ( \tilde{f}_{k_{1}}, \tilde{g}_{k_{2}}, \tilde{h}_{k_{3}} ) \Vert_{Z_{k}}
\\
& \lesssim 2^{k_{1} \sigma } \Vert f_{k_{1}} \Vert_{Z_{k_{1}}} 2^{k_{2} \frac{n-2}{2}} \Vert g_{k_{2}} \Vert_{Z_{k_{2}}} 2^{k_{3} \frac{n-2}{2}} \Vert h_{k_{3}} \Vert_{Z_{k_{3}}}
\end{split}
\end{equation}
combining \eqref{firstestimateforh5r} and \eqref{secondestforh5r} gives us the result in this case. 
\\
\textbf{Case 1b:} If $ 2^{k_{2}} \approx 2^{k_{1}}$ while $ 2^{k_{3}} \ll 2^{k_{1}}$. Then in this case $ \kappa$ is localized around $ 2^{k_{2}} \approx 2^{k_{1}}$ and hence we can apply \eqref{multimultiplier} directly
\begin{equation*}
\begin{split}
& \Vert \Delta_{k} ( \i \partial_{t} - ( - \Delta)^{s} + \i)^{-1}  \ind_{[-T',T']} T^{i,j}_{4} ( \tilde{f}_{k_{1}}, \tilde{g}_{k_{2}} , \tilde{h}_{k_{3}}) \Vert_{Z_{k}}
\\
& \lesssim 2^{-k (2s-1)/2} \max_{e \in \mathscr{E}} \Vert T^{i,j}_{4} ( \tilde{f}_{k_{1}}, \tilde{g}_{k_{2}} , \tilde{h}_{k_{3}} ) \Vert_{L^{1}_{e } L^{2}_{e^{\perp},t}}
\\
& \overset{\eqref{multimultiplier}}{ \lesssim } 2^{-k ( 2s-1)/2} 2^{k} 2^{k_{1} (2s-1)} \Vert \tilde{f}_{k_{1}} \Vert_{L^{\infty}_{e' } L^{2}_{e'^{\perp},t}}  2^{-2k_{1}} \Vert \tilde{g}_{k_{2}} \Vert_{L^{2}_{e'} L^{\infty}_{e'^{\perp},t}} \max_{e \in \mathscr{E}} \Vert \tilde{h}_{k_{3}} \Vert_{L^{2}_{e} L^{\infty}_{e^{\perp},t}}
\\
& \overset{\eqref{local_smoothing},\eqref{maximal_Estimate}}{\lesssim} 2^{k -k_{1}} 2^{(k_{1}- k)(2s-1)/2} \Vert f_{k_{1}} \Vert_{Z_{k_{1}}} 2^{k_{2} \frac{n-2}{2}} \Vert g_{k_{2}} \Vert_{Z_{k_{2}}} 2^{k_{3} \frac{n-2}{2}} \Vert h_{k_{3}} \Vert_{Z_{k_{3}}}
\end{split}
\end{equation*}
this readily implies \eqref{dyadictij4} in this case. 
\\
\textbf{Case 1c:} If $ 2^{\min \{ k_{2} ,k_{3} \}} \ll 2^{k_{1}} $. In view of the analysis above we may assume that $ k_{2} \leq k_{1} -C$ for sufficiently large $C $. We wish to use Lemma \ref{Taylorcalc}. However, if $ 2^{k_{3}} \approx 2^{k_{1}}$ then it might be the case that $ |\xi| \approx |\eta|$ and hence the convergence stated in Lemma \ref{Taylorcalc} might fail. To overcome this, we split the sum in $k$ for when $k \leq k_{2} +C_{0}$ and when $ k \geq k_{2} +C_{0}$. Fix $ k \leq k_{2} +C_{0}$ then 
$$
\Delta_k T^{i,j}_{4} ( \tilde{f}_{k_{1}} ,\tilde{g}_{k_{2}},\tilde{h}_{k_{3}} ) \neq 0
$$
only if $ 2^{k_{3}} \approx 2^{k_{1}}$. Therefore, $ \kappa$ is localized at $ 2^{k_{3}} \approx 2^{k_{1}}$. Then we apply \eqref{multimultiplier} to obtain
\begin{equation*}
\begin{split}
& \Vert \Delta_{k} ( \i \partial_{t} - ( -\Delta)^{s} + \i)^{-1} \ind_{[-T',T']}T^{i,j}_{4} ( \tilde{f}_{k_{1}}, \tilde{g}_{k_{2}},\tilde{h}_{k_{3}} ) \Vert_{Z_{k}}
\\
& \lesssim 2^{-k(2s-1)/2} \max_{e \in \mathscr{E}} \Vert T^{i,j}_{4} ( \tilde{f}_{k_{1}}, \tilde{g}_{k_{2}} ,\tilde{h}_{k_{3}} ) \Vert_{L^{1}_{e} L^{2}_{e^{\perp},t}}
\\
& \lesssim 2^{-k(2s-1)/2} 2^{k} 2^{k_{2}(2s-1)} 2^{-k_{2}} 2^{-k_{1}} \Vert \tilde{f}_{k_{1}} \Vert_{L^{\infty}_{e'} L^{2}_{e'^{\perp},t}} \Vert \tilde{g}_{k_{2}} \Vert_{L^{2}_{e'} L^{\infty}_{e'^{\perp},t}} \max_{e \in \mathscr{E}} \Vert \tilde{h}_{k_{3}} \Vert_{L^{2}_{e} L^{\infty}_{e^{\perp},t}}
\\
& \lesssim 2^{(k_{2} -k)(2s-1)/2} 2^{k-k_{2}} \Vert f_{k_{1}} \Vert_{Z_{k_{1}}} 2^{k_{2} \frac{n-2}{2}} \Vert g_{k_{2}} \Vert_{Z_{k_{2}}} 2^{k_{3} \frac{n-2}{2}} \Vert h_{k_{3}} \Vert_{Z_{k_{3}}}
\end{split}
\end{equation*}
summing over $k \leq k_{2} +C_{0}$ we obtain 
\begin{equation*}
\begin{split}
& \sum_{k \leq k_{2} +C_{0}} 2^{k \sigma } \Vert \Delta_{k} ( \i \partial_{t} - ( -\Delta)^{s} + \i)^{-1} \ind_{[-T',T']} T^{i,j}_{4} ( \tilde{f}_{k_{1}}, \tilde{g}_{k_{2}},\tilde{h}_{k_{3}} ) \Vert_{Z_{k}}
\\
& \lesssim 2^{k_{1} \sigma } \Vert f_{k_{1}} \Vert_{Z_{k_{1}}} 2^{k_{2} \frac{n-2}{2}} \Vert g_{k_{2}} \Vert_{Z_{k_{2}}} 2^{k_{3} \frac{n-2}{2}} \Vert h_{k_{3}} \Vert_{Z_{k_{3}}}.
\end{split}
\end{equation*}
In view of this estimate, for the remainder of the proof we assume $k \geq k_{2} + C_{0}$. This means that $ | \xi| \gg | \eta|$ and we can use Lemma \ref{Taylorcalc} to write
\begin{equation*}
\begin{split}
& \xi_{j}[ \psi( \xi - \eta) - \psi( \xi)]_{i}
\\
& = \sum_{|\alpha| =1}^{\infty} \frac{ m_{\alpha} ( \xi) | \xi|^{2s-1 - |\alpha|} \xi_{j}}{\alpha !} \eta^{\alpha}
\end{split}
\end{equation*}
Then define $ H^{6}_{\alpha} (f,g,h)$ in the frequency space as follows
\begin{equation*}
\begin{split}
& \mathcal{F}_{\R^{n}} ( H^{6}_{\alpha}( f,g,h) )(\xi,t) 
\\
& = \frac{1}{\alpha !} \int_{\R^{n} \times \R^{n}} m_{\alpha} (\xi) | \xi |^{2s-1- | \alpha|} \xi_{j}  \hat{ {f}}( \xi- \kappa,t) 
\\
& \times \eta^{\alpha} \frac{ \langle \kappa - \eta, \kappa \rangle }{ | \kappa |^{2}} \frac{ \hat{ g} ( \eta,t )}{ | \eta |} \frac{ \hat{h}(\kappa - \eta,t)}{ | \kappa - \eta |} d \eta d \kappa
\end{split}
\end{equation*}
Then we write 
\begin{equation*}
\begin{split}
T^{i,j}_{4} ( \tilde{f}_{k_{1}},\tilde{g}_{k_{2}} ,\tilde{h}_{k_{3}}) = \sum_{| \alpha | =1}^{\infty} H^{6}_{\alpha}( \tilde{f}_{k_{1}}, \tilde{g}_{k_{2}} ,\tilde{h}_{k_{3}} ) 
\end{split}
\end{equation*}
And to conclude this case, it suffices to prove 
\begin{equation}\label{h6alphaestimate}
\begin{split}
&  \sum_{k_{2} + C_{0} \leq k \leq k_{1} +C} 2^{k \sigma } \sum_{| \alpha|=1}^{\infty} \Vert \Delta_{k} ( \i \partial_{t} - ( - \Delta)^{s} + \i)^{-1} \ind_{[-T',T']} H^{6}_{\alpha} ( \tilde{f}_{k_{1}} , \tilde{g}_{k_{2}} , \tilde{h}_{k_{3}} ) \Vert_{Z_{k}} 
\\
& \lesssim 2^{k_{1} \sigma } \Vert f_{k_{1}} \Vert_{Z_{k_{1}}} 2^{k_{2} \frac{n-2}{2}} \Vert g_{k_{2}} \Vert_{Z_{k_{2}}} 2^{k_{3} \frac{n-2}{2}} \Vert h_{k_{3}} \Vert_{Z_{k_{3}}}.
\end{split}
\end{equation}
To prove \eqref{h6alphaestimate} set 
$$
a_{\alpha} : = \Vert \mathcal{F}^{-1}_{\R^{n}} \left( \varphi_{[k-C,k+C]}  m_{\alpha}(\xi) \right) \Vert_{L^{1}(\R^{n})}
$$
Next, fix any multi-index $ \alpha$, then we consider two sub-cases. Assume first that $ 2^{ \min \{ k_{2} ,k_{3} \}} \ll 2^{ \max \{ k_{2} , k_{3} \} }$, then notice that this assumption implies that $ \kappa$ is localized at $ 2^{ \max \{ k_{2} , k_{3} \}}$. Then we use \eqref{multimultiplier} to obtain 
\begin{equation*}
\begin{split}
& \Vert \Delta_{k} ( \i \partial_{t} - ( -\Delta)^{s} + \i)^{-1}\ind_{[-T',T']} H^{6}_{\alpha} ( \tilde{f}_{k_{1}}, \tilde{g}_{k_{2}}, \tilde{h}_{k_{3}} ) \Vert_{Z_{k}}
\\
& \lesssim 2^{-k(2s-1)/2} \max_{e \in \mathscr{E}} \Vert H^{6}_{\alpha}( \tilde{f}_{k_{1}},\tilde{g}_{k_{2}}, \tilde{h}_{k_{3}} ) \Vert_{L^{1}_{e} L^{2}_{e^{\perp},t}}
\\
& \lesssim 2^{-k (2s-1)/2}  \Vert \tilde{f}_{k_{1}} \Vert_{L^{\infty}_{e'} L^{2}_{e'^{\perp},t}} \Vert \tilde{g}_{k_{2}} \Vert_{L^{2}_{e'}L^{\infty}_{e'^{\perp},t}} \max_{e \in \mathscr{E}} \Vert \tilde{h}_{k_{3}} \Vert_{L^{2}_{e} L^{\infty}_{e^{\perp},t}} 2^{- \max \{ k_{2} ,k_{3} \}} 
\\
& \times \frac{ a_{\alpha} 2^{-C_{0} | \alpha|}}{\alpha !} 2^{k(2s-1)}
\\
& \lesssim   \frac{ a_{\alpha} 2^{-C_{0} | \alpha|}}{\alpha !}    2^{ (k -k_{1})(2s-1)/2} \Vert f_{k_{1}} \Vert_{Z_{k_{1}}} 2^{k_{2} \frac{n-2}{2}} \Vert g_{k_{2}} \Vert_{Z_{k_{2}}} 2^{k_{3} \frac{n-2}{2}} \Vert h_{k_{3}} \Vert_{Z_{k_{3}}}
\end{split}
\end{equation*}
Summing over $ \alpha$ and using the convergence stated in Lemma \ref{Taylorcalc} we obtain 
\begin{equation*}
\begin{split}
& \sum_{|\alpha|=1}^{\infty} \Vert \Delta_{k} ( \i \partial_{t} - ( -\Delta)^{s} + \i)^{-1} H^{6}_{\alpha} ( \tilde{f}_{k_{1}}, \tilde{g}_{k_{2}}, \tilde{h}_{k_{3}} ) \Vert_{Z_{k}}
\\
& \lesssim  2^{ (k -k_{1})(2s-1)/2} \Vert f_{k_{1}} \Vert_{Z_{k_{1}}} 2^{k_{2} \frac{n-2}{2}} \Vert g_{k_{2}} \Vert_{Z_{k_{2}}} 2^{k_{3} \frac{n-2}{2}} \Vert h_{k_{3}} \Vert_{Z_{k_{3}}}
\end{split}
\end{equation*}
This holds for every $k \in [ k_{2} + C_{0}, k_{1} +C ]$. Now multiplying $ 2^{k \sigma}$ and summing over $k $ yields \eqref{h6alphaestimate} under the assumption $ 2^{\min \{ k_{2} , k_{3} \}} \ll 2^{ \max \{ k_{2} ,k_{3} \} }$. Next, assume $ 2^{k_{2}} \approx 2^{k_{3}}$.  Then we cannot apply \eqref{multimultiplier} since $ \frac{1}{|\kappa|}$ might be singular. To overcome this, define $H^{6}_{\alpha,r}$ via the Fourier transform 
\begin{equation}\label{h6alhadef}
\begin{split}
& \mathcal{F}_{\R^{n}} ( H^{6}_{\alpha,r} ( \tilde{f}_{k_{1}},\tilde{g}_{k_{2}} ,\tilde{h}_{k_{3}}) ) (\xi,t)
\\
& := \frac{1}{\alpha !} \int_{\R^{n} \times \R^{n}} m_{\alpha} (\xi)\xi_{j} | \xi |^{2s-1- | \alpha|} \hat{ {f}}( \xi- \kappa,t) 
\\
& \times \eta^{\alpha} \frac{ \langle \kappa - \eta, \kappa \rangle }{ | \kappa |^{2}} \frac{ \hat{ g} ( \eta,t )}{ | \eta |} \frac{ \hat{h}(\kappa - \eta,t)}{ | \kappa - \eta |} \varphi_{r} ( \kappa) d \eta d \kappa
\end{split}
\end{equation}
Then to prove \eqref{h6alphaestimate} it suffices to prove 
\begin{equation}\label{case1clastone}
\begin{split}
&  \sum_{|\alpha|=1}^{\infty} \sum_{r \leq k_{2} +C} \Vert \Delta_{k} ( \i \partial_{t} - ( - \Delta)^{s} + \i)^{-1} \ind_{[-T',T']} H^{6}_{\alpha,r} ( \tilde{f}_{k_{1}}, \tilde{g}_{k_{2}} ,\tilde{h}_{k_{3}}) \Vert_{Z_{k}}
\\
& \lesssim  \Vert f_{k_{1}} \Vert_{Z_{k_{1}}} 2^{k_{2} \frac{n-2}{2}} \Vert g_{k_{2}} \Vert_{Z_{k_{2}}} 2^{k_{3} \frac{n-2}{2}} \Vert h_{k_{3}} \Vert_{Z_{k_{3}}}.
\end{split}
\end{equation}
To prove \eqref{case1clastone} we use the localized maximal estimate and \eqref{multimultiplier}
\begin{equation*}
\begin{split}
& \Vert \Delta_{k} ( \i \partial_{t} - ( - \Delta)^{s} + \i)^{-1} \ind_{[-T',T']} H^{6}_{\alpha,r} ( \tilde{f}_{k_{1}} , \tilde{g}_{k_{2}} , \tilde{h}_{k_{3}} ) \Vert_{Z_{k}}
\\
& \lesssim 2^{-k_{1} (2s-1)/2}  \Vert H^{6}_{\alpha ,r} ( \tilde{f}_{k_{1}},\tilde{g}_{k_{2}} ,\tilde{h}_{k_{3}} ) \Vert_{L^{1}_{e'} L^{2}_{e'^{\perp},t}}
\\
& \overset{\eqref{multimultiplier}}{\lesssim} \frac{ a_{\alpha} 2^{- C_{0} | \alpha|}}{ \alpha !}  2^{k_{1} (2s-1)/2} 2^{-r} \Vert \tilde{f}_{k_{1}} \Vert_{L^{\infty}_{e'} L^{2}_{e'^{\perp},t}}
\\
& \times \sum_{\tn_{1} ,\tn_{2}} \Vert P_{r,\tn_{1}} \tilde{g}_{k_{2}} \Vert_{L^{2}_{e'} L^{\infty}_{e'^{\perp},t}}  \Vert P_{r,\tn_{2}} \tilde{h}_{k_{3}} \Vert_{L^{2}_{e'} L^{\infty}_{e'^{\perp},t}}
\\
& \overset{\eqref{maximalestcardinality}}{\lesssim}  \frac{ a_{\alpha} 2^{- C_{0} | \alpha|}}{ \alpha !}  2^{k_{1} (2s-1)/2} 2^{-r} \Vert \tilde{f}_{k_{1}} \Vert_{L^{\infty}_{e'} L^{2}_{e'^{\perp},t}}
\\
& \times \left( \sum_{\tn \in 2^{r} \Z^{n}} \Vert  P_{r,\tn} \tilde{g}_{k_{2}} \Vert^{2}_{L^{2}_{e'} L^{\infty}_{e'^{\perp},t}}  \right)^{1/2}  \left( \sum_{\tn \in 2^{r} \Z^{n}} \Vert  P_{r,\tn} \tilde{h}_{k_{3}} \Vert^{2}_{L^{2}_{e'} L^{\infty}_{e'^{\perp},t}}  \right)^{1/2} 
\\
& \lesssim \frac{ a_{\alpha} 2^{-C_{0} | \alpha|}}{ \alpha !} 2^{-r} 2^{r(n-1)} 2^{k_{2}-r} (1 + | r-k_{2}|)^{2} \Vert f_{k_{1}} \Vert
_{Z_{k_{1}}} \Vert g_{k_{2}} \Vert_{Z_{k_{2}}} \Vert h_{k_{3}} \Vert_{Z_{k_{3}}}
\end{split}
\end{equation*}
Summing over $ r \leq k_{2} +C$ and using that $ 2^{k_{2}} \approx 2^{k_{3}}$ we obtain
\begin{equation*}
\begin{split}
& \sum_{r \leq k_{2} +C} \Vert \Delta_{k} ( \i \partial_{t} - ( - \Delta)^{s} + \i)^{-1} \ind_{[-T',T']}  H^{6}_{\alpha,r} ( \tilde{f}_{k_{1}} , \tilde{g}_{k_{2}} , \tilde{h}_{k_{3}} ) \Vert_{Z_{k}}
\\
& \lesssim \frac{ a_{\alpha} 2^{- C_{0} | \alpha|}}{ \alpha !} \Vert f_{k_{1}} \Vert_{Z_{k_{1}}} 2^{k_{2} \frac{n-2}{2}} \Vert g_{k_{2}} \Vert_{Z_{k_{2}}} 2^{k_{3} \frac{n-2}{2}} \Vert h_{k_{3}} \Vert_{Z_{k_{3}}}. 
\end{split}
\end{equation*}
Summing over $ \alpha$ and using the convergence stated in Lemma \ref{Taylorcalc} we obtain \eqref{case1clastone}. This concludes Case 1c and hence Case 1. 
\\
\textbf{Case 2} If $ \max \{ k_{1},k_{2} ,k_{3} \} = k_{2}$. Then in view of the analysis above, the only two sub cases we need to consider are when $ 2^{k_{1}} \ll 2^{k_{2}}$ while $ 2^{k_{3}} \approx 2^{k_{2}}$ and the case in which $ 2^{ \max \{ k_{1} , k_{3} \}} \ll 2^{k_{2}}$.
\\
\textbf{Case 2a:} If $ 2^{k_{1}} \ll 2^{k_{2}}$ and $ 2^{k_{2}} \approx 2^{k_{3}}$. In this case $ \frac{1}{| \kappa|}$ might be singular. Therefore, we need to use the localized maximal estimate. To that end, we write 
$$
T^{i,j}_{4}( \tilde{f}_{k_{1}} , \tilde{g}_{k_{2}} ,\tilde{h}_{k_{3}}) = \sum_{r = - \infty}^{k_{2} +C} H^{5}_{r} ( \tilde{f}_{k_{1}} ,\tilde{g}_{k_{2}} , \tilde{h}_{k_{3}} ).
$$
Where $H^{5}_{r}$ was defined in \eqref{h5rdef}.
In view of the definitions, it suffices to prove 
\begin{equation}\label{qsgha1}
\begin{split}
& \sum_{k \leq k_{2} +C} 2^{k \sigma} \sum_{r \leq k_{2} +C} \Vert \Delta_{k} ( \i \partial_{t} - ( - \Delta)^{s} + \i)^{-1} \ind_{[-T',T']} H^5_{r} ( \tilde{f}_{k_{1}},\tilde{g}_{k_{2}},\tilde{h}_{k_{3}}) \Vert_{Z_{k}}
\\
& \lesssim 2^{k_{1} \frac{n-2}{2}} \Vert f_{k_{1}} \Vert_{Z_{k_{1}}} 2^{k_{2} \sigma} \Vert g_{k_{2}} \Vert_{Z_{k_{2}}} 2^{k_{3} \frac{n-2}{2}} \Vert h_{k_{3}} \Vert_{Z_{k_{3}}}
\\
& + 2^{k_{1} \sigma} \Vert f_{k_{1}} \Vert_{Z_{k_{1}}} 2^{k_{2} \frac{n-2}{2}} \Vert g_{k_{2}} \Vert_{Z_{k_{2}}} 2^{k_{3} \frac{n-2}{2}} \Vert h_{k_{3}} \Vert_{Z_{k_{3}}}.
\end{split}
\end{equation}
 Fix $ r \in [k_{1} -C ,k_{2} +C]$ then $ k \leq r +C$ (otherwise the summand on the left hand side of \eqref{qsgha1} is $0$) and hence 
\begin{equation*}
\begin{split}
& 2^{k \sigma } \Vert \Delta_{k} ( \i \partial_{t} - ( - \Delta)^{s} + \i)^{-1} \ind_{[-T',T']} H^{5}_{r} ( \tilde{f}_{k_{1}} , \tilde{g}_{k_{2}} , \tilde{h}_{k_{3}} ) \Vert_{Z_{k}}
\\
& \lesssim 2^{k \sigma} 2^{-k (2s-1)/2} \max_{e \in \mathscr{E}} \Vert H^{5}_{r} ( \tilde{f}_{k_{1}} , \tilde{g}_{k_{2}} , \tilde{h}_{k_{3}} ) \Vert_{L^{1}_{e} L^{2}_{e^{\perp},t}}
\\
& \overset{\eqref{multimultiplier}}{\lesssim} 2^{ k \sigma} 2^{-k(2s-1)/2} 2^{r} 2^{k_{2}(2s-1)} 2^{-r}  2^{-k_{2}} \Vert \tilde{f}_{k_{1}} \Vert_{L^{2}_{e'} L^{\infty}_{e'^{\perp},t}} \Vert \tilde{g}_{k_{2}} \Vert_{L^{\infty}_{e'} L^{2}_{e'^{\perp},t}} \max_{e \in \mathscr{E}} \Vert \tilde{h}_{k_{3}} \Vert_{L^{2}_{e} L^{\infty}_{e^{\perp},t}}
\\
& \overset{\eqref{local_smoothing},\eqref{maximal_Estimate}}{\lesssim} 2^{k \sigma } 2^{-k(2s-1)/2} 2^{k_{2} (2s-1)/2} 2^{k_{1} \frac{n-2}{2}} \Vert f_{k_{1}} \Vert_{Z_{k_{1}}} 2^{k_{3} \frac{n-2}{2}} \Vert h_{k_{3}} \Vert_{Z_{k_{3}}} \Vert g_{k_{2}} \Vert_{Z_{k_{2}}}
\end{split}
\end{equation*}
Summing over $k \leq r +C$ we obtain 
\begin{equation*}
\begin{split}
& \sum_{k \leq r +C} 2^{k \sigma} \Vert \Delta_{k} ( \i \partial_{t} - ( - \Delta)^{s} + \i)^{-1} \ind_{[-T',T']} H^{5}_{r} ( \tilde{f}_{k_{1}} , \tilde{g}_{k_{2}} , \tilde{h}_{k_{3}} ) \Vert_{Z_{k}}
\\
& \lesssim 2^{r \sigma } 2^{-r (2s-1)/2} 2^{k_{2} (2s-1)/2}   
2^{k_{1} \frac{n-2}{2}} \Vert f_{k_{1}} \Vert_{Z_{k_{1}}} 2^{k_{3} \frac{n-2}{2}} \Vert h_{k_{3}} \Vert_{Z_{k_{3}}} \Vert g_{k_{2}} \Vert_{Z_{k_{2}}}
\end{split}
\end{equation*}
summing over $ r \leq k_{2} +C$ yields 
\begin{equation}\label{poasa}
\begin{split}
& \sum_{k \leq k_{2} +C} 2^{k \sigma} \sum_{k_{1} -C \leq r \leq k_{2} +C} \Vert \Delta_{k} ( \i \partial_{t} - ( - \Delta)^{s} + \i)^{-1} \ind_{[-T',T']} H^{5}_{r} ( \tilde{f}_{k_{1}} ,\tilde{g}_{k_{2}},\tilde{h}_{k_{3}}) \Vert_{Z_{k}}
\\
& \lesssim 2^{k_{1} \frac{n-2}{2}} \Vert f_{k_{1}} \Vert_{Z_{k_{1}}} 2^{k_{2} \sigma} \Vert g_{k_{2}} \Vert_{Z_{k_{2}} } 2^{k_{3} \frac{n-2}{2}} \Vert h_{k_{3}} \Vert_{Z_{k_{3}}} 
\end{split}
\end{equation}
Next, fix $ r \leq k_{1} -C$ for sufficiently large constant $ C$. Notice that in this case, the sum in $k$ runs only for when $ 2^{k} \approx 2^{k_{1}}$. Further, in view of \eqref{multilpierlemma} we may assume $ \tilde{f}_{k_{1}}$ has Fourier transform supported in $ \{ \langle \xi, e\rangle \geq 2^{k_{1} -c} \}$. Then we estimate 
\begin{equation*}
\begin{split}
& \Vert \Delta_{k} ( \i \partial_{t} - ( -\Delta)^{s} + \i)^{-1} \ind_{[-T',T']} H^{5}_{r} ( \tilde{f}_{k_{1}}, \tilde{g}_{k_{2}} , \tilde{h}_{k_{3}} ) \Vert_{Z_{k}}
\\
& \lesssim 2^{-k_{1} (2s-1)/2} \Vert H^{5}_{r} ( \tilde{f}_{k_{1}}, \tilde{g}_{k_{2}} , \tilde{h}_{k_{3}} ) \Vert_{L^{1}_{e} L^{2}_{e^{\perp},t}}
\\
& \lesssim 2^{-k_{1}(2s-1)/2} \sum_{\tn_{1} ,\tn_{2} } \Vert H^{5}_{r}( \tilde{f}_{k_{1}} , P_{r,\tn_{1}} \tilde{g}_{k_{2}}, P_{r,\tn_{2}} \tilde{h}_{k_{3}} ) \Vert_{L^{1}_{e} L^{2}_{e^{\perp},t}}
\\
& \overset{\eqref{multimultiplier}}{\lesssim} 2^{-k_{1} (2s-1)/2} 2^{k_{2} (2s-1)} 2^{k_{1}} 2^{-r} 2^{-k_{2}} \Vert \tilde{f}_{k_{1}} \Vert_{L^{\infty}_{e} L^{2}_{e^{\perp},t}} 
\\
& \times \sum_{\tn_{1} , \tn_{2}} \Vert P_{r,\tn_{1}} \tilde{g}_{k_{2}} \Vert_{L^{2}_{e} L^{\infty}_{{e}^{\perp},t}} \Vert P_{r,\tn_{2}} \tilde{h}_{k_{3}} \Vert_{L^{2}_{e} L^{\infty}_{e^{\perp},t}}
\\
& \overset{\eqref{maximalestcardinality}}{\lesssim} 2^{-k_{1} (2s-1)/2}2^{k_{2}(2s-1)}  2^{-r}  \Vert \tilde{f}_{k_{1}} \Vert_{L^{\infty}_{e} L^{2}_{e^{\perp},t}} 
\\
& \times \left( \sum_{  \tn \in 2^{r} \Z^{n}} \Vert P_{r,\tn} \tilde{g}_{k_{2}} \Vert^{2}_{L^{2}_{e} L^{\infty}_{{e}^{\perp},t}} \right)^{1/2}
\left( \sum_{  \tn \in 2^{r} \Z^{n}} \Vert P_{r,\tn} \tilde{h}_{k_{3}} \Vert^{2}_{L^{2}_{e} L^{\infty}_{{e}^{\perp},t}} \right)^{1/2}
\\
& \overset{\eqref{local_smoothing},\eqref{eq:max:Xkest:loc}}{\lesssim} ( 1+ | k_{2} -r|)^{2} 
 2^{(k_{2} -k_{1})(2s-1)} 2^{(k_{1} -r)} 2^{r(n-2)} \Vert f_{k_{1}} \Vert_{Z_{k_{1}}} \Vert g_{k_{2}} \Vert_{Z_{k_{2}}} \Vert h_{k_{3}} \Vert_{Z_{k_{3}}}
\end{split}
\end{equation*}
Summing over $ r \leq k_{1} -C$ and then over $ k$ we obtain ( since $ n -1 -2s >0 )$
\begin{equation}\label{poasa2}
\begin{split}
& \sum_{k \leq k_{2} +C} 2^{k \sigma } \sum_{ r \leq k_{1} -C} \Vert \Delta_{k} ( \i \partial_{t} - ( - \Delta)^{s} + \i)^{-1} \ind_{[-T',T']} H^{5}_{r} ( \tilde{f}_{k_{1}} ,\tilde{g}_{k_{2}},\tilde{h}_{k_{3}}) \Vert_{Z_{k}}
\\
& \lesssim 2^{k_{1} \sigma } \Vert f_{k_{1}} \Vert_{Z_{k_{1}}} 2^{k_{2} \frac{n-2}{2}} \Vert g_{k_{2}} \Vert_{Z_{k_{2}} } 2^{k_{3} \frac{n-2}{2}} \Vert h_{k_{3}} \Vert_{Z_{k_{3}}}.
\end{split}
\end{equation}
Combining \eqref{poasa} and \eqref{poasa2} we arrive at \eqref{qsgha1}. Which concludes the proof of this case.
\\
\textbf{Case 2b:} If $ 2^{ \max\{ k_{1},k_{3} \}} \ll 2^{k_{2}} $. In this case notice that $ \kappa$ is localized at $ 2^{k_{2}} $ and so we can apply \eqref{multimultiplier} directly to obtain 
\begin{equation*}
\begin{split}
& \Vert \Delta_{k} ( \i \partial_{t} - ( -\Delta)^{s} + \i)^{-1} \ind_{[-T',T']} T^{i,j}_{4}( \tilde{f}_{k_{1}}, \tilde{g}_{k_{2}}, \tilde{h}_{k_{3}} \Vert_{Z_{k}}
\\
& \lesssim 2^{-k_{2}(2s-1)/2} \Vert T^{i,j}_{4} ( \tilde{f}_{k_{1}}, \tilde{g}_{k_{2}}, \tilde{h}_{k_{3}}) \Vert_{L^{1}_{e'} L^{2}_{e'^{\perp},t}}
\\
& \overset{\eqref{multimultiplier}}{\lesssim} 2^{k_{2} (2s-1)/2} 2^{-k_{2}} \Vert \tilde{g}_{k_{2}} \Vert_{L^{\infty}_{e'} L^{2}_{e'^{\perp},t}} \Vert \tilde{f}_{k_{1}} \Vert_{L^{2}_{e'}{L^{\infty}_{e'^{\perp},t}}} \Vert \tilde{h}_{k_{3}} \Vert_{L^{2}_{e'}{L^{\infty}_{e'^{\perp},t}}} 
\\
& \overset{\eqref{local_smoothing},\eqref{maximal_Estimate}}{\lesssim} \Vert g_{k_{2}} \Vert_{Z_{k_{2}}} 2^{k_{1} \frac{n-2}{2}} \Vert f_{k_{1}} \Vert_{Z_{k_{1}}} 2^{k_{3} \frac{n-2}{2}} \Vert h_{k_{3}} \Vert_{Z_{k_{3}}}.
\end{split}
\end{equation*}
\\
\textbf{Case 3:} If $ \max \{ k_{1} ,k_{2} ,k_{3} \} =k_{3}$. In view of the cases above, we may assume $ 2^{\max \{ k_{1} , k_{2} \} } \ll 2^{k_{3}}$. Further, in view of \eqref{multilpierlemma} we may assume $ h_{k_{3}}$ has Fourier transform supported in $ \{  \langle \xi , e \rangle \geq 2^{k_{3} -2} \}$ for some $ e \in \mathscr{E}$. First, consider the sub-case in which $ 2^{k_{1}} \lesssim 2^{k_{2}}$ then we apply \eqref{multimultiplier} together with the local smoothing and maximal estimate to obtain 
\begin{equation*}
\begin{split}
& \Vert \Delta_{k} ( \i \partial_{t} - ( - \Delta)^{s} +\i)^{-1} \ind_{[-T',T']} T^{i,j}_{4} ( \tilde{f}_{k_{1}} , \tilde{g}_{k_{2}} , \tilde{h}_{k_{3}} ) \Vert_{Z_{k}}
\\
& \lesssim 2^{-k_{3} (2s-1)/2} \Vert T^{i,j}_{4} ( \tilde{f}_{k_{1}} ,\tilde{g}_{k_{2}} , \tilde{h}_{k_{3}} ) \Vert_{Z_{k}}
\\
& \lesssim 2^{k_{3} (2s-1)/2} 2^{-k_{2}} \Vert \tilde{h}_{k_{3}} \Vert_{L^{\infty}_{e} L^{2}_{e^{\perp},t}} \Vert \tilde{f}_{k_{1}} \Vert_{L^{2}_{e} L^{\infty}_{e^{\perp},t}} \Vert \tilde{g}_{k_{2}} \Vert_{L^{2}_{e} L^{\infty}_{e^{\perp},t}}
\\
& \lesssim \Vert h_{k_{3}} \Vert_{Z_{k_{3}}} 2^{k_{2} \frac{n-2}{2}} \Vert g_{k_{2}} \Vert_{Z_{k_{2}}} 2^{k_{1} \frac{n-2}{2}} \Vert f_{k_{1}} \Vert_{Z_{k_{1}}}
\end{split}
\end{equation*}
Multiplying by $ 2^{k \sigma} $ and summing over $ k \leq k_{3} +C$ yields the result under the assumption $ 2^{k_{1}} \lesssim 2^{k_{2}}$. Assume next that $ 2^{k_{2}} \ll 2^{k_{1}}$. Then we use Lemma \ref{Taylorcalc} to write 
\begin{equation*}
\begin{split}
& \xi_{j}[ \psi( \xi - \eta) - \psi( \xi)]_{i}
\\
& = \sum_{|\alpha| =1}^{\infty} \frac{ m_{\alpha} ( \xi) | \xi|^{2s-1 - |\alpha|} \xi_{j}}{\alpha !} \eta^{\alpha}
\end{split}
\end{equation*}
Then we use the operator $H^{6}_{\alpha}$ defined in \eqref{h6alhadef} to reduce matters into proving 
\begin{equation*}
\begin{split}
& \sum_{ | \alpha|=1}^{\infty} \Vert \Delta
_{k} ( \i \partial_{t} - ( - \Delta)^{s} + i)^{-1} \ind_{[-T',T']} H^{6}_{\alpha}(\tilde{f}_{k_{1}}, \tilde{g}_{k_{2}} ,\tilde{h}_{k_{3}}) \Vert_{Z_{k}}
\\
& \lesssim \Vert h_{k_{3}} \Vert_{Z_{k_{3}}} 2^{k_{2} \frac{n-2}{2}} \Vert g_{k_{2}} \Vert_{Z_{k_{2}}} 2^{k_{1} \frac{n-2}{2}} \Vert f_{k_{1}} \Vert_{Z_{k_{1}}}.
\end{split}
\end{equation*}
To that end, fix $ \alpha $ and define 
$$
a_{\alpha}: = \Vert \mathcal{F}^{-1} \left( m_{\alpha}(\xi) \varphi_{[k_{3}-c,k_{3}+c]} (\xi
) \right) \Vert_{L^{1}}.
$$
Then we estimate using \eqref{multimultiplier}
\begin{equation*}
\begin{split}
& \Vert \Delta_{k} ( \i \partial_{t} - ( -\Delta)^{s} +\i)^{-1} \ind_{[-T',T']} H^{6}_{\alpha}( \tilde{f}_{k_{1}}, \tilde{g}_{k_{2}}, \tilde{h}_{k_{3}} ) \Vert_{Z_{k}}
\\
& \lesssim 2^{-k_{3} (2s-1)/2} \Vert H^{6}_{\alpha}( \tilde{f}_{k_{1}}, \tilde{g}_{k_{2}} , \tilde{h}_{k_{3}} ) \Vert_{L^{1}_{e} L^{2}_{e^{\perp},t}}
\\
& \lesssim \frac{ a_{\alpha} 2^{- C | \alpha|}}{ \alpha !} 2^{k_{1}(2s-1)/2} 2^{-k_{1}} \Vert \tilde{f}_{k_{1}} \Vert_{L^{2}_{e} L^{\infty}_{e^{\perp},t}} \Vert \tilde{g}_{k_{2}} \Vert_{L^{2}_{e} L^{\infty}_{e^{\perp},t}} \Vert h_{k_{3}} \Vert_{L^{\infty}_{e} L^{2}_{e^{\perp},t}}
\\
& \lesssim \frac{ a_{\alpha} 2^{- C | \alpha|}}{ \alpha !}  2^{k_{1} \frac{n-2}{2}} \Vert f_{k_{1}} \Vert_{Z_{k_{1}}} 2^{k_{2} \frac{n-2}{2}} \Vert g_{k_{2}} \Vert_{Z_{k_{2}}} \Vert h_{k_{3}} \Vert_{Z_{k_{3}}}
\end{split}
\end{equation*}
summing over $ \alpha$ and using the convergence stated in Lemma \ref{Taylorcalc} we obtain
\begin{equation*}
\begin{split}
& \sum_{ | \alpha|=1}^{\infty} \Vert \Delta
_{k} ( \i \partial_{t} - ( - \Delta)^{s} + i)^{-1}\ind_{[-T',T']} H^{6}_{\alpha}(\tilde{f}_{k_{1}}, \tilde{g}_{k_{2}} ,\tilde{h}_{k_{3}}) \Vert_{Z_{k}}
\\
& \lesssim \Vert h_{k_{3}} \Vert_{Z_{k_{3}}} 2^{k_{2} \frac{n-2}{2}} \Vert g_{k_{2}} \Vert_{Z_{k_{2}}} 2^{k_{1} \frac{n-2}{2}} \Vert f_{k_{1}} \Vert_{Z_{k_{1}}}.
\end{split}
\end{equation*}
This concludes case 3 and hence proves the proposition. 
\end{proof}

\section{Nonlinear estimates III}
In this section we prove the nonlinear estimates regarding the time derivatives. 
Namely, we deal with the term
\begin{equation}\label{timederivative}
\begin{split}
& \i V_{i} \langle \partial_{t} e_{2} , e_{1} \rangle.
\end{split}
\end{equation}

\begin{proposition}\label{timederivativenonlinearity}
Let $ n \geq 4$, $ s \in (1/2,1), \sigma \geq \frac{n-2}{2}$, and $T' \in [0,T]$ then for any $i \in \{ 1,...,n \} $ we have the following ( Here denote $ 
E : = \Vert |\nabla| e_{1} \Vert_{\tilde{F}^{\frac{n-2}{2}}(T')} +\Vert |\nabla| e_{2} \Vert_{\tilde{F}^{\frac{n-2}{2}}(T')} ) $
\begin{equation*}
\begin{split}
& \Vert \ind_{[-T',T']} V_{i} \langle \partial_{t} e_{2},e_{1} \rangle \Vert_{N^{\sigma}} 
\\
& \lesssim  \Vert V_{x} \Vert_{\tilde{F}^{\sigma}(T')} \Vert V_{x} \Vert_{\tilde{F}^{\frac{n-2}{2}}(T')}   \left( \Vert V_{x}\Vert_{\tilde{F}^{\frac{n-2}{2}}(T')} +E^{2} + E  \right)
\\
& + \Vert \vert \nabla \vert e_{1}\Vert_{\tilde{F}^{\sigma}(T')} \Vert V_{x} \Vert^{2}_{\tilde{F}^{\frac{n-2}{2}}(T')} \left(  1 +  \Vert V_{x} \Vert_{\tilde{F}^{\frac{n-2}{2}}(T')}  +   E \right) 
\\
& + \Vert \vert \nabla \vert e_{2}\Vert_{\tilde{F}^{\sigma}(T')}  \Vert V_{x} \Vert^{2}_{\tilde{F}^{\frac{n-2}{2}}(T')} \left(  1 + \Vert V_{x} \Vert_{\tilde{F}^{\frac{n-2}{2}}(T')}+   E \right).
\end{split}
\end{equation*}
\end{proposition}
The main idea is to rewrite $ \langle \partial_{t} e_{2} , e_{1} \rangle$ in terms of $ V_{i}'$s. To that end, we first prove the following lemma 
\begin{lemma}\label{ceoffidentities}
Let $ A_{i} : = \langle \partial_{i} e_{2} ,e_{1} \rangle $ for $i \in \{1,..,n
\}$ and let $ A_{0} : = \langle \partial_{t} e_{2} ,e_{1} \rangle$. Then we have the following identity for any $ i \in \{0,...,n \}$
$$
A_{i} = c \sum_{j =1}^{n} |\nabla|^{-1} R_{j} ( \Re( \i \bar{V}_{j} {V}_{i} )).
$$
\end{lemma}
\begin{proof}
By definition of $ V_{i}$ and the fact that $ u, e_{1},e_{2}$ form an orthonormal frame we obtain 
\begin{equation}\label{id1}
\partial_{i} u = \Re(V_{i}) e_{1} + \Im(V_{i}) e_{2}.
\end{equation}
Next, notice that 
\begin{equation}\label{id2}
\langle \partial_{i} e_{1} , u \rangle = - \langle e_{1}, \partial_{i} u  \rangle = -\Re(V_{i})
\end{equation}
moreover, 
\begin{equation}\label{id3}
 \langle \partial_{i} e_{1}, e_{2} \rangle = - \langle e_{1} , \partial_{i} e_{2} \rangle = - A_{i}.
\end{equation}
And since $ | e_{1} | = 1$ then combining \eqref{id2} and \eqref{id3} we obtain
\begin{equation}\label{id4}
\partial_{i} e_{1} = - \Re(V_{i}) u - A_{i} e_{2}.
\end{equation}
Similarly, 
\begin{equation}\label{id4e}
\partial_{i} e_{2} = -\Im(V_{i}) u + A_{i} e_{1}
\end{equation}
Now by definition of $ A_{i}$ we have 
$$
\partial_{j} A_{i} = \langle \partial_{j} e_{1} , \partial_{i} e_{2}\rangle + \langle e_{1} , \partial_{j} \partial_{i} e_{2} \rangle
$$
while
$$
\partial_{i} A_{j} = \langle \partial_{i} e_{1}, \partial_{j}e_{2} \rangle + \langle e_{1}, \partial_{j} \partial_{i} e_{2} \rangle.
$$
This implies 
\begin{equation}\label{id5}
\partial_{j} A_{i}- \partial_{i} A_{j}=\langle \partial_{j} e_{1} , \partial_{i} e_{2}\rangle-\langle \partial_{i} e_{1}, \partial_{j}e_{2} \rangle
\end{equation}
using \eqref{id4} and \eqref{id4e} we obtain
\begin{equation}
\begin{split}
& \langle \partial_{j} e_{1} , \partial_{i} e_{2} \rangle
\\
& = \langle -\Re(V_{j}) u - A_{j} e_{2} , -\Im(V_{i}) u + A_{i} e_{1} \rangle
\\
& = \Re(V_{j}) \Im(V_{i}).
\end{split} 
\end{equation}
We plug this back into \eqref{id5} to obtain
\begin{equation}
\partial_{j} A_{i} - \partial_{i} A_{j} = \Re(V_{j}) \Im(V_{i}) - \Re(V_{i}) \Im(V_{j}).
\end{equation}
Then  
\begin{equation}\label{keyidee}
\partial_{j} A_{i} - \partial_{i} A_{j}= -\Im( V_{j} \bar{V}_{i} ).
\end{equation}
Notice that the Coulomb condition is equivalent to 
$$
\sum_{j=1}^{n} \partial_{j} A_{j} =0. 
$$
Therefore, fix $i \in \{0,1,..,n \}$ and then taking the divergence of both sides in \eqref{keyidee} we obtain
\begin{equation}
\sum_{j=1} \partial^{2}_{j} A_{i} - \partial_{i} \sum_{j=1}^{n} \partial_{j} A_{j} =-\sum_{j=1}^{n} \partial_{j} \Im(V_{j} \bar{V}_{i} )
\end{equation}
which translates to 
\begin{equation}
\Delta A_{i} =- \sum_{j=1}^{n} \partial_{j} \Im(V_{j} \bar{V}_{i} )
\end{equation}
and by inverting $ \Delta$, and using that $ \Im ( V_{j} \bar{V}_{i}) = \Re(\i \bar{V}_{j} V_{i})$, we obtain the result.

\end{proof}
From Lemma \ref{ceoffidentities} we obtain 
$$
V_{i} \langle \partial_{t} e_{2} ,e_{1} \rangle =c V_{i} \sum_{j =1}^{n} | \nabla |^{-1} R_{j} ( \Re (\i \bar{V}_{j} {V}_{0}))
$$
and recall that we have a formula for $V_{0}$, which is \eqref{eqq6}. Therefore, 
\begin{equation}\label{Timederivativeformulavi}
\begin{split}
& V_{i} \langle \partial_{t} e_{2},e_{1} \rangle 
\\
& = -V_{i} \sum_{j,m=1}^{n} |\nabla|^{-1} R_{j} \left( \Re (  \bar{V}_{j} \nabla^{2s-1}_{m}  {V}_{m} ) \right)
\\
& - V_{i} \sum_{j,m=1}^{n} |\nabla |^{-1} R_{j} \left( \Re ( \i \bar{V}_{j} \langle e_{2} , [\nabla^{2s-1}_{m},e_{1}] ( \Re (V_{m})) \rangle \right)
\\
& - V_{i} \sum_{j,m=1}^{n} | \nabla|^{-1} R_{j} \left( \Re ( \i \bar{V}_{j} \langle e_{2}, [ \nabla^{2s-1}_{m} e_{2} ]( \Im (V_{m}) )) \rangle \right) 
\\
& - V_{i} \sum_{j,m=1}^{n} | \nabla|^{-1} R_{j} \Re  \left(  \bar{V}_{j} \langle e_{1} [ \nabla^{2s-1}_{m} ,e_{2}] ( \Im ( V_{m} ) ) \rangle \right) 
\\
& -  V_{i} \sum_{j,m=1}^{n} | \nabla|^{-1} R_{j} \Re \left(   \bar{V}_{j} \langle e_{1} , [ \nabla^{2s-1}_{m},e_{1} ] (\Re(V_{m})) \rangle \right) 
\end{split}
\end{equation}
We will see below that all the nonlinear estimates for the terms on the right hand side of \eqref{Timederivativeformulavi} will follow from the following trilinear estimate 
\begin{proposition}\label{TriTij5}
Let $n \geq 4$ , $ s \in (1/2,1)$ then define 
$$
T^{i,j}_{5}(f,g,h) : = f | \nabla|^{-1} R_{j} ( g  \nabla_{i}^{2s-1} h ) 
$$
Then we have the following estimate for any $ \sigma \geq \frac{n-2}{2}$ and any $T' \in [0,\infty]$
\begin{equation}\label{Tij5}
\begin{split}
& \Vert \ind_{[-T',T']} T^{i,j}_{5} ( \tilde{f},\tilde{g},\tilde{h}) \Vert_{N^{\sigma}}
\\
& \lesssim \Vert f \Vert_{F^{\sigma}} \Vert g \Vert_{F^{\frac{n-2}{2}}} \Vert h \Vert_{F^{\frac{n-2}{2}}}
\\
& + \Vert f \Vert_{F^{\frac{n-2}{2}}} \Vert h \Vert_{F^{\sigma}} \Vert g \Vert_{F^{\frac{n-2}{2}}}
\\
& + \Vert f \Vert_{F^{\frac{n-2}{2}}} \Vert h \Vert_{F^{\frac{n-2}{2}}} \Vert g \Vert_{F^{\sigma}}.
\end{split}
\end{equation}
\end{proposition}
\begin{proof}
To prove \eqref{Tij5} it suffices to prove the following 
\begin{equation}\label{dyadictij5}
\begin{split}
& \sum_{k =-\infty}^{\max\{ k_{1} ,k_{2} ,k_{3} \} +C } 2^{ \sigma k} \Vert \Delta_{k} ( \i \partial_{t} - ( - \Delta)^{s} + \i)^{-1} \ind_{[-T',T']} T^{i,j}_{5}(\tilde{f}_{k_{1}}, \tilde{g}_{k_{2}} , \tilde{h}_{k_{3}} ) \Vert_{Z_{k}} 
\\ 
& \lesssim 2^{ k_{1} \sigma } \Vert f_{k_{1}} \Vert_{Z_{k_{1}}} 2^{k_{2} \frac{n-2}{2}} \Vert g_{k_{2}} \Vert_{Z_{k_{2}}} 2^{k_{3} \frac{n-2}{2}} \Vert h_{k_{3}} \Vert_{Z_{k_{3}}}
\\
& + 2^{k_{1} \frac{n-2}{2}} \Vert f_{k_{1}} \Vert_{Z_{k_{1}}} 2^{k_{2} \sigma } \Vert g_{k_{2}} \Vert_{Z_{k_{2}}} 2^{k_{3} \frac{n-2}{2}} \Vert h_{k_{3}} \Vert_{Z_{k_{3}}}
\\
& + 2^{k_{1} \frac{n-2}{2}} \Vert f_{k_{1}} \Vert_{Z_{k_{1}}} 2^{k_{2} \frac{n-2}{2}} \Vert g_{k_{2}} \Vert_{Z_{k_{2}}} 2^{k_{3} \sigma} \Vert h_{k_{3}} \Vert_{Z_{k_{3}}}.
\end{split}
\end{equation}
To that end, we consider several cases depending on which function has the highest frequency 
\\
\textbf{Case 1} If $ \max \{ k_{1} ,k_{2} ,k_{3} \} =k_{1}$. Then in view of \eqref{multilpierlemma} we may assume $ f_{k_{1}}$ is supported in the set $ \{ \langle e' , \xi \rangle \geq 2^{k_{1} -2} \}$ for some $ e' \in \mathscr{E}$. Then we consider several sub-cases
\\
\textbf{Case 1a:} If $2^{\min \{ k_{2},k_{3} \} } \ll 2^{ \max \{ k_{2} , k_{3} \} } $. Then in this case the product $ \tilde{g}_{k_{2}} \tilde{h}_{k_{3}}$ has frequency localized at $ 2^{ \max \{ k_{2} ,k_{3} \}}$ and so we can apply \eqref{multimultiplier} to obtain 
\begin{equation*}
\begin{split}
& \Vert \Delta_{k} ( \i \partial_{t} - ( -\Delta)^{s} + \i)^{-1} \ind_{[-T',T']} T^{i,j}_{5} ( \tilde{f}_{k_{1}},\tilde{g}_{k_{2}},\tilde{h}_{k_{3}}) \Vert_{Z_{k}}
\\
& \lesssim 2^{-k (2s-1)/2} \max_{e \in \mathscr{E}} \Vert T^{i,j}_{5} ( \tilde{f}_{k_{1}} , \tilde{g}_{k_{2}} , \tilde{h}_{k_{3}} ) \Vert_{L^{1}_{e} L^{2}_{e^{\perp},t}}
\\
& \overset{\eqref{multimultiplier}}{\lesssim} 2^{-k(2s-1)/2} 2^{k_{3}(2s-1)} 2^{ - \max \{ k_{2} ,k_{3} \} } \Vert \tilde{f}_{k_{1}} \Vert_{L^{\infty}_{e'} L^{2}_{e'^{\perp},t}} \Vert \tilde{g}_{k_{2}} \Vert_{L^{2}_{e'}L^{\infty}_{e'^{\perp},t}} \max_{e \in \mathscr{E}} \Vert \tilde{h}_{k_{3}} \Vert_{L^{2}_{e } L^{\infty}_{e^{\perp},t}}
\\
& \overset{\eqref{local_smoothing},\eqref{maximal_Estimate}}{\lesssim} 2^{(k_{1} -k)(2s-1)/2} \Vert f_{k_{1}} \Vert_{Z_{k_{1}}} 2^{k_{2} \frac{n-2}{2}} \Vert g_{k_{2}} \Vert_{Z_{k_{2}}} 2^{k_{3} \frac{n-2}{2}} \Vert h_{k_{3}} \Vert_{Z_{k_{3}}}
\end{split}
\end{equation*}
multiplying by $ 2^{k \sigma}$ and summing over $k \leq k_{1} +C$ yields
\begin{equation*}
\begin{split}
& \sum_{k \leq k_{1} +C} 2^{k \sigma}  \Vert \Delta_{k} ( \i \partial_{t} - ( -\Delta)^{s} + \i)^{-1} \ind_{[-T',T']} T^{i,j}_{5} ( \tilde{f}_{k_{1}},\tilde{g}_{k_{2}},\tilde{h}_{k_{3}}) \Vert_{Z_{k}}
\\
& \lesssim 2^{k_{1} \sigma} \Vert f_{k_{1}} \Vert_{Z_{k_{1}}} 2^{k_{2} \frac{n-2}{2}} \Vert g_{k_{2}} \Vert_{Z_{k_{2}}} 2^{k_{3} \frac{n-2}{2}} \Vert h_{k_{3}} \Vert_{Z_{k_{3}}}.
\end{split}
\end{equation*}
This gives \eqref{dyadictij5} in this case. 
\\
\textbf{Case 1b:} If $ 2^{k_{2}} \approx 2^{k_{3}}$. Then we write 
$$
T^{i,j}_{5} ( \tilde{f}_{k_{1}} ,\tilde{g}_{k_{2}}, \tilde{h}_{k_{3}}) = \sum_{r \leq k_{2} +C} \tilde{f}_{k_{1}} \Delta_{r} | \nabla|^{-1} R_{j} ( \tilde{g}_{k_{2}} \nabla^{2s-1}_{i} \tilde{h}_{k_{3}} ). 
$$
Fix $ r \in [ k_{2} -C , k_{2} +C]$. Then we estimate as before 
\begin{equation*}
\begin{split}
& \Vert \Delta_{k} ( \i \partial_{t} - ( -\Delta)^{s} + \i)^{-1} \ind_{[-T',T']} \tilde{f}_{k_{1}} \Delta_{r} | \nabla|^{-1} R_{j} ( \tilde{g}_{k_{2}} \nabla^{2s-1}_{i} \tilde{h}_{k_{3}} ) \Vert_{Z_{k}}
\\
& \lesssim 2^{-k (2s-1)/2} \max_{e \in \mathscr{E}} \Vert 
  \tilde{f}_{k_{1}} \Delta_{r} | \nabla|^{-1} R_{j} ( \tilde{g}_{k_{2}} \nabla^{2s-1}_{i} \tilde{h}_{k_{3}} )  \Vert_{L^{1}_{e} L^{2}_{e^{\perp},t}}
\\
& \overset{\eqref{multimultiplier}}{\lesssim} 2^{-k(2s-1)/2} 2^{k_{3}(2s-1)} 2^{ - k_{2} } \Vert \tilde{f}_{k_{1}} \Vert_{L^{\infty}_{e'} L^{2}_{e'^{\perp},t}} \Vert \tilde{g}_{k_{2}} \Vert_{L^{2}_{e'}L^{\infty}_{e'^{\perp},t}} \max_{e \in \mathscr{E}} \Vert \tilde{h}_{k_{3}} \Vert_{L^{2}_{e } L^{\infty}_{e^{\perp},t}}
\\
& \overset{\eqref{local_smoothing},\eqref{maximal_Estimate}}{\lesssim} 2^{(k_{1} -k)(2s-1)/2} \Vert f_{k_{1}} \Vert_{Z_{k_{1}}} 2^{k_{2} \frac{n-2}{2}} \Vert g_{k_{2}} \Vert_{Z_{k_{2}}} 2^{k_{3} \frac{n-2}{2}} \Vert h_{k_{3}} \Vert_{Z_{k_{3}}}
\end{split}
\end{equation*}
multiplying by $ 2^{k \sigma}$ and summing over $k \leq k_{1} +C$ yields 
\begin{equation}\label{restimatecase1b}
\begin{split}
& \sum_{k \leq k_{1} +C} 2^{\sigma k_{}} \sum_{r \in [k_{2} -C,k_{2} +C]} \Vert \Delta_{k} ( \i \partial_{t} - ( -\Delta)^{s} + \i)^{-1} \ind_{[-T',T']} \tilde{f}_{k_{1}} \Delta_{r} | \nabla|^{-1} R_{j} ( \tilde{g}_{k_{2}} \nabla^{2s-1}_{i} \tilde{h}_{k_{3}} ) \Vert_{Z_{k}}
\\
& \lesssim 2^{\sigma k_{1} } \Vert f_{k_{1}} \Vert_{Z_{k_{1}}} 2^{k_{2} \frac{n-2}{2}} \Vert g_{k_{2}} \Vert_{Z_{k_{2}}} 2^{k_{3} \frac{n-2}{2}} \Vert h_{k_{3}} \Vert_{Z_{k_{3}}}
\end{split}
\end{equation}
Next, fix $ r \leq k_{2} -C$ then $ 2^{r} \ll 2^{k_{1}}$ and here we use the localized maximal estimate  
\begin{equation*}
\begin{split}
& \Vert \Delta_{k} ( \i \partial_{t} -(-\Delta)^{s} + \i)^{-1} \ind_{[-T',T']} \left(   \tilde{f}_{k_{1}} \Delta_{r} | \nabla|^{-1} R_{j} ( \tilde{g}_{k_{2}} \nabla^{2s-1}_{i} \tilde{h}_{k_{3}} )\right) \Vert_{Z_{k}}
\\
& \lesssim 2^{-k_{1} (2s-1)/2} \Vert \tilde{f}_{k_{1}} \Delta_{r} | \nabla|^{-1} R_{j} ( \tilde{g}_{k_{2}} \nabla^{2s-1}_{i} \tilde{h}_{k_{3}} ) \Vert_{L^{1}_{e'} L^{2}_{e'^{\perp},t}}
\\
& \overset{\eqref{multimultiplier}}{\lesssim} 2^{k_{1} (2s-1)/2} 2^{-r} \Vert \tilde{f}_{k_{1}} \Vert_{L^{\infty}_{e'} L^{2}_{e'^{\perp},t}} \sum_{\tn_{1},\tn_{2}}  \Vert P_{\tn_{1}} \tilde{g}_{k_{2}} \Vert_{L^{2}_{e'}L^{\infty}_{e'^{\perp},t}} \Vert P_{\tn_{2}} \tilde{h}_{k_{3}} \Vert_{L^{2}_{e'}L^{\infty}_{e'^{\perp},t}}
\\
& \overset{\eqref{maximalestcardinality}}{\lesssim }
2^{k_{1} (2s-1)/2} 2^{-r} \Vert \tilde{f}_{k_{1}} \Vert_{L^{\infty}_{e'} L^{2}_{e'^{\perp},t}} \left( \sum_{\tn \in 2^{r} \Z^{n}} \Vert P_{r,\tn} \tilde{g}_{k_{2}} \Vert^{2}_{L^{2}_{e'}L^{\infty}_{e'^{\perp},t}} \right)^{1/2} 
\\
& \times \left( \sum_{\tn \in 2^{r} \Z^{n}} \Vert P_{r,\tn} \tilde{h}_{k_{3}} \Vert^{2}_{L^{2}_{e'}L^{\infty}_{e'^{\perp},t}} \right)^{1/2}
\\
& \overset{\eqref{local_smoothing},\eqref{eq:max:Xkest:loc}}{\lesssim} ( 1+ | k_{2} -r|)^{2} 2^{r(n-3)} 2^{k_{2}} \Vert f_{k_{1}} \Vert_{Z_{k_{1}}} \Vert g_{k_{2}} \Vert_{Z_{k_{2}}} \Vert h_{k_{3}} \Vert_{Z_{k_{3}}}
\end{split}
\end{equation*}
summing over $ r \leq k_{2} -C$ yields 
\begin{equation}\label{case1restimate2}
\begin{split}
& \sum_{k \leq k_{1} +C} 2^{k \sigma} \sum_{r \leq k_{2} -C} \Vert \Delta_{k} ( \i \partial_{t} -(-\Delta)^{s} + \i)^{-1} \ind_{[-T',T']} \left(   \tilde{f}_{k_{1}} \Delta_{r} | \nabla|^{-1} R_{j} ( \tilde{g}_{k_{2}} \nabla^{2s-1}_{i} \tilde{h}_{k_{3}} )\right) \Vert_{Z_{k}}
\\
& \lesssim 2^{\sigma k_{1} } \Vert f_{k_{1}} \Vert_{Z_{k_{1}}} 2^{k_{2} \frac{n-2}{2}} \Vert g_{k_{2}} \Vert_{Z_{k_{2}}} 2^{k_{3} \frac{n-2}{2}} \Vert h_{k_{3}} \Vert_{Z_{k_{3}}}.
\end{split}
\end{equation}
Combining \eqref{restimatecase1b} and \eqref{case1restimate2} gives
\begin{equation}\label{zxcv}
\begin{split}
&  \sum_{k \leq k_{1} +C} 2^{k \sigma} \sum_{r \leq k_{2} +C} \Vert \Delta_{k} ( \i \partial_{t} -(-\Delta)^{s} + \i)^{-1} \ind_{[-T',T']} \left(   \tilde{f}_{k_{1}} \Delta_{r} | \nabla|^{-1} R_{j} ( \tilde{g}_{k_{2}} \nabla^{2s-1}_{i} \tilde{h}_{k_{3}} )\right) \Vert_{Z_{k}}
\\
& \lesssim 2^{\sigma k_{1} } \Vert f_{k_{1}} \Vert_{Z_{k_{1}}} 2^{k_{2} \frac{n-2}{2}} \Vert g_{k_{2}} \Vert_{Z_{k_{2}}} 2^{k_{3} \frac{n-2}{2}} \Vert h_{k_{3}} \Vert_{Z_{k_{3}}}.
\end{split}
\end{equation}
which proves the result in this case. 
\\
\textbf{Case 2:} If $ \max \{ k_{1},k_{2} ,k_{3} \} =k_{2}$. By \eqref{multilpierlemma} we may assume that $ g_{k_{2}}$ has Fourier transform supported in $ \{ \langle \xi , e' \rangle \geq 2^{k_{2} -2} \}$ for some $ e' \in \mathscr{E}$. Then we consider two separate sub-cases. 
\\
\textbf{Case 2a:} if $ 2^{k_{3}} \approx 2^{k_{2}}$ then this is similar to what we did in Case 1b. Namely, to prove \eqref{dyadictij5} in this case, it suffices to prove the following 
\begin{equation}\label{case2btij5}
\begin{split}
&  \sum_{k \leq k_{2} +C} 2^{k \sigma} \sum_{r \leq k_{2} +C} \Vert \Delta_{k} ( \i \partial_{t} -(-\Delta)^{s} + \i)^{-1} \ind_{[-T',T']} \left(   \tilde{f}_{k_{1}} \Delta_{r} | \nabla|^{-1} R_{j} ( \tilde{g}_{k_{2}} \nabla^{2s-1}_{i} \tilde{h}_{k_{3}} )\right) \Vert_{Z_{k}}
\\
& \lesssim 2^{\frac{n-2}{2} k_{1} } \Vert f_{k_{1}} \Vert_{Z_{k_{1}}} 2^{k_{2} \sigma } \Vert g_{k_{2}} \Vert_{Z_{k_{2}}} 2^{k_{3} \frac{n-2}{2}} \Vert h_{k_{3}} \Vert_{Z_{k_{3}}}
\\
&+ 2^{\sigma k_{1} } \Vert f_{k_{1}} \Vert_{Z_{k_{1}}} 2^{k_{2} \frac{n-2}{2}} \Vert g_{k_{2}} \Vert_{Z_{k_{2}}} 2^{k_{3} \frac{n-2}{2}} \Vert h_{k_{3}} \Vert_{Z_{k_{3}}}.
\end{split}
\end{equation}
To that end, we split the sum in $r$. If $ r \leq k_{1} -C$ for sufficiently large constant $ C$, then $2^{k}$ is comparable to $2^{k_{1}}$. Further, by \eqref{multilpierlemma}, we may assume that $ f_{k_{1}}$ has Fourier transform supported in $ \{ \langle \xi , e \rangle \geq 2^{k_{1} -c} \}$ for some $ e \in \mathscr{E}$.
\begin{equation*}
\begin{split}
& \Vert \Delta_{k}( \i \partial_{t} - ( - \Delta)^{s} + \i)^{-1} \ind_{[-T',T']} \left( \tilde{f}_{k_{1}} | \nabla|^{-1} \Delta_{r}R_{j} (\tilde{g}_{k_{2}} \nabla^{2s-1}_{i} \tilde{h}_{k_{3}} ) \right) \Vert_{Z_{k}}
\\
& \lesssim 2^{-k_{1}(2s-1)/2} \Vert  \tilde{f}_{k_{1}} | \nabla|^{-1} \Delta_{r}  R_{j} ( \tilde{g}_{k_{2}} \nabla^{2s-1}_{i} \tilde{h}_{k_{3}} ) \Vert_{L^{1}_{e} L^{2}_{e^{\perp},t}}
\\
& \overset{\eqref{multimultiplier},\eqref{maximalestcardinality}}{\lesssim} 2^{-k_{1}(2s-1)/2} 2^{k_{3} (2s-1)} 2^{-r} \Vert \tilde{f}_{k_{1}} \Vert_{L^{\infty}_{\tilde{e}} L^{2}_{e^{\perp},t}}
\\
& \times  \left( \sum_{\tn \in 2^{r} \Z^{n}} \Vert P_{r,\tn} \tilde{g}_{k_{2}} \Vert^{2}_{L^{2}_{e}L^{\infty}_{e^{\perp},t}} \right)^{1/2} \left( \sum_{\tn \in 2^{r} \Z^{n}} \Vert P_{r,\tn} \tilde{h}_{k_{3}} \Vert^{2}_{L^{2}_{e}L^{\infty}_{e^{\perp},t}} \right)^{1/2}
\\
& \overset{\eqref{local_smoothing},\eqref{eq:max:Xkest:loc}}{\lesssim} (1+|r-k_{2}|)^{2}2^{-k_{1}(2s-1)}  2^{2s k_{2}}  2^{r(n-3)} \Vert f_{k_{1}} \Vert_{Z_{k_{1}}} \Vert g_{k_{2}} \Vert_{Z_{k_{2}}} \Vert h_{k_{3}} \Vert_{Z_{k_{3}}}
\end{split}
\end{equation*}
Since $ (1+ | r -k_{2} |)^{2} \lesssim 2^{( k_{2} - r ) \delta}$ for $ \delta > 0 $ to be chosen later. Then we proved 
\begin{equation*}
\begin{split}
&  \Vert \Delta_{k}( \i \partial_{t} - ( - \Delta)^{s} + \i)^{-1}\ind_{[-T',T']} \left( \tilde{f}_{k_{1}} | \nabla|^{-1} \Delta_{r} R_{j} ( \tilde{g}_{k_{2}} \nabla^{2s-1}_{i} \tilde{h}_{k_{3}} ) \right) \Vert_{Z_{k}}
\\
& \lesssim  2^{(k_{2} - r) \delta} 2^{-k_{1}(2s-1)}  2^{2s k_{2}}  2^{r(n-3)} \Vert f_{k_{1}} \Vert_{Z_{k_{1}}} \Vert g_{k_{2}} \Vert_{Z_{k_{2}}} \Vert h_{k_{3}} \Vert_{Z_{k_{3}}}
\end{split}
\end{equation*}
summing over $ r \leq k_{1}-C$ ( this is summable since for small $ \delta> 0$ we have $ n > 3 + \delta$) we obtain
\begin{equation*}
\begin{split}
& \sum_{r \leq k_{1} -C}  \Vert \Delta_{k}( \i \partial_{t} - ( - \Delta)^{s} + \i)^{-1} \ind_{[-T',T']} \left( \tilde{f}_{k_{1}} \Delta_{r} | \nabla|^{-1} R_{j} ( \tilde{g}_{k_{2}} \nabla^{2s-1}_{i} \tilde{h}_{k_{3}} ) \right) \Vert_{Z_{k}}
\\
& \lesssim 2^{(k_{2} -k_{1})\delta} 2^{-k_{1} (2s-1)} 2^{2s k_{2}} 2^{k_{1}(n-3)} \Vert f_{k_{1}} \Vert_{Z_{k_{1}} } \Vert g_{k_{2}} \Vert_{Z_{k_{2}}} \Vert h_{k_{3}} \Vert_{Z_{k_{3}}}
\\
& \lesssim 2^{(k_{2}- k_{1} ) \delta} 2^{k_{1}(n-2-2s)} 2^{-k_{2} (n-2-2s)} 
\\
& \times \Vert f_{k_{1}} \Vert_{Z_{k_{1}}} 2^{k_{2} \frac{n-2}{2}} \Vert g_{k_{2}} \Vert_{Z_{k_{2}}} 2^{k_{3} \frac{n-2}{2}} \Vert h_{k_{3}}\Vert_{Z_{k_{3}}} 
\end{split}
\end{equation*}
If we choose $ \delta >0 $ so that $ n -2 -2s - \delta > 0 $ ( which is possible since $ n \geq 4$ and $ s < 1$) then the first factor is 
$$
2^{(k_{1} -k_{2} ) ( n-2-2s - \delta)} \lesssim 1
$$
and we conclude the following 
\begin{equation}\label{zzzi}
\begin{split}
& \sum_{k \leq k_{1}+C} 2^{\sigma k} \sum_{r \leq k_{1} -C}  \Vert \Delta_{k}( \i \partial_{t} - ( - \Delta)^{s} + \i)^{-1} \ind_{[-T',T']} \left( \tilde{f}_{k_{1}} \Delta_{r}| \nabla|^{-1}  R_{j} ( \tilde{g}_{k_{2}} \nabla^{2s-1}_{i} \tilde{h}_{k_{3}} ) \right) \Vert_{Z_{k}}
\\
& \lesssim 2^{\sigma k_{1}} \Vert f_{k_{1}} \Vert_{Z_{k_1}} 2^{k_{2} \frac{n-2}{2}} \Vert g_{k_{2}} \Vert_{Z_{k_{2}}} 2^{k_{3} \frac{n-2}{2}} \Vert h_{k_{3}}\Vert_{Z_{k_{3}}}.
\end{split}
\end{equation}
Next, fix $ r \in [k_{1} -C , k_{2} +C]$ then the sum in $k$ runs up to $ r+C$, and  we have 
\begin{equation*}
\begin{split}
& 2^{ \sigma k} \Vert \Delta_{k} ( \i \partial_{t} -(-\Delta)^{s} + \i)^{-1} \ind_{[-T',T']} \left(   \tilde{f}_{k_{1}} \Delta_{r} | \nabla|^{-1} R_{j} ( \tilde{g}_{k_{2}} \nabla^{2s-1}_{i} \tilde{h}_{k_{3}} )\right) \Vert_{Z_{k}}
\\
& \lesssim 2^{k \sigma} 2^{-k(2s-1)/2} 2^{-r} 2^{k_{2} (2s-1)} \max_{e \in \mathscr{E}} \Vert \tilde{f}_{k_{1}} \Vert_{L^{2}_{e}L^{\infty}_{e^{\perp},t}} \Vert \tilde{g}_{k_{2}} \Vert_{L^{\infty}_{e'} L^{2}_{e'^{\perp},t}} \Vert \tilde{h}_{k_{3}} \Vert_{L^{2}_{e'} L^{\infty}_{e'^{\perp},t}}
\\
& \lesssim 2^{k \sigma} 2^{-k(2s-1)/2} 2^{-r} 2^{k_{2} (2s-1)/2} \Vert g_{k_{2}} \Vert_{Z_{k_{2}}} 2^{k_{1} \frac{n-1}{2}} \Vert f_{k_{1}}\Vert_{Z_{k_{1}}} 2^{k_{3} \frac{n-1}{2}} \Vert h_{k_{3}} \Vert_{Z_{k_{3}}}
\end{split}
\end{equation*}
summing over $k \leq r$ first, then summing over $ r \leq k_{2} +C$ yields 
\begin{equation}\label{iiiiz}
\begin{split}
&  \sum_{k \leq k_{2} +C} 2^{k \sigma} \sum_{ k_{1} -C \leq r \leq k_{2} +C} \Vert \Delta_{k} ( \i \partial_{t} -(-\Delta)^{s} + \i)^{-1} \ind_{[-T',T']} \left(   \tilde{f}_{k_{1}} \Delta_{r} | \nabla|^{-1} R_{j} ( \tilde{g}_{k_{2}} \nabla^{2s-1}_{i} \tilde{h}_{k_{3}} )\right) \Vert_{Z_{k}}
\\
& \lesssim 2^{\frac{n-2}{2} k_{1} } \Vert f_{k_{1}} \Vert_{Z_{k_{1}}} 2^{k_{2} \sigma } \Vert g_{k_{2}} \Vert_{Z_{k_{2}}} 2^{k_{3} \frac{n-2}{2}} \Vert h_{k_{3}} \Vert_{Z_{k_{3}}}
\\
&+ 2^{\sigma k_{1} } \Vert f_{k_{1}} \Vert_{Z_{k_{1}}} 2^{k_{2} \frac{n-2}{2}} \Vert g_{k_{2}} \Vert_{Z_{k_{2}}} 2^{k_{3} \frac{n-2}{2}} \Vert h_{k_{3}} \Vert_{Z_{k_{3}}}.
\end{split}
\end{equation}
Combining \eqref{zzzi}, and \eqref{iiiiz} we obtain
\eqref{case2btij5}. This concludes this case.
\\
\textbf{Case 2b:} If $ 2^{k_{3}} \ll 2^{k_{2}}$ then here we apply \eqref{multimultiplier} directly to get 
\begin{equation*}
\begin{split}
& \Vert \Delta_{k} ( \i \partial_{t} - ( - \Delta)^{s} + \i)^{-1} \ind_{[-T',T']} T^{i,j}_{5} ( \tilde{f}_{k_{1}}, \tilde{g}_{k_{2}} , \tilde{h}_{k_{3}}) \Vert_{Z_{k}}
\\
& \lesssim 2^{-k (2s-1)/2} \max_{e \in \mathscr{E}} \Vert T^{i,j}_{5} ( \tilde{f}_{k_{1}}, \tilde{g}_{k_{2}}, \tilde{h}_{k_{3}} ) \Vert_{L^{1}_{e} L^{2}_{e^{\perp},t}}
\\
& \overset{\eqref{multimultiplier}}{\lesssim} 2^{-k(2s-1)/2} 2^{-k_{2}} \max_{e \in \mathscr{E}} \Vert \tilde{f}_{k_{1}} \Vert_{L^{2}_{e} L^{\infty}_{e^{\perp},t}} \Vert \tilde{g}_{k_{2}} \Vert_{L^{\infty}_{e'} L^{2}_{e'^{\perp},t}} 2^{k_{3} (2s-1)} \Vert \tilde{h}_{k_{3}} \Vert_{L^{2}_{e'} L^{\infty}_{e'^{\perp},t}}
\\
& \overset{\eqref{local_smoothing},\eqref{maximal_Estimate}}{\lesssim} 2^{-k(2s-1)/2} 2^{k_{2} (2s-1)/2} \Vert g_{k_{2}} \Vert_{Z_{k_{2}}} 2^{k_{3} \frac{n-2}{2}} \Vert h_{k_{3}} \Vert_{Z_{k_{3}}} 2^{k_{1} \frac{n-2}{2}} \Vert f_{k_{1}} \Vert_{Z_{k_{1}}}
\end{split}
\end{equation*}
Multiplying by $ 2^{\sigma k} $ and summing over $k \leq k_{2} +C $ yields the result in this case. 
\\
\textbf{Case 3:} If $ \max \{ k_{1},k_{2} ,k_{3} \} = k_{3}$. In view of the analysis above, we may assume that $ 2^{k_{2}} \ll 2^{k_{3}}$. Further, in view of \eqref{multilpierlemma} we may assume that $ h_{k_{3}}$ has Fourier transform supported in the set $ \{ \langle \xi,e' \rangle \geq 2^{k_{3} - c} \}$ for some $ e' \in \mathscr{E}$. Then we estimate using \eqref{multimultiplier} 
\begin{equation*}
\begin{split}
& \Vert \Delta_{k} ( \i \partial_{t} - (-\Delta)^{s} + \i)^{-1} \ind_{[-T',T']} T^{i,j}_{5}(\tilde{f}_{k_{1}},\tilde{g}_{k_{2}},\tilde{h}_{k_{3}}) \Vert_{Z_{k}}
\\
& \lesssim 2^{-k(2s-1)/2} \max_{e \in \mathscr{E}} \Vert T^{i,j}_{5}( \tilde{f}_{k_{1}}, \tilde{g}_{k_{2}}, \tilde{h}_{k_{3}} ) \Vert_{L^{1}_{e}L^{2}_{e^{\perp},t}}
\\
& \lesssim 2^{-k(2s-1)/2} 2^{-k_{3}} 2^{k_{3}(2s-1)} \Vert \tilde{f}_{k_{1}} \Vert_{L^{2}_{e'}L^{\infty}_{e'^{\perp},t}} \Vert \tilde{h}_{k_{3}} \Vert_{L^{\infty}_{e'}L^{2}_{e'^{\perp},t}} \max_{e \in \mathscr{E}} \Vert \tilde{g}_{k_{2}} \Vert_{L^{2}_{e} L^{\infty}_{e^{\perp},t}}
\\
& \lesssim 2^{(k_{3}-k)(2s-1)/2} \Vert h_{k_{3}} \Vert_{Z_{k_{3}}} 2^{k_{1} \frac{n-2}{2}} \Vert f_{k_1} \Vert_{Z_{k_{1}}} 2^{k_{2} \frac{n-2}{2}} \Vert g_{k_{2}} \Vert_{Z_{k_{2}}}. 
\end{split}
\end{equation*}
This concludes the third case and hence the proposition. 
\end{proof}
Notice that Proposition \ref{TriTij5} gives us the nonlinear estimates for the first term on the right hand side of \eqref{Timederivativeformulavi}. Namely, we have 
\begin{equation*}
\begin{split}
& \Vert V_{i} |\nabla|^{-1} R_{j} \left( \Re (\i V_{j} \nabla^{2s-1}_{m}  \bar{V}_{m} ) \right \Vert_{N^{\sigma}(T')}
\\
& \lesssim \Vert V_{x} \Vert_{{F}^{\sigma}(T')} \Vert V_{x} \Vert^{2}_{{F}^{\frac{n-2}{2}}(T')}
\end{split}
\end{equation*}
Next, to handle the remaining terms, we first replace each $ e_{i}$ with $ e_{i} -Q_{i} + Q_{i}$ where $ Q_{i}$ was defined in Lemma \ref{eintegra}, and then we may replace $ e_{i} - Q_{i}$ with $ | \nabla|^{-1} | \nabla| e_{i}$. Then we reduce proving the nonlinear estimates for the remaining terms on the right hand side of \eqref{Timederivativeformulavi} into proving the following four  multi-linear estimates 
\begin{equation}\label{firstmultilinearest}
\begin{split}
& \Vert T^{i,j}_{6} ( \tilde{f}_{1}, \tilde{f}_{2},\tilde{f}_{3},\tilde{f}_{4},\tilde{f}_{5})\Vert_{N^{\sigma}(T')} \\
& \lesssim  \sum_{j=1}^{5} \Vert f_{j} \Vert_{F^{\sigma}(T')}  \prod_{i \neq j} \Vert f_{i} \Vert_{F^{\frac{n-2}{2}}(T')}
\end{split}
\end{equation}
here $ T^{i,j}_{6}$ is defined as 
$$
T^{i,j}_{6} (f_{1},f_{2},f_{3},f_{4},f_{5}) : = f_{1} |\nabla|^{-1} R_{j} \left( f_{2}  (|\nabla|^{-1}f_{3}) (|\nabla|^{-1}f_{4}) \nabla^{2s-1}_{i} f_{5} \right).
$$
And we need to prove 
\begin{equation}\label{secondmultilinearest}
\begin{split}
& \Vert T^{i,j}_{7} ( \tilde{f}_{1}, \tilde{f}_{2},\tilde{f}_{3},\tilde{f}_{4},\tilde{f}_{5})\Vert_{N^{\sigma}(T')} \\
& \lesssim  \sum_{j=1}^{5} \Vert f_{j} \Vert_{F^{\sigma}(T')}  \prod_{i \neq j} \Vert f_{i} \Vert_{F^{\frac{n-2}{2}}(T')}
\end{split}
\end{equation}
where $T^{i,j}_{7} $ is defined as follows 
$$
T^{i,j}_{7} (f_{1},f_{2},f_{3},f_{4},f_{5}) : = f_{1} |\nabla|^{-1} R_{j} \left( f_{2} (|\nabla|^{-1} f_{3})  \nabla^{2s-1}_{i}( |\nabla|^{-1}f_{4} f_{5}) \right).
$$
And we also we need to prove 
\begin{equation}\label{thirdmultilinearest}
\begin{split}
& \Vert T^{i,j}_{8} ( \tilde{f}_{1}, \tilde{f}_{2},\tilde{f}_{3},\tilde{f}_{4})\Vert_{N^{\sigma}(T')} \\
& \lesssim  \sum_{j=1}^{4} \Vert f_{j} \Vert_{F^{\sigma}(T')}  \prod_{i \neq j} \Vert f_{i} \Vert_{F^{\frac{n-2}{2}}(T')}
\end{split}
\end{equation}
where $T^{i,j}_{8}$ is defined as follows 
$$
T^{i,j}_{8} (f_{1},f_{2},f_{3},f_{4}) : = f_{1} |\nabla|^{-1} R_{j} \left( f_{2} (|\nabla|^{-1}f_{3})  \nabla^{2s-1}_{i}( f_{4}) \right)
$$
and lastly we need to prove 
\begin{equation}\label{fourthdmultilinearest}
\begin{split}
& \Vert T^{i,j}_{9} ( \tilde{f}_{1}, \tilde{f}_{2},\tilde{f}_{3},\tilde{f}_{4})\Vert_{N^{\sigma}(T')} \\
& \lesssim  \sum_{j=1}^{4} \Vert f_{j} \Vert_{F^{\sigma}(T')}  \prod_{i \neq j} \Vert f_{i} \Vert_{F^{\frac{n-2}{2}}(T')}
\end{split}
\end{equation}
where $T^{i,j}_{9}$ is defined as follows 
$$
T^{i,j}_{9} (f_{1},f_{2},f_{3},f_{4}) : = f_{1} |\nabla|^{-1} R_{j} \left( f_{2}   \nabla^{2s-1}_{i}(   (|\nabla|^{-1}f_{3}) f_{4}) \right).
$$

All the of the terms appearing on the right hand side of \eqref{Timederivativeformulavi} reduces to one of the above multi-linear estimates. There are two approaches to proving \eqref{firstmultilinearest}, \eqref{secondmultilinearest}, \eqref{thirdmultilinearest}, and \eqref{fourthdmultilinearest}. First approach is to repeat the arguments made in the proof of Proposition \ref{TriTij5} and place any additional low frequencies in $L^{\infty}_{t,x}$. The second approach is to rely on Corollary \ref{algebraforftilde} and Proposition \ref{TriTij5}. We will follow the latter approach.
\begin{corollary}
Let $ n \geq 4$, $ s \in (1/2,1)$ then we have the following estimate
for any $ \sigma \geq \frac{n-2}{2}$ and any $T' \in [0,\infty]$ 
\begin{equation}\label{firstmultilinearest1}
\begin{split}
& \Vert T^{i,j}_{6} ( \tilde{f}_{1}, \tilde{f}_{2},\tilde{f}_{3},\tilde{f}_{4},\tilde{f}_{5})\Vert_{N^{\sigma}(T')} \\
& \lesssim  \sum_{j=1}^{5} \Vert f_{j} \Vert_{F^{\sigma}(T')}  \prod_{i \neq j} \Vert f_{i} \Vert_{F^{\frac{n-2}{2}}(T')}
\end{split}
\end{equation}

\end{corollary}
\begin{proof}
Define $ G : = \tilde{f}_{2} | \nabla|^{-1} \tilde{f}_{3} | \nabla|^{-1} \tilde{f}_{4}$. 
Then we have 
$$
T^{i,j}_{6} (\tilde{f}_{1},\tilde{f}_{2} ,\tilde{f}_{3},\tilde{f}_{4},\tilde{f}_{5} ) = T^{i,j}_{5} ( \tilde{f}_{1} , G , \tilde{f}_{5} ) 
$$
and hence we have from Proposition \ref{TriTij5} the following 
\begin{equation}\label{yuyui0}
\begin{split}
& \Vert T_{6}^{i,j} (\tilde{f}_{1},\tilde{f}_{2} ,\tilde{f}_{3},\tilde{f}_{4},\tilde{f}_{5} ) \Vert_{N^{\sigma}(T')}
\\
& \lesssim \Vert f_{1} \Vert_{F^{\sigma}(T')} \Vert G \Vert_{\tilde{F}^{\frac{n-2}{2}}(T')} \Vert f_{5} \Vert_{F^{\frac{n-2}{2}}(T')}
\\
& + \Vert f_{1} \Vert_{F^{\frac{n-2}{2}}(T')} \Vert G \Vert_{\tilde{F}^{\sigma}(T')} \Vert f_{5} \Vert_{F^{\frac{n-2}{2}}(T')} 
\\
& + \Vert f_{1} \Vert_{F^{\frac{n-2}{2}}(T') } \Vert G \Vert_{\tilde{F}^{\frac{n-2}{2}}(T')} \Vert f_{5} \Vert_{F^{\sigma}(T')}.
\end{split}
\end{equation}
And notice that by Corollary \ref{algebraforftilde} we have 
\begin{equation}\label{yuyui1}
\begin{split}
& \Vert G \Vert_{\tilde{F}^{\frac{n-2}{2}}(T')} 
\\
& \overset{\eqref{eqalgebraic1}}{\lesssim} \Vert \tilde{f}_{2 } |\nabla|^{-1} \tilde{f}_{3}   \Vert_{\tilde{F}^{\frac{n-2}{2}}(T')} \Vert f_{4} \Vert_{F^{\frac{n-2}{2}}(T')}
\\
& \overset{\eqref{eqalgebraic1}}{\lesssim} \Vert f_{2} \Vert_{F^{\frac{n-2}{2}}(T')} \Vert f_{3} \Vert_{F^{\frac{n-2}{2}}(T')} \Vert f_{4} \Vert_{F^{\frac{n-2}{2}}(T')}.
\end{split}
\end{equation}
Similarly, for any $ \sigma \geq \frac{n-2}{2}$ we have 
\begin{equation}\label{yuyui2}
\begin{split}
& \Vert G \Vert_{\tilde{F}^{\sigma}(T')}
\\
& \overset{ \eqref{eqalgebraic3}}{\lesssim} \Vert \tilde{f}_{2} | \nabla|^{-1} \tilde{f}_{3} \Vert_{\tilde{F}^{\frac{n-2}{2}}(T')} \Vert f_{4} \Vert_{F^{\sigma}(T')}+  \Vert \tilde{f}_{2} |\nabla|^{-1} \tilde{f}_{3} \Vert_{\tilde{F}^{\sigma}(T')} \Vert f_{4} \Vert_{F^{\frac{n-2}{2}}(T')}
\\
& \overset{ \eqref{eqalgebraic1}, \eqref{eqalgebraic3}}{\lesssim} \Vert f_{3} \Vert_{F^{\frac{n-2}{2}}(T')} \Vert f_{2} \Vert_{F^{\frac{n-2}{2}}(T')} \Vert f_{4} \Vert_{F^{\sigma}(T')}
\\
&+ \left( \Vert f_{2} \Vert_{F^{\frac{n-2}{2}}(T')} \Vert f_{3} \Vert_{F^{\sigma}(T')} + \Vert f_{2} \Vert_{F^{\sigma}(T')} \Vert f_{3} \Vert_{F^{\frac{n-2}{2}}(T')} \right)\Vert f_{4} \Vert_{F^{\frac{n-2}{2}}(T')}
\end{split}
\end{equation}
Plugging \eqref{yuyui1} and \eqref{yuyui2} back into \eqref{yuyui0} gives \eqref{firstmultilinearest1}. 
\end{proof}
The proofs for \eqref{secondmultilinearest}, \eqref{thirdmultilinearest}, \eqref{fourthdmultilinearest} are identical.
\begin{corollary}
Let $T^{i,j}_{7}$ be as above. Then
\begin{equation}\label{secondmultilinearest111}
\begin{split}
& \Vert T^{i,j}_{7} ( \tilde{f}_{1}, \tilde{f}_{2},\tilde{f}_{3},\tilde{f}_{4},\tilde{f}_{5})\Vert_{N^{\sigma}(T')} \\
& \lesssim  \sum_{j=1}^{5} \Vert f_{j} \Vert_{F^{\sigma}(T')}  \prod_{i \neq j} \Vert f_{i} \Vert_{F^{\frac{n-2}{2}}(T')}
\end{split}
\end{equation}
\end{corollary}
\begin{proof}
Set $F_{1} =  \tilde{f}_{2} \vert \nabla \vert^{-1} \tilde{f}_{3}$. And $F_{2} : =  | \nabla|^{-1} \tilde{f}_{4} \tilde{f}_{5}$. Then notice 
$$
T^{i,j}_{7} (f_{1},f_{2},f_{3},f_{4},f_{5}) = T^{i,j}_{5}(f_{1}, F_{1},F_{2}).
$$
And so by Proposition \ref{TriTij5} we obtain 
\begin{equation}\label{moredetails1}
\begin{split}
& \Vert T^{i,j}_{7} ( \tilde{f}_{1}, \tilde{f}_{2},\tilde{f}_{3},\tilde{f}_{4},\tilde{f}_{5})\Vert_{N^{\sigma}(T')} 
\\
& \lesssim \Vert f_{1} \Vert_{F^{\sigma}(T')} \Vert F_{1} \Vert_{\tilde{F}^{\frac{n-2}{2}}(T')} \Vert F_{2} \Vert_{\tilde{F}^{\frac{n-2}{2}}(T')}
\\
&+ \Vert f_{1} \Vert_{F^{\frac{n-2}{2}}(T')} \Vert F_{1} \Vert_{F^{\sigma}(T')} \Vert F_{2} \Vert_{F^{\frac{n-2}{2}}(T')}
\\
& + \Vert f_{1} \Vert_{F^{\frac{n-2}{2}}(T')} \Vert F_{1} \Vert_{F^{\frac{n-2}{2}}(T')} \Vert F_{2} \Vert_{F^{\sigma}(T')}.
\end{split}
\end{equation}
But notice that by Corollary \ref{algebraforftilde} we have 
\begin{equation}\label{moredetails2}
\begin{split}
& \Vert F_{1} \Vert_{\tilde{F}^{\frac{n-2}{2}}(T')}
\\
& \overset{\eqref{eqalgebraic3}}{\lesssim} \Vert f_{2} \Vert_{F^{\frac{n-2}{2}}(T')} \Vert f_{3} \Vert_{F^{\frac{n-2}{2}}}.
\end{split}
\end{equation}
And similarly, 
\begin{equation}\label{moredtails3}
\begin{split}
& \Vert F_{2} \Vert_{\tilde{F}^{\frac{n-2}{2}}(T')}
\\
& \overset{\eqref{eqalgebraic3}}{\lesssim} \Vert f_{4} \Vert_{F^{\frac{n-2}{2}}(T')} \Vert f_{5} \Vert_{F^{\frac{n-2}{2}}}.
\end{split}
\end{equation}
Moreover, 
\begin{equation}\label{moredetails4}
\begin{split}
& \Vert F_{1} \Vert_{\tilde{F}^{\sigma}(T')} 
\\
& \overset{\eqref{eqalgebraic3}}{ \lesssim} \Vert f_{2} \Vert_{F^{\frac{n-2}{2}}(T')} \Vert f_{3} \Vert_{F^{\sigma}(T')} + \Vert f_{3} \Vert_{F^{\frac{n-2}{2}}(T')} \Vert f_{2} \Vert_{F^{\sigma}(T')}.
\end{split}
\end{equation}
And similarly, 
\begin{equation}\label{moredetails5}
\begin{split}
& \Vert F_{2} \Vert_{\tilde{F}^{\sigma}(T')} 
\\
& \overset{\eqref{eqalgebraic3}}{ \lesssim} \Vert f_{4} \Vert_{F^{\frac{n-2}{2}}(T')} \Vert f_{5} \Vert_{F^{\sigma}(T')} + \Vert f_{5} \Vert_{F^{\frac{n-2}{2}}(T')} \Vert f_{4} \Vert_{F^{\sigma}(T')}.
\end{split}
\end{equation}
Plugging \eqref{moredetails2},\eqref{moredetails4}, \eqref{moredtails3}, \eqref{moredetails5} back into \eqref{moredetails1} gives the result.
\end{proof}
The estimates for \eqref{fourthdmultilinearest}, \eqref{thirdmultilinearest} follows in exactly the same way.

The above estimates give the proof of Proposition \ref{timederivativenonlinearity}.

\section{Nonlinear estimates IV}
In this section we prove the remaining nonlinear estimates. Namely, 
we handle the terms 
$$
V_{0} \langle e_{1} , \partial_{i} e_{2} \rangle + \sum_{j=1}^{n} \nabla^{2s-1}_{j} \left( V_{i} \langle e_{1} , \partial_{j} e_{2} \rangle \right)- \nabla^{2s-1}_{j} \left( V_{j} \langle \partial_{i} e_{2} , e_{1} \rangle \right).
$$
Similar to what we did in Section 7. We will only need to prove the nonlinear estimates for the following trilinear expressions
$$
\nabla^{2s-1}_{i}\left( f |\nabla|^{-1} R_{j} ( g h) \right), \text{ and } (\nabla^{2s-1}_{i}f) |\nabla|^{-1} R_{j} ( g h) . 
$$
The remaining terms will follow from the algebra properties of $ \tilde{F}^{\sigma}$ spaces stated in Corollary \ref{algebraforftilde}, and the identities in Lemma \ref{ceoffidentities}. 
\begin{proposition}\label{lasttrilinearguy}
Let $n \geq 4$, $ s \in ( 1/2, 1)$, $ \sigma \geq \frac{n-2}{2}$, and $T' \in [0,\infty]$. Define 
$$
T^{i,j}_{10}(f,g,h):= \nabla^{2s-1}_{i} \left( f |\nabla|^{-1} R_{j} ( g h) \right).
$$
Then we have the following
\begin{equation}\label{lasttri}
\begin{split}
& \Vert T^{i,j}_{10}(\tilde{f},\tilde{g},\tilde{h}) \Vert_{N^{\sigma}(T')}
\\
& \lesssim \Vert f \Vert_{F^{\sigma}} \Vert g \Vert_{F^{\frac{n-2}{2}}} \Vert h \Vert_{F^{\frac{n-2}{2}}}
\\
& + \Vert f \Vert_{F^{\frac{n-2}{2}}} \Vert g \Vert_{F^{\sigma}} \Vert h \Vert_{F^{\frac{n-2}{2}}}
\\
& + \Vert f \Vert_{F^{\frac{n-2}{2}}} \Vert g \Vert_{F^{\frac{n-2}{2}}} \Vert h \Vert_{F^{\sigma}}.
\end{split}
\end{equation}
\end{proposition}
\begin{proof}
To prove \eqref{lasttri}, it suffices to prove the following
\begin{equation}\label{lasttri1}
\begin{split}
& \sum_{k=-\infty}^{\max \{ k_{1},k_{2} ,k_{3} \} +C } 2^{k \sigma} \Vert \Delta_{k}( \i \partial_{t} - ( - \Delta)^{s} + \i)^{-1} \ind_{[-T',T']}  T^{i,j}_{10}(\tilde{f}_{k_{1}},\tilde{g}_{k_{2}}, \tilde{h}_{k_{3}} ) \Vert_{Z_{k}} 
\\
& \lesssim 2^{\sigma k_{1}} \Vert f_{k_{1}} \Vert_{Z_{k_{1}}} 2^{k_{2} \frac{n-2}{2}} \Vert g_{k_{2}} \Vert_{Z_{k_{2}}} 2^{k_{3} \frac{n-2}{2}} \Vert h_{k_{3}} \Vert_{Z_{k_{3}}}
\\
& + 2^{\frac{n-2}{2} k_{1}} \Vert f_{k_{1}} \Vert_{Z_{k_{1}}} 2^{k_{2} \sigma} \Vert g_{k_{2}} \Vert_{Z_{k_{2}}} 2^{k_{3} \frac{n-2}{2}} \Vert h_{k_{3}} \Vert_{Z_{k_{3}}}
\\
& +2^{\frac{n-2}{2} k_{1}} \Vert f_{k_{1}} \Vert_{Z_{k_{1}}} 2^{k_{2} \frac{n-2}{2}} \Vert g_{k_{2}} \Vert_{Z_{k_{2}}} 2^{k_{3} \sigma} \Vert h_{k_{3}} \Vert_{Z_{k_{3}}}.
\end{split}
\end{equation}
To that end, we consider several cases depending on which function has the highest frequency 
\\
\textbf{Case 1:} If $ \max\{ k_{1},k_{2},k_{3} \} = k_{1}$. Then in view of \eqref{multilpierlemma} we may assume that $ f_{k_{1}}$ has Fourier transform supported in the set $ \{ \langle \xi, e' \rangle \geq 2^{k_{1}-c} \}$ for some $ e' \in \mathscr{E}$. To prove \eqref{lasttri1} we consider two sub-cases
\\
\textbf{Case 1a:} If $ 2^{k_{2}} \approx 2^{k_{3}}$. Then in this case we write 
\begin{equation*}
\begin{split}
T^{i,j}_{10} (\tilde{f}_{k_{1}},\tilde{g}_{k_{2}},\tilde{h}_{k_{3}})(x,t) =\sum_{r \leq k_{1}+C} \nabla^{2s-1}_{i} \left( \tilde{f}_{k_{1}} \Delta_{r} \left( |\nabla|^{-1} R_{j} ( \tilde{g}_{k_{2}} \tilde{h}_{k_{3}} \right) \right).
\end{split}
\end{equation*}
Fix $ r \in [k_{2}-C,k_{2}+C]$, then by the local smoothing, and maximal estimate we obtain
\begin{equation}
\begin{split}
& \Vert \Delta_{k} ( \i \partial_{t} - ( -\Delta)^{s} +\i)^{-1} \ind_{[-T',T']} \nabla^{2s-1}_{i}\left( \tilde{f}_{k_{1}} \Delta_{r} \left( |\nabla|^{-1} R_{j} ( \tilde{g}_{k_{2}} \tilde{h}_{k_{3}} \right) \right) \Vert_{Z_{k}}
\\
& \lesssim 2^{-k(2s-1)/2} \max_{e \in \mathscr{E}} \Vert \nabla^{2s-1}_{i} \left(\tilde{f}_{k_{1}} \Delta_{r} \left( |\nabla|^{-1} R_{j} ( \tilde{g}_{k_{2}} \tilde{h}_{k_{3}} \right) \right) \Vert_{L^{1}_{e}L^{2}_{e^{\perp},t}}
\\
& \overset{\eqref{multimultiplier}}{\lesssim} 2^{-k(2s-1)/2} 2^{k_{1}(2s-1)} \Vert \tilde{f}_{k_{1}} \Vert_{L^{\infty}_{e'}L^{2}_{e'^{\perp},t}} 2^{-k_{2}} \Vert \tilde{g}_{k_2} \Vert_{L^{2}_{e'}L^{\infty}_{e'^{\perp},t}} \max_{e \in \mathscr{E}} \Vert \tilde{h}_{k_{3}} \Vert_{L^{2}_{e}L^{\infty}_{e^{\perp},t}}
\\
& \overset{\eqref{local_smoothing},\eqref{maximal_Estimate}}{\lesssim} 2^{(k_{1}-k)(2s-1)/2} \Vert f_{k_{1}} \Vert_{Z_{k_{1}}} 2^{k_{2} \frac{n-2}{2}} \Vert g_{k_{2}} \Vert_{Z_{k_{2}}} 2^{k_{3} \frac{n-2}{2}} \Vert h_{k_{3}} \Vert_{Z_{k_{3}}},
\end{split}
\end{equation}
multiplying by $ 2^{\sigma k}$ and summing over $ k \leq k_{1}+C$ yields 
\begin{equation}\label{ccc1}
\begin{split}
& \sum_{k \leq k_{1}+C} 2^{k \sigma} \Vert \Delta_{k} ( \i \partial_{t} - ( -\Delta)^{s} +\i)^{-1} \ind_{[-T',T']} \nabla^{2s-1}_{i}\left( \tilde{f}_{k_{1}} \Delta_{r} \left( |\nabla|^{-1} R_{j} ( \tilde{g}_{k_{2}} \tilde{h}_{k_{3}} \right) \right) \Vert_{Z_{k}}
\\
& \lesssim 2^{k_{1} \sigma}\Vert f_{k_{1}} \Vert_{Z_{k_{1}}} 2^{k_{2} \frac{n-2}{2}} \Vert g_{k_{2}} \Vert_{Z_{k_{2}}} 2^{k_{3} \frac{n-2}{2}} \Vert h_{k_{3}} \Vert_{Z_{k_{3}}}
\end{split}
\end{equation}
for any $ r\in [k_{2}-C,k_{2}+C]$. 
Next, if $ r \leq k_{2}-C$ for sufficiently large $C$ then we estimate as follows; using \eqref{maximalestcardinality}, \eqref{multimultiplier} we obtain 
\begin{equation}\label{jbv}
\begin{split}
& \Vert \Delta_{k} ( \i \partial_{t} - ( -\Delta)^{s} +\i)^{-1} \ind_{[-T',T']} \nabla^{2s-1}_{i} \left( \tilde{f}_{k_{1}} \Delta_{r} \left( |\nabla|^{-1} R_{j} ( \tilde{g}_{k_{2}} \tilde{h}_{k_{3}} \right) \right) \Vert_{Z_{k}}
\\
& \lesssim 2^{-k_{1}(2s-1)/2}  \Vert \nabla^{2s-1}_{i} \left(\tilde{f}_{k_{1}} \Delta_{r} \left( |\nabla|^{-1} R_{j} ( \tilde{g}_{k_{2}} \tilde{h}_{k_{3}} \right) \right) \Vert_{L^{1}_{e'}L^{2}_{e'^{\perp},t}}
\\
& \overset{\eqref{multimultiplier}, \eqref{maximalestcardinality}}{\lesssim} 2^{k_{1}(2s-1)/2} 2^{-r} \Vert \tilde{f}_{k_{1}} \Vert_{L^{\infty}_{e'}L^{2}_{e'^{\perp},t}} 
\\
& \times \left( \sum_{\tn \in 2^{r} \Z^{n}} \Vert P_{r,\tn} \tilde{g}_{k_{2}} \Vert^{2}_{L^{2}_{e'}L^{\infty}_{e'^{\perp},t}} \right)^{1/2} \left( \sum_{\tn \in 2^{r} \Z^{n}} \Vert P_{r,\tn} \tilde{h}_{k_{3}} \Vert^{2}_{L^{2}_{e'}L^{\infty}_{e'^{\perp},t}} \right)^{1/2}
\\
& \overset{\eqref{local_smoothing},\eqref{eq:max:Xkest:loc}}{\lesssim} \Vert f_{k_{1}} \Vert_{Z_{k_{1}}} 2^{r(n-3)} 2^{k_{2}} (1+|k_{2}-r|)^{2} \Vert g_{k_{2}} \Vert_{Z_{k_{2}}} \Vert h_{k_{3}}\Vert_{Z_{k_{3}}},
\end{split}
\end{equation}
summing over $r \leq k_{2}-C $ we obtain
\begin{equation}\label{ccc2}
\begin{split}
& \sum_{r \leq k_{2}-C} \Vert \Delta_{k} ( \i \partial_{t} - ( -\Delta)^{s} +\i)^{-1} \ind_{[-T',T']} \nabla^{2s-1}_{i} \left( \tilde{f}_{k_{1}} \Delta_{r} \left( |\nabla|^{-1} R_{j} ( \tilde{g}_{k_{2}} \tilde{h}_{k_{3}} \right) \right) \Vert_{Z_{k}}
\\
& \lesssim \Vert f_{k_{1}} \Vert_{Z_{k_{1}}} 2^{k_{2} \frac{n-2}{2}} \Vert g_{k_{2}} \Vert_{Z_{k_{2}}} 2^{k_{3} \frac{n-2}{2}} \Vert h_{k_{3}} \Vert_{Z_{k_{3}}}.
\end{split}
\end{equation}
Combining \eqref{ccc1}, and \eqref{ccc2} we obtain \eqref{lasttri1} in this case. 
\\
\textbf{Case 1b.} If $2^{\max \{ k_{2} , k_{3} \} } \gg 2^{\min \{ k_{2},k_{3} \}}$. Then in this case we use \eqref{multimultiplier} to obtain 
\begin{equation*}
\begin{split}
& \Vert \Delta_{k} ( \i \partial_{t} - ( - \Delta)^{s} + \i)^{-1} \ind_{[-T',T']} T^{i,j}_{10}(\tilde{f}_{k_{1}},\tilde{g}_{k_{2}},\tilde{h}_{k_{3}}) \Vert_{Z_{k}}
\\
& \lesssim 2^{-k(2s-1)/2} \max_{e \in \mathscr{E}} \Vert T^{i,j}_{10}(\tilde{f}_{k_{1}},\tilde{g}_{k_{2}},\tilde{h}_{k_{3}}) \Vert_{L^{1}_{e} L^{2}_{e^{\perp},t}}
\\
& \overset{\eqref{multimultiplier}}{\lesssim} 2^{-k(2s-1)/2} 2^{k_{1}(2s-1)} 2^{- \max \{ k_{2},k_{3} \} } \Vert \tilde{f}_{k_{1}} \Vert_{L^{\infty}_{e'} L^{2}_{e'^{\perp},t}} \Vert \tilde{g}_{k_{2}} \Vert_{L^{2}_{e'} L^{\infty}_{e'^{\perp},t}} \max_{e \in \mathscr{E}} \Vert \tilde{h}_{k_{3}} \Vert_{L^{2}_{e}L^{\infty}_{e^{\perp},t}}
\\
& \overset{\eqref{local_smoothing},\eqref{maximal_Estimate}}{\lesssim} 2^{(k_{1}-k_{})(2s-1)/2} \Vert f_{k_{1}} \Vert_{Z_{k_{1}}} 2^{k_{2} \frac{n-2}{2}} \Vert g_{k_{2}} \Vert_{Z_{k_{2}}} 2^{k_{3} \frac{n-2}{2}} \Vert h_{k_{3}} \Vert_{Z_{k_{3}}}.
\end{split}
\end{equation*}
This proves the result in this case. 
\\
\textbf{Case 2:} If $ \max \{ k_{1},k_{2},k_{3} \} \in \{ k_{2} ,k_{3} \}.$ By symmetry, we may assume without loss of generality that $ \max \{ k_{1},k_{2},k_{3} \} = k_{2}$. In view of \eqref{multilpierlemma} we may assume that $ g_{k_{2}}$ has Fourier transform supported in the set $ \{ \langle \xi, e' \rangle \geq 2^{k_{2} -c} \}$. We consider two sub-cases
\\
\textbf{Case 2a:} If $ 2^{k_{3}} \ll 2^{k_{2}}$ then we obtain using the maximal estimate, smoothing estimate, and \eqref{multimultiplier} the following
\begin{equation*}
\begin{split}
& \Vert \Delta_{k} ( \i \partial_{t} - ( - \Delta)^{s} + \i)^{-1} \ind_{[-T',T']} T^{i,j}_{10}(\tilde{f}_{k_{1}},\tilde{g}_{k_{2}},\tilde{h}_{k_{3}}) \Vert_{Z_{k}}
\\
& \lesssim 2^{-k(2s-1)/2}  \max_{e \in \mathscr{E}} \Vert T^{i,j}_{10}(\tilde{f}_{k_{1}},\tilde{g}_{k_{2}},\tilde{h}_{k_{3}}) \Vert_{L^{1}_{e} L^{2}_{e^{\perp},t}}
\\
& \overset{\eqref{multimultiplier}}{\lesssim} 2^{-k(2s-1)/2} 2^{k_{2}(2s-1)}  2^{-k_{2}} \Vert \tilde{g}_{k_{2}} \Vert_{L^{\infty}_{e'} L^{2}_{e'^{\perp},t}} \Vert \tilde{f}_{k_{1}} \Vert_{L^{2}_{e'}L^{\infty}_{e'^{\perp},t}} \max_{e \in \mathscr{E}} \Vert \tilde{h}_{k_{3}} \Vert_{L^{2}_{e} L^{\infty}_{e^{\perp},t}}
\\
& \overset{\eqref{local_smoothing},\eqref{maximal_Estimate}}{\lesssim} 2^{(k_{2}-k)(2s-1)/2} \Vert g_{k_{2}} \Vert_{Z_{k_{2}}} 2^{k_{1} \frac{n-2}{2}} \Vert f_{k_{1}} \Vert_{Z_{k_{1}}} 2^{k_{3} \frac{n-2}{2}} \Vert h_{k_{3}} \Vert_{Z_{k_{3}}}.
\end{split}
\end{equation*}
This readily implies the result in this case.
\\
\textbf{Case 2b:} If $ 2^{k_{2}} \approx 2^{k_{3}}$. Then we write 
\begin{equation*}
\begin{split}
T^{i,j}_{10} (\tilde{f}_{k_{1}},\tilde{g}_{k_{2}},\tilde{h}_{k_{3}})(x,t) =\sum_{r \leq k_{2}+C} \nabla^{2s-1}_{i} \left(\tilde{f}_{k_{1}} \Delta_{r} \left( |\nabla|^{-1} R_{j} ( \tilde{g}_{k_{2}} \tilde{h}_{k_{3}} \right) \right).
\end{split}
\end{equation*}
And it suffice to prove
\begin{equation}\label{iuop}
\begin{split}
& \sum_{k \leq k_{2} +C} 2^{k \sigma} \sum_{r \leq k_{2} +C} \Vert \Delta_{k} ( \i \partial_{t} - (-\Delta)^{s} +\i)^{-1} \ind_{[-T',T']} \left( \nabla^{2s-1}_{i} \left( \tilde{f}_{k_{1}} \Delta_{r} \left( |\nabla|^{-1} R_{j} ( \tilde{g}_{k_{2}} \tilde{h}_{k_{3}} \right) \right) \right) \Vert_{Z_{k}}
\\
& \lesssim 2^{\sigma k_{1}} \Vert f_{k_{1}} \Vert_{Z_{k_{1}}} 2^{k_{2} \frac{n-2}{2}}\Vert g_{k_{2}}\Vert_{Z_{k_{2}}} 2^{k_{3} \frac{n-2}{2}} \Vert h_{k_{3}} \Vert_{Z_{k_{3}}}
\\
&+ 2^{ k_{1} \frac{n-2}{2}} \Vert f_{k_{1}} \Vert_{Z_{k_{1}}} 2^{k_{2} \sigma }\Vert g_{k_{2}}\Vert_{Z_{k_{2}}} 2^{k_{3} \frac{n-2}{2}} \Vert h_{k_{3}} \Vert_{Z_{k_{3}}}
\end{split}
\end{equation}
Fix $ r \leq k_{1}-C$ for sufficiently large $C$. Then the sum in $k$ on the left hand side of \eqref{iuop} runs only for when $ 2^{k} \approx 2^{k_{1}}$.  Hence, we obtain from \eqref{jbv} 
\begin{equation}\label{ccc22}
\begin{split}
& \sum_{k \leq k_{2} +C} 2^{k\sigma} \sum_{r \leq k_{1}-C} \Vert \Delta_{k} ( \i \partial_{t} - ( -\Delta)^{s} +\i)^{-1} \ind_{[-T',T']} \nabla^{2s-1}_{i}\left( \tilde{f}_{k_{1}} \Delta_{r} \left( |\nabla|^{-1} R_{j} ( \tilde{g}_{k_{2}} \tilde{h}_{k_{3}} \right) \right)\Vert_{Z_{k}}
\\
& \lesssim 2^{\sigma k_{1}} \Vert f_{k_{1}} \Vert_{Z_{k_{1}}} 2^{k_{2} \frac{n-2}{2}} \Vert g_{k_{2}} \Vert_{Z_{k_{2}}} 2^{k_{3} \frac{n-2}{2}} \Vert h_{k_{3}} \Vert_{Z_{k_{3}}}.
\end{split}
\end{equation}
Next, if $ r \in [ k_{1}-C,k_{2} +C]$ then the sum in $k$ on the left hand side of \eqref{iuop} runs only up to $r+C$. And hence we obtain
\begin{equation*}
\begin{split}
& \Vert \Delta_{k} ( \i \partial_{t} - ( -\Delta)^{s} +\i)^{-1} \ind_{[-T',T']} \nabla^{2s-1}_{i} \left( \tilde{f}_{k_{1}} \Delta_{r} \left( |\nabla|^{-1} R_{j} ( \tilde{g}_{k_{2}} \tilde{h}_{k_{3}} \right) \right) \Vert_{Z_{k}}
\\
& \lesssim 2^{-k(2s-1)/2}  \max_{e \in \mathscr{E}} \Vert 
\nabla^{2s-1}_{i} \left( \tilde{f}_{k_{1}} \Delta_{r} \left( |\nabla|^{-1} R_{j} ( \tilde{g}_{k_{2}} \tilde{h}_{k_{3}} \right) \right)
\Vert_{L^{1}_{e} L^{2}_{e^{\perp},t}}
\\
& \lesssim 2^{-k(2s-1)/2} 2^{-r} 2^{k_{2}(2s-1)} 
\Vert \tilde{g}_{k_{2}} \Vert_{L^{\infty}_{e'} L^{2}_{e'^{\perp},t}} \Vert \tilde{f}_{k_{1}} \Vert_{L^{2}_{e'}L^{\infty}_{e'^{\perp},t}} \max_{e \in \mathscr{E}} \Vert \tilde{h}_{k_{3}} \Vert_{L^{2}_{e} L^{\infty}_{e^{\perp},t}}
\\
& \lesssim 2^{(k_{2}-k)(2s-1)/2} 2^{-r}\Vert g_{k_{2}} \Vert_{Z_{k_{2}}} 2^{k_{1} \frac{n-1}{2}} \Vert f_{k_{1}} \Vert_{Z_{k_{1}}} 2^{k_{3} \frac{n-1}{2}} \Vert h_{k_{3}} \Vert_{Z_{k_{3}}}
\end{split}
\end{equation*}
multiplying by $2^{k \sigma}$ and summing over $ k \leq r+C$ and then summing over $ r \in [k_{1}-C , k_{2} +C] \cap \Z $ we obtain
\begin{equation}\label{xcc3}
\begin{split}
& \sum_{k \leq k_{2}+C} 2^{k\sigma} \sum_{k_{1}-C \leq  r \leq k_{2}+C} \Vert \Delta_{k} ( \i \partial_{t} - ( -\Delta)^{s} +\i)^{-1} \ind_{[-T',T']} \nabla^{2s-1}_{i} \left( \tilde{f}_{k_{1}} \Delta_{r} \left( |\nabla|^{-1} R_{j} ( \tilde{g}_{k_{2}} \tilde{h}_{k_{3}} \right) \right) \Vert_{Z_{k}}
\\
& \lesssim 2^{ k_{1} \frac{n-2}{2}} \Vert f_{k_{1}} \Vert_{Z_{k_{1}}} 2^{k_{2} \sigma }\Vert g_{k_{2}}\Vert_{Z_{k_{2}}} 2^{k_{3} \frac{n-2}{2}} \Vert h_{k_{3}} \Vert_{Z_{k_{3}}}.
\end{split}
\end{equation}
Combining \eqref{ccc22}, and \eqref{xcc3} yields \eqref{iuop}.

\end{proof}

\begin{proposition}\label{lasttrilinearguyzzz}
Let $n \geq 4$, $ s \in ( 1/2, 1)$, $ \sigma \geq \frac{n-2}{2}$, and $ T' \in [0,\infty]$. Define 
$$
T^{i,j}_{11}(f,g,h):= \nabla^{2s-1}_{i}  f |\nabla|^{-1} R_{j} ( g h).
$$
Then we have the following
\begin{equation}\label{lasttri77}
\begin{split}
& \Vert T^{i,j}_{11}(\tilde{f},\tilde{g},\tilde{h}) \Vert_{N^{\sigma}(T')}
\\
& \lesssim \Vert f \Vert_{F^{\sigma}} \Vert g \Vert_{F^{\frac{n-2}{2}}} \Vert h \Vert_{F^{\frac{n-2}{2}}}
\\
& + \Vert f \Vert_{F^{\frac{n-2}{2}}} \Vert g \Vert_{F^{\sigma}} \Vert h \Vert_{F^{\frac{n-2}{2}}}
\\
& + \Vert f \Vert_{F^{\frac{n-2}{2}}} \Vert g \Vert_{F^{\frac{n-2}{2}}} \Vert h \Vert_{F^{\sigma}}.
\end{split}
\end{equation}
\end{proposition}
\begin{proof}
To prove \eqref{lasttri77}, it suffices to prove the following
\begin{equation}\label{lasttri177}
\begin{split}
& \sum_{k=-\infty}^{\max \{ k_{1},k_{2} ,k_{3} \} +C } 2^{k \sigma} \Vert \Delta_{k}( \i \partial_{t} - ( - \Delta)^{s} + \i)^{-1} \ind_{[-T',T']} T^{i,j}_{11}(\tilde{f}_{k_{1}},\tilde{g}_{k_{2}}, \tilde{h}_{k_{3}} ) \Vert_{Z_{k}} 
\\
& \lesssim 2^{\sigma k_{1}} \Vert f_{k_{1}} \Vert_{Z_{k_{1}}} 2^{k_{2} \frac{n-2}{2}} \Vert g_{k_{2}} \Vert_{Z_{k_{2}}} 2^{k_{3} \frac{n-2}{2}} \Vert h_{k_{3}} \Vert_{Z_{k_{3}}}
\\
& + 2^{\frac{n-2}{2} k_{1}} \Vert f_{k_{1}} \Vert_{Z_{k_{1}}} 2^{k_{2} \sigma} \Vert g_{k_{2}} \Vert_{Z_{k_{2}}} 2^{k_{3} \frac{n-2}{2}} \Vert h_{k_{3}} \Vert_{Z_{k_{3}}}
\\
& +2^{\frac{n-2}{2} k_{1}} \Vert f_{k_{1}} \Vert_{Z_{k_{1}}} 2^{k_{2} \frac{n-2}{2}} \Vert g_{k_{2}} \Vert_{Z_{k_{2}}} 2^{k_{3} \sigma} \Vert h_{k_{3}} \Vert_{Z_{k_{3}}}.
\end{split}
\end{equation}
To that end, we consider several cases depending on which function has the highest frequency 
\\
\textbf{Case 1:} If $ \max\{ k_{1},k_{2},k_{3} \} = k_{1}$. Then in view of \eqref{multilpierlemma} we may assume that $ f_{k_{1}}$ has Fourier transform supported in the set $ \{ \langle \xi, e' \rangle \geq 2^{k_{1}-c} \}$ for some $ e' \in \mathscr{E}$. To prove \eqref{lasttri177} we consider two sub-cases
\\
\textbf{Case 1a:} If $ 2^{k_{2}} \approx 2^{k_{3}}$. Then in this case we write 
\begin{equation*}
\begin{split}
T^{i,j}_{11} (\tilde{f}_{k_{1}},\tilde{g}_{k_{2}},\tilde{h}_{k_{3}})(x,t) =\sum_{r \leq k_{1}+C} \nabla^{2s-1}_{i}  \tilde{f}_{k_{1}} \Delta_{r} \left( |\nabla|^{-1} R_{j} ( \tilde{g}_{k_{2}} \tilde{h}_{k_{3}} \right).
\end{split}
\end{equation*}
Fix $ r \in [k_{2}-C,k_{2}+C]$, then by the local smoothing, and maximal estimate we obtain
\begin{equation}
\begin{split}
& \Vert \Delta_{k} ( \i \partial_{t} - ( -\Delta)^{s} +\i)^{-1} \nabla^{2s-1}_{i} \tilde{f}_{k_{1}} \Delta_{r} \left( |\nabla|^{-1} R_{j} ( \tilde{g}_{k_{2}} \tilde{h}_{k_{3}} \right)  \Vert_{Z_{k}}
\\
& \lesssim 2^{-k(2s-1)/2} \max_{e \in \mathscr{E}} \Vert \nabla^{2s-1}_{i} \tilde{f}_{k_{1}} \Delta_{r} \left( |\nabla|^{-1} R_{j} ( \tilde{g}_{k_{2}} \tilde{h}_{k_{3}} \right) \Vert_{L^{1}_{e}L^{2}_{e^{\perp},t}}
\\
& \overset{\eqref{multimultiplier}}{\lesssim} 2^{-k(2s-1)/2} 2^{k_{1}(2s-1)} \Vert \tilde{f}_{k_{1}} \Vert_{L^{\infty}_{e'}L^{2}_{e'^{\perp},t}} 2^{-k_{2}} \Vert \tilde{g}_{k_2} \Vert_{L^{2}_{e'}L^{\infty}_{e'^{\perp},t}} \max_{e \in \mathscr{E}} \Vert \tilde{h}_{k_{3}} \Vert_{L^{2}_{e}L^{\infty}_{e^{\perp},t}}
\\
& \overset{\eqref{local_smoothing},\eqref{maximal_Estimate}}{\lesssim} 2^{(k_{1}-k)(2s-1)/2} \Vert f_{k_{1}} \Vert_{Z_{k_{1}}} 2^{k_{2} \frac{n-2}{2}} \Vert g_{k_{2}} \Vert_{Z_{k_{2}}} 2^{k_{3} \frac{n-2}{2}} \Vert h_{k_{3}} \Vert_{Z_{k_{3}}},
\end{split}
\end{equation}
multiplying by $ 2^{\sigma k}$ and summing over $ k \leq k_{1}+C$ yields 
\begin{equation}\label{ccc177}
\begin{split}
& \sum_{k \leq k_{1}+C} 2^{k \sigma} \Vert \Delta_{k} ( \i \partial_{t} - ( -\Delta)^{s} +\i)^{-1} \ind_{[-T',T']} \nabla^{2s-1}_{i} \tilde{f}_{k_{1}} \Delta_{r} \left( |\nabla|^{-1} R_{j} ( \tilde{g}_{k_{2}} \tilde{h}_{k_{3}} \right) \Vert_{Z_{k}}
\\
& \lesssim 2^{k_{1} \sigma}\Vert f_{k_{1}} \Vert_{Z_{k_{1}}} 2^{k_{2} \frac{n-2}{2}} \Vert g_{k_{2}} \Vert_{Z_{k_{2}}} 2^{k_{3} \frac{n-2}{2}} \Vert h_{k_{3}} \Vert_{Z_{k_{3}}}
\end{split}
\end{equation}
for any $ r\in [k_{2}-C,k_{2}+C]$. 
Next, if $ r \leq k_{2}-C$ for sufficiently large $C$ then we estimate as follows; using \eqref{maximalestcardinality}, \eqref{multimultiplier} we obtain 
\begin{equation}\label{jbv77}
\begin{split}
& \Vert \Delta_{k} ( \i \partial_{t} - ( -\Delta)^{s} +\i)^{-1} \ind_{[-T',T']} \nabla^{2s-1}_{i} \tilde{f}_{k_{1}} \Delta_{r} \left( |\nabla|^{-1} R_{j} ( \tilde{g}_{k_{2}} \tilde{h}_{k_{3}} \right) \Vert_{Z_{k}}
\\
& \lesssim 2^{-k_{1}(2s-1)/2}  \Vert \nabla^{2s-1}_{i} \tilde{f}_{k_{1}} \Delta_{r} \left( |\nabla|^{-1} R_{j} ( \tilde{g}_{k_{2}} \tilde{h}_{k_{3}} \right) \Vert_{L^{1}_{e'}L^{2}_{e'^{\perp},t}}
\\
& \overset{\eqref{multimultiplier}, \eqref{maximalestcardinality}}{\lesssim} 2^{k_{1}(2s-1)/2} 2^{-r} \Vert \tilde{f}_{k_{1}} \Vert_{L^{\infty}_{e'}L^{2}_{e'^{\perp},t}} 
\\
& \times \left( \sum_{\tn \in 2^{r} \Z^{n}} \Vert P_{r,\tn} \tilde{g}_{k_{2}} \Vert^{2}_{L^{2}_{e'}L^{\infty}_{e'^{\perp},t}} \right)^{1/2} \left( \sum_{\tn \in 2^{r} \Z^{n}} \Vert P_{r,\tn} \tilde{h}_{k_{3}} \Vert^{2}_{L^{2}_{e'}L^{\infty}_{e'^{\perp},t}} \right)^{1/2}
\\
& \overset{\eqref{local_smoothing},\eqref{eq:max:Xkest:loc}}{\lesssim} \Vert f_{k_{1}} \Vert_{Z_{k_{1}}} 2^{r(n-3)} 2^{k_{2}} (1+|k_{2}-r|)^{2} \Vert g_{k_{2}} \Vert_{Z_{k_{2}}} \Vert h_{k_{3}}\Vert_{Z_{k_{3}}},
\end{split}
\end{equation}
summing over $r \leq k_{2}-C $ we obtain
\begin{equation}\label{ccc277}
\begin{split}
& \sum_{r \leq k_{2}-C} \Vert \Delta_{k} ( \i \partial_{t} - ( -\Delta)^{s} +\i)^{-1} \ind_{[-T',T']} \nabla^{2s-1}_{i} \tilde{f}_{k_{1}} \Delta_{r} \left( |\nabla|^{-1} R_{j} ( \tilde{g}_{k_{2}} \tilde{h}_{k_{3}} \right) \Vert_{Z_{k}}
\\
& \lesssim \Vert f_{k_{1}} \Vert_{Z_{k_{1}}} 2^{k_{2} \frac{n-2}{2}} \Vert g_{k_{2}} \Vert_{Z_{k_{2}}} 2^{k_{3} \frac{n-2}{2}} \Vert h_{k_{3}} \Vert_{Z_{k_{3}}}.
\end{split}
\end{equation}
Combining \eqref{ccc177}, and \eqref{ccc277} we obtain \eqref{lasttri177} in this case. 
\\
\textbf{Case 1b.} If $2^{\max \{ k_{2} , k_{3} \} } \gg 2^{\min \{ k_{2},k_{3} \}}$. Then in this case we use \eqref{multimultiplier} to obtain 
\begin{equation*}
\begin{split}
& \Vert \Delta_{k} ( \i \partial_{t} - ( - \Delta)^{s} + \i)^{-1} \ind_{[-T',T']} T^{i,j}_{11}(\tilde{f}_{k_{1}},\tilde{g}_{k_{2}},\tilde{h}_{k_{3}}) \Vert_{Z_{k}}
\\
& \lesssim 2^{-k(2s-1)/2} \max_{e \in \mathscr{E}} \Vert T^{i,j}_{11}(\tilde{f}_{k_{1}},\tilde{g}_{k_{2}},\tilde{h}_{k_{3}}) \Vert_{L^{1}_{e} L^{2}_{e^{\perp},t}}
\\
& \overset{\eqref{multimultiplier}}{\lesssim} 2^{-k(2s-1)/2} 2^{k_{1}(2s-1)} 2^{- \max \{ k_{2},k_{3} \} } \Vert \tilde{f}_{k_{1}} \Vert_{L^{\infty}_{e'} L^{2}_{e'^{\perp},t}} \Vert \tilde{g}_{k_{2}} \Vert_{L^{2}_{e'} L^{\infty}_{e'^{\perp},t}} \max_{e \in \mathscr{E}} \Vert \tilde{h}_{k_{3}} \Vert_{L^{2}_{e}L^{\infty}_{e^{\perp},t}}
\\
& \overset{\eqref{local_smoothing},\eqref{maximal_Estimate}}{\lesssim} 2^{(k_{1}-k)(2s-1)/2} \Vert f_{k_{1}} \Vert_{Z_{k_{1}}} 2^{k_{2} \frac{n-2}{2}} \Vert g_{k_{2}} \Vert_{Z_{k_{2}}} 2^{k_{3} \frac{n-2}{2}} \Vert h_{k_{3}} \Vert_{Z_{k_{3}}}.
\end{split}
\end{equation*}
This proves the result in this case. 
\\
\textbf{Case 2:} If $ \max \{ k_{1},k_{2},k_{3} \} \in \{ k_{2} ,k_{3} \}.$ By symmetry, we may assume without loss of generality that $ \max \{ k_{1},k_{2},k_{3} \} = k_{2}$. In view of \eqref{multilpierlemma} we may assume that $ g_{k_{2}}$ has Fourier transform supported in the set $ \{ \langle \xi, e' \rangle \geq 2^{k_{2} -c} \}$. We consider two sub-cases
\\
\textbf{Case 2a:} If $ 2^{k_{3}} \ll 2^{k_{2}}$ then we obtain using the maximal estimate, smoothing estimate, and \eqref{multimultiplier} the following
\begin{equation*}
\begin{split}
& \Vert \Delta_{k} ( \i \partial_{t} - ( - \Delta)^{s} + \i)^{-1} \ind_{[-T',T']} T^{i,j}_{11}(\tilde{f}_{k_{1}},\tilde{g}_{k_{2}},\tilde{h}_{k_{3}}) \Vert_{Z_{k}}
\\
& \lesssim 2^{-k(2s-1)/2}  \max_{e \in \mathscr{E}} \Vert T^{i,j}_{11}(\tilde{f}_{k_{1}},\tilde{g}_{k_{2}},\tilde{h}_{k_{3}}) \Vert_{L^{1}_{e} L^{2}_{e^{\perp},t}}
\\
& \overset{\eqref{multimultiplier}}{\lesssim} 2^{-k(2s-1)/2} 2^{k_{2}(2s-1)} 2^{-k_{2}} \Vert \tilde{g}_{k_{2}} \Vert_{L^{\infty}_{e'} L^{2}_{e'^{\perp},t}} \Vert \tilde{f}_{k_{1}} \Vert_{L^{2}_{e'}L^{\infty}_{e'^{\perp},t}} \max_{e \in \mathscr{E}} \Vert \tilde{h}_{k_{3}} \Vert_{L^{2}_{e} L^{\infty}_{e^{\perp},t}}
\\
& \overset{\eqref{local_smoothing},\eqref{maximal_Estimate}}{\lesssim} 2^{(k_{2}-k)(2s-1)/2} \Vert g_{k_{2}} \Vert_{Z_{k_{2}}} 2^{k_{1} \frac{n-2}{2}} \Vert f_{k_{1}} \Vert_{Z_{k_{1}}} 2^{k_{3} \frac{n-2}{2}} \Vert h_{k_{3}} \Vert_{Z_{k_{3}}}.
\end{split}
\end{equation*}
This readily implies the result in this case.
\\
\textbf{Case 2b:} If $ 2^{k_{2}} \approx 2^{k_{3}}$. Then we write 
\begin{equation*}
\begin{split}
T^{i,j}_{11} (\tilde{f}_{k_{1}},\tilde{g}_{k_{2}},\tilde{h}_{k_{3}})(x,t) =\sum_{r \leq k_{2}+C} \nabla^{2s-1}_{i} \tilde{f}_{k_{1}} \Delta_{r} \left( |\nabla|^{-1} R_{j} ( \tilde{g}_{k_{2}} \tilde{h}_{k_{3}} \right).
\end{split}
\end{equation*}
And it suffice to prove
\begin{equation}\label{iuop77}
\begin{split}
& \sum_{k \leq k_{2} +C} 2^{k \sigma} \sum_{r \leq k_{2} +C} \Vert \Delta_{k} ( \i \partial_{t} - (-\Delta)^{s} +\i)^{-1} \ind_{[-T',T']} \left( \nabla^{2s-1}_{i} \tilde{f}_{k_{1}} \Delta_{r} \left( |\nabla|^{-1} R_{j} ( \tilde{g}_{k_{2}} \tilde{h}_{k_{3}} \right) \right) \Vert_{Z_{k}}
\\
& \lesssim 2^{\sigma k_{1}} \Vert f_{k_{1}} \Vert_{Z_{k_{1}}} 2^{k_{2} \frac{n-2}{2}}\Vert g_{k_{2}}\Vert_{Z_{k_{2}}} 2^{k_{3} \frac{n-2}{2}} \Vert h_{k_{3}} \Vert_{Z_{k_{3}}}
\\
&+ 2^{ k_{1} \frac{n-2}{2}} \Vert f_{k_{1}} \Vert_{Z_{k_{1}}} 2^{k_{2} \sigma }\Vert g_{k_{2}}\Vert_{Z_{k_{2}}} 2^{k_{3} \frac{n-2}{2}} \Vert h_{k_{3}} \Vert_{Z_{k_{3}}}
\end{split}
\end{equation}
Fix $ r \leq k_{1}-C$ for sufficiently large $C$. Then the sum in $k$ on the left hand side of \eqref{iuop77} runs only for when $ 2^{k} \approx 2^{k_{1}}$.  Hence, we obtain from \eqref{jbv77} 
\begin{equation}\label{ccc2277}
\begin{split}
& \sum_{k \leq k_{2} +C} 2^{k\sigma} \sum_{r \leq k_{1}-C} \Vert \Delta_{k} ( \i \partial_{t} - ( -\Delta)^{s} +\i)^{-1}\ind_{[-T',T']} \nabla^{2s-1}_{i} \tilde{f}_{k_{1}} \Delta_{r} \left( |\nabla|^{-1} R_{j} ( \tilde{g}_{k_{2}} \tilde{h}_{k_{3}} )\right) \Vert_{Z_{k}}
\\
& \lesssim 2^{\sigma k_{1}} \Vert f_{k_{1}} \Vert_{Z_{k_{1}}} 2^{k_{2} \frac{n-2}{2}} \Vert g_{k_{2}} \Vert_{Z_{k_{2}}} 2^{k_{3} \frac{n-2}{2}} \Vert h_{k_{3}} \Vert_{Z_{k_{3}}}.
\end{split}
\end{equation}
Next, if $ r \in [ k_{1}-C,k_{2} +C]$ then the sum in $k$ on the left hand side of \eqref{iuop77} runs only up to $r+C$. And hence we obtain
\begin{equation*}
\begin{split}
& \Vert \Delta_{k} ( \i \partial_{t} - ( -\Delta)^{s} +\i)^{-1} \nabla^{2s-1}_{i} \tilde{f}_{k_{1}} \Delta_{r} \left( |\nabla|^{-1} R_{j} ( \tilde{g}_{k_{2}} \tilde{h}_{k_{3}}) \right) \Vert_{Z_{k}}
\\
& \lesssim 2^{-k(2s-1)/2}  \max_{e \in \mathscr{E}} \Vert 
\nabla^{2s-1}_{i} \tilde{f}_{k_{1}} \Delta_{r} \left( |\nabla|^{-1} R_{j} ( \tilde{g}_{k_{2}} \tilde{h}_{k_{3}}) \right) 
 \Vert_{L^{1}_{e} L^{2}_{e^{\perp},t}}
\\
& \lesssim 2^{-k(2s-1)/2} 2^{-r} 2^{k_{2}(2s-1)} 
\Vert \tilde{g}_{k_{2}} \Vert_{L^{\infty}_{e'} L^{2}_{e'^{\perp},t}} \Vert \tilde{f}_{k_{1}} \Vert_{L^{2}_{e'}L^{\infty}_{e'^{\perp},t}} \max_{e \in \mathscr{E}} \Vert \tilde{h}_{k_{3}} \Vert_{L^{2}_{e} L^{\infty}_{e^{\perp},t}}
\\
& \lesssim 2^{(k_{2}-k)(2s-1)/2} 2^{-r}\Vert g_{k_{2}} \Vert_{Z_{k_{2}}} 2^{k_{1} \frac{n-1}{2}} \Vert f_{k_{1}} \Vert_{Z_{k_{1}}} 2^{k_{3} \frac{n-1}{2}} \Vert h_{k_{3}} \Vert_{Z_{k_{3}}}
\end{split}
\end{equation*}
multiplying by $2^{k \sigma}$ and summing over $ k \leq r+C$ and then summing over $ r \in [k_{1}-C , k_{2} +C] \cap \Z $ we obtain
\begin{equation}\label{xcc377}
\begin{split}
& \sum_{k \leq k_{2}+C} 2^{k \sigma} \sum_{k_{1}-C \leq  r \leq k_{2}+C} \Vert \Delta_{k} ( \i \partial_{t} - ( -\Delta)^{s} +\i)^{-1} \ind_{[-T',T']} \nabla^{2s-1}_{i} \tilde{f}_{k_{1}} \Delta_{r} \left( |\nabla|^{-1} R_{j} ( \tilde{g}_{k_{2}} \tilde{h}_{k_{3}}) \right) \Vert_{Z_{k}}
\\
& \lesssim 2^{ k_{1} \frac{n-2}{2}} \Vert f_{k_{1}} \Vert_{Z_{k_{1}}} 2^{k_{2} \sigma }\Vert g_{k_{2}}\Vert_{Z_{k_{2}}} 2^{k_{3} \frac{n-2}{2}} \Vert h_{k_{3}} \Vert_{Z_{k_{3}}}.
\end{split}
\end{equation}
Combining \eqref{ccc2277}, and \eqref{xcc377} yields \eqref{iuop77}.

\end{proof}

In view of Proposition \ref{lasttrilinearguy}, Proposition \ref{lasttrilinearguyzzz} and Corollary \ref{algebraforftilde} we obtain the following 
\begin{proposition}\label{lstnonest}
 Let $ n \geq 4$, $ s\in (1/2,1)$, $ \sigma \geq \frac{n-2}{2}$ and $T' \in [0,T]$. Then for any $ i,j \in \{1,...,n \}$ we have the following 
 \begin{equation*}
\begin{split}
& \Vert V_{0} \langle e_{1} , \partial_{i} e_{2} \rangle \Vert_{N^{\sigma}(T')} + \Vert \nabla^{2s-1}_{j}( V_{i} \langle e_{1} , \partial_{j} e_{2} )\rangle \Vert_{N^{\sigma}(T')} + \Vert \nabla^{2s-1}_{j}( V_{j} \langle e_{1} , \partial_{i} e_{2} )\rangle \Vert_{N^{\sigma}(T')}
\\
& \lesssim \Vert V_{x} \Vert_{\tilde{F}^{\sigma}(T')} \Vert V_{x} \Vert_{\tilde{F}^{\frac{n-2}{2}}(T')}   \left( \Vert V_{x}\Vert_{\tilde{F}^{\frac{n-2}{2}}(T')} +E^{2} + E  \right)
\\
& + \Vert \vert \nabla \vert e_{1}\Vert_{\tilde{F}^{\sigma}(T')} \Vert V_{x} \Vert^{2}_{\tilde{F}^{\frac{n-2}{2}}(T')} \left(  1 +  \Vert V_{x} \Vert_{\tilde{F}^{\frac{n-2}{2}}(T')}  +   E \right) 
\\
& + \Vert \vert \nabla \vert e_{2}\Vert_{\tilde{F}^{\sigma}(T')}  \Vert V_{x} \Vert^{2}_{\tilde{F}^{\frac{n-2}{2}}(T')} \left(  1 + \Vert V_{x} \Vert_{\tilde{F}^{\frac{n-2}{2}}(T')}+   E \right).
\end{split}
 \end{equation*}
 Where we define $ E: =  \Vert |\nabla| e_{1} \Vert_{\tilde{F}^{\frac{n-2}{2}}(T')} + \Vert |\nabla| e_{2} \Vert_{\tilde{F}^{\frac{n-2}{2}}(T')}.$
\end{proposition}
\begin{proof}
This is an immediate consequence of Proposition \ref{lasttrilinearguy}, Proposition \ref{lasttrilinearguyzzz}, Corollary \ref{algebraforftilde}, and Lemma \ref{ceoffidentities}. Indeed, using Lemma \ref{ceoffidentities} we have that 
$
\nabla^{2s-1}_{j}( V_{i} \langle e_{1} , \partial_{j} e_{2} )
$
consists of terms of the form 
$$
\nabla^{2s-1}_{j} \left( V_{i} | \nabla|^{-1} R_{m} ( \Re ( \i \bar{V}_{m} V_{j} ) ) \right)   ~~ \text{  , } m \in \{1,..,n \}.
$$
And hence using Proposition \ref{lasttrilinearguy} we obtain 
\begin{equation*}
\begin{split}
& \Vert \nabla^{2s-1}_{j}( V_{i} \langle e_{1} , \partial_{j} e_{2} )\rangle \Vert_{N^{\sigma}(T')}
\\
& \lesssim \Vert V_{x} \Vert_{\tilde{F}^{\sigma}(T')} \Vert V_{x} \Vert_{\tilde{F}^{\frac{n-2}{2}}(T')}   \left( \Vert V_{x}\Vert_{\tilde{F}^{\frac{n-2}{2}}(T')} +E^{2} + E  \right)
\\
& + \Vert \vert \nabla \vert e_{1}\Vert_{\tilde{F}^{\sigma}(T')} \Vert V_{x} \Vert^{2}_{\tilde{F}^{\frac{n-2}{2}}(T')} \left(  1 +  \Vert V_{x} \Vert_{\tilde{F}^{\frac{n-2}{2}}(T')}  +   E \right) 
\\
& + \Vert \vert \nabla \vert e_{2}\Vert_{\tilde{F}^{\sigma}(T')}  \Vert V_{x} \Vert^{2}_{\tilde{F}^{\frac{n-2}{2}}(T')} \left(  1 + \Vert V_{x} \Vert_{\tilde{F}^{\frac{n-2}{2}}(T')}+   E \right).
\end{split}
\end{equation*}
Similarly, 
\begin{equation*}
\begin{split}
& \Vert \nabla^{2s-1}_{j}( V_{j} \langle e_{1} , \partial_{i} e_{2} )\rangle \Vert_{N^{\sigma}(T')}
\\
& \lesssim \Vert V_{x} \Vert_{\tilde{F}^{\sigma}(T')} \Vert V_{x} \Vert_{\tilde{F}^{\frac{n-2}{2}}(T')}   \left( \Vert V_{x}\Vert_{\tilde{F}^{\frac{n-2}{2}}(T')} +E^{2} + E  \right)
\\
& + \Vert \vert \nabla \vert e_{1}\Vert_{\tilde{F}^{\sigma}(T')} \Vert V_{x} \Vert^{2}_{\tilde{F}^{\frac{n-2}{2}}(T')} \left(  1 +  \Vert V_{x} \Vert_{\tilde{F}^{\frac{n-2}{2}}(T')}  +   E \right) 
\\
& + \Vert \vert \nabla \vert e_{2}\Vert_{\tilde{F}^{\sigma}(T')}  \Vert V_{x} \Vert^{2}_{\tilde{F}^{\frac{n-2}{2}}(T')} \left(  1 + \Vert V_{x} \Vert_{\tilde{F}^{\frac{n-2}{2}}(T')}+   E \right).
\end{split}
\end{equation*}
Lastly, using Lemma \ref{ceoffidentities} we get that 
$
V_{0} \langle e_{1} , \partial_{i} e_{2} \rangle 
$
consists of terms of the form 
$$
V_{0} |\nabla|^{-1} R_{m} \Re ( \i \bar{V}_{m} V_{i} ) ~~\text{  , } m \in \{1,..,n \}
$$
and for any fixed $m \in \{ 1,...,n \} $ we use the identity \eqref{eqq6}
\begin{equation}\label{tylsteqq}
\begin{split}
& V_{0} |\nabla|^{-1} R_{m} \Re ( \i \bar{V}_{m} V_{i} ) 
\\
& = \i \sum_{l=1}^{n} |\nabla|^{-1} R_{m} \Re ( \i \tilde{V}_{m} V_{i} ) \nabla^{2s-1}_{l} V_{l}  
\\
& -   |\nabla|^{-1} R_{m} \Re ( \i \bar{V}_{m} V_{i} )  \langle e_{2} , \sum_{l=1}^{n} [ \nabla^{2s-1}_{l} ,e_{1}]( \Re V_{l} )
\\ 
& -|\nabla|^{-1} R_{m} \Re ( \i \bar{V}_{m} V_{i} ) \langle e_{2} , \sum_{l=1}^{n} [ \nabla^{2s-1}_{l} ,e_{2}]( \Im V_{l} )
\\
& + \i |\nabla|^{-1} R_{m} \Re ( \i \bar{V}_{m} V_{i} ) \langle e_{1} , \sum_{l=1}^{n} [ \nabla^{2s-1}_{l} ,e_{2}]( \Im V_{l} )
\\
& + \i |\nabla|^{-1} R_{m} \Re ( \i \bar{V}_{m} V_{i} ) \langle e_{1} , \sum_{l=1}^{n} [ \nabla^{2s-1}_{l} ,e_{1}]( \Re V_{l} ).
\end{split}
\end{equation}
The first term on the right hand side of \eqref{tylsteqq} can be handled using Proposition \ref{lasttrilinearguyzzz}. Indeed, for any $l \in \{ 1,..,n \}$ we have 
$$
|\nabla|^{-1} R_{m} \Re ( \i \tilde{V}_{m} V_{i} ) \nabla^{2s-1}_{l} V_{l}= \frac{1}{2} \left( T^{l,m}_{11} (V_{l}, \i \bar{V}_{m},V_{i}) +T^{l,m}_{11} (V_{l}, -\i {V}_{m},\bar{V}_{i}) \right).
$$
and hence we obtain the nonlinear estimate for this term using Proposition \ref{lasttrilinearguyzzz}. For the remaining terms, replace $ e_{i}$ with $ e_{i} -Q_{i} +Q_{i}$ ( where $Q_{i}$ was defined in Lemma \ref{eintegra}) and then replace $ e_{i} -Q_{i} $ with $ |\nabla|^{-1} ( \vert \nabla \vert e_{i})$. Then all of the nonlinear estimates for the terms on the right hand side of \eqref{tylsteqq} are reduced to proving the following four multi-linear estimates   
\begin{equation}\label{12multi}
\begin{split}
& \Vert T^{i,j}_{12} ( \tilde{f}_{1},\tilde{f}_{2},\tilde{f}_{3},\tilde{f}_{4},\tilde{f}_{5} ) \Vert_{N^{\sigma}(T')}
\\
& \lesssim \sum_{i=1}^{5} \Vert f_{i} \Vert_{F^{\sigma}} \prod_{j\neq i} \Vert f_{j}\Vert_{F^{\frac{n-2}{2}}},
\end{split}
\end{equation}
where 
$$
T^{i,j}_{12}(f_{1},f_{2},f_{3},f_{4},f_{5}) = \nabla^{2s-1}_{i}( f_{1} |\nabla|^{-1}f_{2} ) (| \nabla|^{-1} f_{3}) |\nabla|^{-1} R_{j} ( f_{4} f_{5} ).
$$
Moreover, we need to prove 
\begin{equation}\label{13multi}
\begin{split}
& \Vert T^{i,j}_{13} ( \tilde{f}_{1},\tilde{f}_{2},\tilde{f}_{3},\tilde{f}_{4},\tilde{f}_{5} ) \Vert_{N^{\sigma}(T')}
\\
& \lesssim \sum_{i=1}^{5} \Vert f_{i} \Vert_{F^{\sigma}} \prod_{j\neq i} \Vert f_{j}\Vert_{F^{\frac{n-2}{2}}},
\end{split}
\end{equation}
where 
$$
T^{i,j}_{13}(f_{1},f_{2},f_{3},f_{4},f_{5}) = \nabla^{2s-1}_{i}(f_{1}  ) ( | \nabla|^{-1}f_{2}) (| \nabla|^{-1} f_{3}) |\nabla|^{-1} R_{j} ( f_{4} f_{5} ).
$$
Furthermore, we need 
\begin{equation}\label{14multi}
\begin{split}
& \Vert T^{i,j}_{14} ( \tilde{f}_{1},\tilde{f}_{2},\tilde{f}_{3},\tilde{f}_{4} ) \Vert_{N^{\sigma}(T')}
\\
& \lesssim \sum_{i=1}^{4} \Vert f_{i} \Vert_{F^{\sigma}} \prod_{j\neq i} \Vert f_{j}\Vert_{F^{\frac{n-2}{2}}},
\end{split}
\end{equation}
where 
$$
T^{i,j}_{14} (f_{1},f_{2},f_{3},f_{4})= \nabla^{2s-1}_{i}(f_{1} |\nabla|^{-1} f_{2} ) |\nabla|^{-1} R_{j} (f_{3}  f_{4}).
$$
And lastly 
\begin{equation}\label{15multi}
\begin{split}
& \Vert T^{i,j}_{15} ( \tilde{f}_{1},\tilde{f}_{2},\tilde{f}_{3},\tilde{f}_{4} ) \Vert_{N^{\sigma}(T')}
\\
& \lesssim \sum_{i=1}^{4} \Vert f_{i} \Vert_{F^{\sigma}} \prod_{j\neq i} \Vert f_{j}\Vert_{F^{\frac{n-2}{2}}},
\end{split}
\end{equation}
where 
$$
T^{i,j}_{15} (f_{1},f_{2},f_{3},f_{4})= \nabla^{2s-1}_{i}(f_{1} ) ( |\nabla|^{-1} f_{2}
 ) |\nabla|^{-1} R_{j} (f_{3}  f_{4}).
$$
The proofs of \eqref{12multi},\eqref{13multi}, \eqref{14multi}, and \eqref{15multi} follow from Corollary \ref{algebraforftilde}. Similar to what we did in Section 7. To illustrate we first prove \eqref{12multi} in detail. 
\\
\underline{ Proof of \eqref{12multi}:} 
write 
\begin{equation*}
\begin{split}
& \nabla^{2s-1}_{i} ( f_{1} | \nabla|^{-1} f_{2}) |\nabla|^{-1} f_{3} 
 = | \nabla|^{2s-1} \left( | \nabla|^{-(2s-1)} \left(  \nabla^{2s-1}_{i} ( f_{1} | \nabla|^{-1} f_{2}) |\nabla|^{-1} f_{3} \right) \right).
\end{split}
\end{equation*}
Then set 
$$
G : = | \nabla|^{-(2s-1)} \left(  \nabla^{2s-1}_{i} ( f_{1} | \nabla|^{-1} f_{2}) |\nabla|^{-1} f_{3} \right).
$$
Then from Proposition \ref{lasttrilinearguyzzz} we have 
\begin{equation}\label{ggg1}
\begin{split}
& \Vert T^{i,j}_{12} ( \tilde{f}_{1}, \tilde{f}_{2},\tilde{f}_{3},\tilde{f}_{4},\tilde{f}_{5} ) \Vert_{N^{\sigma}(T')}
\\
& \lesssim \Vert G \Vert_{\tilde{F}^{\sigma}} \Vert f_{4} \Vert_{F^{\frac{n-2}{2}}} \Vert f_{5} \Vert_{F^{\frac{n-2}{2}}}
\\
& + \Vert G \Vert_{\tilde{F}^{\frac{n-2}{2}}} \Vert f_{4} \Vert_{F^{\sigma}} \Vert f_{5} \Vert_{F^{\frac{n-2}{2}}}
\\
& + \Vert G \Vert_{\tilde{F}^{\frac{n-2}{2}}} \Vert f_{4} \Vert_{F^{\frac{n-2}{2}}} \Vert f_{5} \Vert_{F^{\sigma}}.
\end{split}
\end{equation}
Now using Corollary \ref{algebraforftilde} we have
\begin{equation}\label{ggg2}
\begin{split}
& \Vert G \Vert_{\tilde{F}^{\frac{n-2}{2}}}
\\
& \lesssim \Vert \nabla^{2s-1}_{i} ( f_{1} | \nabla|^{-1} f_{2}) |\nabla|^{-1} f_{3} \Vert_{F^{\frac{n-4s}{2}}}
\\
& \overset{\eqref{eqalgebraic1}}{\lesssim} \Vert f_{3} \Vert_{F^{\frac{n-2}{2}}} \Vert \nabla^{2s-1}_{i} ( f_{1} | \nabla|^{-1} f_{2} ) \Vert_{F^{\frac{n-4s}{2}}}
\\
& \lesssim \Vert f_{3} \Vert_{F^{\frac{n-2}{2}}} \Vert   f_{1} | \nabla|^{-1} f_{2}  \Vert_{F^{\frac{n-2}{2}}}
\\
& \overset{\eqref{eqalgebraic1}}{\lesssim} \Vert f_{3} \Vert_{F^{\frac{n-2}{2}}} \Vert f_{1} \Vert_{F^{\frac{n-2}{2}}} \Vert f_{2} \Vert_{F^{\frac{n-2}{2}}}.
\end{split}
\end{equation}
Similarly, we have 
\begin{equation}\label{ggg3}
\begin{split}
& \Vert G \Vert_{\tilde{F}^{\sigma}}
\\
& \lesssim \Vert \nabla^{2s-1}_{i} ( f_{1} | \nabla|^{-1} f_{2}) |\nabla|^{-1} f_{3} \Vert_{F^{\sigma - 2s+1}}
\\
& \overset{\eqref{eqalgebraic6}}{ \lesssim} \Vert \nabla^{2s-1}_{i} ( f_{1} | \nabla |^{-1} f_{2}) \Vert_{F^{\frac{n-4s}{2}}} \Vert  f_{3} \Vert_{F^{\sigma}} + \Vert f_{3} \Vert_{F^{\frac{n-2}{2}}} \Vert f_{1} | \nabla|^{-1} f_{2} \Vert_{F^{\sigma}}
\\
& \overset{\eqref{eqalgebraic4}}{\lesssim} 
\Vert f_{1} \Vert_{F^{\sigma}} \Vert f_{2} \Vert_{F^{\frac{n-2}{2}}} \Vert f_{3} \Vert_{F^{\frac{n-2}{2}}} 
\\
& + \Vert f_{1} \Vert_{F^{\frac{n-2}{2}}} \Vert f_{2} \Vert_{F^{\sigma}} \Vert f_{3} \Vert_{F^{\frac{n-2}{2}}}
\\
&+ \Vert f_{1} \Vert_{F^{\frac{n-2}{2}}} \Vert f_{2} \Vert_{F^{\frac{n-2}{2}}} \Vert f_{3} \Vert_{F^{\sigma}}
\end{split}
\end{equation}
Plugging \eqref{ggg2} and \eqref{ggg3} back into \eqref{ggg1} gives \eqref{12multi}.
\\
\underline{Proof of \eqref{13multi}:} Similar to the proof of \eqref{12multi} set
$$
G : = |\nabla|^{-2s+1} \left( \nabla^{2s-1}_{i}\tilde{f}_{1}) ( | \nabla|^{-1} f_{2} )( | \nabla|^{-1} f_{3} ) \right). 
$$
Then using Proposition \ref{lasttrilinearguyzzz} we obtain
\begin{equation}\label{ddetl1}
\begin{split}
& \Vert T^{i,j}_{13} ( \tilde{f}_{1},\tilde{f}_{2},\tilde{f}_{3}, \tilde{f}_{4},\tilde{f}_{5} ) \Vert_{N^{\sigma}(T')}
\\
& \lesssim \Vert G \Vert_{\tilde{F}^{\sigma}} \Vert f_{4} \Vert_{F^{\frac{n-2}{2}}} \Vert f_{5} \Vert_{F^{\frac{n-2}{2}}}
\\
& + \Vert G \Vert_{\tilde{F}^{\frac{n-2}{2}}} \Vert f_{4} \Vert_{F^{\sigma}} \Vert f_{5} \Vert_{F^{\frac{n-2}{2}}}
\\
& + \Vert G \Vert_{\tilde{F}^{\frac{n-2}{2}}} \Vert f_{4} \Vert_{F^{\frac{n-2}{2}}} \Vert f_{5} \Vert_{F^{\sigma}}.
\end{split}
\end{equation}
But by the algebra properties, Corollary \ref{algebraforftilde}, we have 
\begin{equation}\label{ddetl2}
\begin{split}
& \Vert G \Vert_{\tilde{F}^{\frac{n-2}{2}}}
\\
&\lesssim \Vert f_{1} \Vert_{F^{\frac{n-2}{2}}} \Vert f_{2} \Vert_{F^{\frac{n-2}{2}}} \Vert f_{3} \Vert_{F^{\frac{n-2}{2}}}.
\end{split}
\end{equation}
Similarly, 
\begin{equation}\label{ddetl3}
\begin{split}
& \Vert G \Vert_{\tilde{F}^{\sigma}}
\\
& {\lesssim} 
\Vert f_{1} \Vert_{F^{\sigma}} \Vert f_{2} \Vert_{F^{\frac{n-2}{2}}} \Vert f_{3} \Vert_{F^{\frac{n-2}{2}}} 
\\
& + \Vert f_{1} \Vert_{F^{\frac{n-2}{2}}} \Vert f_{2} \Vert_{F^{\sigma}} \Vert f_{3} \Vert_{F^{\frac{n-2}{2}}}
\\
&+ \Vert f_{1} \Vert_{F^{\frac{n-2}{2}}} \Vert f_{2} \Vert_{F^{\frac{n-2}{2}}} \Vert f_{3} \Vert_{F^{\sigma}}.
\end{split}
\end{equation}
Plugging \eqref{ddetl2}, \eqref{ddetl3} back into \eqref{ddetl1} yields \eqref{13multi}.\\
The proofs for \eqref{14multi}, and \eqref{15multi} follow almost verbatim.  
\end{proof}

\section{Proof of Theorem \ref{mainthm}}\label{s:proofofthm}
Throughout this section fix $ \sigma_{0} \geq \frac{n}{2} +10$ and $ u \in L^{\infty}( [-T,T] : B^{\sigma_{0}+1}_{Q} )$ so that $u$ is a solution of \eqref{SSchrodingereq} with initial data $ u_{0} \in B^{\sigma_{0}+1}_{Q}$, on the time interval $ [- T_{0},T_{0}]$ for some $T_{0} : = T_{0}(s, \sigma_{0} , ||u_0||_{B^{\sigma_{0}}_{Q}}) \in (0,1]$. Existence of such a solution is guaranteed by Proposition \ref{localexistencethm}. Assume further that $ \Vert u_{0} \Vert_{\dot{B}^{\frac{n}{2}}_{Q}} \leq \epsilon \ll 1$. 
\\
We recall the following nonlinear estimate which we proved in Proposition \ref{lstnonest}, Proposition \ref{timederivativenonlinearity}, Proposition \ref{trlCgauge}, Proposition  \ref{errnonlinears}, and Proposition \ref{Firstnonlinearestimate}.
\begin{lemma}\label{finalnonlinearestimates}
Let $ n \geq 4$ , $ s \in (1/2,1)$ and $ \sigma \geq \frac{n-2}{2}$. Then for any $ T' \in (0, T_{0}]$ satisfying 
$$
\vert u(x,t) - Q \vert \leq 1/2 \quad \text{ for all } (x,t) \in \R^{n} \times [-T',T']
$$
we have the following
\begin{equation*}
\begin{split}
& \Vert \mathcal{N} ( V_{x} , e_{1},e_{2} ) \Vert_{N^{\sigma}(T')}
\\
& \lesssim \Vert V_{x} \Vert_{F^{\sigma}(T')} \left( \Vert V_{x} \Vert^{2}_{F^{\frac{n-2}{2}}(T')} + \Vert V_{x} \Vert_{F^{\frac{n-2}{2}}(T')} E  + E^{2} \right)
\\
& +\Vert |\nabla |e_{1} \Vert_{\tilde{F}^{\sigma}(T')}  \Vert V_{x} \Vert_{\tilde{F}^{\frac{n-2}{2}}(T')}  \left(\Vert V_{x} \Vert^{2}_{\tilde{F}^{\frac{n-2}{2}}(T')}+ \Vert V_{x} \Vert_{\tilde{F}^{\frac{n-2}{2}}(T')} + ( 1 + \Vert V_{x} \Vert_{\tilde{F}^{\frac{n-2}{2}}(T')})   E +  E^{2} \right).
\\
&+ \Vert |\nabla |e_{2} \Vert_{\tilde{F}^{\sigma}(T')}  \Vert V_{x} \Vert_{\tilde{F}^{\frac{n-2}{2}}(T')}  \left(\Vert V_{x} \Vert^{2}_{\tilde{F}^{\frac{n-2}{2}}(T')}+ \Vert V_{x} \Vert_{\tilde{F}^{\frac{n-2}{2}}(T')} + ( 1 + \Vert V_{x} \Vert_{\tilde{F}^{\frac{n-2}{2}}(T')})   E +  E^{2} \right),
\end{split}
\end{equation*}
where $ E : = \Vert |\nabla| e_{1} \Vert_{\tilde{F}^{\frac{n-2}{2}}(T')} + \Vert |\nabla| e_{2} \Vert_{\tilde{F}^{\frac{n-2}{2}}(T')}$.
\end{lemma}

The next lemma gives us control of the norm of the frames $ e_1, e_{2}$ in terms of $V_{x}$. This is the only time in which the assumption $n \geq 5$ is needed.
\begin{lemma}\label{EcontroledbyV}
Let $ n \geq 5$, $ s \in (1/2,1)$, and $ T' \in [0,T_{0}]$ so that we have 
$$
\vert u(x,t) - Q \vert \leq 1/2 \quad \text{ for all } (x,t) \in \R^{n} \times [-T',T'].
$$
Then there exists an $ \epsilon_{1}:= \epsilon_{1} (n,s) > 0$ so that whenever 
$$
\Vert V_{x} \Vert_{F^{\frac{n-2}{2}}(T')} \leq \epsilon_{1}
$$
then we have 
$$
\Vert |\nabla| u \Vert_{\tilde{F}^{\frac{n-2}{2}}(T')} + \Vert | \nabla| e_{1} \Vert_{\tilde{F}^{\frac{n-2}{2}}(T')} + \Vert | \nabla| e_{2} \Vert_{\tilde{F}^{\frac{n-2}{2}}(T')}   \lesssim \Vert V_{x} \Vert_{F^{\frac{n-2}{2}}(T')}.
$$
\end{lemma}
\begin{proof}
Define 
$$
Z : = \Vert |\nabla| u \Vert_{\tilde{F}^{\frac{n-2}{2}}(T')} + \Vert | \nabla| e_{1} \Vert_{\tilde{F}^{\frac{n-2}{2}}(T')} + \Vert | \nabla| e_{2} \Vert_{\tilde{F}^{\frac{n-2}{2}}(T')}.
$$
Notice that 
$$
\partial_{i} u = \Re(V_{i}) e_{1} + \Im(V_{i}) e_{2}
$$
and 
$$
\partial_{i} e_{1} = -\Re(V_{i}) u - A_{i} e_{2}
$$
where $ A_{i} = \langle \partial_{i} e_{2} ,e_{1} \rangle$. 
\\
Similarly, we have 
$$
\partial_{i} e_{2} = -\Im(V_{i}) u + A_{i} e_{1}.
$$
Then using \eqref{eqalgebraic1} we obtain
\begin{equation}\label{qwe1}
\begin{split}
& \Vert | \nabla| u \Vert_{\tilde{F}^{\frac{n-2}{2}}(T')} 
\\
& 
\lesssim \sum_{i=1}^{n} \Vert \Re(V_{i}) e_{1} + \Im(V_{i}) e_{2} \Vert_{\tilde{F}^{\frac{n-2}{2}}(T') }
\\
& \lesssim \Vert V_{x} \Vert_{F^{\frac{n-2}{2}}(T')} + \sum_{i=1}^{n} \Vert \Re(V_{i}) ( e_{1} - Q_{1}) \Vert_{\tilde{F}^{\frac{n-2}{2}}(T')} + \sum_{i=1}^{n} \Vert \Im(V_{i}) ( e_{2}-Q_{2}) \Vert_{\tilde{F}^{\frac{n-2}{2}}(T')}
\\
& \lesssim \Vert V_{x} \Vert_{F^{\frac{n-2}{2}}(T')} + \sum_{i=1}^{n} \Vert \Re(V_{i}) |\nabla|^{-1} (|\nabla| e_{1} ) \Vert_{\tilde{F}^{\frac{n-2}{2}}(T')} + \sum_{i=1}^{n} \Vert \Im(V_{i}) |\nabla|^{-1} | \nabla|e_{2} \Vert_{\tilde{F}^{\frac{n-2}{2}}(T')}
\\
&
\overset{\eqref{eqalgebraic1}}{\lesssim} \Vert V_{x} \Vert_{F^{\frac{n-2}{2}}(T')} (1+ Z).
\end{split}
\end{equation}
Similarly, 
\begin{equation}\label{qwe2}
\begin{split}
& \Vert |\nabla| e_{j} \Vert_{\tilde{F}^{\frac{n-2}{2}}(T')}
\\
& \lesssim \Vert V_{x} \Vert_{F^{\frac{n-2}{2}(T')}} ( 1+Z) + \Vert A \Vert_{\tilde{F}^{\frac{n-2}{2}}(T')} (1+Z).
\end{split}
\end{equation}
Where $ A : = (A_{1},....,A_{n})$. \\
\textbf{Claim: } For any $ i \in \{ 1,..,n \}$ we have the estimate 
$$
\Vert A_{i} \Vert_{\tilde{F}^{\frac{n-2}{2}}(T')} \lesssim \Vert V_{x} \Vert^{2}_{F^{\frac{n-2}{2}}(T')}.
$$
Once this claim is proven, we can conclude as follows; we take $ \epsilon_{1}$ small enough so that 
$$
C \epsilon_{1}(1+\epsilon_{1}) \leq \frac{1}{2}.
$$
Then combining \eqref{qwe1} and \eqref{qwe2} we obtain
$$
Z \lesssim \Vert V_{x} \Vert_{F^{\frac{n-2}{2}}(T') }
$$
which is the desired result. It remains to prove the claim. 
\\
\underline{Proof of Claim:}
By Lemma \ref{ceoffidentities} we can write
$$
A_{i} = c \sum_{j=1}^{n} | \nabla|^{-1} R_{j} ( \Re ( \i \bar{V_{j}} {V}_{i}))
$$
then we estimate using the algebra properties of $F^{\sigma}(T')$ spaces 
\begin{equation*}
\begin{split}
& \Vert A_{i} \Vert_{\tilde{F}^{\frac{n-2}{2}}(T')}
\\
& \lesssim \sum_{j=1}^{n} \Vert |\nabla|^{-1} \Re(\i V_{i} \bar{V}_{j}) \Vert_{\tilde{F}^{\frac{n-2}{2}}(T')}
\\
& \overset{\eqref{eqalgebraic4}}{\lesssim} \Vert V_{x} \Vert^{2}_{F^{\frac{n-2}{2}}(T')}.
\end{split}
\end{equation*}
\end{proof}
The next lemma gives us higher regularity control of the norm of the frames $ e_{1},e_{2}$ in terms of $V_{x}$
\begin{lemma}\label{EcontrolledVhigherregularity}
Fix $ \sigma \in [ \frac{n-2}{2}, \sigma_{0} +1]$ and $ T' \in [0,T_{0} ]$ satisfying 
$$
\vert u(x,t) - Q \vert \leq 1/2 \quad \text{ for all } (x,t) \in \R^{n} \times [-T',T'].
$$
 Then there exists $ \epsilon_{2} : = \epsilon_{2} (n,s,\sigma) > 0$ so that whenever we have 
$$
\Vert V_{x} \Vert_{F^{\frac{n-2}{2}}(T')} \leq \epsilon_{2}.
$$
Then we have that 
$$
\Vert | \nabla| e_{1} \Vert_{\tilde{F}^{\sigma}(T')} + \Vert | \nabla| e_{2} \Vert_{\tilde{F}^{\sigma}(T')} + \Vert | \nabla| u \Vert_{\tilde{F}^{\sigma}(T')} \lesssim \Vert V_{x} \Vert_{F^{\sigma}(T')}.
$$
\end{lemma}
\begin{proof}
Take $ \epsilon_{2} \leq \epsilon_{1}$, then by Lemma \ref{EcontroledbyV} we obtain that 
$$
\Vert |\nabla| u \Vert_{\tilde{F}^{\frac{n-2}{2}}(T')} + \Vert | \nabla| e_{1} \Vert_{\tilde{F}^{\frac{n-2}{2}}(T')} + \Vert | \nabla| e_{2} \Vert_{\tilde{F}^{\frac{n-2}{2}}(T')}   \lesssim \Vert V_{x} \Vert_{F^{\frac{n-2}{2}}(T')}.
$$
Furthermore, we have 
\begin{equation*}
\begin{split}
& \Vert | \nabla| u \Vert_{\tilde{F}^{\sigma}(T')} 
\\
& 
\lesssim \sum_{i=1}^{n} \Vert \Re(V_{i}) e_{1} + \Im(V_{i}) e_{2} \Vert_{\tilde{F}^{\sigma} (T')}
\\
& \lesssim \Vert V_{x} \Vert_{F^{\sigma}(T')} + \sum_{i=1}^{n} \Vert \Re(V_{i}) ( e_{1} - Q_{1}) \Vert_{\tilde{F}^{\sigma}(T')} + \sum_{i=1}^{n} \Vert \Im(V_{i}) ( e_{2}-Q_{2}) \Vert_{\tilde{F}^{\sigma}(T')}
\\
& \lesssim \Vert V_{x} \Vert_{F^{\sigma}(T')} + \sum_{i=1}^{n} \Vert \Re(V_{i}) |\nabla|^{-1} (|\nabla| e_{1} ) \Vert_{\tilde{F}^{\sigma}(T')} + \sum_{i=1}^{n} \Vert \Im(V_{i}) |\nabla|^{-1} | \nabla|e_{2} \Vert_{\tilde{F}^{\sigma}(T')}
\\
& \overset{\eqref{eqalgebraic5}}{\lesssim} \Vert V_{x} \Vert_{F^{\frac{n-2}{2}}(T')} \left( \Vert | \nabla| e_{1} \Vert_{\tilde{F}^{\sigma}(T')} + \Vert | \nabla| e_{2} \Vert_{\tilde{F}^{\sigma}(T')} \right) + \Vert V_{x} \Vert_{F^{\sigma}(T')}
\\
& \leq \epsilon_{2} Z + \Vert V_{x} \Vert_{F^{\sigma}(T')}
\end{split}
\end{equation*}
where we defined
$$
Z:= \Vert | \nabla| e_{1} \Vert_{\tilde{F}^{\sigma}(T')} + \Vert | \nabla| e_{2} \Vert_{\tilde{F}^{\sigma}(T')} + \Vert | \nabla| u \Vert_{\tilde{F}^{\sigma}(T')}.
$$
Similarly, we have ( here $ A: = (A_{1},...,A_{n}),$ where recall that $A_{i} = \langle \partial_{i} e_{2},e_{1} \rangle$)
\begin{equation*}
\begin{split}
& \Vert | \nabla| e_{j} \Vert_{\tilde{F}^{\sigma}(T')}
\\
& \lesssim \Vert V_{x} \Vert_{F^{\sigma}(T') }+ \Vert A \Vert_{\tilde{F}^{\sigma}(T')} + \Vert V_{x} \Vert^{2}_{F^{\frac{n-2}{2}}(T')} \left( \Vert | \nabla| e_{1} \Vert_{\tilde{F}^{\sigma}(T')} + \Vert | \nabla| e_{2} \Vert_{\tilde{F}^{\sigma}(T')} \right)
\\
& \overset{\eqref{eqalgebraic4}}{\lesssim} \Vert V_{x} \Vert_{F^{\sigma}(T') } + \epsilon_{2} (1+ \epsilon_{2}) Z.
\end{split}
\end{equation*}
Combining the above inequalities we obtain, for sufficiently small $ \epsilon_{2}$, 
$$
Z \lesssim \Vert V_{x} \Vert_{F^{\sigma}(T') }.
$$
\end{proof}
Next, we define the following  function 
$$
L^{\sigma}(r) : = \Vert \mathcal{N}(V_{x} ,e_{1},e_{2} ) \Vert_{N^{\sigma}(r)}.
$$
Then we prove the following continuity property of $L^{\sigma
}$. 
\begin{lemma}\label{continuityofL}
Let $ \sigma \in [ \frac{n-2}{2},\sigma_{0} -1]$, and $T' \in [0,T_{0}]$ so that 
$$
\vert u(x,t) - Q \vert \leq 1/2 \quad \text{ for all } (x,t) \in \R^{n} \times [-T',T']
$$
Then the function $L^{\sigma}: [0,T'] \to [0,\infty)$ satisfy the following 
\begin{enumerate}
    \item $$L^{\sigma}(0) = 0$$
    \item $$
    | L^{\sigma}(r_{1} ) -L^{\sigma}(r_{2})| \leq C |r_{1} -r_{2} |^{1/2}.
    $$
    here the constant $C$ may depend on the Besov norm of $V_{x}$, $e_{1},e_{2}$. 
\end{enumerate}
\end{lemma}
\begin{proof}
The assertion (1) is clear since $ \Vert 0 \Vert_{L^{2}_{x,t}} =0  $. and since $ \Vert \Delta_{k} ( \i \partial_{t} - ( -\Delta)^{s} +\i)^{-1} f \Vert_{Z_{k}} \lesssim   \Vert \Delta_{k} f \Vert_{L^{2}_{x,t}}$. We need to prove the second assertion. To that end, set $ v = \mathcal{N}(V_{x},e_{1},e_{2})$. Then in view of Lemma \ref{eintegra}, Lemma \ref{controlvu}, and the nonlinear estimates Lemma \ref{finalnonlinearestimates} we have  $ v \in N^{\sigma}(T') \cap L^{\infty}( [-T',T']: B^{\sigma_{0} -(2s-1)}_{2,1})$. Then for  any $ r_{1},r_{2} \in [0,T']$ we have 
\begin{equation*}
\begin{split}
& | \Vert v \Vert_{N^{\sigma}(r_{1})} - \Vert v \Vert_{N^{\sigma}(r_{2} )}|
\\
& \lesssim \Vert E_{r_{1}} v - E_{r_{2}} v \Vert_{N^{\sigma}}
\\
& \lesssim | r_{1} -r_{2} |^{\frac{1}{2}} \sum_{k \in \Z } 2^{\sigma k} \sup_{t \in [-T',T']} \Vert \Delta_{k} v \Vert_{L^{2} (\R^{n})}
\\
& = | r_{1} - r_{2} |^{1/2} \sum_{k  \leq 0} 2^{\sigma k} \sup_{t \in [-T',T']} \Vert \Delta_{k} v \Vert_{L^{2}( \R^{n})} + | r_{1} - r_{2} |^{1/2} \sum_{k \geq 1} 2^{\sigma k} \sup_{t \in [-T',T']} \Vert \Delta_{k} v \Vert_{L^{2}( \R^{n})} 
\\
& \lesssim | r_{1} -r_{2} |^{1/2} \sup_{t \in [-T',T']} \Vert v \Vert_{B^{\sigma_{0} -(2s-1)}_{2,1}} + | r_{1} -r_{2} |^{1/2}   \sup_{t \in [-T',T']} \Vert v \Vert_{B^{\sigma_{0} -(2s-1)}_{2,1}} \sum_{k \geq 1} 2^{- k ( \sigma_{0} -(2s-1) - \sigma)} 
\\
& \lesssim | r_{1} -r_{2} |^{1/2} \sup_{t \in [-T',T']} \Vert v \Vert_{B^{\sigma_{0} -(2s-1)}_{2,1}}.
\end{split}
\end{equation*}
\end{proof}
\begin{remark}
    Notice that in the above Lemma, we cannot hope to get continuity of the function $L^{\sigma}$ for every  $  \sigma \in [ \frac{n-2}{2},\sigma_{0}]$. This is because the nonlinearity $\mathcal{N}(V_{x} ,e_{1} ,e_{2} )$ contain terms of the form $ | \nabla|^{2s-1} (V_{x})$ while $ V_{x} \in \dot{B}^{\sigma_{0}}_{2,1}$. So to run the above argument, we need $ \sigma < \sigma_{0} -(2s-1)$.  
\end{remark}
Now we prove that the functions $ V_{i}'$s are in fact small in the space $ F^{\frac{n-2}{2}}(T')$. This is the content of the next lemma
\begin{lemma}\label{smallnessofvx}
Assume the initial data $ u_{0}$ satisfy 
$$
\Vert u_{0} \Vert_{\dot{B}^{n/2}_{Q}} \leq \epsilon_{3}
$$
for some $ \epsilon_{3}(n,s)>0$ sufficiently small. Then for any $T'\in [0,T_{0}]$ which satisfies 
$$
\vert u(x,t) - Q \vert \leq 1/2 \quad \text{ for all } (x,t) \in \R^{n} \times [-T',T'],
$$
we obtain that 
$$
\Vert V_{x} \Vert_{F^{\frac{n-2}{2}}(T')} \lesssim \epsilon_{3}.
$$
\end{lemma}
\begin{proof}
Define the Duhmal extension as follows 
\begin{equation}\label{duhmalextension}
\Tilde{E}_{T'} ( V_{i} ) (x,t): = \phi(t) e^{-\i (-\Delta)^{s} t} V_{i}|_{t=0} + \phi(t) \int_{0}^{t} e^{-\i ( t-r) (-\Delta)^{s}} E_{T'}(\mathcal{N}(V_{x},e_{1},e_{2}))(x,r) dr.
\end{equation}

Then notice that for any $ t \in [-T',T']$ and any $ x \in \R^{n}$ we have 
$$
\tilde{E}_{T'}(V_{i}) (x,t) = V_{i}(x,t).
$$
Therefore, $ \tilde{E}_{T'}(V_{i})$ is an extension for $V_{i}$ and hence
$$
\Vert V_{i} \Vert_{F^{\frac{n-2}{2}}(T')} \leq \Vert \tilde{E}_{T'}(V_{i}) \Vert_{F^{\frac{n-2}{2}}}.
$$
Then we can use the linear estimates, Lemma \ref{LinearEstimates}, for each $ i \in \{ 1,...,n \}$ to obtain
\begin{equation}\label{hgf1}
\begin{split}
& \Vert V_{x} \Vert_{F^{\frac{n-2}{2}}(T')}
\\
& \leq C_{1} \epsilon_{3} + C_{2} \Vert \mathcal{N}(V_{x},e_{1},e_{2}) \Vert_{N^{\frac{n-2}{2}}(T')}
\end{split}
\end{equation}
Now to prove smallness, we assume first that 
$$
\Vert \mathcal{N}(V_{x},e_{1},e_{2}) \Vert_{N^{\frac{n-2}{2}}(T')} \leq C_{1} \epsilon_{3}.
$$
If this is the case, then we obtain that 
$$
\Vert V_{x} \Vert_{F^{\frac{n-2}{2}}(T')} \leq C_{1} ( 1+ C_{2})\epsilon_{3}. 
$$
Let $ \epsilon_{3}$ be small enough so that $ C_{1}(1+C_{2}) \epsilon_{3} \leq \epsilon_{1}$, and hence from Lemma \ref{EcontroledbyV} we get smallness of $ |\nabla|e_{1}$, $ |\nabla|e_{2}$.
Therefore, we can use the nonlinear estimates in Lemma \ref{finalnonlinearestimates} to obtain 
\begin{equation*}
\begin{split}
& \Vert \mathcal{N}(V_{x},e_{1},e_{2} ) \Vert_{N^{\frac{n-2}{2}}(T')} \leq  C_{3} \epsilon^{2}_{3} \Vert V_{x} \Vert_{F^{\frac{n-2}{2}}(T')}
\end{split}
\end{equation*}
using the linear estimates again we obtain
\begin{equation*}
\begin{split}
& \Vert \mathcal{N}(V_{x},e_{1},e_{2} ) \Vert_{N^{\frac{n-2}{2}}(T')}
\\
& \leq C_{3} \epsilon^{2}_{3} \left( C_{1} \epsilon_{3} + C_{2} \Vert \mathcal{N}(V_{x} ,e_{1},e_{2} ) \Vert_{N^{\frac{n-2}{2}}(T')} \right)
\end{split}
\end{equation*}
Hence, if $ \epsilon_{3}$ is taken small enough then we obtain 
$$
\Vert \mathcal{N}(V_{x},e_{1},e_{2} ) \Vert_{N^{\frac{n-2}{2}}(T')} \leq \frac{ C_{1} \epsilon_{3}}{2}.
$$
Now using lemma \ref{continuityofL} we obtain that $L^{\frac{n-2}{2}}(r) \leq C_{1} \epsilon_{3}$ for all $r \in [0,T']$ and hence  
$$
\Vert V_{x} \Vert_{F^{\frac{n-2}{2}}(T')} \lesssim \epsilon_{3}.
$$
\end{proof}
Next, we prove that if we take $ 0 < \epsilon_{4} (\sigma_{0},n,s) \leq \epsilon_{3} $ small enough, then we can obtain an analogous result to Lemma \ref{smallnessofvx} but for any $ \sigma \in [ \frac{n-2}{2},\sigma_{0} -1]$.  
\begin{lemma}\label{higherregularityfcontrolofv}
There exists $ \epsilon_{4} : = \epsilon_{4}(n,s,\sigma_{0}) > 0$ so that whenever
$$
\Vert u_{0} \Vert_{\dot{B}^{\frac{n}{2}}_{Q}} \leq \epsilon_{4}
$$
then
for any $ \sigma \in [ \frac{n-2}{2}, \sigma_{0} -1]$ and any $T' \in [0,T_{0}]$ which satisfies 
$$
\vert u(x,t) - Q \vert \leq 1/2 \quad \text{ for all } (x,t) \in \R^{n} \times [-T',T']
$$
we have 
$$
\Vert V_{x} \Vert_{F^{\sigma}(T')} \lesssim \Vert u|_{t=0} \Vert_{\dot{B}^{\sigma+1}_{Q}}.
$$
\end{lemma}
\begin{proof}
We use the Duhmal extension $ \tilde{E}_{T'}$, the linear estimates in Lemma \ref{LinearEstimates}, and Lemma \ref{controlvu} to obtain
\begin{equation*}
\begin{split}
& \Vert V_{x} \Vert_{F^{\sigma}(T')} \leq C_{1,\sigma_{0}} \Vert u_{0} \Vert_{\dot{B}^{\sigma+1}_{Q}}    + C_{2,\sigma_{0}} \Vert \mathcal{N}(V_{x},e_{1},e_{2} ) \Vert_{N^{\sigma}(T')}
\end{split}
\end{equation*}
Assume first that 
$$
L^{\sigma}(T') = \Vert \mathcal{N}(V_{x},e_{1},e_{2} ) \Vert_{N^{\sigma}(T')} \leq C_{1,\sigma_{0}} \Vert u_{0} \Vert_{\dot{B}^{\sigma+1}_{Q}}
$$
then we obtain 
$$
\Vert V_{x} \Vert_{F^{\sigma}(T')} \leq ( C_{1,\sigma_{0}} +C_{2,\sigma_{0}}) C_{1,\sigma_{0}} \Vert u_{0} \Vert_{\dot{B}^{\sigma+1}_{Q}} . 
$$
Next, choose $ \epsilon_{4}$ small enough so that the conclusions of Lemma \ref{smallnessofvx}, Lemma \ref{EcontrolledVhigherregularity} holds. Then using the nonlinear estimates we obtain the following 
\begin{equation*}
\begin{split}
& \Vert \mathcal{N}(V_{x},e_{1},e_{2} ) \Vert_{N^{\sigma}(T')}
\\
& \leq C \epsilon^{2}_{4} \Vert V_{x} \Vert_{F^{\sigma}(T')}
\end{split}
\end{equation*}
But applying the linear estimates again we easily obtain
\begin{equation*}
\begin{split}
& \Vert \mathcal{N}(V_{x},e_{1},e_{2} ) \Vert_{N^{\sigma}(T')} \leq \frac{C_{1,\sigma_{0}} \Vert u_{0} \Vert_{\dot{B}^{\sigma+1}_{Q}}}{2}.
\end{split}
\end{equation*}
Hence, by Lemma \ref{continuityofL} we obtain 
$$
 \Vert \mathcal{N}(V_{x},e_{1},e_{2} ) \Vert_{N^{\sigma}(T')} \leq C_{1,\sigma_{0}} \Vert u_{0} \Vert_{\dot{B}^{\sigma+1}_{Q}}.
$$
\end{proof}
As a corollary of the above we obtain the following
\begin{corollary}\label{keyestimatesforfixedpoint}
If the initial data $ u_{0}$ satisfy 
$$
\Vert u_{0} \Vert_{\dot{B}^{\frac{n}{2}}_{Q}} \leq \epsilon_{4}
$$
then for any $ \sigma \in [\frac{n-2}{2},\sigma_{0}-1 ]$ and any $T'$ which satisfies 
$$
\vert u(x,t) - Q \vert \leq 1/2 \quad \text{ for all } (x,t) \in \R^{n} \times [-T',T']
$$
we have 
$$
\sup_{t \in [-T',T']} \Vert V_{x} \Vert_{ \dot{B}^{\sigma}_{2,1}} \lesssim \Vert u_{0} \Vert_{\dot{B}^{\sigma+1}_{Q}}.
$$
\end{corollary}
Now we prove the following apriori estimate
\begin{proposition}\label{shorttimeexs}
Let $n \geq 5$, $ s \in (1/2,1) $. Then for any $ u_{0} \in B^{\sigma_{0} + 1}_{Q}$ with 
$$
\Vert u_{0} \Vert_{\dot{B}^{n/2}_{Q}} \leq \epsilon_{4}
$$
we have the following: given $T \in (0,1]$ and any 
$$
u \in C([-T,T] : B^{\sigma_{0}}_{Q}) \cap L^{\infty} ([-T,T]: \dot{B}^{\sigma_{0}+1)}_{Q})
$$ which solves \eqref{SSchrodingereq} on the time interval $[-T,T]$ with initial data $u_{0}$, we have  
\begin{equation}\label{uniformbound}
\sup_{t \in [-T,T]} \Vert u \Vert_{B^{\sigma_{0}}_{Q}} \lesssim \Vert u_{0} \Vert_{{B}^{\sigma_{0}}_{Q}},
\end{equation}
with constant independent of $T$. Moreover, for any $ \sigma \in [ \frac{n}{2}, \sigma_{0} ]$ we have 
\begin{equation}\label{uniformbound2}
\sup_{t \in [-T,T]} \Vert u \Vert_{\dot{B}^{\sigma}_{Q}} \lesssim \Vert u_{0} \Vert_{\dot{B}^{\sigma_{}}_{Q}},
\end{equation}
with constant independent of $T$. And lastly, for any $ \sigma \in [\frac{n}{2}, \sigma_{0} ]$ we have
$$
u \in C([-T,T] : \dot{B}^{\sigma}_{Q} ).
$$

\end{proposition}
\begin{proof}
First, notice the conservation of energy law 
$$
\Vert u(\cdot,t) -Q \Vert_{L^{2}(\R^{n})} = \Vert u|_{t=0}  -Q \Vert_{L^{2}(\R^{n})},
$$
for any $ t \in [-T,T]$. 
Indeed, this follows from integration by parts
\begin{equation*}
\begin{split}
& \partial_{t} \left( \Vert u(\cdot,t) -Q \Vert^{2}_{L^{2}(\R^{n})}\right)
\\
& = 2 \int_{\R^{n}} \langle u(x,t) - Q , \partial_{t} u(x,t) \rangle dx
\\
& = -2 \int_{\R^{n}} \langle u(x,t) - Q , u(x,t) \wedge (- \Delta)^{s} u(x,t) \rangle dx
\\
& = -2 \int_{\R^{n}} \langle u(x,t) -Q, Q \wedge ( - \Delta)^{s} u(x,t) \rangle dx 
\\
& = 2 \int_{\R^{n}} \langle \left( u(x,t)-Q \right) \wedge Q , (- \Delta)^{s}( u(x,t) -Q) \rangle dx
\\
&= 2 \int_{\R^{n}} \langle (-\Delta)^{s/2} \left( u(x,t) - Q\right) )\wedge Q, (-\Delta)^{s/2}( u(x,t) -Q) \rangle dx
\\
&=0
\end{split}
\end{equation*}
Hence, 
$$
\Vert u(\cdot,t) -Q\Vert_{L^{2}(\R^{n})} = \Vert u|_{t=0} -Q \Vert_{L^{2}(\R^{n})},
$$
for any $ t \in [-T,T]$. Therefore, to prove \eqref{uniformbound} it suffices to prove \eqref{uniformbound2}, namely for any $ \sigma \in [\frac{n}{2},\sigma_{0}]$ we have to prove
\begin{equation}\label{uniformbound1}
\sup_{t \in [-T,T]} \Vert u \Vert_{\dot{B}^{\sigma}_{Q}} \lesssim \Vert u|_{t=0} \Vert_{\dot{B}^{\sigma}_{Q}}.
\end{equation}
First,  assume $ | u(x,t) - Q | \leq 1/2$ for any $ (x,t) \in \R^{n} \times [-T,T]$. Then
in view of Lemma \ref{EcontrolledVhigherregularity} we have 
$$
\Vert |\nabla| u \Vert_{\tilde{F}^{\sigma_{} -1}(T')} \lesssim \Vert V_{x} \Vert_{F^{\sigma_{} -1}(T')}
$$
for any $T' \in [0,T]$. Using Lemma \ref{higherregularityfcontrolofv} we obtain
\begin{equation}\label{cnus}
  \Vert |\nabla| u \Vert_{\tilde{F}^{\sigma_{} -1}(T')} \lesssim \Vert u|_{t=0} \Vert_{\dot{B}_{Q}^{\sigma_{}}},  
\end{equation}
for any $ T' \in [0,T]$.
Now using the energy estimate, \eqref{energy}, we end up with \eqref{uniformbound2} and hence \eqref{uniformbound1}. \\
This proves \eqref{uniformbound2}, and \eqref{uniformbound}  under the assumption that $ | u(x,t) - Q | \leq 1/2$ for any $ (x,t) \in \R^{n} \times [-T,T]$. If this is not the case, then we restrict to a smaller interval $[-\tilde{T},\tilde{T}]$ so that $ | u(x,t) - Q | \leq 1/2$ for any $ (x,t) \in \R^{n} \times [- \tilde{T},\tilde{T}]$. Then repeating the argument on the smaller interval $[-\tilde{T},\tilde{T}]$ we obtain 
$$
\sup_{t \in [-\tilde{T},\tilde{T}]} \Vert u - Q \Vert_{\dot{B}^{\frac{n}{2}}_{2,1}} \leq C \epsilon_{4}
$$
if $ \epsilon_{4} $ is chosen small enough. Then this implies 
$$
\sup_{t \in [-\tilde{T},\tilde{T}]}\Vert u - Q \Vert_{L^\infty_{x}} \leq \frac{1}{4}
$$
this readily implies that $ \vert u(x,t) - Q \vert \leq 1/2$ for any $ (x,t) \in \R^{n} \times [-T,T]$. Which proves \eqref{uniformbound2}, and \eqref{uniformbound}. Lastly, in view of the embedding $ F^{\sigma} \hookrightarrow C ( \R : B^{\sigma}_{2,1} ) $ and \eqref{cnus} we obtain that $ u \in C ( [-T,T ] : \dot{B}^{\sigma}_{Q})$.
\end{proof}
An immediate consequence of the above is the following 
\begin{corollary}\label{suffc}
Let $ n \geq 5$, $ s \in (1/2,1)$. Then there exists $ \epsilon : = \epsilon(n,s, \sigma_{0}) >0$ so that the following holds; for any $ u_{0} \in B^{\sigma_{0} + 1}_{Q}$ with
$$
\Vert u_{0} \Vert_{\dot{B}^{\frac{n}{2}}_{Q}} \leq \epsilon
$$
there exists $u \in C ( [-1,1]: \dot{B}_{Q}^{\sigma_{0}}) \cap L^{\infty} ([-1,1] : \dot{B}^{\sigma_{0}+1}_{Q}) $ which solves \eqref{SSchrodingereq} on the time interval $ [-1,1]$, and satisfies 
$$
\sup_{t \in [-1,1]} \Vert u \Vert_{\dot{B}^{n/2}_{Q}} \lesssim \Vert u_{0} \Vert_{\dot{B}^{n/2}_{Q}}.
$$
Moreover, for any $ \sigma \in [\frac{n}{2},\sigma_{0}]$ we have 
$$
\sup_{t \in [-1,1]} \Vert u \Vert_{\dot{B}^{\sigma}_{Q}} \lesssim \Vert u_{0} \Vert_{\dot{B}^{\sigma}_{Q}}.
$$
\end{corollary}

\begin{proof}
Take $ \epsilon = \epsilon_{4}$, where $ \epsilon_{4} $ is the same as in Proposition \ref{shorttimeexs}. Then
using Proposition \ref{localexistencethm} we can find $T_{0}:= T_{0}(\sigma_{0}, \Vert u_0 \Vert_{{B}^{\sigma}_{Q}})> 0$, and $u$ so that $ u $ is a solution of \eqref{SSchrodingereq} on the time interval $[-T_{0},T_{0}]$. Now using \eqref{uniformbound} this solution can be extended to $[-T_{0}-T,T_{0}+T]$ with $T>0$ independent of $T_{0}$. Hence, $u $ can be extended to a solution on $[-1,1]$. Lastly, the uniform estimates follow from Proposition \ref{shorttimeexs}.
\end{proof}

Now we can prove Theorem \ref{mainthm}
\begin{proof}[Proof of Theorem \ref{mainthm}]
In view of Corollary \ref{suffc}, 
it suffices to show that given a solution $u$ of \eqref{SSchrodingereq}  on $ \R^{n} \times [-R,R]$, with $R> 1$ then we have 
\begin{equation}\label{finalpriori}
    \sup_{t \in [-R,R]} \Vert u \Vert_{\dot{B}^{\sigma}_{Q}} \lesssim \Vert u_{0} \Vert_{\dot{B}^{\sigma}_{Q}}
\end{equation}
for any $ \sigma \in [\frac{n}{2},\sigma_{0}]$,
with constants independent of $R$. 
Indeed, once we prove \eqref{finalpriori}, we obtain Theorem \ref{mainthm} as follows; Take $ \epsilon \ll \epsilon_{4}$, and by Corollary \ref{suffc} we can find  $u^{(1)}$ that solves \eqref{SSchrodingereq} on the time interval $ [-1,1]$ and satisfies 
\begin{equation}\label{constc0}
\sup_{t \in [-1,1]} \Vert u \Vert_{\dot{B}^{\frac{n}{2}}_{Q}} \leq C_{0} \epsilon.
\end{equation}
If $ \epsilon > 0$ was chosen so that 
$$
C_{0} \epsilon < \epsilon_{4},
$$
then we have 
$$
\Vert u^{(1)} ( \cdot,1) \Vert_{\dot{B}^{n/2}_{Q}} \leq \epsilon_{4}
$$
then we can apply Corollary \ref{suffc} to the initial data $ u^{(1)}(\cdot,1)$, to obtain $u_{+}^{(2)}$ which is a solution of \eqref{SSchrodingereq} on the time interval $[-1,1]$ but with initial data
$$
u_{+}^{(2)}(\cdot,0) = u^{(1)}(\cdot,1).
$$
We can apply the same argument to obtain a solution $u^{(2)}_{-}$ of \eqref{SSchrodingereq} on the time interval $[-1,1]$ but with initial data $ u^{(1)}(\cdot,-1)$. Now we construct $ u^{(2)}$ as follows
$$
u^{(2)}(x,t) := \begin{cases}
    u^{(1)}(x,t)  ~~\text{  , if } t \in [-1,1]
    \\
    u^{(2)}_{+}(x,t-1) ~~ \text{ , if } t \in [1,2]
    \\
     u^{(2)}_{-}(x,t+1) ~~ \text{  , if } t \in [-2,-1]
\end{cases}
$$
Then notice that $u^{(2)}$ solves \eqref{SSchrodingereq} with initial data $ u_{0}.$ Now assuming we have proven \eqref{finalpriori}, then we obtain
$$
\sup_{t \in [-2,2]} \Vert u^{(2)} (\cdot,t) \Vert_{\dot{B}^{n/2}_{Q}} \leq C \epsilon
$$
And if $ \epsilon >0$ was chosen so that $ C \epsilon < \epsilon_{4}$ then we can repeat the process to obtain a solution $u^{(3)}$ on the time interval $[-3,3]$, which satisfy 
$$
\sup_{t \in [-3,3]} \Vert u^{(3)} (\cdot,t) \Vert_{\dot{B}^{n/2}_{Q}} \leq C \epsilon
$$
 And lastly, since $C$ is the same constant in each iteration, then we can extend $u $ to all of $\R^{n} \times \R$. Thus, it remains to prove \eqref{finalpriori}.
\\
To that end, assume $u$ is a solution of \eqref{SSchrodingereq} on the time interval $[-R,R]$. Then  define $v$ as follows 
$$
v(x,t) : = {u}( R^{1/2s} x , R t)  ~~~\text{  ,for } (x,t) \in \R^{n} \times [-1,1]
$$
Then notice that $ v $ solves \eqref{SSchrodingereq} with initial data $ u_{0}( R^{1/2s}\cdot)$ on the time interval $[-1,1]$. Further, 
notice 
$$
\Vert u_{0}(R^{1/2s} \cdot) \Vert_{\dot{B}^{n/2}_{Q}} \approx \Vert u_{0} \Vert_{\dot{B}^{n/2}_{Q}} \leq \epsilon.
$$
Hence, using Proposition \ref{shorttimeexs} we obtain 
\begin{equation}\label{maay}
\sup_{t \in [-1,1]} \Vert v \Vert_{\dot{B}_{Q}^{\sigma}} \lesssim \Vert u_{0}( R^{1/2s} \cdot) \Vert_{\dot{B}^{\sigma}_{Q}} \approx_{n,s,\sigma} ( R^{1/2s})^{(\sigma - \frac{n}{2})}\Vert u_{0} \Vert_{\dot{B}^{\sigma}_{Q}}
\end{equation}
but for any $t \in [-1,1]$ we have 
\begin{equation}\label{maay2}
 \Vert v(\cdot,t) \Vert_{\dot{B}_{Q}^{\sigma_{}}} \approx (R^{1/2s})^{(\sigma_{} - \frac{n}{2})} \Vert {u}(\cdot,R t) \Vert_{\dot{B}_{Q}^{\sigma_{}}} 
\end{equation}
Combining \eqref{maay} and \eqref{maay2} we obtain 
$$
\sup_{t \in [-R,R]} \Vert u \Vert_{\dot{B}^{\sigma_{}}_{Q}} \lesssim \Vert u_{0} \Vert_{\dot{B}^{\sigma_{}}_{Q}}.
$$
\end{proof}

\appendix 
\section{Taylor estimates}\label{s:taylor}
We start with the following elementary estimates. This is implicitly contained in the proof of \cite[Lemma 5.1]{subcritical} but we include the proof here for completeness.
\begin{lemma}\label{Taylorcalc}
Let $ \xi \in \R^{n} \setminus \{ 0 \}$ be fixed with $ |\xi | \approx 2^{k}$, $k \in \Z$. Then for any $ \eta \in \R^{n} $ with $ | \eta | \ll | \xi|$ we have the following identity 
\begin{equation}\label{Taylorcal111}
\vert \xi \vert^{2s} - \vert \xi - \eta \vert^{2s} = \sum_{ | \alpha | =1}^{\infty} m^{i}_{\alpha}( \xi) | \xi |^{2s - | \alpha|} \frac{\eta^{\alpha}}{\alpha !}
\end{equation}
Similarly, we have, for any $i \in \{1,...,n \}$ 
\begin{equation}\label{Taylorcal22}
\xi_{i} | \xi |^{2s-2} - (\xi - \eta)_{i} | \xi - \eta|^{2s-2} = \sum_{ | \alpha | =1}^{\infty} m^{2}_{\alpha}( \xi) | \xi |^{2s-1 - | \alpha|} \frac{\eta^{\alpha}}{\alpha !}
\end{equation}
and lastly we have, for any $i,j \in \{1,...,n \}$ the following 
\begin{equation}\label{Taylorcal33}
\xi_{i} \xi_{j} | \xi |^{2s-2} - (\xi - \eta)_{i}(\xi-\eta)_{j} | \xi - \eta|^{2s-2} = \sum_{ | \alpha | =1}^{\infty} m^{3}_{\alpha}( \xi) | \xi |^{2s - | \alpha|} \frac{\eta^{\alpha}}{\alpha !}.
\end{equation}
Here $m^{i}_{\alpha} $ satisfy the following; there exists a constant $C_{0} : = C_{0}(n) \geq 1$ so that 
\begin{equation}\label{qpppa}
\sum_{|\alpha|=1}^{\infty} \frac{ \Vert \mathcal{F}^{-1} ( \varphi_{k} m^{i}_{\alpha}) \Vert_{L^{1}}}{ \alpha !} 2^{ -C_{0}  |\alpha |} < \infty.    
\end{equation}

\end{lemma}
\begin{proof}
fix such a $ \xi$ and define the function $ h: \R^{n} \to \R$ as follows
$$
h(\eta) : = \vert \xi + \eta \vert^{2s}
$$
since this function (as a function of $\eta$ in the domain $ |\eta | \ll |\xi |$) is real analytic at $0$ we can use Taylor theorem to obtain 
\begin{equation}\label{c2}
\vert \xi \vert^{2s} - \vert \xi - \eta \vert^{2s} = -\sum_{ \vert \alpha \vert = 1}^{\infty} \frac{1}{ \alpha !} D^{\alpha}h(0) (-\eta)^{\alpha}
\end{equation}
We need to compute $ D^{\alpha}h(0)$. 
Write $ D^{\alpha} = \partial_{i_{|\alpha|}} .... \partial_{i_{1}} $ with each $ i_{k} \in \{ 1,...,n\}$. To write $D^{\alpha}h(0)$ in closed form, we need to introduce the following notation:
Let $k \in [ 0,...,[\frac{|\alpha|}{2}] ]$, $[\frac{|\alpha|}{2}]$ is the greatest integer less than $\frac{|\alpha|}{2}$, and consider the set $P_{k}$ as the set of all distinct $k$ pairings of indices of $\{ i_{1},...,i_{|\alpha|} \}$. In other words, 
$$
P_{k}: = \{ \{ \{ i_{t_{1}},i_{r_{1}}\},...,\{i_{t_{k}},i_{r_{k}}\} \} : t_{q} \neq r_{q} \text{ and } \{{t_{q}},{r_{q}}\} \cap \{ {t_{q'}},{r_{q'}} \} = \emptyset \}.
$$
Simple computations give us that the cardinality of $P_{k}$ is 
$$
|P_{k} | = \frac{ |\alpha|!}{(|\alpha|-2k)! 2^{k} k!}.
$$
Next, if $p \in P_{k}$ then $p = \{ \{i_{t_{1}},i_{r_{1}}\},..., \{i_{t_{k}}, i_{r_{k}} \} \}$ with $i_{t},i_{r}$ satisfying the above constraints. Then define   
$$
S(p) = \{ i_{1},...,i_{\vert \alpha \vert} \} \setminus \{ i_{t_{1}} ,i_{r_{1}},....,i_{t_{k}}, i_{r_{k}} \}.
$$
then $S(p)$ has cardinality $ |\alpha| -2k$. Lastly, define for any $j \in S(p)$ the following 
$$
R^{j}: = \frac{\xi_{j}}{\vert \xi \vert}
$$
and for any $p \in P_{k}$ define $ \delta^{p} $ as the product of the corresponding kronecker deltas. More precisely,
$$
\delta^{p} : = \delta_{i_{q_{1}},i_{r_{1}}}.... \delta_{i_{q_{k}},i_{r_{k}}}.
$$
With this notation we claim that $D^{\alpha} h(0)$ is of the following form
\begin{equation}\label{www4}
\begin{split}
& D^{\alpha} h(0)
\\
& = \vert \xi \vert^{2s - |\alpha|} \big[C^{0}_{s,|\alpha|} R^{i_{1}}...R^{i_{|\alpha|}}+ \sum_{k=1}^{ [|\alpha|/2]} C^{k}_{s,|\alpha|} \sum_{p \in P_{k}} \delta^{p}  \prod_{j \in S(p)} R^{j}  \big]
\end{split}
\end{equation}
where
$$
C^{j}_{s,|\alpha|} : = 2s ( 2s-2)(2s-4)....(2s-2(|\alpha|-j-1)).
$$
This formula can be proven via induction. If $ | \alpha | =1 $ then $ D^{\alpha} = \partial_{i} $ for some $ i \in \{1,...,n \}$ and clearly we have 
$$
D^{\alpha} h(0) = 2s \vert  \xi  \vert^{2s-1} \frac{ \xi_{i}}{\vert \xi \vert}
$$
which is consistent with \eqref{www4}. Assume the formula \eqref{www4} holds for any multi-index with $ \vert \alpha \vert = m$. Consider a multi-index $ \alpha $ with $ | \alpha | = m+1$. Then write $ D^{\alpha} = \partial_{i_{m+1}} D^{\alpha'} $ where $ |\alpha' | = m$. Then by the induction hypothesis we have
\begin{equation}\label{www5}
\begin{split}
& D^{\alpha} h(0) 
\\
&= \partial_{i_{m+1}} ( \vert \xi \vert^{2s - m} \big[C^{0}_{s,m} R^{i_{1}}...R^{i_{m}}+ \sum_{k=1}^{ [m/2]} C^{k}_{s,m} \sum_{p \in P_{k}} \delta^{p}  \prod_{j \in S(p)} R^{j}  \big]
 )
 \\
 & = \partial_{i_{m+1}} ( \vert \xi \vert^{2s-2m} C_{s,m}^{0} \prod_{k=1}^{m} \xi_{i_{k}} ) + \partial_{i_{m+1}} ( \vert \xi \vert^{2s-m}  \big[ \sum_{k=1}^{[m/2]} C_{s,m}^{k} \sum_{p \in P_{k}} \delta^{p} \prod_{j \in S(p)} R^{j} \big]
 \\
 & = \vert \xi \vert^{2s-m-1}  C^{0}_{s,m+1} R^{i_{1}}...R^{i_{m+1}}
  + \vert \xi \vert^{2s-m-1} \sum_{ a=1}^{m} C^{1}_{s,m+1} \delta_{i_{m+1},i_{a}} \prod_{ j \in S ( \{ i_{m+1},i_{a} \})} R^{j} 
 \\
 & + \sum_{k=1}^{ [m/2]} C_{s,m}^{k} \sum_{p \in P_{k}} \delta^{p} \partial_{i_{m+1}}(  \vert \xi \vert^{2s-m} \prod_{j \in S(p)} R^{j} ) 
\end{split}
\end{equation}
and notice that for any $ p \in P_{k}$ we have 
\begin{equation}\label{www6}
\begin{split}
& \partial_{i_{m+1}} ( \vert \xi \vert^{2s -m } \prod_{j \in S(p)} R^{j} )
\\
& = \partial_{i_{m+1}} ( \vert \xi \vert^{2s - 2m +2k} \xi_{i_{q_{1}}} .... \xi_{q_{m-2k}} ) 
\\
& = (2s-2m+2k) \vert \xi \vert^{2s-m-1} R^{i_{q_{1}}}...R^{i_{q_{m-2k}}} R^{i_{m+1}}  + \vert \xi \vert^{2s-m-1} \sum_{j=1}^{m-2k} \delta_{i_{m+1}, i_{q_j}} \prod_{t \in S(p)\setminus \{ i_{q_{j}} \} } R^{t}
\end{split}
\end{equation}
plugging \eqref{www6} back into \eqref{www5} yields that $ D^{\alpha} h(0)$ is of the form given in \eqref{www4}. Therefore, the pointwise identity, \eqref{Taylorcal111}, holds with $m_{\alpha}$ being 
$$
m_{\alpha}(\xi) : = (-1)^{|\alpha|+1} \big[C^{0}_{s,|\alpha|} R^{i_{1}}...R^{i_{|\alpha|}}+ \sum_{k=1}^{ [|\alpha|/2]} C^{k}_{s,|\alpha|} \sum_{p \in P_{k}} \delta^{p}  \prod_{j \in S(p)} R^{j}  \big]
$$
Clearly, we have 
$$
\Vert \mathcal{F}^{-1}( \varphi_{k_{1}} m_{\alpha}) \Vert_{L^{1}} \lesssim C^{|\alpha|} a_{|\alpha|} 
$$
for any $k_{1} \in \Z$.
Here
\begin{equation}
a_{|\alpha|} : = \sum_{k=0}^{[|\alpha|/2]} | C_{s,|\alpha|}^{k} | | P_{k}| 
\end{equation}
and it remains to show that the series has the claimed convergence. To that end, we first bound $a_{|\alpha|}$. By definition, we had 
$$
a_{|\alpha|} = \sum_{k=0}^{[|\alpha|/2]} |C^{k}_{s,|\alpha|}| \vert P_{k} \vert
$$
We bound $C^{k}_{s,|\alpha|}$. Recall that 
$C^{k}_{s,|\alpha|} = 2s (2s-2) ....(2s-2(|\alpha|-k-1))$.
an obvious bound is 
$$
| C^{k}_{s,|\alpha|}| \lesssim_{s} 2^{|\alpha|-k} (|\alpha|- k)!
$$
and $|P_{k} |$ was explicit, we had
$$
|P_{k}| = \frac{|\alpha|!}{(|\alpha|-2k)! 2^{k} k!}
$$
therefore,
$$
|C^{k}_{s,|\alpha|}| |P_{k} | \lesssim |\alpha|!  \frac{2^{|\alpha|-2k}}{k!} \frac{ (|\alpha|-k)!}{(|\alpha|-2k)!}
$$
hence,
$$
\frac{a_{|\alpha|} 2^{-C_{0}|\alpha|}}{\alpha !} \lesssim  \frac{|\alpha|!}{\alpha !} 2^{(-C_{0}+1)|\alpha|} \sum_{k=0}^{[|\alpha|/2]} \frac{2^{-2k}}{k!} \frac{(|\alpha|-k)!}{(|\alpha|-2k)!}
$$
using the bound $ \frac{|\alpha|!}{\alpha !} \lesssim (n)^{|\alpha|}$ then we obtain
\begin{equation}
    \frac{a_{|\alpha|} 2^{-C_{0}|\alpha|}}{\alpha !} \lesssim  (n)^{|\alpha|} 2^{(-C_{0}+1)|\alpha|} \sum_{k=0}^{[|\alpha|/2]} \frac{2^{-2k}}{k!} \frac{(|\alpha|-k)!}{(|\alpha|-2k)!}
 \end{equation}
the right hand side can be bounded by 
$$
(n 2^{(-C_{0} +1)} e)^{|\alpha|}.
$$
Thus, we obtain
$$
\frac{ \Vert \mathcal{F}^{-1} ( \varphi_{k} m^{i}_{\alpha}) \Vert_{L^{1}}}{ \alpha !} 2^{ -C_{0}  |\alpha |} \leq (C n 2^{(-C_{0} +1)} e)^{|\alpha|}.
$$
If $C_{0}$ was chosen so that $ C_{0} \gg C n$ then summing $|\alpha|$, the series converges.\\
This prove \eqref{Taylorcal111} with the stated convergence. The proofs for \eqref{Taylorcal22}, and \eqref{Taylorcal33} follow by the same argument.

\end{proof}

The following lemma is an immediate consequence of Lemma \ref{Taylorcalc}, and \eqref{multiplierint}.

\begin{lemma}\label{Taylorstuff}
Let $ f,g \in L^{2}(\R^{n+1}), $ $k_{1} ,k_{2} \in  \Z$ with $ k_{2} \leq k_{1} -C_{0} $, and assume $ \hat{f} $ is supported in $ \{ ( \xi,\tau) : \vert \xi \vert  \in [2^{k_{1}-1},2^{k_{1} +1}] \} \times \R$ and $ \hat{g} $ is supported in $\{ \vert \xi \vert \in [2^{k_{2}-1},2^{k_{2}+1} ]\} \times \R $. 
Then for any $p,q \in [1,\infty], s \in (0,1)$ we have the following
$$
\Vert \Delta_{k_{1}} [(-\Delta)^{s},g](f) \Vert_{L^{p}_{e}L^{q}_{e^{\perp},t}} \lesssim \sum_{\vert \alpha \vert =1}^{\infty} \frac{a_{\vert \alpha \vert}}{\alpha !} 2^{k_{1} (2s -\vert \alpha \vert)} \Vert f D^{\alpha}(g) \Vert_{L^{p}_{e} L^{q}_{e^{\perp},t}}
$$
where the constants $ a_{\vert \alpha \vert}$ satisfy the following
$$
\sum_{ \vert \alpha \vert =1}^{\infty} \frac{ a_{ \vert \alpha \vert}}{\alpha !} 2^{-C_{0}  \vert \alpha \vert} < \infty.
$$
\end{lemma}
\begin{proof}

We rewrite $[(-\Delta)^{s},g](f)$ using the Taylor expansion in Lemma \ref{Taylorcalc}, and using \eqref{multiplierint} we obtain 
\begin{equation}\label{www3}
\begin{split}
& \Vert \Delta_{k_{1}}[(-\Delta)^{s},g](f) \Vert_{L^{p}_{e}L^{q}_{e}}
\\
& \lesssim \sum_{\vert \alpha \vert =1}^{\infty} \frac{a_{\alpha}}{\alpha !} 2^{k_{1}(2s-\vert \alpha \vert )} \Vert f D^{\alpha} g \Vert_{L^{p}_{e} L^{q}_{e^{\perp},t}}
\end{split}
\end{equation}
where 
\begin{equation}
a_{|\alpha|} : =  \Vert \mathcal{F}^{-1} ( \varphi_{k} m^{1}_{\alpha}) \Vert_{L^{1}}
\end{equation}
and we conclude by \eqref{qpppa}.

\end{proof}
Next, we will need the following Taylor estimates for the commutator of the operator $ \nabla^{2s-1}_{i}$ for any $ i \in \{ 1,..,n\}$. 
\begin{lemma}\label{Taylorstuff2}
Let $ s \in (0,1) $ and $ n \geq 3$, and $ f,g $ are as in Lemma \ref{Taylorstuff}. Then we have the following inequality for any $p,q \in [1,\infty]$
\begin{equation*}
\Vert \Delta_{k_{1}} [ \nabla^{s}_{i}, g ](f) \Vert_{L^{p}_{e} L^{q}_{e^{\perp},t}} \lesssim \sum_{| \alpha |=1}^{\infty} \frac{a_{|\alpha|}}{\alpha !} 2^{k_{1} (s - |\alpha|) } \Vert f D^{\alpha} (g) \Vert_{L^{p}_{e} L^{q}_{e^{\perp},t}}
\end{equation*}
where the constants $ a_{|\alpha|}$ satisfy the convergence given in Lemma \ref{Taylorstuff}.
\end{lemma}
And lastly, we will need the following
\begin{lemma}\label{Taylorstuff3}
Let $ s \in (\frac{1}{2},1) $ , $ n \geq 3$ and $ f,g $ as in Lemma \ref{Taylorstuff}. Then define $ \nabla^{2s-1}_{i,j} : = \partial_{j} \nabla^{2s-1}_{i}$. Then for any $ p,q \in [1,\infty]$ we have the following 
$$
\Vert \Delta_{k_{1}} [ \nabla^{2s-1}_{i,j} ,g ]( f) \Vert_{L^{p}_{e } L^{q}_{e^{\perp},t}} \lesssim \sum_{ \vert \alpha \vert =1}^{\infty} \frac{a_{\vert \alpha \vert}}{\alpha !} 2^{k_{1} (2s -\vert \alpha \vert)} \Vert f D^{\alpha}(g) \Vert_{L^{p}_{e}L^{q}_{e^{\perp},t}}
$$
\end{lemma}

\section{Elementary estimates}\label{s:elementary}

\subsection{Elementary facts about Besov spaces}
We start with the following elementary estimate
\begin{lemma}\label{elemantaryestbesov}
Let $ \sigma > 0$. Then we have 
\begin{equation}\label{adad}
  \Vert g \partial_{i} f \Vert_{\dot{B}^{\sigma}_{2,1}} \lesssim \Vert f \Vert_{L^{\infty}}  \Vert g \Vert_{\dot{B}^{\sigma+1}_{2,1}} + \Vert g \Vert_{L^{\infty}} \Vert f \Vert_{\dot{B}^{\sigma+1}_{2,1}}.  
\end{equation}

\end{lemma}
\begin{proof}
Notice that
\begin{equation}\label{corr1}
\begin{split}
& \sum_{k \in \Z } 2^{k \sigma} \Vert \Delta_{k} \left( g \partial_{i} f \right) \Vert_{L^{2}}
\\
& \lesssim \sum_{j \in \Z } \sum_{k : 2^{k} \approx 2^{j} } 2^{ k \sigma } \Vert \Delta_{k} \left( \Delta_{j} g \Delta_{\leq j-C} \partial_{i} f \right) \Vert_{L^{2}}
\\
& + \sum_{j \in \Z } \sum_{k : 2^{k} \approx 2^{j} } 2^{ k \sigma } \Vert \Delta_{k} \left( \Delta_{j} \partial_{i} f \Delta_{\leq j-C} g \right) \Vert_{L^{2}}
\\
& + \sum_{| j_{1} - j_{2} | \lesssim 1 } \sum_{k \leq j_{1}+C} 2^{k \sigma} \Vert \Delta_{k} \left( \Delta_{j_{1}} \partial_{i} f \Delta_{j_{2}}g \right) \Vert_{L^{2}}.
\end{split}
\end{equation}
We estimate each term on the right hand side of \eqref{corr1}. For the first term 
\begin{equation*}
\begin{split}
& 2^{\sigma j } \Vert \Delta_{k} \left( \Delta_{j}g  \Delta_{j-C} \partial_{i} f \right) \Vert_{L^{2}}
\\
& \lesssim 2^{(\sigma+1) j} \Vert \Delta_{ \leq j-C} f \Vert_{L^{\infty}} \Vert \Delta_{j} g \Vert_{L^{2}}
\\
& \lesssim \Vert f \Vert_{L^{\infty}} 2^{ (\sigma +1) j} \Vert \Delta_{j} g \Vert_{L^{2}}.
\end{split}
\end{equation*}
Therefore, summing over $j$ we obtained 
\begin{equation}\label{bzx1}
\begin{split}
& \sum_{j \in \Z } \sum_{k : 2^{k} \approx 2^{j} } 2^{ k \sigma } \Vert \Delta_{k} \left( \Delta_{j} g \Delta_{\leq j-C} \partial_{i} f \right) \Vert_{L^{2}}
\\
& \lesssim \Vert f \Vert_{L^{\infty}}  \Vert g \Vert_{\dot{B}^{\sigma+1}_{2,1}}.
\end{split}
\end{equation}
Similarly, we have 
\begin{equation}
\begin{split}
& \sum_{j \in \Z } \sum_{k : 2^{k} \approx 2^{j} } 2^{ k \sigma } \Vert \Delta_{k} \left( \Delta_{j} \partial_{i} f \Delta_{\leq j-C} g \right) \Vert_{L^{2}}
\\
&\lesssim \Vert g \Vert_{L^{\infty}} \Vert f \Vert_{\dot{B}^{\sigma+1}_{2,1}}.
\end{split}
\end{equation}
We need to estimate the third term on the right hand side of \eqref{corr1}. 
\begin{equation*}
\begin{split}
& 2^{\sigma k } \Vert \Delta_{k} \left( \Delta_{j_{2}}g  \Delta_{j_{1}} \partial_{i} f \right) \Vert_{L^{2}}
\\
& \lesssim 2^{\sigma k} 2^{j_{1}}  \Vert g \Vert_{L^{\infty}} \Vert \Delta_{j_{1}} f\Vert_{L^{2}}
\end{split}
\end{equation*}
this easily implies
\begin{equation*}
\begin{split}
&  \sum_{| j_{1} - j_{2} | \lesssim 1 } \sum_{k \leq j_{1}+C} 2^{k \sigma} \Vert \Delta_{k} \left( \Delta_{j_{1}} \partial_{i} f \Delta_{j_{2}}g \right) \Vert_{L^{2}}
\\
& \lesssim \Vert g \Vert_{L^{\infty}} \Vert f \Vert_{\dot{B}^{\sigma +1}_{2,1}}.
\end{split}
\end{equation*}

\end{proof}

We also need the following. The proof is based on \cite[Section 2]{Bahouri}. 
\begin{lemma}\label{Besovstuff}
\begin{enumerate}
    \item Let $ \theta : \R^{n} \to \R$ , 
and $ F : \R \to \R$ be smooth with $F(0) =0$. Then one has the following inequality for any $ \sigma >0$
\begin{equation}\label{Besovalgest1}
\Vert F \circ \theta \Vert_{\dot{B}^{\sigma}_{2,1}} \lesssim C(\Vert \theta \Vert_{L^{\infty}})\Vert \theta \Vert_{\dot{B}^{\sigma}_{2,1}}.
\end{equation}
And the same is true for the inhomogeneous Besov space. Namely, for $ \sigma >0$
\begin{equation}\label{Besovalgest2}
\Vert F \circ \theta \Vert_{B^{\sigma}_{2,1}} \lesssim C(\Vert \theta \Vert_{L^{\infty}} ) \Vert \theta \Vert_{B^{\sigma}_{2,1}}.
\end{equation}
\item Let $ u : \R^{n} \to \S^{2}$ satisfy $ \Vert u \Vert_{\dot{B}^{\frac{n}{2}}_Q} \leq \epsilon \ll 1 $ where $ Q = (0,0,1)$. Then if $ f : = L (u) $, where $ L : \S^{2} \to \C$ is the stereographic projection, then 
\begin{equation}\label{Besovalgest3}
\Vert f \Vert_{\dot{B}^{\sigma}_{2,1}} \lesssim \Vert u \Vert_{\dot{B}^{\sigma}_{Q}},
\end{equation}
for any $ \sigma \geq \frac{n}{2}$.\\
Moreover, the same is true for the inhomogeneous Besov space. Namely, 
\begin{equation}\label{Besovalgest4}
\Vert f \Vert_{B^{\sigma}_{2,1}} \lesssim \Vert u \Vert_{B_{Q}^{\sigma}},
\end{equation}
for any $ \sigma \geq \frac{n}{2}$.
\item For any $ \sigma \geq \frac{n}{2}$ we have
$$
\Vert fg \Vert_{{B}^{\sigma}_{2,1}} \lesssim \Vert f \Vert_{{B}^{\frac{n}{2}}_{2,1}} \Vert g \Vert_{{B}^{\sigma}_{2,1}} + \Vert g \Vert_{{B}^{\frac{n}{2}}_{2,1}} \Vert f \Vert_{{B}^{\sigma}_{2,1}} 
$$
In particular, 
$$
\Vert fg \Vert_{B^{\sigma}_{2,1}} \lesssim \Vert f \Vert_{B^{\sigma}_{2,1}} \Vert g \Vert_{B^{\sigma}_{2,1}}.
$$
and for the homogeneous Besov spaces 
$$
\Vert fg \Vert_{\dot{B}^{\sigma}_{2,1}} \lesssim \Vert f \Vert_{\dot{B}^{\frac{n}{2}}_{2,1}} \Vert g \Vert_{\dot{B}^{\sigma}_{2,1}} + \Vert g \Vert_{\dot{B}^{\frac{n}{2}}_{2,1}} \Vert f \Vert_{\dot{B}^{\sigma}_{2,1}} 
$$
\end{enumerate}

\end{lemma}
\subsection{Commutator identities, estimates:} we start with proving the following pointwise identity which was used in the proof of Proposition \ref{trilineares1}.
\begin{lemma}\label{pointwiseidentitycomm}
Let $f,g,h \in \mathscr{S}(\R^{n})$. Fix $i \in \{1,...,n \}$,  then for $ s \in (1/2,1)$ define 
$$
W^{i}(f,g,h)(x) : = \int_{\R^{n}} \frac{ f(y) (g(x) -g(y)) (h(x) -h(y)) (x-y)_{i} }{ |x -y|^{n +2s}} dy
$$
then we have the following identity (up to a multiplicative constant $c_{n,s} $) 
\begin{equation}\label{pointwscomm}
\begin{split}
& W^{i} (f,g,h) (x) 
\\
& = \nabla^{2s-1}_{i} ( f g h ) - h \nabla^{2s-1}_{i} (fg) 
\\
& - g \nabla_{i}^{2s-1} (f h ) + g h \nabla^{2s-1}_{i} (f)
\end{split}
\end{equation}
\end{lemma}
\begin{proof}
To prove \eqref{pointwscomm}, we simplify the right hand side. Notice that in physical space, we have 
$$
\nabla^{2s-1}_{i} (f)(x)  = p.v  \int_{\R^{n}} \frac{ (x-y)_{i}}{|x-y|^{n+2s}} f(y) dy. 
$$
And therefore, the right hand side of \eqref{pointwscomm} becomes 
\begin{equation*}
\begin{split}
& \nabla^{2s-1}_{i} ( f g h )(x) - h(x) \nabla^{2s-1}_{i} (fg)(x) 
 - g(x) \nabla_{i}^{2s-1} (f h )(x) + g(x) h(x) \nabla^{2s-1}_{i} (f)(x)
 \\
 & = p.v \int_{\R^{n}} \frac{ (x-y)_{i}}{|x-y|^{n+2s}} \left( f(y) g(y) h(y) - h(x) f(y) g(y) - g(x) f(y) h(y) + g(x) h(x) f(y) \right) dy 
\end{split}
\end{equation*}
and notice that we have 
\begin{equation*}
\begin{split}
& f(y) g(y) h(y) - h(x) f(y) g(y) - g(x) f(y) h(y) + g(x) h(x) f(y)
\\
& = f(y) (g(x) -g(y)) (h(x) - h(y))
\end{split}
\end{equation*}
Therefore, we end up with 
\begin{equation*}
\begin{split}
& \nabla^{2s-1}_{i} ( f g h )(x) - h(x) \nabla^{2s-1}_{i} (fg)(x) 
 - g(x) \nabla_{i}^{2s-1} (f h )(x) + g(x) h(x) \nabla^{2s-1}_{i} (f)(x)
 \\
 & = W^{i} (f,g,h) (x)
 \end{split}
\end{equation*}
Notice that the $p.v$ goes away since we have the double difference. 
\end{proof}
Next, we prove the following elementary estimate which will be used to prove short time existence 
\begin{lemma}
Let $ u : \R^{n} \to \S^{2} $ be Lipschitz, and $v \in H^{2} ( \R^{n} : \R^{3} ) $ so that $v ( x) \cdot u(x) = 0$ for every

$x\in \R^{n}$. Then the following holds
\begin{equation}\label{comestzz1}
\vert \int_{\R^{n}} \langle v(x),  u(x) \wedge (- \Delta)^{s} v(x) \rangle  dx \vert \lesssim \Vert \nabla u \Vert^{2s}_{L^{\infty}} \Vert v \Vert^{2}_{L^{2}}.
\end{equation}
  \end{lemma}
\begin{proof}
First, notice the following elementary estimates; for any vectors $ q,p \in \S^{2}$ and $ \tilde{q} ,\tilde{p} \in \R^{3}$ so that $ p \cdot \tilde{p} =0 $, and $ q \cdot \tilde{q} =0 $ we have 
\begin{equation}\label{orthoest}
\vert \langle \tilde{p} , ( p-q ) \wedge \tilde{q} \rangle \vert \lesssim \vert p -q \vert^{2} \vert \tilde{p} \vert \vert \tilde{q} \vert.
\end{equation}
Now using that $ |u |=1$, and the fact that $ \langle v (x) , u(x) \wedge v(x) \rangle = 0 $, we can write 
$$
\int_{\R^{n}} \langle v(x) , (u(x) \wedge (- \Delta)^{s} v(x)) \rangle  = c_{s,n} \int_{\R^{n} \times \R^{n}} \frac{ \langle v(x), (u(x)-u(y)) \wedge v(y) \rangle }{|x-y|^{n+2s}} dy  dx
$$
Hence, using \eqref{orthoest} we obtain
\begin{equation}
\begin{split}
& \vert  \int_{\R^{n}} \langle v(x)  ,(u(x) \wedge (- \Delta)^{s} v(x) ) \rangle   \vert 
\\
& \overset{\eqref{orthoest}}{\lesssim} \int_{\R^{n} \times \R^{n}} \frac{ \vert u (x) - u(y) \vert^{2} \vert v(y) \vert \vert v(x) \vert}{\vert x-y \vert^{n+2s}} dy dx 
\\
& \lesssim \Vert v \Vert^{2}_{L^{2}(\R^{n})} \int_{\R^{n}} \frac{ \left( \min \{ \Vert \nabla u \Vert_{L^{\infty}}  |y| , 2 \}\right)^{2} }{ \vert y \vert^{n+2s}} dy
\end{split}
\end{equation}
where in  the last inequality we used H\"older inequality, and Young inequality. Now the result follows by noticing that 
$$
\int_{\R^{n}} \frac{ \left(\min \{ \Vert \nabla u \Vert_{L^{\infty}}  |y| , 2 \}\right)^{2} }{ \vert y \vert^{n+2s}} dy \lesssim \Vert \nabla u \Vert^{2s}_{L^{\infty}}.
$$
\end{proof}
We close this subsection with the following 
\begin{lemma}
Let $u : \R^{n} \to \S^{2}$ with $ u \in B^{\sigma}_{Q}$ , $ \sigma \geq \frac{n}{2} +10$. Then for any $ f : \R^{n} \to \C$ with $ f \in H^{1}$ we have
\begin{equation}\label{comestzz2}
\Vert [ ( -\Delta)^{s} , u ] (f) \Vert_{L^{2}(\R^{n})} \lesssim ( 1+ \Vert u -Q \Vert_{B^{2}_{\infty,1}}) \Vert f \Vert_{H^{1}}.
\end{equation}
\end{lemma}
\begin{proof}
This follows from the integral representation for $ ( -\Delta)^{s}$ when $ s \in (0,1)$ and then integrating in polar coordinates. Namely, 
\begin{equation*}
\begin{split}
& \Vert [ ( -\Delta)^{s} , u ] (f) \Vert_{L^{2}(\R^{n})}
\\
& \lesssim \Vert (-\Delta)^{s} u \Vert_{L^{\infty}} \Vert f \Vert_{L^{2}} + \Vert \int_{\R^{n}}  \frac{ \left( u(x) - u(x-y) \right) \left( f(x) - f(x-y) \right) }{ | y|^{n+2s}} dy  \Vert_{L^{2}(\R^{n})}
\\
& \lesssim  \Vert (-\Delta)^{s} u \Vert_{L^{\infty}} \Vert f \Vert_{L^{2}} + \Vert \nabla u \Vert_{L^{\infty}} \Vert \nabla f \Vert_{L^{2}} \int_{0}^{1} r^{1-2s} dr 
\\
& + \Vert f \Vert_{L^{2}} \int_{1}^{\infty} r^{-1 -2s} dr
\\
& \lesssim ( 1+ \Vert u -Q \Vert_{B^{2}_{\infty,1}}) \Vert f \Vert_{H^{1}}.
\end{split}
\end{equation*}
Where in the last line we used 
$$
\Vert \nabla u \Vert_{L^{\infty}} + \Vert ( -\Delta)^{s} u \Vert_{L^{\infty}} \lesssim \Vert  u -Q \Vert_{B^{2}_{\infty,1}}.
$$
\end{proof}

\subsection{Product structures}

In this subsection we prove several estimates regarding the size of the modulation of a product of two functions. These estimates were used in the proof of the algebra properties of the spaces $F^{\sigma}$. We start with the following 
\begin{lemma}\label{prod1}
Let $ k_{1},k_{2} \in \Z$ with $ k_{1} \leq k_{2} +10$ and $ j_{1},j_{2} \geq 0$. 
Let $ f_{k_{1}}, f_{k_{2}} $ be two functions with Fourier transform supported in $ D_{k_{1}, \leq j_{1}}$, $ D_{k_{2}, \leq j_{2}}$ respectively. Then given $ j \geq 0$ with $j\geq \max\{ k_{1}+ (2s-1)k_{2},0 \} + 100$ if $j \geq  \max\{ j_{1} , j_{2} \} +50 $ then we have 
$$
Q_{j} \left( \Delta_{ \leq k_{2} +20} ( \tilde{f}_{k_{1}} f_{k_{2}}) \right) =0
$$
where $ \tilde{f}_{k_{1}} \in \{ f_{k_{1}}, \bar{f}_{k_{1}} \}$.
\end{lemma}
\begin{proof}
Take the Fourier transform 
\begin{equation*}
\begin{split}
& \mathcal{F} \left( Q_{j} \left( \Delta_{\leq k_{2} +20} ( \tilde{f}_{k_{1}} f_{k_{2}}) \right) \right) (\xi,\tau)
\\
& = \varphi_{j} ( \vert \xi \vert^{2s} + \tau) \varphi_{\leq k_{2} +20} ( \xi) \int_{\R^{n+1}} \hat{ \tilde{f}}_{k_{1}}( \eta,r) \hat{ f}_{k_{2}}(\xi-\eta, \tau -r ) d \eta dr,
\end{split}
\end{equation*}
and in order for this to be nonzero we must have the following 
\begin{enumerate}
    \item 
    $
    2^{j-1} \leq \vert \vert \xi \vert^{2s} + \tau \vert \leq 2^{j+1}
    $
    \item  $ \vert r  \pm \vert \eta \vert^{2s}  \vert \leq 2^{j_{1} +1} \ll 2^{j}$
    \item  $ \vert \xi - \eta \vert \approx 2^{k_{2}} \geq 2^{-30} \vert \xi \vert$.
\end{enumerate}
Here the sign of $ |\eta|^{2s}$ depend on whether we have a complex conjugate or not. 
Then notice the following
\begin{equation*}
\begin{split}
& \vert \vert \xi - \eta \vert^{2s} + \tau -r \vert
\\
& \geq \vert \vert \xi \vert^{2s} + \tau \vert - \left( \vert \pm \vert \eta \vert^{2s} + r \vert \right) 
\\
& - \vert \left( \vert \xi - \eta \vert^{2s} - \vert \xi \vert^{2s}  \pm \vert \eta \vert^{2s} \right) \vert  
\\
& \geq 2^{j-1} - 2^{j-50} - \vert \left( \vert \xi - \eta \vert^{2s} - \vert \xi \vert^{2s} \pm \vert \eta \vert^{2s} \right) \vert
\end{split}
\end{equation*}
it remains to show that 
\begin{equation}
\vert \left( \vert \xi - \eta \vert^{2s} - \vert \xi \vert^{2s} \pm \vert \eta \vert^{2s} \right) \vert \ll 2^{j}.
\end{equation}
This follows easily since $ \vert \eta \vert^{2s} \leq 2^{2s k_{1}} \ll 2^{j}$ and 
\begin{equation*}
\begin{split}
& | \vert \xi - \eta \vert^{2s} - \vert \xi \vert^{2s} |
\\
& \lesssim \max\{ \vert \xi - \eta \vert^{2s-1}, \vert \xi \vert^{2s-1} \} \vert \eta \vert
\\
&  \ll 2^{j}.
\end{split}
\end{equation*}
\end{proof}

We have a similar result for lower bounds
\begin{lemma}\label{prod2}
Let $ k_{1},k_{2} \in \Z$ with $ k_{1} \leq k_{2} -10$, and $ j_{1} \ge \max\{ k_{1} + (2s-1)k_{2}, 0 \} +100$. Take any $ f_{k_{1}}, f_{k_{2}} $ in $ Z_{k_{1}}, Z_{k_{2}}$ respectively. Then there exists  a constant $ C$ so that for any $j \leq j_{1} -C $ we have
\begin{equation}\label{lft}
Q_{j} \left( Q_{\geq j_{1} +100} (f_{k_{2}}) Q_{j_{1}} (\tilde{f}_{k_{1}}) \right) =0
\end{equation}

where $ \tilde{f}_{k_{1}} \in \{ f_{k_{1}} , \bar{f}_{k_{1}} \}.$
\end{lemma}
\begin{proof}
We take the Fourier transform,
\begin{equation*}
\begin{split}
& \mathcal{F} \left( Q_{j} \left( Q_{\geq j_{1} +100} (f_{k_{2}}) Q_{j_{1}} (\tilde{f}_{k_{1}}) \right) \right)(\xi,\tau)
\\
& = \varphi_{j} ( \vert \xi \vert^{2s} + \tau) \int_{\R^{n+1}} \varphi_{\geq j_{1} + 100} ( \vert \xi - \eta \vert^{2s} + \tau -r) \hat{f}_{k_{2}} ( \xi - \eta, \tau -r) \varphi_{j}( \vert \eta \vert^{2s} + r ) \hat{ \tilde{ f}}_{k_{1}}(\eta,r) d \eta dr
\end{split}
\end{equation*}
and we notice that a necessary condition for this to be nonzero is to have 
\begin{enumerate}
    \item $ \vert \vert \xi - \eta \vert^{2s} + \tau - r \vert \geq 2^{j_{1} +100}$
    \item  $ \vert r \pm \vert \eta \vert^{2s} \vert \in [ 2^{j_{1}-1}, 2^{j_{1} +1}]$
    \item $ 2^{j_{1} -10} \geq \vert \xi \vert^{2s-1} \vert \eta \vert$
    \item  $ \vert \eta \vert \ll \vert \xi \vert$.
\end{enumerate}
Here the sign of $ |\eta|^{2s}$ depend on whether we have a complex conjugate on $ f_{k_{1}}$ or not. In either case, notice 
\begin{equation*}
\begin{split}
& \vert \vert \xi \vert^{2s} + \tau \vert
\\
& \geq \vert \vert \xi - \eta \vert^{2s} + \tau -r \vert \vert - \vert r \pm \vert \eta \vert^{2s} \vert - \vert \left(  \vert \xi - \eta \vert^{2s} - \vert \xi \vert^{2s} \pm \vert \eta \vert^{2s} \right) \vert
\\
& \geq C( 2^{j_{1} + 100} - 2^{j_{1}}) 
\\
& \geq C 2^{j_{1}}.
\end{split}
\end{equation*}
\end{proof}

\section{Short time existence}\label{s:shorttime}
In this section we prove existence of a local solution $u$ of \eqref{SSchrodingereq}. This was needed to derive a modified $s$-Schr\"odinger equation in terms of the functions 
$$
V_{i} : = \langle \partial_{i} u , e_{1} \rangle + \i \langle \partial_{i} u , e_{2} \rangle.
$$
 Namely, we prove
\begin{proposition}\label{localexistencethm}
Let $n \geq 1$, $ s\in ( 1/2,1)$ and $ \sigma_{0} \geq \frac{n}{2} +10$. Then given $u_{0} : \R^{n} \to \S^{2}$ with $ u_{0} \in B^{\sigma_{0}}_{Q}$ there exists $ T : = T( \Vert u_{0} \Vert_{B^{\sigma_{0}}_{Q}}) > 0 $ and $u : \R^{n+1} \to \S^{2}$ so that $u \in L^{\infty}( [-T,T] : B^{\sigma_{0}}_{Q} )$ and $u $ solves \eqref{SSchrodingereq} on the time interval $[-T,T]$ with initial data $u_{0}$. Moreover, we have persistence of regularity. In particular, whenever $ u_{0} \in B^{\sigma}_{Q}$ with $ \sigma > \sigma_{0}$ then $ u \in L^{\infty} ( [ -T,T] : B^{\sigma}_{Q})$. Lastly, for any $ q < \sigma_{0}$ we have $ u \in C ( [-T,T] : B^{q}_{Q} )$. 
\end{proposition}
To prove Proposition \ref{localexistencethm} we first find functions $ u^{\epsilon}$ that solve an "approximate" PDE to the $s$-Schr\"odinger map. Then we prove uniform estimates for these approximates solutions, which will then allow us to pass to the limit. This idea of taking parabolic perturbations of Schr\"odinger maps is not new. See for example \cite{DingWang2001}. In fact, many of the arguments in this section are relativaly standard. However, since we are unable to find a precise reference for Proposition \ref{localexistencethm}, we include all of the proofs here.
\subsection{Approximate solutions} First, we start with constructing approximate solutions. In particular, we have 
\begin{proposition}\label{parapolicguys}
Let $n \geq 1$, $ s\in (1/2,1)$ and $ u_{0} : \R^{n} \to \S^{2}$ with $ u_{0} \in B^{\sigma_{0}}_{Q}$ where $ \sigma_{0} \geq \frac{n}{2} +10$. Given $ \epsilon \in (0,1)$ there exists $ T^{\epsilon} > 0$ and $ u^{\epsilon}: \R^{n+1} \to \S^{2}, ~ u^{\epsilon} \in C ( [0,T^{\epsilon}] : B^{\sigma_{0}}_{Q})$ which solves 
\begin{equation}\label{SSparabolic}
\begin{cases}
 \partial_{t} u^{\epsilon} - \epsilon \Delta u^{\epsilon}  = - u^{\epsilon} \wedge (-\Delta)^s u^{\epsilon} + \epsilon (  |\nabla u^{\epsilon} |^{2} u^{\epsilon}) \mbox{    } \mbox{ on } \R^{n} \times [0,T^{\epsilon}] \\
 u^{\epsilon}(x,0) = u_{0}.
\end{cases}
\end{equation}
Moreover, the solution $ u^{\epsilon}$ is unique. Namely, given any $ v \in C ( [0,T^{\epsilon}] : B^{\sigma_{0}}_{Q})$ which solves \eqref{SSparabolic}, then we have $ v = u^{\epsilon}$. Lastly, we have persistence of regularity, in the sense that if $ u_{0} \in B^{\sigma}_{Q} $ for some $ \sigma > \sigma_{0}$ then we obtain $ u^{\epsilon} \in C ( [0,T^{\epsilon}] : B^{\sigma}_{Q})$. 
\end{proposition}
This is will follow from a standard fixed point argument. First, notice that the left hand side of \eqref{SSparabolic} annihilates constants. Therefore, taking $ w = u-Q$, it suffices to find $ w^{\epsilon} : \R^{n+1} \to \R^{3},  ~ w^{\epsilon} \in C( [0,T^{\epsilon}] : B^{\sigma}_{2,1} )$, $ |w^{\epsilon} + Q |=1$, which solves  
\begin{equation}\label{SSparabolicscalar}
\begin{cases}
 \partial_{t} w^{\epsilon} - \epsilon \Delta w^{\epsilon}  = - (w^{\epsilon}+Q) \wedge (-\Delta)^s w^{\epsilon} + \epsilon (  |\nabla w^{\epsilon} |^{2} (w^{\epsilon}+Q)) \mbox{    } \mbox{ on } \R^{n} \times [0,T^{\epsilon}] \\
 w^{\epsilon}(x,0) = u_{0} -Q,
\end{cases}
\end{equation}
To that end, notice that the nonlinearity on the right hand side of \eqref{SSparabolicscalar} is Lipschitz. 
\begin{lemma}\label{lipschitsbound}
Let $ \mathcal{N}^{\epsilon}$ denote the nonlinearity on the right hand side of \eqref{SSparabolicscalar}. Then we have the following Lipschitz bound; for any $ \sigma \geq \frac{n}{2} +10$ and any $ w, v \in B^{\sigma}_{2,1}$ with $ \Vert w \Vert_{B^{\sigma}_{2,1}} + \Vert v \Vert_{B^{\sigma}_{2,1}} \leq R$ we have 
\begin{equation}\label{lipschitznepsilon}
\Vert \mathcal{N}^{\epsilon}( w ) - \mathcal{N}^{\epsilon}(v) \Vert_{B^{\sigma -2s}_{2,1}} \lesssim_{\sigma , R} \Vert w - v \Vert_{B^{\sigma}_{2,1}}.
\end{equation}
\end{lemma}
 \begin{proof}
By definition 
\begin{equation*}
\begin{split}
& \mathcal{N}^{\epsilon}( w ) - \mathcal{N}^{\epsilon}(v)
\\
& = -Q \wedge ( - \Delta)^{s} ( w-v) - (w-v) \wedge (- \Delta)^{s}(w) - v \wedge (-\Delta)^{s} (w-v)
\\
& + \epsilon \left((w +Q) \sum_{j =1}^{n} \partial_{j} ( w-v) \cdot ( \partial_{j} w + \partial_{j} v )   + \vert \nabla v \vert^{2} (w-v) \right) 
\end{split}
\end{equation*}
And hence, \eqref{lipschitznepsilon} follows from the algebra properties of Besov spaces, see Lemma \ref{Besovstuff}. 
 \end{proof}
 In order to run the contraction argument, we will need the following known estimate about the heat semigroup $  e^{ \epsilon t\Delta} $. We include the proof for completeness. 
 \begin{lemma}\label{heatsemigr}
For any  $ \sigma > 0$ and $ t> 0$, we have the following 
\begin{enumerate}
    \item for $ f \in B^{\sigma}_{2,1} $ we have for any $ \beta \geq 0$
    \begin{equation}\label{htest1}
        \Vert e^{\epsilon t \Delta} f \Vert_{B^{\sigma + \beta}_{2,1}} \lesssim_{\sigma, \beta} (1+ ( \epsilon t)^{- \beta/2} ) \Vert f \Vert_{B^{\sigma}_{2,1}}.
    \end{equation}
    \item Given $T >0$,  $ F \in C ( [0,T] : B^{\sigma}_{2,1}) $ we have for any $ t \in (0, T]$, $ \beta \in [0,2)$ 
    \begin{equation}\label{htest2}
        \Vert \int_{0}^{t} e^{\epsilon ( t- \tilde{t}) \Delta} F(\cdot,\tilde{t}) d\tilde{t}  \Vert_{B^{\sigma + \beta}_{2,1}} \lesssim_{\sigma, \beta} \left( t +  t^{1- \beta/2}  \epsilon^{-\beta/2} \right) \sup_{ \tilde{t} \in [0,t]} \Vert F ( \cdot, \tilde{t}) \Vert_{B^{\sigma}_{2,1}}.
    \end{equation}
\end{enumerate}
 \end{lemma}
 \begin{proof}
\eqref{htest1} follows from Plancheral. Indeed, for any $k \geq 1$ we have 
\begin{equation*}
\begin{split}
&   \Vert  e^{\epsilon t \Delta} \Delta_{k} f \Vert_{L^{2}(\R^{n})}
\\
& \lesssim e^{ - c \epsilon t 2^{2k}} \Vert \Delta_{k}f  \Vert_{L^{2}(\R^{n})}
\\
& \lesssim 2^{- \beta k} ( \epsilon t)^{-\beta/2} \Vert \Delta_{k}f  \Vert_{L^{2}(\R^{n})}.
\end{split}
\end{equation*}
Thus, 
\begin{equation}\label{htest3}
\sum_{k \geq 1} 2^{ ( \sigma+ \beta) k }  \Vert  e^{\epsilon t \Delta} \Delta_{k} f \Vert_{L^{2}(\R^{n})} \lesssim ( \epsilon t)^{-\beta/2} \Vert f  \Vert_{B^{\sigma}_{2,1}}.
\end{equation}
As for $k=0$, we have the obvious estimate 
\begin{equation}\label{htest4}
\Vert e^{\epsilon t \Delta} \Delta_{ \leq 0} f \Vert_{L^{2}(\R^{n})} \lesssim \Vert f \Vert_{B^{\sigma}_{2,1}}. 
\end{equation}
Combining \eqref{htest3}, and \eqref{htest4} gives \eqref{htest1}. Next, notice that  \eqref{htest2} follows from Minkowki's inequality and \eqref{htest1}. Indeed, 
\begin{equation*}
\begin{split}
& \Vert \int_{0}^{t} e^{\epsilon ( t- \tilde{t}) \Delta} F(\cdot,\tilde{t}) d\tilde{t}  \Vert_{B^{\sigma + \beta}_{2,1}}
\\
& \lesssim \int_{0}^{t} \Vert e^{\epsilon ( t- \tilde{t}) \Delta} F(\cdot,\tilde{t})  \Vert_{B^{\sigma + \beta}_{2,1}} d \tilde{t}
\\
& \overset{\eqref{htest1}}{\lesssim} \sup_{\tilde{t} \in [0,t]} \Vert F(\cdot,\tilde{t} ) \Vert_{B^{\sigma}_{2,1}} | \int_{0}^{t} (1+ \epsilon^{-\beta/2} (t - \tilde{t})^{-\beta/2} ) d\tilde{t}|
\\
& \left( t +  t^{1- \beta/2}  \epsilon^{-\beta/2} \right) \sup_{ \tilde{t} \in [0,t]} \Vert F ( \cdot, \tilde{t}) \Vert_{B^{\sigma}_{2,1}}
\end{split}
\end{equation*}

 \end{proof}
 With the help of Lemma \ref{heatsemigr} we can now run the fixed point argument to obtain 
 \begin{lemma}\label{prepreloc}
Let $ Q \in \S^{2}$, $ w_{0} \in B^{\sigma_{0}}_{2,1}(\R^{n} : \R^{3})$, $ \sigma_{0} \geq \frac{n}{2} +10$. Then for any $ \epsilon \in (0,1)$ there exists $ T:= T( \Vert w\Vert_{B^{\sigma_{0}}_{2,1}}, \epsilon) > 0$ and $ w^{\epsilon} \in C( [ 0, T] : B^{\sigma_{0}}_{2,1} )$ so that $w^{\epsilon}$ solves
\begin{equation}\label{SSparabolicscalar1}
\begin{cases}
 \partial_{t} w^{\epsilon} - \epsilon \Delta w^{\epsilon}  = - (w^{\epsilon}+Q) \wedge (-\Delta)^s w^{\epsilon} + \epsilon (  |\nabla w^{\epsilon} |^{2} (w^{\epsilon}+Q)) \mbox{    } \mbox{ on } \R^{n} \times [0,T^{}] \\
 w^{\epsilon}(x,0) = w_{0}.
\end{cases}
\end{equation}
Moreover, the solution $ w^{\epsilon}$ is unique. In the sense, that given any $ v \in C ( [ 0, T] : B^{\sigma_{0}}_{2,1}) $ which solves \eqref{SSparabolicscalar1} then $ w^{\epsilon} = v$. Lastly, if $w_{0} \in B^{\sigma}_{2,1} ( \R^{n} : \R^{3})$ for $ \sigma >\sigma_{0}$. Then we obtain $ w^{\epsilon} \in C ( [0,T] : B^{\sigma}_{2,1})$

 \end{lemma}
\begin{proof}
Let $ N^{\epsilon}( w)$ be the nonlinearity on the right hand side of \eqref{SSparabolicscalar1}. Fix $T > 0$ to be chosen later and choose $ R \geq 1$ such that $ \Vert w_{0} \Vert_{B^{\sigma_{0}}_{2,1}} \ll R$. Then set 
$$
B_{R} : = \{ w \in C ( [0,T] : B^{\sigma_{0}}_{2,1} ) : \sup_{t \in [0,T]} \Vert w \Vert_{B^{\sigma_{0}}_{2,1}} \leq R \}. 
$$
Consider the operator given by 
$$
\Phi(w) : = e^{\epsilon t \Delta} w_{0} + \int_{0}^{t} e^{ \epsilon ( t- \tilde{t}) \Delta} \mathcal{N}^{\epsilon} (w) d\tilde{t}.
$$
By Lemma \ref{lipschitznepsilon}, and Lemma \ref{heatsemigr} we obtain
\begin{equation*}
\begin{split}
& \sup_{t \in [0,T]} \Vert \Phi (w) \Vert_{B^{\sigma}_{2,1}}
\\
& \leq C \Vert w_{0} \Vert_{B^{\sigma}_{2,1}} +  \left( T+ { \epsilon^{-s} T^{1-s}} \right) R 
\\
& \leq \frac{R}{4} +  C ( T + \epsilon^{-s} T^{1-s} ) R.
\end{split}
\end{equation*}
Hence, for small enough $T$ we obtain
$$
\sup_{t \in [0,T]} \Vert \Phi (w) \Vert_{B^{\sigma_{0}}_{2,1}} \leq R.
$$
Moreover, Lemma \ref{lipschitznepsilon}, and Lemma \ref{heatsemigr} imply that $ \Phi( w) \in C ( [0,T] : B^{\sigma_{0}}_{2,1})$. Therefore, 
$$
\Phi : B_{R} \to B_{R}.
$$
Lastly, for small $T$, $ \Phi$ is a contraction. Indeed, by Lemma \ref{heatsemigr}, Lemma \ref{lipschitznepsilon} we obtain for any $w,v \in B_{R}$ 
$$
\sup_{t \in [0,T]} \Vert \Phi(w) - \Phi(v) \Vert_{B^{\sigma_{0}}_{2,1}} \lesssim C_{R,s } ( T + \epsilon^{-s} T^{1-s} ) \Vert w - v \Vert_{B^{\sigma_{0}}_{2,1}}.
$$
Hence, for small $ T > 0$, we obtain 
$$
\sup_{t \in [0,T]} \Vert \Phi(w) - \Phi(v) \Vert_{B^{\sigma_{0}}_{2,1}} \leq \frac{1}{2} \Vert w - v \Vert_{B^{\sigma_{0}}_{2,1}}.
$$
Thus, a unique fixed point exist. Namely, there exist a unique $ w^{\epsilon } \in B_{R}$ which solves \eqref{SSparabolicscalar1}. \\
It remains to show uniqueness of $ w^{\epsilon}$ in the larger space $ C ( [0,T] : B^{\sigma_{0}}_{2,1})$. To that end, take $ v \in C( [0,T] : B^{\sigma_{0}}_{2,1}) $ which solves \eqref{SSparabolicscalar1}. It suffices to show that there exists $ h \in (0,T]$ such that for any $ t_{0} \in [0,T]$ for which $ w(\cdot, t_{0}) = v(\cdot,t_{0})$ we have $ w( \cdot,t) = v(\cdot,t)$ for all $ t\in [t_{0} , t_{0}+h] \cap [t_{0} , T]$. To that end, fix $ t_{0}$ so that $ w(\cdot, t_{0}) = v(\cdot,t_{0})$. And fix $ h > 0$ to be chosen later. Notice that we can write
$$
w(x,t) = e^{(t-t_{0}) \epsilon \Delta} w(\cdot,t_{0}) + \int_{t_{0}}^{t} e^{ ( t - \tilde{t}) \epsilon \Delta} \mathcal{N}^{\epsilon}( w)(x,\tilde{t}) d \tilde{t}. 
$$
Similarly, for $ v $ we can write 
$$
v(x,t) = e^{(t-t_{0}) \epsilon \Delta} v(\cdot,t_{0}) + \int_{t_{0}}^{t} e^{ ( t - \tilde{t}) \epsilon \Delta} \mathcal{N}^{\epsilon}( v)(x,\tilde{t}) d \tilde{t}. 
$$
Then using \eqref{htest2}, \eqref{lipschitznepsilon} we obtain
\begin{equation*}
\begin{split}
&\sup_{t \in [t_{0},t_{0} +h]} \Vert w - v \Vert_{B^{\sigma_{0}}_{2,1}}
\\
& \leq C (h + \epsilon^{-s} h^{1-s} ) \sup_{t \in [t_{0},t_{0} + h]} \Vert w - v \Vert_{B^{\sigma_{0}}_{2,1}},
\end{split}
\end{equation*}
where $ C $ depends on the Besov norms $ \sup_{t \in [0,T]} \Vert w \Vert_{B^{\sigma}_{2,1}}$ and $ \sup_{t \in [0,T]} \Vert v \Vert_{B^{\sigma}_{2,1}}.$ Hence, taking $ h $ small enough we obain
$$
\sup_{t \in [t_{0},t_{0} +h]} \Vert w - v \Vert_{B^{\sigma_{0}}_{2,1}} \leq \frac{1}{2} \sup_{t \in [t_{0},t_{0} +h]} \Vert w - v \Vert_{B^{\sigma_{0}}_{2,1}}
$$
which implies 
$$
\sup_{t \in [t_{0},t_{0} +h]} \Vert w - v \Vert_{B^{\sigma_{0}}_{2,1}} =0.
$$
Which gives uniqueness. For persistence of regularity, we can apply the above fixed point argument with $ \sigma > \sigma_{0}$ and see that the unique solution $ w^{\epsilon}$ satisfies 
$$
w^{\epsilon} \in C( [0, \tilde{T}^{\epsilon}] : B^{\sigma}_{2,1}),
$$
for some $ \tilde{T}^{\epsilon} \leq T^{\epsilon}$. But using \eqref{htest2} we arrive at 
$$
w^{\epsilon} \in C( [0, {T}^{\epsilon}] : B^{\sigma}_{2,1}),
$$

\end{proof}
Now we can prove Proposition \ref{parapolicguys}.

\begin{proof}[Proof of Proposition \ref{parapolicguys}]
Using Lemma \ref{prepreloc} with initial data $ u_{0} -Q $ we obtain $ w^{\epsilon} : \R^{n+1} \to \R^{3} $ which solves
    \begin{equation}\label{SSparabolicscalar2}
\begin{cases}
 \partial_{t} w^{\epsilon} - \epsilon \Delta w^{\epsilon}  = - (w^{\epsilon}+Q) \wedge (-\Delta)^s w^{\epsilon} + \epsilon (  |\nabla w^{\epsilon} |^{2} (w^{\epsilon}+Q)) \mbox{    } \mbox{ on } \R^{n} \times [0,T^{}] \\
 w^{\epsilon}(x,0) = u_{0} -Q.
\end{cases}
\end{equation}
Therefore, taking $ u^{\epsilon} = w^{\epsilon} + Q$, we see that $ u^{\epsilon}$ has the required regularity and solves \eqref{SSparabolic}. Moreover, uniqueness follows from uniqueness of the solutions of \eqref{SSparabolicscalar2}. It remains to show that 
$ | u^{\epsilon } | = 1$. To that end, set 
$$
\rho(x,t) : = | u^{\epsilon} |^{2} -1.
$$
Then $ \rho (\cdot,0) = 0$. A direct computation gives
$$
\partial_{t} \rho = \epsilon \Delta \rho +2 \epsilon | \nabla u^{\epsilon} |^{2} \rho.
$$
Therefore, integration by parts yields
$$
\frac{d}{dt} \Vert \rho (\cdot,t) \Vert_{L^{2}(\R^{n})} \lesssim \Vert \nabla u^{\epsilon}(\cdot,t) \Vert^{2}_{L^{\infty}(\R^{n})} \Vert \rho (\cdot,t) \Vert_{L^{2}(\R^{n})}.
$$
Hence, by Grownell's inequality we obtain $ \rho(\cdot,t) = 0$. Which implies that $ u^{\epsilon} : \R^{n+1} \to \S^{2}$.  
\end{proof}

 \subsection{Uniform estimates for the approximate solutions} In this subsection, we show that the solutions obtained in Proposition \ref{parapolicguys}  can be extended to an interval $ [0,T]$ where $ T $ does not depend on $\epsilon$. Namely, we prove 
 \begin{proposition}\label{mainuniformsec}
For any $ \epsilon \in (0,1]$ let $ u^{\epsilon}$ denote the solution obtained in Proposition \ref{parapolicguys}. Then there exists $ T : = T( \Vert u_{0} \Vert_{B^{\sigma_{0}}_{Q}}) > 0$ so that $ u^{\epsilon}$ solves \eqref{SSparabolic} on the time interval $ [0, T]$. 
 Moreover, whenever $ u_{0} \in B^{\sigma}_{Q}$ with $ \sigma \geq \sigma_{0}$ then  we have the following uniform estimate 
 \begin{equation}\label{uniformepsilon}
    \sup_{0< \epsilon \leq 1} \sup_{t \in [0,T]} \Vert u^{\epsilon} \Vert_{B^{\sigma}_{Q}} \lesssim \Vert u_{0} \Vert_{B^{\sigma}_{Q}}.
 \end{equation}
 \end{proposition}
 The proof of Proposition \ref{mainuniformsec} will be a sequence of lemmata.
 To begin, for any $k \in \N$ consider the following operator 
 \begin{equation}\label{ckdef}
    C_{k}(f,g) : = \Delta_{k} ( fg) - f \Delta_{k} (g) - g \Delta_{k}(f). 
 \end{equation}
We start with the following elementary estimate
\begin{lemma}\label{1lm}
For any $ \sigma >1  $ and any $ f,g \in L^{2}(\R^{n})$ we have 
\begin{equation}\label{mnm1}
\sum_{k \geq 1} 2^{k \sigma} \Vert C_{k} (f,g) \Vert_{L^{2}(\R^{n})}\lesssim_{\sigma} \Vert f \Vert_{B^{1}_{{\infty},1}} \Vert g \Vert_{B^{\sigma-1}_{2,1}} +\Vert g \Vert_{B^{1}_{{\infty},1}} \Vert f \Vert_{B^{\sigma-1}_{2,1}} 
\end{equation}

\end{lemma}
\begin{proof}
Write 
$$
f = \sum_{k_{1} \geq 1} \Delta_{k_{1}} f + \Delta_{\leq 0} f
$$
and similarly, for $g$
$$
g = \sum_{k_{2} \geq 1 } \Delta_{k_{2}} g + \Delta_{\leq 0} g.
$$
Write $ f_{k_{1}} : = \Delta_{k_{1}} f$ , $g_{k_{2}} := \Delta_{k_{2}} g$. Assume first $ k_{1} \leq k_{2}$. Then notice that for any $k \leq k_{2} +C$ we have by Minkowski's inequality 
\begin{equation}\label{mnm2}
\begin{split}
& \Vert C_{k} (f_{k_{1}},g_{k_{2}}) \Vert_{L^{2}(\R^{n}) }
\\
& \lesssim  \int_{\R^{n}} \vert \mathcal{F}^{-1} ( \varphi_{k}) (y) \vert  \left( \Vert (f_{k_{1} } ( x -y) - f_{k_{1}}(x) ) ( g_{k_{2}}(x-y) - g_{k_{2}}(x))  \Vert_{L^{2}_{x}} \right) dy 
\\
& \lesssim \Vert \nabla f_{k_{1}} \Vert_{L^{\infty}(\R^{n})} \Vert g_{k_{1}} \Vert_{L^{2}(\R^{n})} \Vert \int_{\R^{n}} \vert  \mathcal{F}^{-1} ( \varphi_{k}) (y)  \vert \vert y \vert dy 
\\
& \lesssim 2^{k_{1}}  \Vert f_{k_{1}} \Vert_{L^{\infty}(\R^{n})} \Vert g_{k_{1}} \Vert_{L^{2}(\R^{n})} \Vert \int_{\R^{n}} \vert  \mathcal{F}^{-1} ( \varphi_{k}) (y)  \vert \vert y \vert dy 
\\
& \lesssim 2^{-k} 2^{k_{1}} \Vert f_{k_{1}} \Vert_{L^{\infty}(\R^{n})} \Vert g_{k_{2}} \Vert_{L^{2}(\R^{n})}.
\end{split}
\end{equation}
Multiplying by $ 2^{k \sigma} $ and summing over $k \leq k_{2} +C$ yields 
\begin{equation}\label{mnm3}
\begin{split}
& \sum_{k \leq k_{2} +C} 2^{k \sigma} \Vert C_{k} (f_{k_{1}},g_{k_{2}}) \Vert_{L^{2}(\R^{n}) }
\\
& \lesssim 2^{k_{1}} \Vert f_{k_{1}} \Vert_{L^{\infty}(\R^{n})} 2^{k_{2} ( \sigma -1)} \Vert g_{k_{2}} \Vert_{L^{2}( \R^{n})},
\end{split}
\end{equation}
whenever $k_{1} \leq k_{2}$. If $ k_{1} \geq k_{2}$ then by reversing the roles of $f_{k_{1}}, g_{k_{2}}$ in \eqref{mnm2} we obtain 

\begin{equation}\label{mnm4}
\begin{split}
& \sum_{k \leq k_{1} +C} 2^{k \sigma} \Vert C_{k} (f_{k_{1}},g_{k_{2}}) \Vert_{L^{2}(\R^{n}) }
\\
& \lesssim 2^{k_{1}(\sigma -1)} \Vert f_{k_{1}} \Vert_{L^{2}(\R^{n})} 2^{k_{2}} \Vert g_{k_{2}} \Vert_{L^{\infty}( \R^{n})},
\end{split}
\end{equation}
whenever $k_{1} \geq k_{2}$. Combining \eqref{mnm4}, \eqref{mnm3} yields \eqref{mnm1}. 
\end{proof}
%Next, fix $ u \in B^{\sigma}_{Q}$ where $ \sigma \geq \frac{n}{2} +10$. And set $ v : = u - Q$. Then 
%$$
%v_{l} \in B^{\sigma_{0}}_{2,1} \quad \text{ for } l =1,2,3.
%$$
%In particular, $ v = \sum_{k \geq 1} \Delta_{k} v + \Delta_{\leq 0} v$. And by definition of $ \Delta_{k}$ we have $ \Delta_{k} v = \Delta_{k} u$ whenever $ k \geq 1$. 

An immediate consequence of Lemma \ref{1lm} is following estimate. In what follows, for $u,v : \R^{n+1} \to \R^{3}$ we write
$$
C_{k} ( u,v) := \Delta_{k} ( \langle u , v \rangle ) - \langle u, \Delta_{k} v \rangle - \langle \Delta_{k } u , v \rangle.
$$
\begin{lemma}\label{2lm}
Let $u: \R^{n} \to \S^{2}$ , $u \in B^{\sigma}_{Q} $ with $ \sigma \geq \frac{n}{2} +10$. Then we have 
\begin{equation}\label{mnm5}
\sum_{k \geq 1} 2^{k ( \sigma +1)} \Vert \langle u , \Delta_{k} u \rangle \Vert_{L^{2}( \R^{n})} 
\lesssim   \Vert u -Q \Vert_{B^{\sigma}_{2,1}} \Vert u -Q \Vert_{B^{1}_{\infty,1}}.
\end{equation}
Moreover, for the Sobolev norm we get 
\begin{equation}\label{mnm6}
\sum_{k \geq 1} 2^{k \sigma} \Vert \langle u , \Delta_{k} u \rangle \Vert_{H^{1}( \R^{n})} 
\lesssim   \Vert u -Q \Vert_{B^{\sigma}_{2,1}} \Vert u -Q \Vert_{B^{2}_{\infty,1}}.
\end{equation}

\end{lemma}
\begin{proof}
Notice that for any $ k \geq 1$ we have $ \Delta_{k} ( u - Q) = \Delta_{k} u $. Moreover, since $ \vert u \vert^{2} =1 $ then we have
\begin{equation*}
0 = \Delta_{k} ( \vert u \vert^{2} ) = 2 \langle u , \Delta_{k} u \rangle  + C_{k} ( u ,u ).
\end{equation*}
Therefore, 
\begin{equation}\label{1fmnm}
 \langle u , \Delta_{k} u \rangle = -\frac{1}{2}  C_{k} ( u,u) =- \frac{1}{2} C_{k} ( u -Q, u -Q).
\end{equation}
Then \eqref{mnm5} follows from \eqref{mnm1}.

Next, for \eqref{mnm6} notice that by definition 
$$
\Vert \langle u , \Delta_{k} u \rangle \Vert_{H^{1}(\R^{n})} = \Vert \langle u , \Delta_{k} u \rangle \Vert_{L^{2}( \R^{n})} + \Vert \nabla \langle u , \Delta_{k} u \rangle \Vert_{L^{2}( \R^{n})}. 
$$
In view of \eqref{mnm5} we only need to control the derivatives of $ \langle u, \Delta_{k} u \rangle$.  To that end, for any $j \in \{ 1,...,n \}$ we have from \eqref{1fmnm}
$$
\partial_{j} \langle u , \Delta_{k} u \rangle =-  C_{k}( u-Q, \partial_{j} u) .
$$
Hence, applying \eqref{mnm1} gives the result.
\end{proof}
Lemma \ref{2lm} gives us control of the normal projection of $ \Delta_{k} u$. More importantly, an immediate consequence of \eqref{mnm5} is the following 
\begin{equation}\label{3fmnm}
    \sum_{k \geq j} 2^{k ( \sigma )} \Vert \langle u , \Delta_{k} u \rangle \Vert_{L^{2}( \R^{n})} 
\lesssim   2^{-j} \Vert u -Q \Vert_{B^{\sigma}_{2,1}} \Vert u -Q \Vert_{B^{1}_{\infty,1}}, \quad j \geq 1.
\end{equation}
The point of the estimate \eqref{3fmnm} is that, roughly speaking, we control the Besov norm of $ u $ in terms of its $L^{2}$ norm and the Besov norm of the tangential projection of $ \Delta_{k} u$. More precisely, for $ u \in B^{\sigma}_{Q}$, $ \sigma \geq \frac{n}{2}+10 $, and $j \in \N$ consider the energy  
\begin{equation}\label{enE}
\mathcal{E}_{j} (u) = 2^{j \sigma} \Vert u - Q \Vert_{L^{2}(\R^{n})} + \sum_{ k \geq j} 2^{k \sigma} \Vert \Delta_{k} u - \langle u , \Delta_{k} u \rangle u \Vert_{L^{2} ( \R^{n})}.
\end{equation}
Then we have the following
\begin{corollary}\label{enest1}
There exist $K \in \N$ and $C_{0}:= C_{0}(n,s,\sigma) \geq 1 $ so that 
\begin{equation}\label{4fmnm}
\Vert u -Q \Vert_{B^{\sigma}_{2,1}} \leq C_{0} \mathcal{E}_{K}(u) \lesssim C_{0} 2^{\sigma K} \Vert u -Q \Vert_{B^{\sigma}_{2,1}}
\end{equation}
\end{corollary}
\begin{proof}
The estimate 
$$
\mathcal{E}_{K} (u) \lesssim 2^{\sigma K} \Vert u -Q \Vert_{B^{\sigma}_{2,1}}
$$
is self evident. We only need to prove 
$$
\Vert u -Q \Vert_{B^{\sigma}_{2,1}} \leq C_{0} \mathcal{E}_{K}(u).
$$
Take $K $ large enough so that 
$$
 2^{-K}  \Vert u -Q \Vert_{B^{1}_{\infty,1}} \ll 1.
$$
Now write 
$$
\Delta_{k} u = \langle u , \Delta_{k} u \rangle u + \left( \Delta_{k} - \langle u , \Delta_{k} u \rangle u \right). 
$$
Then we estimate 
\begin{equation*}
\begin{split}
& \Vert u- Q \Vert_{B^{\sigma}_{2,1}} = \sum_{k =0}^{K} 2^{\sigma k} \Vert \Delta_{k}( u -Q) \Vert_{L^{2}(\R^{n})} + \sum_{k > K} 2^{\sigma k} \Vert \Delta_{k} u \Vert_{L^{2}(\R^{n})} 
\\
& \lesssim 2^{\sigma K} \Vert u -Q \Vert_{L^{2}(\R^{n})} + \sum_{k > K} 2^{\sigma k} \Vert \Delta_{k} u  \Vert_{L^{2}(\R^{n})} 
\\
& \lesssim 2^{\sigma K} \Vert u -Q \Vert_{L^{2}(\R^{n})} + \sum_{k >K} 2^{\sigma k} \Vert \langle u , \Delta_{k} u \rangle \Vert_{L^{2}(\R^{n})} + \sum_{k >K} 2^{\sigma k} \Vert  \Delta_{k} u - \langle u , \Delta_{k} u \rangle u   \Vert_{L^{2}(\R^{n})}
\\
& \overset{\eqref{3fmnm}}{\lesssim} \mathcal{E}_{K}(u) + 2^{-K} \Vert u -Q \Vert_{B^{\sigma}_{2,1}} \Vert u -Q \Vert_{B^{1}_{\infty,1}}
\\
& \leq \mathcal{E}_{K}(u) + \frac{1}{2} \Vert u- Q \Vert_{B^{\sigma}_{2,1}}.
\end{split}
\end{equation*}
This gives the first inequality in \eqref{4fmnm}.
\end{proof}

Next, we need the following technical lemma
\begin{lemma}\label{3lm}
Let $ u \in B^{\sigma}_{Q}$ with $ \sigma \geq \frac{n}{2} +10$. Then for $ j \in \{1,...,n \}$ and $k \geq 1$ consider the following operator 
\begin{equation}\label{ldingterm}
T_{k,j}( f)(x) = \int_{\R^{n}} y_{j} \mathcal{F}^{-1}(\varphi_{k})(y) ( - \Delta)^{s}f(x-y) dy.
\end{equation}
Then we have 
\begin{equation}\label{2fmnm}
  \sum_{k \geq 1} 2^{\sigma k} \Vert  \langle u , T_{k,j} (u)  \rangle \Vert_{L^{2}(\R^{n}) }  \lesssim \Vert u -Q \Vert_{B^{1}_{\infty,1}} + \Vert u -Q \Vert_{B^{1}_{\infty,1}} \Vert u -Q \Vert_{B^{\sigma}_{2,1}}.
\end{equation}
\end{lemma}
\begin{proof}
We would like to use a similar argument to Lemma \ref{2lm}. Namely, we want to rewrite $ \langle u, T_{k,j} u \rangle$ in terms of a bilinear operator which annihilates constants and satisfy a conclusion similar to that of Lemma \ref{1lm}. To that end, for $f,g \in L^{2}(\R^n) $ define 
\begin{equation*}
\tilde{C}_{k}(f,g) = T_{k,j} ( fg) - f T_{k,j}(g) - g T_{k,j} (f).
\end{equation*}
Then we first prove 
\begin{equation}\label{mnm7}
\sum_{k \geq 1} 2^{k \sigma} \Vert \tilde{C}_{k} (f,g) \Vert_{L^{2}(\R^{n})}\lesssim_{} \Vert f \Vert_{B^{1}_{{\infty},1}} \Vert g \Vert_{B^{\sigma+2s-2}_{2,1}} +\Vert g \Vert_{B^{1}_{{\infty},1}} \Vert f \Vert_{B^{\sigma+2s-2}_{2,1}} 
\end{equation}
To prove \eqref{mnm7} write $f_{k_{1}} = \Delta_{k_{1}} f$ and $g_{k_{2}} = \Delta_{k_{2}} g$ and assume without loss of generality that $k_{1} \leq k_{2}$. Moreover, define $L_{k,j} $ via the Fourier transform
$$
\hat{L}_{k,j}(\xi) = (\i \partial_{\xi_{j}}  \varphi_{k} (\xi) )\vert \xi \vert^{2s} .
$$
Then we have 
$$
T_{k,j} (f ) = L_{k,j} * f,
$$
and the kernel $L_{k,j}$ satisfy 
$$
\int_{\R^{n}} L_{k,j} (y) dy =0 \quad \int_{\R^{n}} |y| |L_{k,j}(y)| dy\lesssim 2^{(2s-2)k}.
$$
Then we estimate 
\begin{equation}\label{mnm8}
\begin{split}
& \Vert \tilde{C}_{k} (f_{k_{1}},g_{k_{2}}) \Vert_{L^{2}(\R^{n}) }
\\
& = \Vert \int_{\R^{n}} | L_{k,j}(y) |\left( (f_{k_{1} } ( x -y) - f_{k_{1}}(x) ) ( g_{k_{2}}(x-y) - g_{k_{2}}(x))  \right) dy \Vert_{L^{2}(\R^{n})}
\\
& \lesssim \int_{\R^{n}} \vert  L_{k,j}(y) \vert  \left( \Vert (f_{k_{1} } ( x -y) - f_{k_{1}}(x) ) ( g_{k_{2}}(x-y) - g_{k_{2}}(x))  \Vert_{L^{2}_{x}} \right) dy 
\\
& \lesssim 2^{(2s-2)k} \Vert \nabla f_{k_{1}} \Vert_{L^{\infty}(\R^{n})} \Vert g_{k_{2}} \Vert_{L^{2}(\R^{n})}  
\\
& \lesssim 2^{(2s-2)k} 2^{k_{1}}  \Vert f_{k_{1}} \Vert_{L^{\infty}(\R^{n})} \Vert g_{k_{2}} \Vert_{L^{2}(\R^{n})}  
\\
& \lesssim 2^{-2k} 2^{k_{1}} 2^{2sk_{2}} \Vert f_{k_{1}} \Vert_{L^{\infty}(\R^{n})} \Vert g_{k_{2}} \Vert_{L^{2}(\R^{n})}.
\end{split}
\end{equation}
Multiplying by $2^{\sigma k} $ and summing over $k \leq k_{2} +C$ gives \eqref{mnm7}. 

Now notice that $T_{k,j}(c) =0 $ for any constant. Therefore,
$$
0 = T_{k,j}( | u |^{2}) = 2 \langle u , T_{k,j}(u) \rangle  + \tilde{C}_{k} ( u,u).
$$
Using that $ \tilde{C}_{k} ( u,u) = \tilde{C}_{k}(u-Q,u-Q)$ we end up with 
$$
\langle u , T_{k,j}(u) \rangle = - \frac{1}{2} \tilde{C}_{k}(u-Q,u-Q).
$$
Thus, \eqref{2fmnm} follows from \eqref{mnm7} and the assumption that $ s \in (1/2,1)$.
\end{proof}
Next, for any $k \geq 1$, $u,v : \R^{n} \to \R^{3}$ consider the commutator given by 
\begin{equation}\label{wedgecommutator}
[\Delta_{k},u]^{\wedge}(v) : =\Delta_{k} ( u \wedge v ) -u \wedge \Delta_{k} v.
\end{equation}

Recall the definition of $T_{k,j}$ given in \eqref{ldingterm}. Then write
\begin{equation}\label{defofrk}
[\Delta_{k},u]^{\wedge}((-\Delta)^{s}u)= -\sum_{j=1}^{n} ( \partial_{j} u ) \wedge T_{k,j} u + R_{k}
\end{equation}
for some reminder $R_{k}$. The following gives control of $R_{k}$
\begin{lemma}\label{taylrreminder}
Let $ \sigma \geq \frac{n}{2} +10$, $ w \in B^{\sigma}_{2,1}(\R^{n} :\R^{3})$. 
And let $R_{k}$ be given as above, i.e 
$$
R_{k} = [\Delta_{k},w]^{\wedge}( ( - \Delta)^{s}w)+\sum_{j=1}^{n} ( \partial_{j} w ) \wedge T_{k,j} w
$$
Then we have the following 
$$
\sum_{k \geq 1} 2^{\sigma k} \Vert R_{k} \Vert_{L^{2}(\R^{n})} \lesssim \Vert w \Vert_{B^{2}_{\infty,1}} \Vert w \Vert_{B^{\sigma}_{2,1}}
$$
\end{lemma}
\begin{proof}
for $ k_{1},k_{2} \geq 1$  write $ w_{k_{i}} = \Delta_{k_{i}} w$. And $ w_{0} = \Delta_{\leq 0 } w$. Then we can write 
$$
R_{k} = \sum_{ k_{1},k_{2} \geq 0} R_{k,k_{1},k_{2}},
$$
where 
$$
R_{k,k_{1},k_{2}} : = [\Delta_{k},w_{k_{1}}]^{\wedge}( ( - \Delta)^{s}w_{k_{2}})+\sum_{j=1}^{n} ( \partial_{j} w_{k_{1}} ) \wedge T_{k,j} w_{k_{2}}.
$$
Carrying out the convolution and using that $ \mathcal{F}^{-1}(\varphi_{k})$ has integral $0$ for $k \geq 1$, we obtain
\begin{equation*}
\begin{split}
& R_{k,k_{1},k_{2}}(x)
\\
& = \int_{\R^{n}}  \mathcal{F}^{-1}(\varphi_{k})(y) \left( w_{k_{1}}(x-y) - w_{k_{1}}(x) +y \cdot \nabla w_{k_{1}}(x ) \right) \wedge ( - \Delta)^{s} w_{k_{2}}(x-y) dy.
\end{split}
\end{equation*}
If $ 2^{k_{1}} \lesssim 2^{k_{2}}$ then $ 2^{k} \lesssim 2^{k_{2}}$ and we obtain by Minkowki's inequality, \eqref{multiplier2}, and Taylor's Theorem 
\begin{equation*}
\begin{split}
& \Vert R_{k,k_{1},k_{2}} \Vert_{L^{2}(\R^{n})}
\\
& \lesssim 2^{2sk_{2}} \Vert \nabla^{2} w_{k_{1}} \Vert_{L^{\infty}} \Vert w_{k_{2}} \Vert_{L^{2}(\R^{n})} \Vert \left( \vert y \vert^{2} | \mathcal{F}^{-1} ( \varphi_{k} ) (y) dy \right) \Vert_{L^{1}_{y}}
\\
& \lesssim  2^{2k_{1}} \Vert  w_{k_{1}} \Vert_{L^{\infty}(\R^{n})}   2^{k_{2}2s} 2^{-2k}\Vert w_{k_{2}} \Vert_{L^{2}(\R^{n})}
\end{split}
\end{equation*}
Multiplying by $ 2^{k \sigma}$ and summing over $k \leq k_{2}+C$, and using the fact that $ s\in (1/2,1)$ we obtained 
\begin{equation}\label{rmnd1}
\begin{split}
& \sum_{k=1}^{k_{2} +C} 2^{k \sigma} \Vert R_{k,k_{1},k_{2}} \Vert_{L^{2}(\R^{n})}
\\
& \lesssim 2^{2k_{1}} \Vert  w_{k_{1}} \Vert_{L^{\infty}(\R^{n})} 2^{\sigma k_{2}} \Vert w_{k_{2}} \Vert_{L^{2}(\R^{n})},
\end{split}
\end{equation}
whenever $ 2^{k_{1}} \lesssim 2^{k_{2}}$. On the other hand, if $ 2^{k_{2}} \ll 2^{k_{1}}$. Then we place $ w_{k_{1}} $ in the $L^{2}$ norm, that is 
\begin{equation*}
\begin{split}
& \Vert R_{k,k_{1},k_{2}} \Vert_{L^{2}(\R^{n})}
\\
& \lesssim \Vert w_{k_{1}} \Vert_{L^{2}(\R^{n})} 2^{2sk_{2}} \Vert w_{k_{2}} \Vert_{L^{\infty}(\R^{n})} 
+ \sum_{j =1}^{n} \Vert \partial_{j} w_{k_{1} } \wedge T_{k,j} (w_{k_{2}}) \Vert_{L^{2}}
\end{split}
\end{equation*}
The second term is only non-zero if $ k \leq k_{2} +C$. And so for $k \leq k_{2} +C$ we have 
\begin{equation}
\begin{split}
\sum_{j =1}^{n} \Vert \partial_{j} w_{k_{1} } \wedge T_{k,j} (w_{k_{2}}) \Vert_{L^{2}}
\lesssim \Vert w_{k_{1}} \Vert_{L^{2}(\R^{n})} 2^{2sk_{2}+k_{1}} 2^{-k} \Vert w_{k_{2}} \Vert_{L^{\infty}(\R^{n})}
\end{split}
\end{equation}
using this and the assumption that $ s < 1$, we obtain 
\begin{equation}\label{rmnd2}
\begin{split}
& \sum_{k=1}^{k_{1} +C} 2^{k \sigma} \Vert R_{k,k_{1},k_{2}} \Vert_{L^{2}(\R^{n})}
\\
& \lesssim 2^{\sigma k_{1}} \Vert  w_{k_{1}} \Vert_{L^{2}(\R^{n})} 2^{2 k_{2}} \Vert w_{k_{2}} \Vert_{L^{\infty}(\R^{n})} + \Vert w_{k_{1}} \Vert_{L^{2}(\R^{n})} 2^{2sk_{2}+k_{1}} \Vert w_{k_{2}} \Vert_{L^{\infty}(\R^{n})} \sum_{k=1}^{k_{2}+C} 2^{k ( \sigma -1)}
\\
& \lesssim 2^{\sigma k_{1}} \Vert  w_{k_{1}} \Vert_{L^{2}(\R^{n})} 2^{2 k_{2}} \Vert w_{k_{2}} \Vert_{L^{\infty}(\R^{n})}.
\end{split}
\end{equation}
whenever $ 2^{k_{2}} \ll 2^{k_{1}}$. Combining \eqref{rmnd1} and \eqref{rmnd2} gives the result. 
\end{proof}
Now we use the above estimates to prove the following 
\begin{lemma}\label{tngvk}
Let $ u : \R^{n} \to \S^{2}$, $ u \in B^{\sigma}_{Q}$ with $ \sigma \geq \frac{n}{2} +10$. For $k \geq 1$ define 
$$
v_{k} := \Delta_{k} u - \langle u , \Delta_{k} u \rangle u. 
$$
Then we have the following 
\begin{equation*}
\begin{split}
& \vert \int_{\R^{n}} \langle v_{k}(x), \Delta_{k} ( u \wedge (- \Delta)^{s} u ) \rangle dx \vert
\\
& \lesssim \Vert v_{k} \Vert_{L^{2}(\R^{n})} e_{k},
\end{split}
\end{equation*}
where the constants $e_{k}$ satisfy 
\begin{equation}\label{ekbounds}
\sum_{k \geq 1} 2^{\sigma k} e_{k} \lesssim \left( 1+ \Vert u -Q\Vert_{B^{2}_{\infty,1}} \right)^{2} \Vert u - Q \Vert_{B^{\sigma}_{2,1}}.
\end{equation}
\end{lemma}
\begin{proof}
For $k \geq 1$ we can write ( using the definition of $R_{k}$, \eqref{defofrk})
$$
\Delta_{k} ( u \wedge (- \Delta)^{s} u )= u \wedge ( - \Delta)^{s} \Delta_{k} u - \sum_{j=1}^{n} ( \partial_{j} u) \wedge T_{k,j} u + R_{k}.
$$
Next, since $ \Delta_{k} u = v_{k} + \langle u , \Delta_{k} u \rangle u$, then we have 
$$
u \wedge (( - \Delta)^{s} \Delta_{k}u) = u \wedge  ( - \Delta)^{s} v_{k} + u \wedge [ ( - \Delta)^{s} , u  ] ( \langle u , \Delta_{k } u \rangle ).
$$
Therefore, using that $ \langle v_{k} , u \rangle = 0$ for any $k \geq 1$, we obtain 
\begin{equation}\label{vbv1}
\begin{split}
& \vert \int_{\R^{n}} \langle v_{k} (x) , \Delta_{k} ( u \wedge (- \Delta)^{s} u ) \rangle dx \vert 
\\
& \lesssim \vert \int_{\R^{n}} \langle v_{k} (x) , u \wedge (- \Delta)^{s} v_{k} \rangle dx \vert
\\
& + \vert \int_{\R^{n}} \vert v_{k} \vert  \vert u  \wedge [ (- \Delta)^{s} u]( \langle u , \Delta_{k} u \rangle) \vert dx
\\
& \overset{\eqref{comestzz1}, \text{H\"older}}{\lesssim} \Vert \nabla u \Vert^{2s}_{L^{\infty}} \Vert v_{k} \Vert^{2}_{L^{2}} + \Vert v_{k} \Vert_{L^{2}} \Vert [ (- \Delta)^{s} , u ] ( \langle u , \Delta_{k} u \rangle ) \Vert_{L^{2}}
\\
& \overset{\eqref{comestzz2}}{\lesssim} \Vert \nabla u \Vert^{2s}_{L^{\infty}} \Vert v_{k} \Vert^{2}_{L^{2}} +  \Vert v_{k} \Vert_{L^{2}} \left( 1 + \Vert u -Q \Vert_{B^{2}_{\infty,1}} \right) \Vert \langle u , \Delta_{k} u \rangle \Vert_{H^{1}}
\\
& \lesssim \left( 1 + \Vert u -Q \Vert_{B^{2}_{\infty,1}} \right)^{2} \Vert v_{k} \Vert^{2}_{L^{2}} +  \Vert v_{k} \Vert_{L^{2}} \left( 1 + \Vert u -Q \Vert_{B^{2}_{\infty,1}} \right) \Vert \langle u , \Delta_{k} u \rangle \Vert_{H^{1}}
\end{split}
\end{equation}
where in the last line we used that $ s < 1$.
Next, by H\"older inequality we have 
\begin{equation}\label{vbv2}
\vert \int_{\R^{n}} \langle v_{k} ( x) , R_{k} (x) \rangle d x \vert  \leq \Vert v_{k} \Vert_{L^{2}} \Vert R_{k} \Vert_{L^{2}}.
\end{equation}
So it remains to estimate the terms involving $\partial_{j} u \wedge T_{k,j} (u)$. Notice that since $ v_{k}, \partial_{j} u $, and $ T_{k,j} u - \langle u , T_{k,j} u \rangle u $ are all orthogonal to $u$ then we can write 
$$
\langle v_{k} , \partial_{j} u \wedge T_{k,j} u \rangle = \langle u , T_{k,j} u \rangle \langle v_{k} , \partial_{j} u \wedge u \rangle.
$$
Hence, 
\begin{equation}\label{vbv3}
\vert \int_{\R^{n}} \langle v_{k} ( x) , \partial_{j} u(x) \wedge T_{k,j} u(x) \rangle dx \vert \leq \Vert \partial_{j} u \Vert_{L^{\infty}} \Vert v_{k} \Vert_{L^{2}} \Vert \langle u , T_{k,j} u \rangle \Vert_{L^{2}}.
\end{equation}
Combining \eqref{vbv1},\eqref{vbv2} and \eqref{vbv3} we obtain
\begin{equation}
\vert \int_{\R^{n}} \langle v_{k}(x), \Delta_{k} ( u \wedge (- \Delta)^{s} u ) \rangle dx \vert
 \lesssim \Vert v_{k} \Vert_{L^{2}(\R^{n})} e_{k}
\end{equation}

where 
\begin{equation}
\begin{split}
& e_{k} = C \left( 1+ \Vert u - Q \Vert_{B^{2}_{\infty,1}} \right) \left( \Vert \langle u , \Delta_{k} u \rangle \Vert_{H^{1}} + \Vert R_{k} \Vert_{L^{2}} + \Vert \langle u , T_{k,j} u \rangle \Vert_{L^{2}}   \right)
\\
&+  \left( 1+ \Vert u - Q \Vert_{B^{2}_{\infty,1}} \right)^{2}  \Vert v_{k} \Vert_{L^{2}}
\end{split}
\end{equation}
Then the estimate \eqref{ekbounds} follows from Lemma \ref{taylrreminder}, Lemma \ref{3lm}, and Lemma \ref{2lm}. 
\end{proof}
We apply the above estimates to the approximate solutions $ u^{\epsilon}$ to obtain the following 
\begin{lemma}\label{estuepsilon}
Let $ u^{\epsilon} \in C([ 0,T] : B^{\sigma}_{Q})$ be a solution to \eqref{SSparabolic} for some $T > 0$. Let $ R > 0 $ be large enough so that 
$$
1+ \sup_{t \in [0,T^{}]} \Vert u^{\epsilon }-Q \Vert_{B^{2}_{\infty,1}} \leq R.
$$
Then there exists $K \in \N$ which depends on $ \sigma , R$ and a constant $C:= C ( n,s, \sigma) \geq 1$ so that 
$$
\mathcal{E}_{K} ( u^{\epsilon}(t))  \leq \mathcal{E}_{K} ( u_{0} ) \text{exp} ( C R^{2} t ) 
$$
for any $ t \in [0,T ]$. 
Here $ \mathcal{E}_{K}$ is defined in \eqref{enE}. 
\end{lemma}
\begin{proof}
Choose $K:= K(R,\sigma)$ to be large enough so that the conclusion of Corollary \ref{enest1} holds for every $ t \in [0,T]$. Namely,  
\begin{equation}\label{c0kr0}
 \Vert u^{\epsilon} (\cdot,t) -Q \Vert_{B^{\sigma}_{2,1}} \leq C_{0} \mathcal{E}_{K}(u^{\epsilon}(\cdot,t)) \lesssim C_{0} 2^{\sigma K} \Vert u^{\epsilon} (\cdot,t) -Q \Vert_{B^{\sigma}_{2,1}} ,\quad t \in [0,T].
\end{equation}
For better readability, in what follows we write $ u^{\epsilon} (t)$ to mean the function $ u^{\epsilon}(\cdot,t)$. By definition 
\begin{equation}\label{engdefprf}
\mathcal{E}_{K} ( u^{\epsilon} (t))= 2^{K \sigma} \Vert u^{\epsilon}(t) - Q \Vert_{L^{2}(\R^{n})} + \sum_{ k \geq K} 2^{k \sigma} \Vert \Delta_{k} u^{\epsilon}(t) - \langle u^{\epsilon}(t) , \Delta_{k} u^{\epsilon}(t) \rangle u^{\epsilon} (t)\Vert_{L^{2} ( \R^{n})}.
\end{equation}
We first estimate the tangential projection. Set 
$$
v_{k}(t) : = \Delta_{k} u^{\epsilon}(t) - \langle u^{\epsilon}(t) , \Delta_{k} u^{\epsilon}(t) \rangle u^{\epsilon}(t). 
$$
Then by definition $ \langle v_{k}(t) , u^{\epsilon} (t) \rangle =0 $. Hence,
taking the derivative of $ \Vert  v_{k} \Vert^{2}_{L^{2}}$ we obtain
\begin{equation}\label{bpz1}
\frac{1}{2}\frac{d}{dt} \Vert v_{k} (t) \Vert^{2}_{L^{2}(\R^{n})} = \int_{\R^{n} } \langle v_{k} (t), \Delta_{k} \partial_{t} u^{\epsilon}(t) \rangle dx - \int_{\R^{n}} \langle u^{\epsilon}(t) , \Delta_{k} u^{\epsilon}(t) \rangle \langle v_{k} (t), \partial_{t} u^{\epsilon} \rangle dx.
\end{equation}
We estimate the first term on the right hand side of \eqref{bpz1}. Using that 
$$
\partial_{t} u^{\epsilon} = \epsilon \Delta u^{\epsilon} - u^{\epsilon} \wedge (-\Delta)^s u^{\epsilon} + \epsilon (  |\nabla u^{\epsilon} |^{2} u^{\epsilon})
$$
we have 
\begin{equation}\label{bpz2}
\begin{split}
& \int_{\R^{n} } \langle v_{k} (t), \Delta_{k} \partial_{t} u^{\epsilon}(t) \rangle dx 
\\
& =  \int_{\R^{n}} \langle v_{k}(t), \Delta_{k} \left( (-\Delta)^{s} u^{\epsilon}(t) \wedge u^{\epsilon}(t) \right) \rangle dx  
\\
& + \epsilon  \int_{\R^{n}} \langle v_{k} (t) , \Delta_{k} \Delta u^{\epsilon}(t) \rangle dx 
\\
& + \epsilon  \int_{\R^{n}} \langle v_{k}(t) , \Delta_{k} \left( \vert \nabla u^{\epsilon}(t) \vert^{2} u^{\epsilon}(t) \right) \rangle dx.
\end{split}
\end{equation}
For the first term on the right hand side of \eqref{bpz2} we use Lemma \ref{tngvk} to obtain
\begin{equation}\label{bpz3}
\vert \int_{\R^{n}} \langle v_{k}(t) ,\Delta_{k} \left( ( - \Delta)^{s} u^{\epsilon}(t) \wedge u^{\epsilon}(t) \right) \rangle dx \vert \lesssim \Vert v_{k}(t) \Vert_{L^{2}} e_{k}(t),
\end{equation}
where $ e_{k}(t)$ satisfy \eqref{ekbounds} with $u^{\epsilon}(t)$. For the second term on the right hand side of \eqref{bpz2} notice that we have 
\begin{equation*}
\begin{split}
& \Delta (\Delta_{k} u^{\epsilon} (t)) = \Delta v_{k}(t) + u^{\epsilon}(t)  \Delta \langle u^{\epsilon} (t) , \Delta_{k} u^{\epsilon}(t) \rangle + 2 \sum_{j =1}^{n} \partial_{j} u^{\epsilon}(t) \partial_{j} \langle u^{\epsilon}(t) , \Delta_{k} u^{\epsilon}(t) \rangle + \Delta u^{\epsilon} \langle u^{\epsilon}, \Delta_{k} u^{\epsilon } \rangle.
\end{split}
\end{equation*}
By integration by parts and using that $ \langle u^{\epsilon} , v_{k} \rangle =0$ we obtain 
\begin{equation*}
\int_{\R^{n}} \langle v_{k} (t) , \Delta (\Delta_{k} u^{\epsilon}(t)) \rangle dx \leq - \Vert \nabla v_{k} (t) \Vert^{2}_{L^{2}(\R^{n})} + C R \Vert v_{k}(t) \Vert_{L^{2}(\R^{n})} \Vert \langle u^{\epsilon}(t) , \Delta_{k} u^{\epsilon}(t) \rangle \Vert_{H^{1}(\R^{n})}.
\end{equation*}
In particular, 
\begin{equation}\label{bpz4}
\int_{\R^{n}} \langle v_{k} (t) , \Delta (\Delta_{k} u^{\epsilon}(t)) \rangle dx \lesssim R \Vert v_{k} \Vert_{L^{2}(\R^{n})}  \Vert \langle u^{\epsilon}(t) , \Delta_{k} u^{\epsilon}(t) \rangle \Vert_{H^{1}(\R^{n})}.
\end{equation}
Lastly, for the third term on the right hand side of \eqref{bpz2} we have 
$$
\Delta_{k} \left( \vert \nabla u^{\epsilon} (t) \vert^{2} u^{\epsilon} \right) = \vert \nabla u^{\epsilon} (t) \vert^{2} \Delta_{k} ( u^{\epsilon} (t) ) + u^{\epsilon} ( t) \Delta_{k} ( \vert \nabla u^{\epsilon}(t) \vert^{2} ) + C_{k} ( \vert \nabla u^{\epsilon}(t) \vert^{2} , u^{\epsilon}(t) - Q),
$$
where $C_{k}$ is defined in \eqref{ckdef}. Then we have (again using that $\langle u^{\epsilon}, v_{k} \rangle =0$) 
\begin{equation}\label{bpz5}
\vert \int_{\R^{n}} \langle v_{k}(t) , \Delta_{k} \left( \vert \nabla u^{\epsilon}(t) \vert^{2} u^{\epsilon}(t) \right) \rangle dx\vert  \lesssim R^{2} \Vert v_{k} (t) \Vert^{2}_{L^{2}(\R^{n})} + \Vert v_{k}(t) \Vert_{L^{2}(\R^{n})} \Vert C_{k}( \vert \nabla u^{\epsilon} (t)\vert^{2}, u^{\epsilon}(t) - Q) \Vert_{L^{2}(\R^{n})} 
\end{equation}
Combining \eqref{bpz3}, \eqref{bpz4}, and \eqref{bpz5} we obtain
\begin{equation}\label{bpz6}
\frac{1}{2} \frac{d}{dt} \left( \Vert v_{k} (t) \Vert^{2}_{L^{2}( \R^{n})} \right) \leq \Vert v_{k}(t) \Vert_{L^{2}(\R^{n})} a_{k}(t)
\end{equation}
where  
\begin{equation*}
\begin{split}
& a_{k}(t) : = C \left( e_{k} (t) + R \Vert \langle u^{\epsilon}(t) , \Delta_{k} u^{\epsilon}(t) \rangle \Vert_{H^{1}(\R^{n})} + \Vert C_{k} ( \vert \nabla u^{\epsilon}(t)  \vert^{2}, u^{\epsilon}(t) -Q )\Vert_{L^{2}(\R^{n})} \right)
\\
& + C R^{2} \left( \Vert v_{k} (t) \Vert_{L^{2}(\R^{n})} + \Vert \langle u^{\epsilon}(t), \Delta_{k} u^{\epsilon}(t) \rangle \Vert_{L^{2}(\R^{n})} \right).
\end{split}
\end{equation*}
In particular, using \eqref{ekbounds}, Lemma \ref{taylrreminder},  Lemma \ref{2lm}, and Lemma \ref{1lm} we have 
\begin{equation}\label{bpz7}
\begin{split}
\sum_{k \geq 1} 2^{k \sigma} a_{k}(t) \lesssim R^{2} \Vert u^{\epsilon}(t) - Q \Vert_{B^{\sigma}_{2,1}}.
\end{split}
\end{equation}
The inequalities \eqref{bpz6}, \eqref{bpz7} readily implies  
\begin{equation}\label{bpz8}
\begin{split}
& \sum_{k \geq K} 2^{k \sigma} \Vert v_{k} (t) \Vert_{L^{2}(\R^{n})}
\\
& \leq \sum_{k \geq K} 2^{\sigma k} \Vert v_{k}(0) \Vert_{L^{2}(\R^{n})} + C R^{2} \int_{0}^{t} \Vert u^{\epsilon}(\tilde{t}) - Q \Vert_{B^{\sigma}_{2,1}} d \tilde{t}. 
\end{split}
\end{equation}
This gives us the estimate needed for the second term on the right hand side of \eqref{engdefprf}. It remains to estimate $ \Vert u^{\epsilon}(t) - Q \Vert_{L^{2}(\R^{n})}$. To that end, 
\begin{equation*}
\begin{split}
& \frac{1}{2} \frac{d}{dt} \left( \Vert u^{\epsilon}(t) - Q \Vert^{2}_{L^{2}(\R^{n})} \right) 
\\
& = - \epsilon \Vert \nabla u^{\epsilon} (t) \Vert^{2}_{L^{2}(\R^{n})} + \frac{ \epsilon}{2} \int_{\R^{n}} | \nabla u^{\epsilon}(t) |^{2} | u^{\epsilon} -Q |^{2} dx  
\\
& \lesssim R^{2} \Vert u^{\epsilon} (t) - Q \Vert^{2}_{L^{2}(\R^{n})}.
\end{split}
\end{equation*}
This readily implies 
\begin{equation}\label{bpz9}
\begin{split}
& \Vert u^{\epsilon}(t) - Q \Vert_{L^{2}(\R^{n})} \leq \Vert u_{0} -Q \Vert_{L^{2}( \R^{n})} + C_{0} R^{2} \int_{0}^{t} \Vert u^{\epsilon}(\tilde{t} ) - Q \Vert_{L^{2}(\R^{n})} d \tilde{t}.
\end{split}
\end{equation}
Multiplying \eqref{bpz9} by $2^{\sigma K}$ and combining it with \eqref{bpz8} we obtain, using \eqref{c0kr0},
\begin{equation*}
\begin{split}
& \mathcal{E}_{K} ( u^{\epsilon}(t))  
\\
& \leq \mathcal{E}_{K}( u_{0}) + C R^{2} \int_{0}^{t} \mathcal{E}_{K}(u^{\epsilon}(\tilde{t})) d \tilde{t}.
\end{split}
\end{equation*}
Now the result follows by Gronwall's inequality.
\end{proof}
Now we can prove Proposition \ref{mainuniformsec}
\begin{proof}[Proof of Proposition \ref{mainuniformsec}]
Fix $ \epsilon \in (0,1]$ and let $ u^{\epsilon}$ be the solution to \eqref{SSparabolic} on $[0,T^{\epsilon}]$. Then let $ [0,T^{\epsilon}_{\text{max}})$ the maximal interval for which $ u^{\epsilon}$ solves \eqref{SSparabolic}.   Clearly whenever $ T^{\epsilon}_{\text{max}} < \infty$ then we have $ \lim_{ t \to T^{\epsilon}_{\text{max}}}  \Vert u^{\epsilon}(t) \Vert_{B^{\sigma_{0}}_{Q}} = \infty $. Next, fix a constant $ C$ so that 
$$
\Vert f \Vert_{B^{2}_{\infty,1}} \leq C \Vert f \Vert_{B^{\sigma_{0}}_{2,1}}.
$$
Then set 
$$
R : = 2 C \left( 1+ \Vert u_{0} \Vert_{B^{\sigma_{0}}_{Q}} \right).
$$
Choose $K : = K (R , \sigma_{0})$ as in Lemma \ref{estuepsilon}. And set 
$$
L:= 2 C_{\sigma_{0}} 2^{K \sigma_{0}} \left( 1 + \Vert u_{0} \Vert_{B^{\sigma_0}_{Q}} \right)
$$
with $ C_{\sigma_{0}}$ large enough so that 
$$
\mathcal{E}_{K}(u_{0}) < L/2.
$$
Such a choice is possible in view of Corollary \ref{enest1}.
Now notice that the bounds 
\begin{equation}\label{recr}
1+ \Vert u^{\epsilon}(t) -Q \Vert_{B^{2}_{\infty,1}} \leq R , \quad \mathcal{E}_{K}(u^{\epsilon}(t)) \leq L
\end{equation}
both hold at time $ t = 0$. Assume next that \eqref{recr} holds for some fixed $ t \in [0,T^{\epsilon}_{\text{max}} )$. 
Then by Lemma \ref{estuepsilon} we have 
$$
\mathcal{E}_{K} ( u^{\epsilon}(t)) \leq \mathcal{E}_{K}(u_{0}) \text{exp}\left(  C R^{2} t \right).
$$
So if $T >0$ is chosen small enough so that 
$$
\text{exp}\left(  C R^{2} t \right) \leq \frac{3}{2}
$$
for any $ t \in [0,T]$.
Then we have 
$$
\mathcal{E}_{K} ( u^{\epsilon}(t)) \leq \frac{3}{2} \mathcal{E}_{K}(u_{0}) \leq \frac{3}{4} L.
$$
Similarly, the second bound in \eqref{recr} improves. 
Notice  \eqref{c0kr0} implies 
$$
\Vert u^{\epsilon}  -Q\Vert_{B^{\sigma_{0}}_{2,1}} \leq C_{0} L.
$$
Moreover, since $ u^{\epsilon}$ solves \eqref{SSparabolic} then we have by the Algebra properties of Besov spaces
\begin{equation*}
\begin{split}
& \Vert \partial_{t} u^{\epsilon} \Vert_{B^{2}_{\infty,1}}
\\
& \lesssim \Vert \partial_{t} u^{\epsilon} \Vert_{B^{\sigma_{0} -2}_{2,1}}
\\
& \lesssim \left( 1+ \Vert u^{\epsilon} \Vert_{B^{\sigma_{0}}_{2,1}} \right)^{3}
\\
& \leq C' \left( 1+ C_{0} L \right)^{3}
\end{split}
\end{equation*}
for some constant $C'$ that depends only on $n,s$ and $ \sigma_{0}$.
This readily implies 
$$
1+ \Vert u^{\epsilon}(t) -Q\Vert_{B^{2}_{\infty,1}} \leq C'( 1+ C_{0} L)^{3}T + \frac{R}{2}   
$$
and if $T $ was chosen so that 
$$
C'( 1+ C_{0} L)^{3}T \leq \frac{R}{4} 
$$
then we obtain
$$
1+ \Vert u^{\epsilon}(t) -Q\Vert_{B^{2}_{\infty,1}} \leq \frac{3}{4} R.
$$
Therefore, we conclude that the bounds in \eqref{recr} hold for any $ t \in [0,T]$. Consequently 
 \begin{equation}\label{uniformepsilonz}
    \sup_{0< \epsilon \leq 1} \sup_{t \in [0,T]} \Vert u^{\epsilon} \Vert_{B^{\sigma_{0}}_{Q}} \lesssim \Vert u_{0} \Vert_{B^{\sigma_{0}}_{Q}}.
 \end{equation}
Hence,
$T < T^{\epsilon}_{\text{max }}.$ And the solutions $u^{\epsilon}$ exist on the common interval $[0,T]$. It remains to show persistent of regularity. Assume $u_{0} \in B^{\sigma}_{Q}$ for some $ \sigma \geq \sigma_{0}$. Then by Proposition \ref{parapolicguys} we obtain that $ u^{\epsilon} \in C( [0,T] : B^{\sigma}_{Q})$. Using \eqref{uniformepsilonz} we have for any $t \in [0,T]$ and any $ \epsilon \in (0,1]$
$$
 \Vert u^{\epsilon}(t)-Q \Vert_{B^{2}_{\infty,1}} \leq C \Vert u_{0} \Vert_{B^{\sigma_{0}}_{Q}}.
$$
Set $R :=  C \Vert u_{0} \Vert_{B^{\sigma_{0}}_{Q}} +1$.
Hence choosing $K := K( R, \sigma)$ so that we have the conclusion of Lemma \ref{estuepsilon} holds, we obtain
\begin{equation*}
\begin{split}
& \Vert u^{\epsilon}(t) \Vert_{B^{\sigma}_{Q}}
\\
& \lesssim \mathcal{E}_{K}(u_{0}) \text{exp}(C R^{2} T)
\\
& \overset{\eqref{4fmnm}}{\lesssim} 2^{K \sigma}  \text{exp}(C R^{2} T) \Vert u_{0} \Vert_{B^{\sigma}_{Q}}.
\end{split}
\end{equation*}
This gives \eqref{uniformepsilon} since $ R,K,$ depend only on $\Vert u \Vert_{B^{\sigma_{0}}_{Q}} $  and $ \sigma$.
\end{proof}
\subsection{Convergence of the approximate solutions} In this section we show that the solutions $u^{\epsilon}$ converge to a function $ u$ which solve the $s$-Schr\"odinger map on the same uniform time interval $[0,T]$. We start with the following
\begin{lemma}\label{convergenceueps}
Let $ \epsilon , \delta \in (0,1)$ and let $u^{\epsilon}$, $u^{\delta} :\R^{n+1} \to \S^{2}$ be the solutions to \eqref{SSparabolic} obtained in Proposition \ref{parapolicguys} on the time interval $[0,T]$ where $T$ does not depend on $ \epsilon$ or $ \delta$. Then we have
\begin{equation}\label{cauchyuep}
\sup_{t \in [0,T]} \Vert u^{\epsilon}(t) - u^{\delta}(t) \Vert_{H^{s}} \lesssim \epsilon + \delta.
\end{equation}
\end{lemma}
\begin{proof}
Set $ v : = u^{\epsilon} - u^{\delta}$. Then notice that a direct computation gives ( recall the definition $ [ (- \Delta)^{s} ,f ]^{\wedge} g := (- \Delta)^{s} ( f \wedge g ) - f \wedge ( - \Delta)^s g$ )
\begin{equation}\label{dervv}
\begin{split}
& \frac{1}{2} \frac{d}{dt} \Vert v(t) \Vert^{2}_{H^{s}}
\\
& = \int_{\R^{n}} \langle [ ( - \Delta)^{s/2},u^{\epsilon}(t)]^{\wedge} v(t) , (-\Delta)^{s/2} v(t) \rangle dx
\\
&+ \int_{\R^{n}} \langle [ ( - \Delta)^{s/2}, (-\Delta)^{s} u^{\delta}(t)]^{\wedge} v(t) , ( - \Delta)^{s/2} v(t) \rangle dx
\\
& + \int_{\R^{n}} \langle v(t) ,  E \rangle dx + \int_{\R^{n}} \langle ( - \Delta)^{s/2} v(t) , (-\Delta)^{s/2}  E \rangle dx
\end{split}
\end{equation}
where 
$$
E:= \epsilon \left( \Delta u^{\epsilon} + | \nabla u^{\epsilon}|^{2} u^{\epsilon} \right) - \delta \left( \Delta u^{\delta} + | \nabla u^{\delta}|^{2} u^{\delta} \right).
$$
Notice that by the embeddings $ B^{\sigma_{0}}_{2,1} \hookrightarrow W^{1,\infty}$, $ B_{2,1}^{\sigma_{0} -2s} \hookrightarrow W^{1,\infty}$ and the uniform estimates  \eqref{uniformepsilon} we have for any $ t\in [0,T]$
\begin{equation}\label{jhg}
\Vert u^{\epsilon}(t) \Vert_{W^{1,\infty}} + \Vert u^{\delta}(t) \Vert_{W^{1,\infty}} \leq C
\end{equation}
where $C := C(n,s\sigma_{0}, \Vert u_{0} \Vert_{B^{\sigma_{0}}_{Q}})$. Similarly, 
\begin{equation}\label{jhg1}
\Vert E \Vert_{H^{s}} \leq C ( \epsilon  + \delta).
\end{equation}
Moreover, 
the same argument that lead to the proof of \eqref{comestzz2} gives   for any $ f,g: \R^{n} \to \R^{3}$ the following 
\begin{equation}\label{esyestimate}
\begin{split}
& \Vert [ ( - \Delta)^{s/2}, f ]^{\wedge} g \Vert_{L^{2}}
\\
& \lesssim \left( \Vert f \Vert_{L^{\infty}} + \Vert \nabla f \Vert_{L^{\infty}} \right) \Vert g \Vert_{L^{2}}.
\end{split}
\end{equation}
Hence using  \eqref{esyestimate}, \eqref{jhg}, and \eqref{jhg1} we have 
\begin{equation*}
\begin{split}
& \frac{d}{dt} \Vert v(t) \Vert^{2}_{H^{s}}
\\
& \leq C \Vert v(t) \Vert^{2}_{H^{s}} + C ( \epsilon + \delta) \Vert v(t) \Vert_{H^{s}}.
\end{split}
\end{equation*}
Since $ v(0) =0$ then this readily implies 
$$
\Vert v(t) \Vert_{H^{s}} \leq C ( \epsilon + \delta)
$$
for any $ t \in [0,T]$. This proves the result.

\end{proof}
Now we finish the proof of Proposition \ref{localexistencethm}
\begin{proof}[Proof of Proposition \ref{localexistencethm}]
Let $ u_{0} \in B^{\sigma_{0}}_{Q}$ with $ \sigma_{0} \geq \frac{n}{2} +10$. Then denote by  $ u^{\epsilon}$ the solutions to \eqref{SSparabolic}. Set $ \epsilon_{m} := 2^{-m}$ then by Lemma \ref{convergenceueps}, there exists $ u : \R^{n}\times [0,T] \to \R^{3}$ with $ u -Q \in C ( [0,T], H^{s})$ and
\begin{equation}\label{convergence}
u^{\epsilon_{m}} -Q \to u-Q \quad \text{in } C([0,T]: H^{s} ).
\end{equation}
Now we show that $u \in L^{\infty} ( [ 0,T] :B^{\sigma}_{Q})$ whenever $ u_{0} \in B^{\sigma}_{Q}$ with $ \sigma \geq \sigma_{0}$. To that end, fix an $M \in \N$ then
\begin{equation*}
\begin{split}
& \Vert \Delta_{\leq 0} (u(\cdot,t) -Q) \Vert_{L^{2}(\R^{n})} + \sum_{k=1}^{M} 2^{k \sigma} \Vert \Delta_{k} (u(\cdot,t) -Q) \Vert_{L^{2}(\R^{n})}
\\
& = \lim_{m \to \infty} \left( \Vert \Delta_{\leq 0} (u^{\epsilon_{m}}(\cdot,t) -Q) \Vert_{L^{2}(\R^{n})} + \sum_{k=1}^{M}2^{k \sigma} \Vert \Delta_{k} (u^{\epsilon_{m}}(\cdot,t) -Q) \Vert_{L^{2}(\R^{n})} \right)
\\
& \overset{\eqref{uniformepsilon}}{\lesssim}  \Vert u_{0} \Vert_{B^{\sigma}_{Q}}.
\end{split}
\end{equation*}
Letting $M \to \infty$
we obtain
\begin{equation}\label{xcxc}
\sup_{t \in [0,T]} \Vert u(\cdot,t) \Vert_{B^{\sigma}_{Q}} \lesssim \Vert u_{0} \Vert_{B^{\sigma}_{Q}}
\end{equation}
hence
$ u \in L^{\infty} ( [0,T] : B^{\sigma}_{Q})$. Now we show that $u$ solves the $ s$-Schr\"odinger map equation. To that end, we need to show for any $ r \in [0, \sigma_{0})$ the following convergence holds 
\begin{equation}\label{stongerconv}
u^{\epsilon_{m}} -Q \to u -Q \quad \text{ in } C ( [0,T] : B^{r}_{2,1} ).
\end{equation}
In view of \eqref{convergence}, we obtain \eqref{stongerconv} once we prove 
\begin{equation}\label{strongconv2}
\sum_{ k \geq 1} 2^{rk} \Vert \Delta_{k} ( u^{\epsilon_{m}}(t) - u(t) ) \Vert_{L^{2}(\R^{n})} \to 0
\end{equation}
as $m \to \infty$. Fix $ K \in \N$ then 
\begin{equation*}
\begin{split}
& \sum_{ k \geq 1} 2^{rk} \Vert \Delta_{k} ( u^{\epsilon_{m}}(t) - u(t) ) \Vert_{L^{2}(\R^{n})}
\\
& \leq \sum^{K}_{ k \geq 1} 2^{rk} \Vert \Delta_{k} ( u^{\epsilon_{m}}(t) - u(t) ) \Vert_{L^{2}(\R^{n})} + \sum_{ k > K} 2^{rk} \Vert \Delta_{k} ( u^{\epsilon_{m}}(t) - u(t) ) \Vert_{L^{2}(\R^{n})}
\\
& \lesssim \sum^{K}_{ k \geq 1} 2^{rk} \Vert \Delta_{k} ( u^{\epsilon_{m}}(t) - u(t) ) \Vert_{L^{2}(\R^{n})} 
\\
& + 2^{- ( \sigma_{0}-r)K} \left( \Vert u^{\epsilon_{m}}(t) - Q \Vert_{B^{\sigma_{0}}_{2,1}} + \Vert u(t) - Q \Vert_{B^{\sigma_{0}}_{2,1}} \right)
\\
& \overset{\eqref{uniformepsilon}, \eqref{xcxc}}{\lesssim}  \sum^{K}_{ k \geq 1} 2^{rk} \Vert \Delta_{k} ( u^{\epsilon_{m}}(t) - u(t) ) \Vert_{L^{2}(\R^{n})} + 2^{-(\sigma_{0}-r) K} \Vert u_{0} \Vert_{B^{\sigma_{0}}_{Q}} 
\end{split}
\end{equation*}
letting $m \to \infty$ first, and then letting $ K \to \infty $ yields \eqref{strongconv2} which proves \eqref{stongerconv}. Notice that since $ \sigma_{0} \geq \frac{n}{2} +10$ then using  \eqref{stongerconv} with $ r= \sigma_{0} -1 $ implies that $ u^{\epsilon_{m}} \to u$ uniformly. In particular $ | u| =1$. Moreover,
 using \eqref{stongerconv} and \eqref{uniformepsilon} we obtain 
$$
\mathcal{N}^{\epsilon_{m}} ( u^{\epsilon_{m}}) \to - u \wedge (- \Delta)^{s}u \quad \text{ in } C ( [0,T] : B^{\sigma_{0} -1 -2s}_{2,1}) 
$$
where $ \mathcal{N}^{\epsilon_{m}} ( u^{\epsilon_{m}})$ is the right hand side of \eqref{SSparabolic}. 
In particular, since $ \sigma_{0} \geq \frac{n}{2} +10$ then
$$
\mathcal{N}^{\epsilon_{m}} ( u^{\epsilon_{m}}) \to - u \wedge (- \Delta)^{s}u 
$$
uniformly as $m \to \infty$. This gives the result since by construction 
$$
u^{\epsilon_{m}}(x,t) = u_{0}(x) + \int_{0}^{t} \mathcal{N}^{\epsilon_{m}}(u^{\epsilon_{m}} ) (x,\tilde{t}) + \epsilon_{m} \Delta u^{\epsilon_{m}} (x, \tilde{t}) d \tilde{t} 
$$
taking the limit of both sides  
$$
u(x,t) = u_{0} (x) - \int_{0}^{t} u (x,\tilde{t}) \wedge (-\Delta)^{s} u (x,\tilde{t}) d \tilde{t}.  
$$
This proves Proposition \ref{localexistencethm} on the interval $[0,T]$. To show that the solution can be extended to $ [-T,T]$. Denote the solution we obtained above by $u^{+}$. Then we use the same argument but for the equation 
\begin{equation}\label{sreverse}
\begin{cases}
 \partial_{t} u =  u \wedge (-\Delta)^s u \mbox{    } \mbox{ on } \R^{n} \times [0,T] \\
 u(x,0) = u_{0}.
\end{cases}
\end{equation}
Namely, we omit the negative sign. Then one obtains a solution $ {u}^{-} $ of \eqref{sreverse} on the same interval $[0,T]$. Then to conclude we simply set
\begin{equation}\
u(x,t)=
\begin{cases}
  u^{+}(x,t) \quad t \in [0,T] \\
 u^{-}(x,-t) \quad t \in [-T,0]
\end{cases}
\end{equation}

\end{proof}

\section{Ghost of a null structure}\label{null}
In this section we include the calculations mentioned in the introduction regarding the null structure when $ s \in (0,1)$. These calculations were not used in the proof of Theorem \ref{mainthm}. However, since this reduction is valid for any $ s \in (0,1)$, we include it for future use.
% This could serve as a starting point for investigating well posedness for the regime $ s \in (0,1/2)$.\\
First, we recall the following lemma, which was proven in \cite{subcritical}
\begin{lemma}\label{nullred}
Let $u : \R^{n+1} \to \S^{2}$ with $ \Vert u -Q \Vert_{L^{\infty}_{x,t}} \ll 1$, $ u \in C( \R :B^{n/2 + 10}_{Q})$. Then set $ f= L(u)$ where $L : \S^{2} \to \C $ is the stereographic projection. Then $u $ solves \eqref{SSchrodingereq} if and only if $ f $ solves 
\begin{equation}\label{null1}
( \i \partial_{t} - ( - \Delta)^{s}) f =  H_{ ( - \Delta)^{s}} ( f, \frac{1}{1+|f|^{2}}) + f H_{(-\Delta)^{s}} ( f, \frac{ \bar{f}}{1+|f|^{2}} ).
\end{equation}
\end{lemma}
Next, we simplify the right hand side of \eqref{null1}.
\begin{lemma}
Let $ f : \R^{n} \to \C $ and $ s \in (0,1)$. Define 
$$
\Phi(f) : = \frac{ \bar{f}}{1+ | f |^{2}}.
$$
Then we have 
\begin{equation}\label{null3}
\begin{split}
& H_{ ( - \Delta)^{s}} ( f, \frac{1}{1+|f|^{2}}) + f H_{(-\Delta)^{s}} ( f, \frac{ \bar{f}}{1+|f|^{2}})
\\
& = \frac{-\bar{f}}{1+|f|^{2}} H_{(-\Delta)^{s}}(f,f) + C_{s} \int_{\R^{n}} \frac{(f(x)-f(y))^{2} ( \Phi(f)(x) - \Phi(f)(y) }{|x-y|^{n+2s}} dy.
\end{split}
\end{equation}
\end{lemma}
\begin{proof}
Observe that
\[
H_{(-\Delta)^s}\left(f,\frac{1}{1+|f|^2}\right)
=-H_{(-\Delta)^s}\left(f,\frac{|f|^2}{1+|f|^2}\right)
=-H_{(-\Delta)^s}\bigl(f,f\Phi(f)\bigr).
\]
Since $s\in(0,1)$, we have
\[
H_{(-\Delta)^s}(f,g)(x)
=C_s\int_{\mathbb R^n}
\frac{(f(x)-f(y))(g(x)-g(y))}
{|x-y|^{n+2s}}\,dy.
\]
Therefore,
\begin{align*}
&H_{(-\Delta)^s}\left(f,\frac{1}{1+|f|^2}\right)(x)
+f(x)H_{(-\Delta)^s}\bigl(f,\Phi(f)\bigr)(x)
+\Phi(f)(x)H_{(-\Delta)^s}(f,f)(x)\\
&=-H_{(-\Delta)^s}\bigl(f,f\Phi(f)\bigr)(x)
+f(x)H_{(-\Delta)^s}\bigl(f,\Phi(f)\bigr)(x)
+\Phi(f)(x)H_{(-\Delta)^s}(f,f)(x)\\
&=-C_s\int_{\mathbb R^n}
\frac{(f(x)-f(y))
\bigl(f(x)\Phi(f)(x)-f(y)\Phi(f)(y)\bigr)}
{|x-y|^{n+2s}}\,dy\\
&\quad+C_s\int_{\mathbb R^n}
\frac{f(x)(f(x)-f(y))
\bigl(\Phi(f)(x)-\Phi(f)(y)\bigr)}
{|x-y|^{n+2s}}\,dy\\
&\quad+C_s\int_{\mathbb R^n}
\frac{\Phi(f)(x)(f(x)-f(y))^2}
{|x-y|^{n+2s}}\,dy.
\end{align*}
The numerator of the sum of the above three integrals is
\begin{equation*}
\begin{split}
&-(f(x)-f(y))
\bigl(f(x)\Phi(f)(x)-f(y)\Phi(f)(y)\bigr)\\
&\quad+f(x)(f(x)-f(y))
\bigl(\Phi(f)(x)-\Phi(f)(y)\bigr)
+\Phi(f)(x)(f(x)-f(y))^2\\
&=(f(x)-f(y))
\Bigl[
-f(x)\Phi(f)(x)+f(y)\Phi(f)(y)
+f(x)\Phi(f)(x)-f(x)\Phi(f)(y)\\
&\hspace{42mm}
+f(x)\Phi(f)(x)-f(y)\Phi(f)(x)
\Bigr]\\
&=(f(x)-f(y))
\Bigl[
f(y)\Phi(f)(y)-f(x)\Phi(f)(y)
+f(x)\Phi(f)(x)-f(y)\Phi(f)(x)
\Bigr]\\
&=(f(x)-f(y))^2
\bigl(\Phi(f)(x)-\Phi(f)(y)\bigr).
\end{split}
\end{equation*}
Hence,
\begin{equation*}
\begin{split}
&H_{(-\Delta)^s}\left(f,\frac{1}{1+|f|^2}\right)(x)
+f(x)H_{(-\Delta)^s}\bigl(f,\Phi(f)\bigr)(x)
+\Phi(f)(x)H_{(-\Delta)^s}(f,f)(x) 
\\
&=C_s\int_{\mathbb R^n}
\frac{(f(x)-f(y))^2
\bigl(\Phi(f)(x)-\Phi(f)(y)\bigr)}
{|x-y|^{n+2s}}\,dy.
\end{split}
\end{equation*}
Moving $\Phi(f)(x)H_{(-\Delta)^s}(f,f)(x)$ to the right hand side
gives the result.
\end{proof}
The reminder term on the right hand side of \eqref{null3} is well behaved with respect to our resolution spaces. Namely, for any $ f \in F^{\frac{n}{2}}$, if we define 
$$
R(f)(x,t) := \int_{\R^{n}} \frac{(f(x,t)-f(y,t))^{2} ( \Phi(f)(x,t) - \Phi(f)(y,t) }{|x-y|^{n+2s}} dy
$$
then by a similar argument to Proposition \ref{trilineares1}, one can obtain (for $ s \in (1/2,1) $ ) 
\begin{equation}\label{qj}
\Vert R(f) \Vert_{N^{\frac{n}{2}}} \lesssim \Vert f \Vert^{3}_{F^{\frac{n}{2}}}.
\end{equation}
In view of Lemma  \ref{nullred}, the formula \eqref{null3}, and the estimate \eqref{qj}, it is natural to reduce studying \eqref{SSchrodingereq} to studying the scalar equation 
\begin{equation}
\begin{split}
& ( \i \partial_{t} - ( - \Delta)^{s} ) (f) = - \Phi(f) H_{( - \Delta)^{s}}(f,f) + \int_{\R^{n}} \frac{(f(x,t)-f(y,t))^{2} ( \Phi(f)(x,t) - \Phi(f)(y,t) }{|x-y|^{n+2s}} dy
\end{split}
\end{equation}
The issue however is in the leading term 
$$
H_{(-\Delta)^{s}} (f,f).
$$
Namely, one exploits the null property of this nonlinearity by writing 
$$
H_{(-\Delta)^{s}}(f,g)  = ( \i \partial_{t} - ( - \Delta)^{s} ) (f g) - f ( \i \partial_{t} - ( - \Delta)^{s} ) (g) - g ( \i \partial_{t} - ( - \Delta)^{s} )(f).
$$
Then the lost of derivative alluded to in the introduction is present in the following case: Let $k_{1},k_{2} \geq 100$ with $ 2^{k_{2}} \gg 2^{k_{1}}$,  $ f_{k_{1},j_{1}} \in Z_{k_{1}}$, $ f_{k_{2},j_{2}} \in Z_{k_{2}}$ with $ f_{k_{i},j_{i}}$ having modulation $ 2^{j_{i}}$. Assume further that $j_{1},j_{2} $ satisfy 
$$
\max \{ j_{1} , j_{2} \} \leq k_{1} + (2s-1)k_{2} +20.
$$
Then adapting the argument in \cite[Lemma 7.1]{cmp} to the fractional case, one obtains the following bilinear estimate 
\begin{equation*}
\begin{split}
&\Vert \Delta_{k} ( \i \partial_{t} - ( - \Delta)^{s} +\i )^{-1} \left( f_{k_{2},j_{2}} ( \i \partial_{t} - ( - \Delta)^{s} ) f_{k_{1},j_{1}} \right)\Vert_{Z_{k}}
\\
& \lesssim 2^{ k_{2} (1-s)} \Vert f_{k_{2}} \Vert_{Z_{k_{2}}} 2^{k_{1} \frac{n}{2}} \Vert f_{k_{1}} \Vert_{Z_{k_{1}}}.
\end{split}
\end{equation*}
It is unclear whether the factor  $ 2^{k_{2} (1-s)}$  can be removed or not in this specific case. It is worth mentioning that all of the other cases in the proof of \cite[Lemma 7.1]{cmp} can be adapted to when $ s \in (1/2, 1 )$ without the factor $ 2^{k_{2}(1-s)}$. The only exception is the above case. 

\bibliographystyle{abbrv}
\bibliography{bib}

\end{document}